\documentclass{article}

\usepackage{microtype}
\usepackage{graphicx}
\usepackage{subfigure}
\usepackage{booktabs} 

\usepackage{hyperref}

\usepackage[accepted]{icml2024}

\usepackage{amsmath}
\usepackage{amssymb}
\usepackage{mathtools}
\usepackage{amsthm}

\usepackage[capitalize,noabbrev]{cleveref}

\theoremstyle{plain}
\newtheorem{theorem}{Theorem}[section]

\newtheorem{lemma}[theorem]{Lemma}

\theoremstyle{definition}

\newtheorem{assumption}[theorem]{Assumption}
\newtheorem{remark}[theorem]{Remark}

\definecolor{mycolor}{RGB}{0,150,70}
\newcommand{\revision}[1]{{\color{black}#1}}

\icmltitlerunning{High-Probability Convergence for Composite and Distributed Optimization with Heavy-Tailed Noise}

\usepackage{url}            
\usepackage{booktabs}       
\usepackage{amsfonts}       
\usepackage{nicefrac}       
\usepackage{microtype}      

\usepackage{amssymb, amsmath, amsthm, latexsym}
\usepackage{url}
\usepackage{algorithm}
\usepackage{algorithmic}
\usepackage{tabularx}
\usepackage{paralist}
\usepackage{mathtools}

\usepackage{bbm} 
\usepackage{wrapfig}
\usepackage{makecell}
\usepackage{multirow}
\usepackage{booktabs}

\usepackage{nicefrac}       

\usepackage[flushleft]{threeparttable} 

\usepackage{caption}
\usepackage{multirow}
\usepackage{colortbl}
\definecolor{bgcolor}{rgb}{0.8,1,1}
\definecolor{bgcolor2}{rgb}{0.8,1,0.8}
\definecolor{niceblue}{rgb}{0.0,0.19,0.56}

\usepackage{hyperref}
\hypersetup{colorlinks,linkcolor={blue},citecolor={niceblue},urlcolor={blue}}

\renewcommand{\algorithmicrequire}{\textbf{Input:}}
\renewcommand{\algorithmicensure}{\textbf{Output:}}

\usepackage{pifont}
\definecolor{PineGreen}{RGB}{0,110,51}
\definecolor{BrickRed}{RGB}{143,20,2}
\newcommand{\cmark}{{\color{PineGreen}\ding{51}}}%
\newcommand{\xmark}{{\color{BrickRed}\ding{55}}}%

\usepackage{tikz-cd} 

\usepackage{subfigure}

\newcommand{\R}{\mathbb{R}}
\newcommand{\eqdef}{\stackrel{\text{def}}{=}}

\def\<#1,#2>{\left\langle #1,#2\right\rangle}

\newcolumntype{Y}{>{\centering\arraybackslash}X}

\usepackage{xspace}

\newcommand{\algname}[1]{{\sf  #1}\xspace}

\newcommand{\circledOne}{\text{\ding{172}}}
\newcommand{\circledTwo}{\text{\ding{173}}}
\newcommand{\circledThree}{\text{\ding{174}}}
\newcommand{\circledFour}{\text{\ding{175}}}
\newcommand{\circledFive}{\text{\ding{176}}}
\newcommand{\circledSix}{\text{\ding{177}}}
\newcommand{\circledSeven}{\text{\ding{178}}}
\newcommand{\circledEight}{\text{\ding{179}}}
\newcommand{\circledNine}{\text{\ding{180}}}
\newcommand{\circledTen}{\text{\ding{181}}}

\usepackage[colorinlistoftodos,bordercolor=orange,backgroundcolor=orange!20,linecolor=orange,textsize=scriptsize]{todonotes}

\newcommand{\cC}{{\cal C}}
\newcommand{\cD}{{\cal D}}

\newcommand{\cO}{{\cal O}}

\newcommand{\EE}{\mathbb{E}}

\newcommand{\prox}{\mathop{\mathrm{prox}}\nolimits}

\newcommand{\proxkPsi}{\prox_{\gamma \Psi}}

\newcommand{\PP}{\mathbb{P}}

\newcommand{\tx}{\widetilde{x}}
\newcommand{\tX}{\widetilde{X}}
\newcommand{\tg}{\widetilde{g}}

\newcommand{\hg}{\widehat{g}}

\newcommand{\thh}{\widetilde{h}}
\newcommand{\hh}{\widehat{h}}

\def\clip{\texttt{clip}}

\def\clip{\texttt{clip}}
\def\avg{\texttt{avg}}
\def\gap{\texttt{Gap}}

\usepackage{hyperref}
\graphicspath{{plots/}}

\usepackage{makecell}

\usepackage{accents}
\newlength{\dhatheight}

\usepackage{pgfplotstable} 
\usetikzlibrary{automata, positioning, arrows, shapes, fit, calc, intersections}
\usepgfplotslibrary{statistics}
\include{figure}

\def\la{\langle}
\def\ra{\rangle}

\usepackage{natbib}

\def\pd#1{{\color{black}#1}}

\begin{document}

\twocolumn[
\icmltitle{High-Probability Convergence for Composite and Distributed Stochastic Minimization and Variational Inequalities with Heavy-Tailed Noise}



\icmlsetsymbol{equal}{*}

\begin{icmlauthorlist}
\icmlauthor{Eduard Gorbunov}{mbzuai}
\icmlauthor{Abdurakhmon Sadiev}{kaust}
\icmlauthor{Marina Danilova}{mipt}
\icmlauthor{Samuel Horv\'ath}{mbzuai}
\icmlauthor{Gauthier Gidel}{udem,cifar}
\icmlauthor{Pavel Dvurechensky}{wias}
\icmlauthor{Alexander Gasnikov}{innopolis,isp,mipt,skoltech}
\icmlauthor{Peter Richt\'arik}{kaust}
\end{icmlauthorlist}

\icmlaffiliation{kaust}{King Abdullah University of Science and Technology, KSA}
\icmlaffiliation{mipt}{Moscow Institute of Physics and Technology, Russia}
\icmlaffiliation{mbzuai}{Mohamed bin Zayed University of Artificial Intelligence, UAE}
\icmlaffiliation{udem}{Universit\'e de Montr\'eal and Mila, Canada}
\icmlaffiliation{cifar}{Canada CIFAR AI Chair}
\icmlaffiliation{wias}{Weierstrass Institute for Applied Analysis and Stochastics, Germany}
\icmlaffiliation{innopolis}{University Innopolis, Russia}
\icmlaffiliation{isp}{Ivannikov Institute for System Programming RAS, Russia}
\icmlaffiliation{skoltech}{Skolkovo Institute of Science and Technology, Russia}

\icmlcorrespondingauthor{Eduard Gorbunov}{eduard.gorbunov@mbzuai.ac.ae}

\icmlkeywords{Stochastic Optimization, High-Probability Convergence, Heavy-Tailed Noise, Gradient Clipping, Distributed Problems, Variational Inequalities, Machine Learning, ICML}

\vskip 0.3in
]



\printAffiliationsAndNotice{}  

\begin{abstract}
High-probability analysis of stochastic first-order optimization methods under mild assumptions on the noise has been gaining a lot of attention in recent years. Typically, gradient clipping is one of the key algorithmic ingredients to derive good high-probability guarantees when the noise is heavy-tailed. However, if implemented na\"ively, clipping can spoil the convergence of the popular methods for composite and distributed optimization (\algname{Prox-SGD}/\algname{Parallel SGD}) even in the absence of any noise. Due to this reason, many works on high-probability analysis consider only unconstrained non-distributed problems, and the existing results for composite/distributed problems do not include some important special cases (like strongly convex problems) and are not optimal. To address this issue, we propose new stochastic methods for composite and distributed optimization based on the clipping of stochastic gradient differences and prove tight high-probability convergence results (including nearly optimal ones) for the new methods. In addition, we also develop new methods for composite and distributed variational inequalities and analyze the high-probability convergence of these methods.
\end{abstract}

\section{Introduction}
Many recent works on stochastic optimization have the ultimate goal of bridging the theory and practice in machine learning. This is mostly reflected in the attempts at the theoretical analysis of optimization methods under weaker assumptions than the standard ones. Moreover, some phenomena cannot be explained using classical in-expectation convergence analysis (see the motivating example from \citep{gorbunov2020stochastic}) that results in the growing interest in more accurate ways to the analysis of stochastic methods, for example, \emph{high-probability convergence analysis}.

However, despite the significant attention to this topic \citep{nazin2019algorithms, davis2021low, gorbunov2020stochastic, gorbunov2022clipped, cutkosky2021high, sadiev2023high, nguyen2023high, liu2023stochastic, liu2023high}, several important directions remain unexplored. In particular, all mentioned works either consider unconstrained problems or consider general composite/constrained minimization/variational inequality problems but have some noticeable limitations, such as bounded domain assumption, extra logarithmic factors \pd{in the complexity bounds}, not optimal (not accelerated) \pd{convergence} rates, or no analysis of (quasi-) strongly convex (monotone) case. The importance of composite/constrained formulations for the machine learning community can be justified in many ways. For example, composite optimization and distributed optimization have a lot of similarities, i.e., one can view a distributed optimization problem as a special composite optimization problem \citep{parikh2014proximal}. Due to the large sizes of modern machine learning models and datasets, many important problems can be solved in a reasonable time only via distributed methods. Moreover, some problems have a distributed nature, e.g., data determining the problem can be (privately) stored on multiple devices connected with a central server, which is a classical setup for Federated Learning \citep{konevcny2016federated, kairouz2021advances}. Next, composite formulations are very useful for handling different regularizations popular in machine learning and statistics \citep{zou2005regularization, shalev2014understanding, beck2017first}. Finally, variational inequalities are usually considered with constraints as well.

The discrepancy between the importance of composite/constrained and distributed formulations and the lack of high-probability convergence results in this setup can be partially explained as follows. SOTA high-probability convergence results are derived for the algorithms that use \emph{gradient clipping} \citep{pascanu2013difficulty}, i.e., the clipping operator defined as $\clip(x,\lambda) = \min\{1, \nicefrac{\lambda}{\|x\|}\}x$ for $x\neq 0$ and $\clip(0,\lambda) = 0$ with some clipping level $\lambda > 0$ is applied to the stochastic gradients. If $\lambda$ is too small, then na\"ive Proximal Gradient Descent with gradient clipping is not a fixed point method, i.e., the method escapes the solution even if it is initialized there (see a technical explanation in Section~\ref{sec:the_issue_and_star_method}). This fact implies that one has either to increase the clipping level or to decrease the stepsize to converge to the exact solution asymptotically; the latter approach leads to a slower convergence rate. On the other hand, even in the unconstrained case, the existing results with acceleration/linear convergence are derived for the methods using decreasing clipping level \citep{gorbunov2020stochastic,sadiev2023high}. Therefore, new algorithms and analyses are required to handle this issue.

In this work, we close this gap by proposing new stochastic methods for composite and distributed problems via the clipping of \emph{gradient differences} that converge to zero with high probability. This allows us to achieve the desirable acceleration and linear convergence. Before we move on to the presentation of the main contributions, we need to introduce the problem settings formally.

\subsection{Setup}

\paragraph{Notation.} The standard Euclidean norm of vector $x \in \R^d$ is denoted as $\|x\| = \sqrt{\langle x,x \rangle}$. $B_R(x) = \{y\in \R^d\mid \|y-x\| \leq R\}$ is the ball centered at $x$ with radius $R$. Bregman divergence w.r.t.\ function $f$ is denoted as $D_f(x,y) \eqdef f(x) - f(y) - \langle \nabla f(y), x - y \rangle$. In $\cO(\cdot)$, we omit the numerical factors, and in $\widetilde\cO(\cdot)$, we omit numerical and logarithmic factors. For natural $n\geq 1$ the set $\{1,2,\ldots,n\}$ is denoted as $[n]$. We use $\EE_{\xi}[\cdot]$ to denote the expectation w.r.t.\ the randomness coming from $\xi$.

\paragraph{Considered problems.} In this work, we focus on stochastic composite minimization problems:
\begin{equation}
    \min\limits_{x\in\R^d} \left\{\Phi(x) = f(x) + \Psi (x)\right\}, \label{eq:min_problem}
\end{equation}
where $f(x) = \EE_{\xi\sim\cD}[f_{\xi}(x)]$ is a differentiable function satisfying some properties to be defined later and $\Psi(x)$ is a proper, closed, convex function (composite/regularization term). The examples of problem \eqref{eq:min_problem} arise in various applications, e.g., machine learning \citep{shalev2014understanding}, signal processing \citep{combettes2011proximal}, image processing \citep{luke2020proximal}. \revision{We also consider variational inequality problems, see Appendix~\ref{appendix:VIs}.}

The distributed version of \eqref{eq:min_problem} has the following structure:
\begin{equation}
    f(x) = \frac{1}{n}\sum\limits_{i=1}^n \left\{f_i(x) = \EE_{\xi_i\sim \cD_i}[f_{\xi_i}(x)]\right\}. \label{eq:distributed_problem}
\end{equation}
In this case, there are $n$ workers connected in a centralized way with some parameter server; worker $i$ can query some noisy information (stochastic gradients/estimates) about $f_i$.

\paragraph{In-expectation and high-probability convergence.} In-expectation convergence guarantees provide the upper bounds on the number of iterations/oracle calls $\hat K = \hat K(\varepsilon)$ for a method needed to find point $x^{\hat K}$ such that $\EE[\cC(x^{\hat K})] \leq \varepsilon$ for given convergence criterion $\cC(x)$ (e.g., $\cC(x)$ can be $f(x) - f(x^*)$, $\|x - x^*\|^2$, $\|\nabla f(x)\|^2$) and given accuracy $\varepsilon > 0$. High-probability convergence guarantees give the upper bounds on the  number of iterations/oracle calls $K = K(\varepsilon, \beta)$ for a method needed to find point $x$ such that $\PP\{\cC(x^K) \leq \varepsilon\} \geq 1 - \beta$, where $\beta \in (0,1)$ is a confidence level. It is worth noting that Markov's inequality implies $\PP\{\cC(x^K) > \varepsilon\} < \nicefrac{\EE[\cC(x^K)]}{\varepsilon}$, meaning that it is sufficient to take $K = \hat K(\beta\varepsilon) = \hat K$: $\PP\{\cC(x^{\hat K}) > \varepsilon\} < \nicefrac{\EE[\cC(x^{\hat K})]}{\varepsilon} \leq \beta$. However, this typically leads to the polynomial dependence on $\nicefrac{1}{\beta}$ that significantly spoils the complexity of the method when $\beta$ is small. Therefore, we focus on the high-probability convergence guarantees that depend on $\nicefrac{1}{\beta}$ poly-logarithmically. Moreover, such high-probability results are more sensitive to the noise distribution (and, thus, more accurate) than in-expectation ones \citep{gorbunov2020stochastic, sadiev2023high}.

\paragraph{Proximal operator.} We assume that function $\Psi(x)$ has a relatively simple structure such that one can efficiently compute \emph{proximal operator}: $\prox_{\gamma\Psi}(x) = \arg\min_{y\in \R^d}\{\gamma\Psi(y) + \tfrac{1}{2}\|y - x\|^2\}$. For the properties of the proximal operator and examples of functions $\Psi(x)$ such that $\prox_{\gamma\Psi}(x)$ can be easily computed, we refer the reader to \citep{beck2017first}.

\paragraph{Bounded central $\alpha$-th moment.} We consider the situation when $f_i$ and $F_i$ are accessible through the stochastic oracle calls. The stochastic estimates satisfy the following assumption.\footnote{Following \citep{sadiev2023high}, we consider all assumptions only on some bounded set $Q \subseteq \R^d$; the diameter of $Q$ depends on the starting point. \revision{We emphasize that we do not assume boundedness of the domain of the original problem. Instead, we prove via induction that the iterates of the considered methods stay in some ball around the solution with high probability (see the details in Section~\ref{sec:proofs_structure}). Thus, it is sufficient for us to assume everything just on this ball, though our analysis remains unchanged if we introduce all assumptions on the whole domain.}}
\begin{assumption}\label{as:bounded_alpha_moment}
    There exist some set $Q \subseteq \R^d$ and values $\sigma \geq 0$, $\alpha \in (1,2]$ such that for all $x \in Q$ we have $\EE_{\xi_i\sim \cD_i}[\nabla f_{\xi_i}(x)] = \nabla f_i(x)$ and
    \begin{equation}
        \EE_{\xi_i\sim \cD_i}[\left\|\nabla f_{\xi_i}(x) - \nabla f_i(x)\right\|^\alpha] \leq \sigma^\alpha. \label{eq:bounded_alpha_moment_gradient}
    \end{equation}
\end{assumption}
For $\alpha = 2$, Assumption~\ref{as:bounded_alpha_moment} reduces to the bounded variance assumption, and for $\alpha \in (1,2)$ variance of the stochastic estimator can be unbounded, e.g., the noise can have L\'evy $\alpha$-stable distribution \citep{zhang2020adaptive}, which is heavy-tailed.

\paragraph{Assumptions on $f_i$.} We assume that functions $\{f_i\}_{i\in [n]}$ are $L$-smooth.
\begin{assumption}\label{as:L_smoothness}
    We assume that there exist some set $Q \subseteq \R^d$ and constant $L > 0$ such that for all $x, y \in Q, i\in [n]$ and \revision{for all $x^* \in \arg\min_{x\in \R^d} \Phi(x)$}
    \begin{eqnarray}
        \|\nabla f_i(x) - \nabla f_i(y)\| &\leq& L\|x - y\|, \label{eq:L_smoothness}\\
        \|\nabla f_i(x) - \nabla f_i(x^*)\|^2 &\leq& 2LD_{f_i}(x, x^*). \label{eq:L_smoothness_cor_2}
    \end{eqnarray}
\end{assumption}
As noted in \citep[Appendix B]{sadiev2023high}, \eqref{eq:L_smoothness_cor_2} is satisfied on the set $Q \neq \R^d$ if \eqref{eq:L_smoothness} holds on a slightly larger set in the case of $\Psi \equiv 0$, $n = 1$ (unconstrained single-node case). For simplicity, we assume that both \eqref{eq:L_smoothness} and \eqref{eq:L_smoothness_cor_2} hold on $Q$. This is always the case for $L$-smooth functions on $Q = \R^d$ when $\Psi \equiv 0$, $n = 1$. In a more general situation, condition \eqref{eq:L_smoothness_cor_2} can be viewed as an assumption on the structured non-convexity of $\{f_i\}_{i\in [n]}$. \revision{Finally, if $\{f_i\}_{i\in [n]}$ are convex and $L$-smooth on the whole domain of the problem \eqref{eq:min_problem}, then Assumption~\ref{as:L_smoothness} holds.}

Next, for each particular result about the convergence of methods for \eqref{eq:min_problem}, we make one of the following assumptions.
\begin{assumption}\label{as:str_cvx}
    There exist some set $Q \subseteq \R^d$ and constant $\mu \geq 0$ such that $f$ is $\mu$-strongly convex: $\forall \, x, y \in Q$
    \begin{equation}
        f(y) \geq f(x) + \langle \nabla f(x), y - x \rangle + \frac{\mu}{2}\|y - x\|^2. \label{eq:str_cvx}
    \end{equation}
    When $\mu = 0$, function $f$ is called convex on $Q$.
\end{assumption}
This is a standard assumption for optimization literature \citep{nesterov2018lectures}. We also consider a relaxation of strong convexity.
\begin{assumption}\label{as:QSC}
    There exist some set $Q \subseteq \R^d$ and constant $\mu \geq 0$ such that $f_1,\ldots,f_n$ are $(\mu, x^*)$-quasi-strongly convex \revision{for all $x^* \in \arg\min_{x\in \R^d} \Phi(x)$}: $\forall \, x \in Q,\, i\in[n]$
    \begin{equation}
        f_i(x^*) \geq f_i(x) + \langle \nabla f_i(x), x^* - x \rangle + \frac{\mu}{2}\|x - x^*\|^2. \label{eq:QSC}
    \end{equation}
\end{assumption}
Condition \eqref{eq:QSC} is weaker than \eqref{eq:str_cvx} and holds even for some non-convex functions \citep{necoara2019linear}.

\subsection{Our Contributions}

\begin{table*}[t]
    \centering
    \scriptsize
    \caption{\scriptsize Summary of known and new high-probability complexity results for solving (non-) composite (non-) distributed smooth optimization problem \eqref{eq:min_problem}. Column ``Setup'' indicates the assumptions made in addition to Assumptions~\ref{as:bounded_alpha_moment} and \ref{as:L_smoothness}. All assumptions are made only on some ball around the solution with radius $\sim R \geq \|x^0 - x^*\|$. Complexity is the number of stochastic oracle calls (per worker) needed for a method to guarantee that $\PP\{\text{Metric} \leq \varepsilon\} \geq 1 - \beta$ for some $\varepsilon> 0$, $\beta \in (0,1]$ and ``Metric'' is taken from the corresponding column. Numerical and logarithmic factors are omitted for simplicity. Column ``C?'' shows whether the problem \eqref{eq:min_problem} is composite, ``D?'' indicates whether the problem \eqref{eq:min_problem} is distributed. Notation: $L$ = Lipschitz constant; $\sigma$ = parameter from Assumption~\ref{as:bounded_alpha_moment}; $R$ = any upper bound on $\|x^0 - x^*\|$; \revision{$\zeta_*$ = any upper bound on $\sqrt{\frac{1}{n}\sum_{i=1}^n \|\nabla f_i(x^*)\|^2}$; $\widehat{R}^2 = R\left(3R + L^{-1}(2\eta \sigma + \|\nabla f(x^0)\|)\right)$ for some $\eta > 0$ (for the result from \citep{nguyen2023improved}; one can show that $\widehat{R}^2 = \Theta(R^2 + \nicefrac{R\zeta_*}{L})$ when $n=1$, see the discussion after Theorem~\ref{thm:D_prox_clipped_SGD_main_cvx});} 
    $\mu$ = (quasi-)strong convexity parameter. The results of this paper are highlighted in blue.
    }
    \label{tab:comparison_of_rates_minimization}
    \begin{threeparttable}
        \begin{tabular}{|c|c c c c c|}
        \hline
        Setup & Method & Metric & Complexity & C? & D?\\
        \hline\hline
        \multirow{12}{*}{\makecell{As.~\ref{as:str_cvx}\\ ($\mu = 0$)}} &\begin{tabular}{c}
             \algname{clipped-SGD} \\ \citep{sadiev2023high}
        \end{tabular}  & $f(\overline{x}^K) - f(x^*)$ & $\max\left\{\frac{LR^2}{\varepsilon}, \left(\frac{\sigma R}{\varepsilon}\right)^{\frac{\alpha}{\alpha-1}}\right\}$ & \xmark & \xmark \\
        & \begin{tabular}{c}
             \algname{clipped-SSTM} \\ \citep{sadiev2023high}
        \end{tabular} & $f(y^K) - f(x^*)$ & $\max\left\{\sqrt{\frac{LR^2}{\varepsilon}}, \left(\frac{\sigma R}{\varepsilon}\right)^{\frac{\alpha}{\alpha-1}}\right\}$ &\xmark &\xmark\\
        & \begin{tabular}{c}
             \algname{Clipped-SMD}\tnote{\color{blue}(1),(2)} \\ \citep{nguyen2023improved}
        \end{tabular} &$\Phi(\overline{x}^K) - \Phi(x^*)$ &$\max\left\{\frac{L \revision{\widehat R^2}}{\varepsilon}, \left(\frac{\sigma R}{\varepsilon}\right)^{\frac{\alpha}{\alpha-1}}\right\}$ &\cmark &\xmark \\
        &\begin{tabular}{c}
             \algname{Clipped-ASMD}\tnote{\color{blue}(1)} \\ \citep{nguyen2023improved}
        \end{tabular} &$\Phi(y^K) - \Phi(x^*)$ &$\max\left\{\sqrt{\frac{L R^2}{\varepsilon}}, \left(\frac{\sigma R}{\varepsilon}\right)^{\frac{\alpha}{\alpha-1}}\right\}$ &\cmark\xmark \tnote{\color{blue}(3)}  &\xmark \\
        & \cellcolor{bgcolor} \begin{tabular}{c}
        
             \algname{DProx-clipped-SGD-shift} \\ Theorem~\ref{thm:D_prox_clipped_SGD_main_cvx}
        \end{tabular}  & \cellcolor{bgcolor}$\Phi(\overline{x}^K) - \Phi(x^*)$ & \cellcolor{bgcolor}$\max\left\{\frac{L\revision{R^2}}{\varepsilon}, \revision{\frac{R\zeta_*}{\sqrt{n}\varepsilon},} \frac{1}{n}\left(\frac{\sigma \revision{R}}{\varepsilon}\right)^{\frac{\alpha}{\alpha-1}}\right\}$ &\cellcolor{bgcolor} \cmark &\cellcolor{bgcolor} \cmark \\
        & \cellcolor{bgcolor} \begin{tabular}{c}

             \algname{DProx-clipped-SSTM-shift} \\ Theorem~\ref{thm:D_prox_clipped_SSSTM_main}
        \end{tabular} &\cellcolor{bgcolor} $\Phi(y^K) - \Phi(x^*)$ &\cellcolor{bgcolor} $\max\left\{\sqrt{\frac{L\revision{R^2}}{\varepsilon}}, \revision{\sqrt{\frac{R\zeta_*}{\sqrt{n}\varepsilon}},} \frac{1}{n}\left(\frac{\sigma \revision{R}}{\varepsilon}\right)^{\frac{\alpha}{\alpha-1}}\right\}$ &\cellcolor{bgcolor} \cmark &\cellcolor{bgcolor} \cmark \\
        \hline
        \multirow{3}{*}{\makecell{As.~\ref{as:QSC}\\ ($\mu > 0$)}} &  \begin{tabular}{c}
             \algname{clipped-SGD} \\ \citep{sadiev2023high}
        \end{tabular} & $\|x^K - x^*\|^2$ & $\max\left\{\frac{L}{\mu}, \left(\frac{\sigma^2}{\mu^2\varepsilon}\right)^{\frac{\alpha}{2(\alpha-1)}}\right\}$ &\xmark
        &\xmark \\
        & \cellcolor{bgcolor} \begin{tabular}{c}
             
             \algname{DProx-clipped-SGD-shift} \\ Theorem~\ref{thm:D_prox_clipped_SGD_main}
        \end{tabular} &\cellcolor{bgcolor} $\|x^K - x^*\|^2$ &\cellcolor{bgcolor} $\max\left\{\frac{L}{\mu}, \frac{1}{n}\left(\frac{\sigma^2}{\mu^2\varepsilon}\right)^{\frac{\alpha}{2(\alpha-1)}}\right\}$ &\cellcolor{bgcolor}\cmark &\cellcolor{bgcolor}\cmark \\
        \hline
    \end{tabular}
    \begin{tablenotes}
        {\scriptsize 
        \item [{\color{blue}(1)}] All assumptions are made on the whole domain.
        \item [{\color{blue}(2)}] \revision{The authors additionally assume that for a chosen point $\widehat x$ from the domain and for $\eta > 0$ one can compute an estimate $\widehat g$ such that $\mathbb{P}\lbrace \| \widehat g - \nabla f(\hat x) \| > \eta \sigma \rbrace \leq \epsilon$. Such an estimate can be found using geometric median of $\cO(\ln \epsilon^{-1})$ samples \citep{minsker2015geometric}.}
        \item [{\color{blue}(3)}] The authors assume that $\nabla f(x^*) = 0$, which is not true for general composite optimization.
        }
    \end{tablenotes}
    \end{threeparttable}
\end{table*}

$\bullet$ \textbf{Methods with clipping of gradient differences for distributed composite minimization.} We develop two stochastic methods for composite minimization problems -- Proximal Clipped \algname{SGD} with shifts (\algname{Prox-clipped-SGD-shift}) and Proximal Clipped Similar Triangles Method with shifts (\algname{Prox-clipped-SSTM-shift}). Instead of clipping stochastic gradients, these methods clip the difference between the stochastic gradients and the shifts that are updated on the fly. This trick allows us to use decreasing clipping levels, and, as a result, we derive the first accelerated high-probability convergence rates and tight high-probability convergence rates for the non-accelerated method in the quasi-strongly convex case. We also generalize the proposed methods to the distributed case (\algname{DProx-clipped-SGD-shift} and \algname{DProx-clipped-SSTM-shift}) and prove that they benefit from parallelization. To the best of our knowledge, our results are the first showing linear speed-up under Assumption~\ref{as:bounded_alpha_moment}.

$\bullet$ \textbf{Methods with clipping of gradient differences for distributed composite VIPs.} We also apply the proposed trick to the methods for variational inequalities. In particular, we propose \algname{DProx-clipped-SGDA-shifts} and \algname{DProx-clipped-SEG-shifts} and rigorously analyze their high-probability convergence. As in the minimization case, the proposed methods have provable benefits from parallelization.

$\bullet$ \textbf{Tight convergence rates.} As a separate contribution, we highlight the tightness of our analysis: in the known special cases ($\Psi \equiv 0$ and/or $n=1$), the derived complexity bounds either recover or outperform previously known ones (see Table~\ref{tab:comparison_of_rates_minimization} and also Table~\ref{tab:comparison_of_rates_VIP} in the appendix). Moreover, in certain regimes, the results have optimal (up to logarithms) dependencies on $\varepsilon$. This is achieved under quite general assumptions.

\subsection{Closely Related Work}

We discuss closely related work here and defer additional discussion to Appendix~\ref{appendix:extra_related}.

\paragraph{High-probability bounds for unconstrained convex problems.} Standard high-probability convergence results are obtained under the so-called light-tails assumption (sub-Gaussian noise) \citep{nemirovski2009robust, juditsky2011solving, ghadimi2012optimal}. The first work addressing this limitation is \citep{nazin2019algorithms}, where the authors derive the first high-probability complexity bounds for the case of minimization on a bounded set under bounded variance assumption. In the unconstrained case, these results are extended and accelerated by \citet{gorbunov2020stochastic} for smooth convex and strongly convex minimization problems. \citet{gorbunov2021near} tightens them and generalizes to the case of problems with H\"older-continuous gradients and \citet{gorbunov2022clipped} derives high-probability convergence rates in the case of VIPs. \citet{sadiev2023high} relaxes the assumption of bounded variance to Assumption~\ref{as:bounded_alpha_moment} for all problem classes mentioned above, and the results under the same assumption are also derived for \algname{clipped-SGD} (without acceleration) by \citet{nguyen2023high} in the convex and non-convex cases.

\paragraph{High-probability bounds for composite convex problems.} \citet{nazin2019algorithms} propose a truncated version of Mirror Descent for convex and strongly convex composite problems and prove non-accelerated rates of convergence under bounded variance and \emph{bounded domain} assumptions. Accelerated results under bounded variance assumption for strongly convex composite problems are proven by \citet{davis2021low}, who propose an approach based on robust distance estimation. Since this approach requires solving some auxiliary problem at each iteration of the method, the complexity bound from \citet{davis2021low} contains extra logarithmic factors independent of the confidence level. Finally, in their very recent work, \citet{nguyen2023improved} prove high-probability convergence for Clipped Stochastic Mirror Descent (\algname{Clipped-SMD}) for \emph{convex} composite problems. Moreover, the authors also propose Accelerated \algname{Clipped-SMD} (\algname{Clipped-ASMD}) and show that the algorithm is indeed accelerated \emph{but only under the additional assumption that $\nabla f(x^*) = 0$}.


\section{Main Results}\label{sec:the_issue_and_star_method}

In this section, we consider problem \eqref{eq:min_problem} and methods for it.


\paragraph{Failure of the na\"ive approach.}For simplicity, consider a non-stochastic case with strongly convex $f(x)$, $n = 1$. The standard deterministic first-order method for solving problems like \eqref{eq:min_problem} is Proximal Gradient Descent (\algname{Prox-GD}) \citep{combettes2011proximal,nesterov2013gradient}: $x^{k+1} = \prox_{\gamma\Psi}(x^k - \gamma \nabla f(x^k))$. 
Due to the good interplay between the structure of the problem, properties of the proximal operator, and the structure of the method, \algname{Prox-GD} has the same (linear) convergence rate as \algname{GD} for minimization of $f(x)$. One of the key reasons for that is that any solution $x^*$ of problem \eqref{eq:min_problem} satisfies $x^* = \prox_{\gamma \Psi}(x^* - \gamma \nabla f(x^*))$, 
i.e., the solutions of \eqref{eq:min_problem} are fixed points of \algname{Prox-GD} (and vice versa), which is equivalent to $-\nabla f(x^*) \in \partial \Psi(x^*)$, where $\partial \Psi(x^*)$ is a subdifferential of $\Psi$ at $x^*$. However, if we apply gradient clipping to \algname{Prox-GD} na\"ively
\begin{equation}
    x^{k+1} = \prox_{\gamma\Psi}\left(x^k - \gamma \clip(\nabla f(x^k), \lambda)\right), \label{eq:prox_clip_gd_naive}
\end{equation}
then the method loses a fixed point property if $\|\nabla f(x^*)\| > \lambda$, because in this case, $-\clip(\nabla f(x^*), \lambda)$ does not necessarily belongs to $\partial \Psi(x^*)$ and $x^* \neq \prox_{\gamma \Psi}(x^* - \gamma \clip(\nabla f(x^*), \lambda))$ in general. Therefore, for such $\lambda$, one has to decrease the stepsize $\gamma$ to achieve any accuracy of the solution. This approach slows down the convergence making it sublinear even without any stochasticity in the gradients. To avoid this issue, it is necessary to set $\lambda$ large enough. This strategy works in the deterministic case but becomes problematic for a stochastic version of \eqref{eq:prox_clip_gd_naive}:
\begin{equation}
    x^{k+1} = \prox_{\gamma\Psi}\left(x^k - \gamma \clip(\nabla f_{\xi^k}(x^k), \lambda_k)\right), \label{eq:prox_clip_sgd_naive}
\end{equation}
where $\xi^k$ is sampled independently from previous iterations. The problem comes from the fact the existing analysis in the unconstrained case (which is a special case of the composite case) requires taking decreasing $\lambda_k$ \citep{gorbunov2021near, sadiev2023high} that contradicts the requirement that clipping level has to be large enough. Therefore, more fundamental algorithmic changes are needed.

\paragraph{Non-implementable solution.} Let us reformulate the issue: (i) to handle the heavy-tailed noise, we want to use decreasing clipping level $\lambda_k$, (ii) but the method should also converge linearly without the noise, i.e., when $\nabla f_{\xi^k}(x^k) = \EE_{\xi^k}[\nabla f_{\xi^k}(x^k)] = \nabla f(x^k)$. In other words, \emph{the expectation of the vector that is clipped in the method should converge to zero with the same rate as $\lambda_k$}. The method should converge, i.e., with high probability, we should have $\nabla f(x^k) \to \nabla f(x^*)$. These observations lead us to the following purely theoretical algorithm that we call \algname{Prox-clipped-SGD-star}\footnote{The idea behind and the name of this method is inspired by \algname{SGD-star} proposed by \citet{gorbunov2020unified, hanzely2019one}.}:
\begin{gather}
    x^{k+1} = \proxkPsi \left(x^k - \gamma \tg^k\right),\;\; \text{where} \label{eq:prox_clipped_SGD_star}\\
    \tg^k = \nabla f(x^*) + \clip\left(\nabla f_{\xi^k}(x^k) - \nabla f(x^*), \lambda_k\right). 
\end{gather}
The method is non-implementable since $\nabla f(x^*)$ is unknown in advance. Nevertheless, as we explain in the next subsection, the method is useful in designing and analyzing implementable versions. The following theorem gives the complexity of \algname{Prox-clipped-SGD-star}.

\begin{theorem}\label{thm:prox_clipped_SGD_main}
    Let $n=1$ and Assumptions~\ref{as:bounded_alpha_moment}, \ref{as:L_smoothness}, and \ref{as:QSC} with $\mu > 0$ hold for $Q = B_{2R}(x^*)$, $R \geq \|x^0 - x^*\|$\revision{, for some\footnote{\revision{If all of our results, one can use any solution $x^*$, e.g., one can take $x^*$ being a projection of $x^*$ on the solution set.}} $x^* \in \arg\min_{x\in \R^d} \Phi(x)$.} Assume that $K \geq 1$, $\beta \in (0,1)$, $A = \ln \tfrac{4(K+1)}{\beta}$, 
    \revision{\begin{gather*}
        0 < \gamma = \cO\left(\min\left\{\frac{1}{LA}, \frac{\ln(B_K)}{\mu(K+1)}\right\}\right),\\
        B_K = \Theta\left(\max\left\{2, \frac{(K+1)^{\nicefrac{2(\alpha-1)}{\alpha}}\mu^2 R^2}{\sigma^2 A^{\nicefrac{2(\alpha-1)}{\alpha}}\ln^2(B_K)}\right\}\right),\\
        \lambda_k = \Theta\left(\frac{\exp(-\gamma\mu(1+\nicefrac{k}{2}))R}{\gamma A}\right).
    \end{gather*}}
    Then to guarantee $\|x^K - x^*\|^2 \leq \varepsilon$ with probability $\geq 1 - \beta$ \algname{Prox-clipped-SGD-star} requires
    \begin{equation}
        \widetilde\cO\left(\max\left\{\frac{L}{\mu}, \left(\frac{\sigma^2}{\mu^2\varepsilon}\right)^{\frac{\alpha}{2(\alpha-1)}}\right\}\right) \label{eq:SGD_star_main_result_QSC}
    \end{equation}
    iterations/oracle calls.
\end{theorem}
\begin{proof}[Sketch of the proof]
    Following \citet{gorbunov2020stochastic, sadiev2023high}, we prove by induction\footnote{We use the induction to apply Bernstein's inequality for the estimation of the sums appearing due to the stochasticity of the gradients. We refer to Section~\ref{sec:proofs_structure} for the details.} that $\|x^k - x^*\|^2 \leq 2\exp(-\gamma\mu k)R^2$ with high probability. This and $L$-smoothness imply that $\|\nabla f(x^k) - \nabla f(x^*)\| \sim \exp(-\nicefrac{\gamma\mu k}{2})$ and $\|\nabla f(x^k) - \nabla f(x^*)\| \leq \nicefrac{\lambda_k}{2}$ with high probability. These facts allow us to properly clip the heavy-tailed noise without sacrificing the convergence rate. See the complete formulation of Theorem~\ref{thm:prox_clipped_SGD_main} and the full proof in Appendix~\ref{appendix:prox_clipped_SGD_star}.
\end{proof}

The above complexity bound for \algname{Prox-clipped-SGD-star} coincides with the known one for \algname{clipped-SGD} for the unconstrained problems under the same assumptions \citep{sadiev2023high} -- similarly as the complexity of \algname{Prox-GD} coincides with the complexity of \algname{GD} for unconstrained smooth problems.

\begin{remark}[On the logarithmic factors.]
    In our results, we do not focus on the logarithmic factors to keep the proofs simpler. Therefore, in the main part of the paper, we omit logarithmic dependencies, but we provide them in the Appendix. We believe it is possible to improve the logarithmic factors in our complexity bounds following the proof technique from \citep{nguyen2023improved}.
\end{remark}

\begin{remark}[Dependence of the parameters on $R$.]
    Similarly to many previous works \citep{gorbunov2020stochastic, gorbunov2022clipped, sadiev2023high, nguyen2023high, nguyen2023improved}, in all our results, the choice of the stepsize and clipping level depends on some upper bound\footnote{One can avoid this via considering problems defined on bounded domain \citep{nazin2019algorithms} (then one can take $R$ equal to the diameter of the domain) or assuming that the gradient is bounded \citep{cutkosky2021high} (in this case, one can at least remove the dependence on $R$ from the clipping level). However, both options have clear limitations.} $R$ for $\|x^0 - x^*\|$. Deriving matching high-probability results for parameter-free methods is an important open problem going beyond the scope of this paper. In practice, one can run the method with $R = \widetilde R$ for a sufficiently large number of steps (prescribed by theory). If the target accuracy is not achieved, then one needs to take $R = 2\widetilde R$ and repeat the procedure. Then, with a high probability after the logarithmic number of restarts, either parameter $R$ will be chosen appropriately, or the method will reach the desired accuracy.
\end{remark}


\paragraph{Prox-clipped-SGD-shift.} As mentioned before, the key limitation of \algname{Prox-clipped-SGD-star} is that it explicitly uses shift $\nabla f(x^*)$, which is not known in advance. Therefore, guided by the literature on variance reduction and communication compression \citep{gorbunov2020unified, gower2020variance, mishchenko2019distributed}, it is natural to approximate $\nabla f(x^*)$ via shifts $h^k$. This leads us to a new method called \algname{Prox-clipped-SGD-shift}: as before $x^{k+1} = \proxkPsi \left(x^k - \gamma \tg^k\right)$ but now
\begin{gather}
    \tg^k = h^k + \hat\Delta^k,\;\; h^{k+1} = h^k + \nu \hat \Delta^k, \label{eq:prox_clipped_SGD_shift_shifts}\\
    \hat\Delta^k = \clip\left(\nabla f_{\xi^k}(x^k) - h^k, \lambda_k\right), 
\end{gather}
where $\nu > 0$ is a stepsize for learning shifts. Similar shifts are proposed by \citet{mishchenko2019distributed} in the context of distributed optimization with communication compression. Despite the similarities between compression operators and clipping, the latter does not have a constant contraction factor, making it non-trivial to extend the techniques from the literature on communication compression to the methods with clipping, e.g., see \citep[Section 5]{khirirat2023clip21}. Moreover, the proofs from \citep{mishchenko2019distributed} rely on the unbiasedness of the compression operator. In contrast, gradient clipping introduces bias, and we consider high-proabbility analysis that differs significantly from the in-expectation one considered by \citet{mishchenko2019distributed}.

Since \algname{Prox-clipped-SGD-shift} is a special case of its distributed variant, we continue our discussion with the distributed version of the method.

\paragraph{Distributed Prox-clipped-SGD-shift.} We propose a generalization of \algname{Prox-clipped-SGD-shift} to the distributed case \eqref{eq:distributed_problem} called Distributed \algname{Prox-clipped-SGD-shift} (\algname{DProx-clipped-SGD-shift}):
\begin{gather}
    x^{k+1} = \proxkPsi \left(x^k - \gamma \tg^k\right),\;\; \text{where}\label{eq:Dprox_clipped_SGD_shift}\\
    \tg^k = \frac{1}{n}\sum\limits_{i=1}^n \tg_i^k,\;\; \tg_i^k = h_i^k + \hat\Delta_i^k,\;\; h_i^{k+1} = h_i^k + \nu \hat \Delta_i^k, \notag\\
    \hat\Delta_i^k = \clip\left(\nabla f_{\xi_i^k}(x^k) - h_i^k, \lambda_k\right), \label{eq:Dprox_clipped_SGD_shift_shifts}
\end{gather}
where $\xi_1^k,\ldots, \xi_n^k$ are sampled independently from each other and previous steps. In this method, worker $i$ updates the shift $h_i^k$ and sends clipped vector $\hat \Delta_i^k$ to the server. Since $\tg^k = h^k + \tfrac{1}{n}\sum_{i=1}^n\hat \Delta_i^k$ and $h^{k+1} = h^k + \tfrac{\nu}{n}\sum_{i=1}^n\hat \Delta_i^k$, where $h^k = \frac{1}{n}\sum_{i=1}^n h_i^k$, workers do not need to send $h_i^k$ to the server for $k > 0$. \revision{We notice that even when $\Psi \equiv 0$, i.e., the problem is unconstrained, individual gradients $\{\nabla f_i(x^*)\}_{i\in [n]}$ of the clients' function at the solution of problem \eqref{eq:min_problem} are not necessary zero, though their sum equals to zero. However, if applied without any shifts to the local (stochastic) gradients, then, similarly to the case of non-distributed \algname{Prox-GD} \eqref{eq:prox_clip_gd_naive}, the clipping operation also breaks the fixed point property, since $\frac{1}{n}\sum_{i=1}^n \clip(\nabla f_i(x^*), \lambda) \neq 0$ for small values of $\lambda$. This highlights the importance of the shifts for distributed unconstrained case.}

For the proposed method, we derive the following result. 

\revision{\begin{theorem}[Convergence of \algname{DProx-clipped-SGD-shift}: quasi-strongly convex case]\label{thm:D_prox_clipped_SGD_main}
    Let $K \geq 1$, $\beta \in (0,1)$, $A = \ln \frac{48n(K+1)}{\beta}$. Let Assumptions~\ref{as:bounded_alpha_moment}, \ref{as:L_smoothness}, and \ref{as:QSC} with $\mu > 0$ hold for $Q = B_{3n\sqrt{2}R}(x^*)$, where $R \geq \|x^0 - x^*\|^2.$ 
    \revision{Assume that $\zeta_* \geq \sqrt{\frac{1}{n}\sum_{i=1}^n \|\nabla f_i(x^*)\|^2}$,} $\nu = \Theta(\nicefrac{1}{A})$,
    \begin{gather*}
        0 < \gamma = \cO\left(\min\left\{\frac{1}{LA}, \frac{\sqrt{n}R}{A\zeta_*}, \frac{\ln(B_K)}{\mu(K+1)}\right\}\right),\\
        B_K = \Theta\left(\max\left\{2, \frac{(K+1)^{\nicefrac{2(\alpha-1)}{\alpha}}\mu^2 n^{\nicefrac{2(\alpha-1)}{\alpha}} R^2}{\sigma^2 A^{\nicefrac{2(\alpha-1)}{\alpha}}\ln^2(B_K)}\right\}\right),\\
        \lambda_k = \Theta\left(\frac{n\exp(-\gamma\mu(1+\nicefrac{k}{2}))R}{\gamma A}\right).
    \end{gather*}
    Then to guarantee $\|x^K - x^*\|^2 \leq \varepsilon$ with probability $\geq 1 - \beta$ \algname{DProx-clipped-SGD-shift} requires
    \begin{align}
        \widetilde\cO\left(\max\left\{\frac{L}{\mu}, \frac{\zeta_*}{\sqrt{n}\mu R}, \frac{1}{n}\left(\frac{\sigma^2}{\mu^2\varepsilon}\right)^{\frac{\alpha}{2(\alpha-1)}}\right\}\right) \label{eq:d_prox_clipped_SGD_shift_main_result_QSC}
    \end{align}
    iterations/oracle calls per worker.
\end{theorem}}
\begin{proof}[Sketch of the proof]
    The proof follows similar steps to the proof of Theorem~\ref{thm:prox_clipped_SGD_main} up the change of the Lyapunov function: by induction, we prove that $V_k \leq 2\exp(-\gamma\mu k)V$ with high probability, where $V_k = \|x^k - x^*\|^2 + \frac{C^2\gamma^2 A^2}{n}\sum_{i=1}^n \|h_i^k - \nabla f_i(x^*)\|^2$. The choice of the Lyapunov function reflects the importance of the ``quality'' of shifts $\{h_i^k\}_{i\in[n]}$, i.e., their proximity to $\{\nabla f_i(x^*)\}_{i\in[n]}$. Moreover, we increase the clipping level $n$ times to balance the bias and variance of $\tg^k$; see Appendix~\ref{appendix:aux_results}. This allows us to reduce the last term in the complexity bound $n$ times. See the complete formulation of Theorem~\ref{thm:D_prox_clipped_SGD_main} and the full proof in Appendix~\ref{appendix:dprox_clipped_SGD_shift}.
\end{proof}

\revision{The next theorem gives the result in the convex case.

\begin{theorem}[Convergence of \algname{DProx-clipped-SGD-shift}: convex case]\label{thm:D_prox_clipped_SGD_main_cvx}
    Let $K \geq 1$, $\beta \in (0,1)$, $A = \ln \frac{48n(K+1)}{\beta}$. Let Assumptions~\ref{as:bounded_alpha_moment}, \ref{as:L_smoothness},  and \ref{as:str_cvx} with $\mu = 0$ hold for $Q = B_{\sqrt{2R}}(x^*)$, where $R \geq \|x^0 - x^*\|$. Assume that $\nu = 0$, $\zeta_* \geq \sqrt{\frac{1}{n}\sum_{i=1}^n \|\nabla f_i(x^*)\|^2}$,
    \begin{gather*}
        0 < \gamma = \cO\left(\min\left\{\frac{1}{LA}, \frac{\sqrt{n}R}{A\zeta_*}, \frac{n^{\nicefrac{(\alpha-1)}{\alpha}}R}{\sigma K^{\nicefrac{1}{\alpha}}A^{\nicefrac{(\alpha-1)}{\alpha}}}\right\}\right),\\
        \lambda_k = \lambda = \Theta\left(\frac{nR}{\gamma A}\right).
    \end{gather*}
    Then to guarantee $\Phi(\bar{x}^K) - \Phi(x^*) \leq \varepsilon$ for $\bar{x}^K = \frac{1}{K+1}\sum_{k=0}^K x^k$ with probability $\geq 1 - \beta$ \algname{DProx-clipped-SGD-shift} requires
    \begin{align}
        \widetilde\cO\left(\max\left\{\frac{LR^2}{\varepsilon}, \frac{R\zeta_*}{\sqrt{n}\varepsilon}, \frac{1}{n}\left(\frac{\sigma R}{\varepsilon}\right)^{\frac{\alpha}{\alpha-1}}\right\}\right) \label{eq:d_prox_clipped_SGD_shift_main_result_cvx}
    \end{align}
    iterations/oracle calls per worker.
\end{theorem}}

\revision{Discussion of the results for \algname{DProx-clipped-SGD-shift}.} Up to the difference between $V$ and $\|x^0 - x^*\|^2$, in the single-node case, the derived results coincide with ones known for \algname{clipped-SGD} in the unconstrained case \citep{sadiev2023high}. \revision{In the composite non-distributed case ($n=1$), the result of Theorem~\ref{thm:D_prox_clipped_SGD_main} is the first known of its type, and Theorem~\ref{thm:D_prox_clipped_SGD_main_cvx} recovers (up to logarithmic factors) the result from \citep{nguyen2023improved} for a version of Stochastic Mirror Descent with gradient clipping (\algname{Clipped-SMD}), see Table~\ref{tab:comparison_of_rates_minimization}. Indeed, parameter $\widehat{R}^2 = R\left(3R + L^{-1}(2\eta \sigma + \|\nabla f(x^0)\|)\right)$ for some $\eta > 0$ from the result by \citet{nguyen2023improved} equals $\Theta(R^2 + \nicefrac{R\zeta_*}{L})$, when $\eta$ is sufficiently small (otherwise $\widehat R$ can be worse than $\Theta(R^2 + \nicefrac{R\zeta_*}{L})$), which can be seen from the following inequalities following smoothness: $\|\nabla f(x^0)\| \leq \|\nabla f(x^*)\| + \|\nabla f(x^0) - \nabla f(x^*)\| \leq \|\nabla f(x^*)\| + L\|x^0 - x^*\|$ and $\|\nabla f(x^*)\| \leq \|\nabla f(x^0)\| + \|\nabla f(x^0) - \nabla f(x^*)\| \leq \|\nabla f(x^0)\| + L\|x^0 - x^*\|$. Since in this work we do not focus on the logarithmic factors, we do not show them in the main text and provide the complete expressions in the appendix. \citet{nguyen2023improved} has better dependencies on the parameters under logarithms than our results. In particular, the logarithmic factor in front of the last term in \eqref{eq:d_prox_clipped_SGD_shift_main_result_cvx} is $\ln\left(\beta^{-1}\left(\nicefrac{\sqrt{2}\sigma R}{\varepsilon}\right)^{\nicefrac{\alpha}{(\alpha-1)}}\right)$ (see \eqref{eq:prox_clipped_SGD_convex_case_complexity} from Theorem~\ref{thm:prox_clipped_SGD_convex_case}), which matches the logarithmic factors from the result obtained by \citet{sadiev2023high}. In the same setting with $n=1$, \citet{nguyen2023improved} obtained a high-probability complexity result with $\ln(\beta^{-1})$ logarithmic factor, which is smaller than the one we have. However, this discrepancy can be relatively small in many situations. For example, if $\sigma \sim 10^3$ (high variance), $R \sim 10^3$ (starting point is far from the solution), $\alpha = \nicefrac{3}{2}$, $\varepsilon \sim 10^{-6}$, $\beta \sim 10^{-6}$ (very accurate solution and small failure probability), then $\ln\left(\beta^{-1}\left(\nicefrac{\sqrt{2}\sigma R}{\varepsilon}\right)^{\nicefrac{\alpha}{(\alpha-1)}}\right) \approx 98$ and $\ln(\beta^{-1}) \approx 14$, i.e., our logarithmic factor is just $\sim 7$ times larger in this case. We conjecture that adjusting the proof technique from \citep{nguyen2023improved} one can improve the logarithmic factors in our results as well.}

It is worth mentioning that shifts are not needed in the \revision{convex} case because the method does not have fast \revision{enough} convergence, which makes it work with a constant clipping level\revision{, i.e., the method in the convex case requires less tight gradient estimates and is more robust to the bias than in strongly convex}. In the \revision{quasi-strongly convex} case, \revision{the shifts' stepsize is chosen as $\nu \sim \Theta(\nicefrac{1}{A})$ and it does not explicitly affect the rate since $\gamma\mu = \Theta(\nicefrac{1}{A})$, see the details in Section~\ref{sec:proofs_structure} and Appendix~\ref{appendix:dprox_clipped_SGD_shift}.} Moreover, as we show later, shifts are important for achieving acceleration in the convex case as well. A similar tradeoff between the needed accuracy of the estimate and the speed of convergence of the method is observed in the literature on the convergence of methods with inexact oracle \citep{devolder2013exactness, devolder2014first}.

Next, as expected for a distributed method, the terms in the complexity bounds related to the noise improve with the growth of $n$. \revision{More precisely, the terms depending on the noise level $\sigma$ are proportional to $\nicefrac{1}{n}$, i.e., our results show so-called linear speed-up in the complexity -- a desirable feature for a stochastic distributed method.} This \revision{aspect} highlights the benefits of parallelization. \revision{To the best of our knowledge, the results for the distributed methods proposed in our work are the only existing ones under Assumption~\ref{as:bounded_alpha_moment} (even if we take into account the in-expectation convergence results). In the special case of $\alpha = 2$, our results match (up to logarithmic factors) the SOTA ones from \citep{gorbunov2021near} since parallelization with linear speed-up follows for free under the bounded variance assumption, if the clipping is applied after averaging as it should be in the parallelized version of methods from \citep{gorbunov2021near} to keep the analysis from \citep{gorbunov2021near} unchanged. Indeed, when $\{\nabla f_{\xi_i}(x)\}_{i\in [n]}$ are independent stochastic gradients satisfying Assumption~\ref{as:bounded_alpha_moment} with parameters $\sigma > 0$ and $\alpha =2$, then $\frac{1}{n}\sum_{i\in[n]}\nabla f_{\xi_i}(x)$ also satisfies Assumption~\ref{as:bounded_alpha_moment} with parameters $\nicefrac{\sigma}{\sqrt{n}}$ and $\alpha =2$. However, when $\alpha < 2$ achieving linear speed-up is not that straightforward. If $\{\nabla f_{\xi_i}(x)\}_{i\in [n]}$ are independent stochastic gradients satisfying Assumption~\ref{as:bounded_alpha_moment} with parameters $\sigma > 0$ and $\alpha < 2$, then the existing results \citep[Lemma 7]{wang2021convergence} give a weaker guarantee: $\frac{1}{n}\sum_{i\in[n]}\nabla f_{\xi_i}(x)$ satisfies Assumption~\ref{as:bounded_alpha_moment} with parameters $\frac{2^{2-\alpha}d^{\frac{1}{\alpha}-\frac{1}{2}}\sigma}{n^{\frac{\alpha-1}{\alpha}}}$, which is \emph{dimension dependent}, and the same $\alpha$. Therefore, if one applies this result to the known ones from \citep{sadiev2023high, nguyen2023improved}, then the resulting complexity will have an extra factor of $d^{\frac{1}{\alpha-1} - \frac{\alpha}{2(\alpha-1)}}$ in the term that depends on $\sigma$. For large-scale or even medium-scale heavy-tailed problems, this factor can be huge, e.g., when $d = 1000$ and $\alpha = \frac{7}{6}$, this factor is $1000^{6 - \frac{7}{3}} > 1000^3 = 10^9$.}

\revision{To avoid these issues,} we apply gradient clipping on the workers and then average clipped vectors, not vice versa. This is \revision{also} partially motivated by the popularity of gradient clipping for ensuring differential privacy guarantees \citep{abadi2016deep, chen2020understanding} in Federated Learning. \revision{Therefore, the proposed distributed methods can be useful for differential privacy as well, though we do not study this aspect in our work.}

\begin{remark}[Dependence of the parameters on $\zeta_*$.]
    In our results, $\gamma$ and $\{\lambda_k\}_{k\geq 0}$ depend on $\zeta_*$, which is an upper bound for $\sqrt{\frac{1}{n}\sum_{i=1}^n \|\nabla f_i(x^*)\|^2}$. Since $f_i$ are $L$-smooth for all $i\in [n]$, we have $\sqrt{\frac{1}{n}\sum_{i=1}^n \|\nabla f_i(x^*)\|^2} \leq G_0 + \sqrt{\frac{1}{n}\sum_{i=1}^n \|\nabla f_i(x^*) - \nabla f_i(x^0)\|^2} \leq G_0 + L\|x^0 - x^*\|$, where $G_0 = \sqrt{\frac{1}{n}\sum_{i=1}^n \|\nabla f_i(x^0)\|^2}$. As suggested by \citet{nazin2019algorithms, nguyen2023improved}, one can estimate $\|\nabla f_i(x^0)\|$ with probability at least $1 - \beta$ using the procedure from \citep{minsker2015geometric} requiring $\cO(\ln(n\beta^{-1}))$ samples on each client $i\in [n]$. Therefore, to estimate $\zeta_*$, it is sufficient to know $R$.
\end{remark}

\paragraph{Acceleration.} Next, we propose a distributed version of clipped Stochastic Similar Triangles Method \citep{gorbunov2020stochastic, gasnikov2016universal} for composite problems (\algname{DProx-clipped-SSTM-shift}): $x^0 = y^0 = z^0$, $A_0 = \alpha_0 = 0$, $\alpha_{k+1} = \frac{k+2}{2aL}$, $A_{k+1} = A_k + \alpha_{k+1}$ and
\begin{gather}
    x^{k+1} = \frac{A_k y^k + \alpha_{k+1}z^k}{A_{k+1}},\\
    z^{k+1} = \prox_{\alpha_{k+1}\Psi}\left(z^k - \alpha_{k+1}\tg(x^{k+1})\right), \label{eq:D_prox_clipped_SSTM_shift_x_z}\\
    y^{k+1} = \frac{A_k y^k + \alpha_{k+1}z^{k+1}}{A_{k+1}} \label{eq:D_prox_clipped_SSTM_shift_y}
\end{gather}
with
\begin{gather}
    \tg(x^{k+1}) = \frac{1}{n}\sum\limits_{i=1}^n \tg_i(x^{k+1}),\;\; \tg_i(x^{k+1}) = h_i^k + \hat\Delta_i^k, \label{eq:D_prox_clipped_SSTM_shift_estimator}\\
    h_i^{k+1} = h_i^k + \nu_k \hat \Delta_i^k,\\
    \hat\Delta_i^k = \clip\left(\nabla f_{\xi_i^k}(x^{k+1}) - h_i^k, \lambda_k\right), \label{eq:D_prox_clipped_SSTM_shift_estimator_2}
\end{gather}
where $\xi_1^k,\ldots, \xi_n^k$ are sampled independently from each other and previous steps. For the proposed method, we derive the following result.

\revision{\begin{theorem}[Convergence of \algname{DProx-clipped-SSTM-shift}]\label{thm:D_prox_clipped_SSSTM_main}
    Let Assumptions~\ref{as:bounded_alpha_moment}, \ref{as:L_smoothness},  and \ref{as:str_cvx} with $\mu = 0$ hold for $Q = B_{5\sqrt{2}nR}(x^*)$, where $R \geq \|x^0 - x^*\|^2$. Let $\zeta_* \geq \sqrt{\frac{1}{n}\sum_{i=1}^n \|\nabla f_i(x^*)\|^2}$, $C = \Theta(\nicefrac{A}{\sqrt{n}})$, $K_0 = \Theta(A^2)$, where $K \geq 1$, $\beta \in (0,1)$, $A = \ln \tfrac{10nK}{\beta}$. Let
    \begin{gather*}
        \nu_k = \begin{cases}
            \frac{2k+5}{(k+3)^2},& \text{if } k > K_0,\\ \frac{(k+2)^2}{C^2(K_0+2)^2n},& \text{if } k \leq K_0,
        \end{cases},\;\; \lambda_k = \Theta\left(\frac{n R}{\alpha_{k+1}A}\right),\\
        a = \Theta\left(\max\left\{2, \frac{A^4}{n}, \frac{A^3\zeta_*}{L\sqrt{n}R}, \frac{\sigma K^{\nicefrac{(\alpha+1)}{\alpha}}A^{\nicefrac{(\alpha-1)}{\alpha}}}{LRn^{\nicefrac{\alpha-1}{\alpha}}}\right\}\right).
    \end{gather*}
    Then to guarantee $\Phi(y^K) - \Phi(x^*) \leq \varepsilon$ with probability $\geq 1 - \beta$ \algname{DProx-clipped-SSTM-shift} requires
    \begin{align}
        \widetilde\cO\left(\max\left\{\sqrt{\frac{LR^2}{\varepsilon}}, \sqrt{\frac{R\zeta_*}{\sqrt{n}\varepsilon}}, \frac{1}{n}\left(\frac{\sigma R}{\varepsilon}\right)^{\frac{\alpha}{\alpha-1}}\right\}\right) \label{eq:d_prox_clipped_SSTM_shift_main_result_cvx}
    \end{align}
    iterations/oracle calls per worker.
\end{theorem}}
\begin{proof}[Sketch of the proof]
    The proof resembles the one for \algname{clipped-SSTM} from \citep{sadiev2023high} but has some noticeable differences. In addition to handling the extra technical challenges appearing due to the composite structure (e.g., one cannot apply some useful formulas like $z^k - z^{k+1} = \alpha_{k+1}\tg(x^{k+1})$ that hold in the unconstrained case), we use a non-standard potential function $M_k$ defined as $M_k = \|z^k - x^*\|^2 + \frac{C^2\alpha_{K_0+1}^2}{n}\sum_{i=1}^n \|h_i^k - \nabla f_i(x^*)\|^2$ for $k \leq K_0$ and $M_k = \|z^k - x^*\|^2 + \tfrac{C^2\alpha_{k+1}^2}{n}\sum_{i=1}^n \|h_i^k - \nabla f_i(x^*)\|^2$ for $k > K_0$. See more details in Appendix~\ref{appendix:dprox_clipped_SSTM_shift}. 
\end{proof}

When $n=1$, the derived result has optimal dependence on $\varepsilon$ (up to logarithmic factors) \citep{nemirovskij1983problem, zhang2020adaptive}. In contrast to the result from \citep{nguyen2023improved}, we do not assume that $\nabla f(x^*) = 0$. Moreover, as \algname{DProx-clipped-SGD-shift}, \algname{DProx-clipped-SSTM-shift} benefits from parallelization since the second term in \eqref{eq:d_prox_clipped_SSTM_shift_main_result_cvx} is proportional $\nicefrac{1}{n}$. When $n$ is sufficiently large, the effect of acceleration can become significant even for large $\sigma$. In Appendix~\ref{appendix:SSTM_str_cvx}, we also provide the convergence results for the restarted version of \algname{DProx-clipped-SSTM-shift} assuming additionally that $f$ is strongly convex and one can compute starting shifts $h_i^0$ as $\nabla f_i(x^0)$.

\section{\revision{On the Proofs Structure}\label{sec:proofs_structure}}

\begin{table*}[t]
    \centering
    \scriptsize
    \caption{\scriptsize Summary of known and new high-probability complexity results for solving (non-) composite (non-) distributed variational inequality problem \eqref{eq:main_problem}. Column ``Setup'' indicates the assumptions made in addition to Assumptions~\ref{as:bounded_alpha_moment}. All assumptions are made only on some ball around the solution with radius $\sim R \geq \|x^0 - x^*\|$ (for the results from \citep{sadiev2023high}) or radius $\sim\sqrt{V}$ (Theorems~\ref{thm:D_prox_clipped_SGDA_main} and \ref{thm:D_prox_clipped_SEG_main}). Complexity is the number of stochastic oracle calls(per worker) needed for a method to guarantee that $\PP\{\text{Metric} \leq \varepsilon\} \geq 1 - \beta$ for some $\varepsilon> 0$, $\beta \in (0,1]$ and ``Metric'' is taken from the corresponding column. Numerical and logarithmic factors are omitted for simplicity. Column ``C?'' shows whether the problem \eqref{eq:min_problem} is composite, ``D?'' indicates whether the problem \eqref{eq:min_problem} is distributed. Notation: $\tx^K_{\text{avg}} = \frac{1}{K+1}\sum_{k=0}^K \tx^k$ (for \algname{SEG}-type methods), $x^K_{\avg} = \frac{1}{K+1}\sum_{k=0}^K x^k$ (for \algname{SGDA}-type methods); $L$ = Lipschitz constant; $\sigma$ = parameter from Assumption~\ref{as:bounded_alpha_moment}; $R$ = any upper bound on $\|x^0 - x^*\|$ (for the results from \citep{sadiev2023high}); $V$ = any upper bound on $\|x^0 - x^*\|^2 + \frac{409600\gamma^2\ln^2\frac{48n(K+1)}{\beta}}{n^2}\sum_{i=1}^n\|F_i(x^*)\|^2$ (for the results of this paper); $\mu$ = quasi-strong monotonicity parameter; $\ell$ = star-cocoercivity parameter. The results of this paper are highlighted in blue.}
    \label{tab:comparison_of_rates_VIP}
    \begin{threeparttable}
        \begin{tabular}{|c|c c c c c|}
        \hline
        Setup & Method  & Metric & Complexity & C? & D?\\
        \hline\hline
        \multirow{3}{*}{As.~\ref{as:L_Lip} \& \ref{as:monotonicity}} & \begin{tabular}{c}
            \algname{clipped-SEG} \\ \citep{sadiev2023high}
        \end{tabular} & $\gap_{R}(\tx_{\avg}^K)$ & $\max\left\{\frac{LR^2}{\varepsilon}, \left(\frac{\sigma R}{\varepsilon}\right)^{\frac{\alpha}{\alpha-1}}\right\}$ & \xmark & \xmark \\
        & \cellcolor{bgcolor}\begin{tabular}{c}
            \algname{DProx-clipped-SEG-shift}\\ Theorem~\ref{thm:D_prox_clipped_SEG_monotone_appendix}
        \end{tabular} & \cellcolor{bgcolor}$\gap_{\sqrt{V}}(\tx_{\avg}^K)$ &\cellcolor{bgcolor} $\max\left\{\frac{LV}{\varepsilon}, \frac{1}{n}\left(\frac{\sigma \sqrt{V}}{\varepsilon}\right)^{\frac{\alpha}{\alpha-1}}\right\}$ &\cellcolor{bgcolor} \cmark &\cellcolor{bgcolor} \cmark \\
        \hline
        \multirow{3}{*}{As.~\ref{as:L_Lip} \& \ref{as:str_monotonicity}} & \begin{tabular}{c}
            \algname{clipped-SEG} \\ \citep{sadiev2023high}
        \end{tabular} & $\|x^k - x^*\|^2$ & $\max\left\{\frac{L}{\mu}, \left(\frac{\sigma^2}{\mu^2\varepsilon}\right)^{\frac{\alpha}{2(\alpha-1})}\right\}$ & \xmark & \xmark\\
        & \cellcolor{bgcolor}\begin{tabular}{c}
            \algname{DProx-clipped-SEG-shift}\\ Theorem~\ref{thm:D_prox_clipped_SEG_QSM_appendix}
        \end{tabular} & \cellcolor{bgcolor}$\|x^k - x^*\|^2$ &\cellcolor{bgcolor} $\max\left\{\frac{L}{\mu}, \frac{1}{n}\left(\frac{\sigma^2}{\mu^2\varepsilon}\right)^{\frac{\alpha}{2(\alpha-1})}\right\}$ &\cellcolor{bgcolor} \cmark &\cellcolor{bgcolor} \cmark \\
        \hline\hline
        \multirow{3}{*}{As.~\ref{as:monotonicity} \& \ref{as:star_cocoercivity} \& \ref{as:cocoercivity}} & \begin{tabular}{c}
            \algname{clipped-SGDA} \\ \citep{sadiev2023high}
        \end{tabular} & $\gap_{R}(x_{\avg}^K)$ & $\max\left\{\frac{\ell R^2}{\varepsilon}, \left(\frac{\sigma R}{\varepsilon}\right)^{\frac{\alpha}{\alpha-1}}\right\}$ & \xmark & \xmark\\
        & \cellcolor{bgcolor}\begin{tabular}{c}
            \algname{DProx-clipped-SGDA-shift}\\ Theorem~\ref{thm:main_result_gap_prox_clipped_SGDA}
        \end{tabular} & \cellcolor{bgcolor}$\gap_{\sqrt{V}}(x_{\avg}^K)$ &\cellcolor{bgcolor} $\max\left\{\frac{\ell V}{\varepsilon}, \frac{1}{n}\left(\frac{\sigma \sqrt{V}}{\varepsilon}\right)^{\frac{\alpha}{\alpha-1}}\right\}$ &\cellcolor{bgcolor} \cmark &\cellcolor{bgcolor} \cmark \\
        \hline
        \multirow{3}{*}{As.~\ref{as:str_monotonicity} \& \ref{as:star_cocoercivity}} & \begin{tabular}{c}
            \algname{clipped-SGDA} \\ \citep{sadiev2023high}
        \end{tabular} & $\|x^K - x^*\|^2$ & $\max\left\{\frac{\ell}{\mu}, \left(\frac{\sigma^2}{\mu^2\varepsilon}\right)^{\frac{\alpha}{2(\alpha-1})}\right\}$ & \xmark & \xmark\\
        & \cellcolor{bgcolor}\begin{tabular}{c}
            \algname{DProx-clipped-SGDA-shift}\\ Theorem~\ref{thm:main_result_prox_clipped_SGDA_2}
        \end{tabular} & \cellcolor{bgcolor}$\|x^K - x^*\|^2$ &\cellcolor{bgcolor} $\max\left\{\frac{\ell}{\mu}, \frac{1}{n}\left(\frac{\sigma^2}{\mu^2\varepsilon}\right)^{\frac{\alpha}{2(\alpha-1})}\right\}$ &\cellcolor{bgcolor} \cmark &\cellcolor{bgcolor} \cmark \\
        \hline
    \end{tabular}
    \end{threeparttable}
\end{table*}

\revision{In this section, we elaborate on the proofs structure of our results and highlight additional challenges appearing due to the presence of the composite term and distributed nature of the methods. The proof of each result consists of two parts: optimization/descent lemma and the analysis of the sums appearing due to the stochasticity and biasedness of the updates (due to the clipping). In the first part, we usually follow some standard analysis of corresponding deterministic method without clipping and separate the stochastic part from the deterministic one (though for \algname{DProx-clipped-SSTM-shift} we use quite non-standard Lyapunov function, which can be interesting on its own). For example, in the analysis\footnote{\revision{In the appendix, we analyze this case in the generality of variational inequalities. Here we provide a simplified version for minimization.}} of \algname{DProx-clipped-SGD-shift} under Assumption~\ref{as:QSC}, we prove the following inequality:
\begin{align}
        &V_{K+1} \leq (1 - \gamma \mu)^{K+1}V_0 \notag\\
        &\;\; + \frac{2\gamma}{n} \sum\limits_{k=0}^K\sum\limits_{i=1}^n (1-\gamma\mu)^{K-k}\langle s_k, \omega_{i,k} \rangle\notag\\
        &\;\; + \gamma^2\sum\limits_{k=0}^K(1-\gamma \mu)^{K-k}\left(\|\omega_k\|^2 + \frac{1}{n^2}\sum\limits_{i=1}^n  \|\omega_{i,k}\|^2\right), \notag
\end{align}
where $V_k = \|x^k - x^*\|^2 + \frac{C^2\gamma^2 A^2}{n}\sum_{i=1}^n \|h_i^k - \nabla f_i(x^*)\|^2$ for some numerical constant $C> 0$, $s_k = x^k - x^* - \gamma (\nabla f(x^k)- \nabla f(x^*))$ and vectors $\omega_{i,k} = \nabla f_i(x^k) - \tg_i^k$ represent the discrepancy between the full gradients and their estimates. Moreover, to use this inequality for some $K = T \geq 0$ we need to show that $\{x^k\}_{k=0}^T$ belong to the set where the assumptions hold (in this particular case, to $B_{3n\sqrt{2}R}(x^*)$) with high probability. We do it always by induction. More precisely, we prove that $\PP\{E_k\} \geq 1 - \nicefrac{k\beta}{(K+1)}$ for the probability event $E_k$ defined as follows: inequalities $V_t \leq 4\exp(-\gamma\mu t)R^2$ and $\left\|\frac{\gamma}{n}\sum_{i=1}^{r-1}\omega_{i,t-1}^u\right\| \leq \exp(-\nicefrac{\gamma\mu(t-1)}{2})\sqrt{\nicefrac{R^2}{2}}$ hold for $t = 0,1,\ldots,k$ and $r = 1, 2, \ldots, n$ simultaneously, where $\omega_{i,t}^u = \EE_{\xi_i^t}[\tg_i^t] - \tg_i^t$ and $\EE_{\xi_i^t}[\cdot]$ denotes an expectation w.r.t.\ $\xi_i^t$. To prove this, we use Bernstein inequality for martingale difference (see Lemma~\ref{lem:Bernstein_ineq}). However, to apply Bernstein inequality we need to circumvent multiple technical difficulties related to the estimation of the norm of the clipped vector (that involves derivations related to the shifts $\{h_i^k\}_{i\in [n]}$), proper choice of the clipping level to control the bias and variance and achieve desired linear speed-up (see Lemma~\ref{lem:bias_and_variance_clip_distributed} and the following discussion). Moreover, when $n > 1$ (distributed case), we also need to apply additional induction over clients to estimate sums like $\circledSix$ from \eqref{eq:SGDA_2_prox_clipped_12345_bound}.
}

\section{Extensions for Variational Inequalities}

In addition to the minimization problems, we also consider stochastic composite variational inequality problems (VIPs):
\begin{gather}
    \text{find }\; x^* \in \R^d \; \text{ such that } \notag\\
    \langle F(x^*), x - x^* \rangle + \Psi(x) - \Psi (x^*) \geq  0, \label{eq:main_problem}
\end{gather}
where the assumptions on operator $F(x) = \EE_{\xi\sim \cD}[F_{\xi}(x)]: \R^d \to \R^d$ will be specified later and, as in the case of minimization, $\Psi(x)$ is a proper, closed, convex function. When $f(x)$ is convex problem \eqref{eq:min_problem} is a special case of \eqref{eq:main_problem} with $F(x) = \nabla f(x)$. For the examples of problems of type \eqref{eq:main_problem}, we refer to \citep{alacaoglu2022stochastic, beznosikov2022stochastic}.

The distributed version of \eqref{eq:main_problem} has the following structure:
\begin{equation}
    F(x) = \frac{1}{n}\sum\limits_{i=1}^n \left\{F_i(x) = \EE_{\xi_i\sim \cD_i}[F_{\xi_i}(x)]\right\}. \label{eq:distributed_problem_VI}
\end{equation}
In this case, there are $n$ workers connected in a centralized way with some parameter server; worker $i$ can query some noisy information (stochastic gradients/estimates) about $F_i$.

Similarly to the minimization case, we consider the situation when $F_i$ is accessible through the stochastic oracle calls. The stochastic estimates satisfy the following assumption.
\begin{assumption}\label{as:bounded_alpha_moment_VI}
    There exist some set $Q \subseteq \R^d$ and values $\sigma \geq 0$, $\alpha \in (1,2]$ such that for all $x \in Q$
    we have $\EE_{\xi_i\sim \cD_i}[F_{\xi_i}(x)] = F_i(x)$ and
    \begin{equation}
        \EE_{\xi_i\sim \cD_i}[\left\|F_{\xi_i}(x) - F_i(x)\right\|^\alpha] \leq \sigma^\alpha. \label{eq:bounded_alpha_moment_operator}
    \end{equation}
\end{assumption}

Under this and additional assumptions on $\{F_i\}_{i=1}^n$, we obtain several convergence results for the distributed versions\footnote{\algname{GDA}-type methods are analogs of \algname{GD} iteration $x^{k+1} = x^k - \gamma F(x^k)$. Such methods converge under (star-)cocoercivity but fail for general monotone problems. To address this issue, versions of Extragradient-type iteration \citep{korpelevich1976extragradient} $x^{k+1} = x^k - \gamma F(x^k - \gamma F(x^k))$ are used; see further details in Appendix~\ref{appendix:VIs}.} of \algname{clipped-SGDA} and \algname{clipped-SEG} from \citep{gorbunov2022clipped}. The results are summarized in Table~\ref{tab:comparison_of_rates_VIP}. As one can see from the table, our results generalize the ones from \citep{sadiev2023high} to the composite distributed case and achieve linear speedup. Further details on the assumptions and our main results for distributed composite variational inequalities are deferred to Appendix~\ref{appendix:VIs}.

\section*{Acknowledgements}
The work of A.~Gasnikov was supported by a grant for research centers in the field of artificial intelligence, provided by the Analytical Center for the Government of the Russian Federation in accordance with the subsidy agreement (agreement identifier 000000D730321P5Q0002) and the agreement with the Ivannikov Institute for System Programming of the Russian Academy of Sciences dated November 2, 2021 No. 70-2021-00142.

\section*{Impact Statement}
This paper presents work whose goal is to advance the field of Machine Learning. There are many potential societal consequences of our work, none of which we feel must be specifically highlighted here.



\bibliography{refs}
\bibliographystyle{icml2024}

\newpage
\appendix
\onecolumn



\section{Extra Related Work}\label{appendix:extra_related}

\paragraph{Non-convex case.} \citet{li2020high} analyze the high-probability convergence rate of \algname{SGD} for finding first-order stationary points for smooth non-convex unconstrained problems. The first high-probability result under Assumption~\ref{as:bounded_alpha_moment} for the same class of functions is derived by \citet{cutkosky2021high}. However, the result of \citet{cutkosky2021high} relies on the additional assumption that the gradients are bounded. \citet{sadiev2023high} remove the bounded gradient assumption but derive a slightly worse rate. \citet{nguyen2023high} improve the result and achieve the same rate as in \citep{cutkosky2021high} without assuming boundedness of the gradients. It is worth mentioning that \citet{cutkosky2021high, sadiev2023high, nguyen2023high} derive their main results for the methods that use gradient clipping.

\paragraph{Gradient clipping} is a very useful algorithmic tool in the training of deep neural networks \citep{pascanu2013difficulty, goodfellow2016deep}. Gradient clipping also has some good theoretical properties, e.g., it can be useful for minimization of $(L_0, L_1)$-smooth functions \citep{zhang2020gradient}, in differential privacy \citep{abadi2016deep}, Byzantine-robustness \citep{karimireddy2021learning}. Moreover, as we already mentioned, almost all existing high-probability results that do not rely on the light-tailed noise assumption are derived for the methods with clipping. Recently, \citet{sadiev2023high} theoretically showed that \algname{SGD} has worse high-probability convergence than \algname{clipped-SGD} even when the noise in the gradient has bounded variance.



\clearpage

\section{Auxiliary and Technical Results}\label{appendix:aux_results}

\paragraph{Bernstein inequality.} In the final stages of our proofs, we need to estimate certain sums of random variables. The main tool that we use to handle such sums is {\it Bernstein inequality for martingale differences} \citep{bennett1962probability,dzhaparidze2001bernstein,freedman1975tail}.
\begin{lemma}\label{lem:Bernstein_ineq}
    Let the sequence of random variables $\{X_i\}_{i\ge 1}$ form a martingale difference sequence, i.e.\ $\EE\left[X_i\mid X_{i-1},\ldots, X_1\right] = 0$ for all $i \ge 1$. Assume that conditional variances $\sigma_i^2\eqdef\EE\left[X_i^2\mid X_{i-1},\ldots, X_1\right]$ exist and are bounded and also assume that there exists deterministic constant $c>0$ such that $|X_i| \le c$ almost surely for all $i\ge 1$. Then for all $b > 0$, $G > 0$ and $n\ge 1$
    \begin{equation}
        \PP\left\{\Big|\sum\limits_{i=1}^nX_i\Big| > b \text{ and } \sum\limits_{i=1}^n\sigma_i^2 \le G\right\} \le 2\exp\left(-\frac{b^2}{2G + \nicefrac{2cb}{3}}\right).
    \end{equation}
\end{lemma}

\paragraph{Impact of clipping on the bias and variance.} The following lemma also helps to handle the aforementioned sums of random variables.

\begin{lemma}[Lemma~5.1 from \cite{sadiev2023high}]\label{lem:bias_and_variance_clip}
    Let $X$ be a random vector in $\R^d$ and $\tX = \clip(X,\lambda)$. Then, $\|\tX - \EE[\tX]\| \leq 2\lambda$. 
    Moreover, if for some $\sigma \geq 0$ and $\alpha \in (1,2]$ we have $\EE[X] = x\in\R^d$, $\EE[\|X - x\|^\alpha] \leq \sigma^\alpha$, 
    and $\|x\| \leq \nicefrac{\lambda}{2}$, then
    \begin{eqnarray}
        \left\|\EE[\tX] - x\right\| &\leq& \frac{2^\alpha\sigma^\alpha}{\lambda^{\alpha-1}}, \label{eq:bias_X_main}\\
        \EE\left[\left\|\tX - \EE[\tX]\right\|^2\right] &\leq& 18 \lambda^{2-\alpha} \sigma^\alpha. \label{eq:variance_X_main}
    \end{eqnarray}
\end{lemma}

\paragraph{Intuition behind the choice of clipping level in the distributed case.} To better illustrate why we increase clipping level $n$ times, we prove the following lemma.

\begin{lemma}\label{lem:bias_and_variance_clip_distributed}
    Let $X_1, X_2, \ldots, X_n$ be independent random vectors in $\R^d$ and $\tX_i = \clip(X_i,\lambda)$ for all $i \in [n]$. Then, for $\tX = \frac{1}{n}\sum_{i=1}^n \tX_i$ we have  $\|\tX - \EE[\tX]\| \leq 2\lambda$. Moreover, if for some $\sigma \geq 0$ and $\alpha \in (1,2]$ we have $\EE[X_i] = x_i\in\R^d$, $\EE[\|X_i - x_i\|^\alpha] \leq \sigma^\alpha$, and $\|x_i\| \leq \nicefrac{\lambda}{2}$ for all $i \in [n]$, then for $x = \frac{1}{n}\sum_{i=1}^n x_i$ the following inequalities hold
    \begin{eqnarray}
        \left\|\EE[\tX] - x\right\| &\leq& \frac{2^\alpha\sigma^\alpha}{\lambda^{\alpha-1}}, \label{eq:bias_X_main_distr}\\
        \EE\left[\left\|\tX - \EE[\tX]\right\|^2\right] &\leq& \frac{18 \lambda^{2-\alpha} \sigma^\alpha}{n}. \label{eq:variance_X_main_distr}
    \end{eqnarray}
\end{lemma}
\begin{proof}
    From Lemma~\ref{lem:bias_and_variance_clip} we have for all $i \in [n]$ that $\|\tX_i - \EE[\tX_i]\| \leq 2\lambda$ and
    \begin{eqnarray}
        \left\|\EE[\tX_i] - x_i\right\| &\leq& \frac{2^\alpha\sigma^\alpha}{\lambda^{\alpha-1}}, \label{eq:bias_X_main_1111111}\\
        \EE\left[\left\|\tX_i - \EE[\tX_i]\right\|^2\right] &\leq& 18 \lambda^{2-\alpha} \sigma^\alpha. \label{eq:variance_X_main_1111111}
    \end{eqnarray}
    Jensen's inequality implies
    \begin{eqnarray*}
        \left\|\tX - \EE[\tX]\right\| &=& \left\| \frac{1}{n}\sum\limits_{i=1}^n \left( \tX_i - \EE[\tX_i]\right) \right\| \leq \frac{1}{n}\sum\limits_{i=1}^n \left\|\tX_i - \EE[\tX_i]\right\| \leq 2\lambda,\\
        \left\|\EE[\tX] - x\right\| &=& \left\| \frac{1}{n}\sum\limits_{i=1}^n \left(\EE[\tX_i] - x_i\right) \right\| \leq \frac{1}{n}\sum\limits_{i=1}^n \left\|\EE[\tX_i] - x_i\right\| \overset{\eqref{eq:bias_X_main_1111111}}{\leq} \frac{2^\alpha\sigma^\alpha}{\lambda^{\alpha-1}}.
    \end{eqnarray*}
    Finally, using the independence of $\tX_1,\ldots, \tX_n$, we derive
    \begin{eqnarray*}
        \EE\left[\left\|\tX - \EE[\tX]\right\|^2\right] &=& \EE\left[\left\| \frac{1}{n}\sum\limits_{i=1}^n \left( \tX_i - \EE[\tX_i]\right) \right\|^2\right] = \frac{1}{n^2}\sum\limits_{i=1}^n\EE\left[\left\|\tX_i - \EE[\tX_i]\right\|^2\right]\\
        &\overset{\eqref{eq:variance_X_main_1111111}}{\leq}& \frac{18 \lambda^{2-\alpha} \sigma^\alpha}{n}
    \end{eqnarray*}
    that concludes the proof.
\end{proof}

From \eqref{eq:bias_X_main_distr}-\eqref{eq:variance_X_main_distr}, we see that number of workers $n$ appears differently in the bound on bias and variance. However, if we replace $\lambda$ with $n\lambda$, then both bounds will transform to \eqref{eq:bias_X_main}-\eqref{eq:variance_X_main} respectively with $\sigma^\alpha = \nicefrac{\sigma^\alpha}{n^{\alpha-1}}$ (in other words, bias and variance will have the same dependence on $n$). These observations hint that the complexity bounds for distributed methods should be similar to the ones proven for non-distributed methods (in the unconstrained case) by \citet{sadiev2023high} up to the replacement of $\sigma^\alpha$ with $\nicefrac{\sigma^\alpha}{n^{\alpha-1}}$. Nevertheless, our analysis of the distributed case does not rely on Lemma~\ref{lem:bias_and_variance_clip_distributed} and has some important differences with the single-node case (even when $\Psi \equiv 0$).

\paragraph{Useful inequality related to prox-operator.} In the analysis of \algname{DProx-clipped-SGDA-shift}, we use the following standard result.

\begin{lemma}[Theorem 6.39 (iii) from \citep{beck2017first}]\label{lem:prox_lemma}
    Let $\Psi$ be a proper lower semicontinuous convex function and $x^+ = \proxkPsi(x)$. Then for all $y \in \R^d$ the following inequality holds:
    \begin{equation*}
        \langle x^+ - x, y - x^+ \rangle \geq \gamma \left(\Psi(x^+) - \Psi(y)\right).
    \end{equation*}
\end{lemma}

\clearpage

\section{\revision{Composite Distributed Variational Inequalities}}\label{appendix:VIs}

In this section, we provide further details on the results we obtained for variational inequalities.




\subsection{Assumptions}


In addition to Assumption~\ref{as:bounded_alpha_moment_VI}, we consider the following assumptions in the case of variational inequalities.

\paragraph{Assumptions on $F_i$.} We make standard assumptions on $\{F_i\}_{i\in [n]}$. The first one is Lipschitzness. 

\begin{assumption}\label{as:L_Lip}
    There exist some set $Q \subseteq \R^d$ such that operators $F_i$ are $L$-Lipschitz:
    \begin{eqnarray}
        \|F_i(x) - F_i(y)\| \leq L\|x - y\| \quad \forall \, x, y \in Q, \, i\in [n]. \label{eq:L_Lip}
    \end{eqnarray}
\end{assumption}

Next, for each particular result, we make one or two of the following assumptions.

\begin{assumption}\label{as:monotonicity}
    There exist some set $Q \subseteq \R^d$ such that $F$ is monotone on $Q$:
    \begin{equation}
        \langle F(x) - F(y), x - y \rangle \geq  0 \quad \forall \, x, y \in Q.\label{eq:monotonicity}
    \end{equation}
\end{assumption}

\begin{assumption}\label{as:str_monotonicity}
    There exist some set $Q \subseteq \R^d$ such that $F$ is $(\mu, x^*)$-quasi strongly monotone on $Q$ for some $\mu \geq 0$ and \revision{any solution $x^*$ of \eqref{eq:main_problem}}:
    \begin{gather}
        \langle F(x) - F(x^*) , x - x^* \rangle \geq \mu\|x - x^*\|^2, \quad \forall\, x\in Q.
        \label{eq:str_monotonicity}
    \end{gather}
\end{assumption}

\begin{assumption}\label{as:star_cocoercivity}
    There exist some set $Q \subseteq \R^d$ such that $\{F_i\}_{i\in[n]}$ are $(\ell, x^*)$-star-cocoercive on $Q$ for some $\ell > 0$ and \revision{any solution $x^*$ of \eqref{eq:main_problem}}:
    \begin{gather}
        \|F_i(x) - F_i(x^*)\|^2 \leq \ell\langle F_i(x) - F_i(x^*), x - x^* \rangle, \quad \forall\, x \in Q,\, i\in [n]. \label{eq:star_cocoercivity}
    \end{gather}
\end{assumption}

\begin{assumption}\label{as:cocoercivity}
    There exist some set $Q \subseteq \R^d$ such that $F$ is $\ell$-cocoercive on $Q$ for some $\ell > 0$:
    \begin{gather}
        \|F(x) - F(y)\|^2 \leq \ell\langle F(x) - F(y), x - y \rangle, \quad \forall\, x,y \in Q. \label{eq:cocoercivity}
    \end{gather}
\end{assumption}

Assumption~\ref{as:monotonicity} is a standard assumption for the literature on VIPs. Quasi-strong monotonicity \citep{mertikopoulos2019learning, song2020optimistic, loizou2021stochastic} is weaker than standard strong monotonicity\footnote{Operator $F$ is called $\mu$-strongly monotone on $Q$ if $\langle F(x) - F(y), x - y \rangle \geq  \mu\|x - y\|^2$.} and star-cocoercivity is weaker than standard cocoercivity (Assumption~\ref{as:cocoercivity}), which implies monotonicity and Lipschitzness but not vice versa. Both conditions \eqref{eq:str_monotonicity} and \eqref{eq:star_cocoercivity} imply neither monotonicity nor Lipschitzness \citep{loizou2021stochastic}.

\subsection{DProx-clipped-SGDA-shift}

For composite variational inequalities, we start with Distributed \algname{Prox-clipped-SGDA-shift} (\algname{DProx-clipped-SGDA-shift}) that is defined in \eqref{eq:Dprox_clipped_SGD_shift}-\eqref{eq:Dprox_clipped_SGD_shift_shifts} with the following change: $\hat\Delta_i^k = \clip\left(F_{\xi_i^k}(x^k) - h_i^k, \lambda_k\right)$, 
where $\xi_1^k,\ldots, \xi_n^k$ are sampled independently from each other and previous steps. For the proposed method, we derive the following result.
\begin{theorem}[Convergence of \algname{DProx-clipped-SGDA-shift}]\label{thm:D_prox_clipped_SGDA_main}
    Let $K \geq 1$, $\beta \in (0,1)$, $A = \ln \frac{48n(K+1)}{\beta}$, $V \geq \|x^0 - x^*\|^2 + \frac{25600\gamma^2A^2}{n^2}\sum_{i=1}^n \|F_i(x^*)\|^2$.\newline
    \textbf{Case 1.} Let Assumptions~\ref{as:bounded_alpha_moment}, \ref{as:str_monotonicity} with $\mu > 0$, and \ref{as:star_cocoercivity} hold for $Q = B_{3\sqrt{V}}(x^*)$. Assume that $0 < \nu = \cO(\nicefrac{1}{\sqrt{n}A})$, $0 < \gamma = \cO\left(\min\{\nicefrac{1}{\sqrt{n}A\mu}, \nicefrac{1}{\ell A}, \nicefrac{\ln(B_K)}{\mu(K+1)}\}\right)$, $B_K = \Theta\left(\max\{2, \nicefrac{(K+1)^{\nicefrac{2(\alpha-1)}{\alpha}}\mu^2 n^{\nicefrac{2(\alpha-1)}{\alpha}} V}{\sigma^2 A^{\nicefrac{2(\alpha-1)}{\alpha}}\ln^2(B_K)}\}\right)$, $\lambda_k = \Theta(\nicefrac{n\exp(-\gamma\mu(1+\nicefrac{k}{2}))\sqrt{V}}{\gamma A})$. \newline
    \textbf{Case 2.} Let Assumptions~\ref{as:bounded_alpha_moment}, \ref{as:monotonicity},  and \ref{as:star_cocoercivity} hold for $Q = B_{3\sqrt{V}}(x^*)$. Assume that $\nu = 0$, $0 < \gamma = \cO(\min\{\nicefrac{1}{\ell A}, \nicefrac{n^{\nicefrac{(\alpha-1)}{\alpha}}\sqrt{V}}{\sigma K^{\nicefrac{1}{\alpha}}A^{\nicefrac{(\alpha-1)}{\alpha}}}\})$, $\lambda_k = \lambda = \Theta(\nicefrac{n\sqrt{V}}{\gamma A})$.\newline
    Then to guarantee $\|x^K - x^*\|^2 \leq \varepsilon$ in \textbf{Case 1} and $\gap_{\sqrt{V}}(x_{\text{avg}}^K) = \max_{y \in B_{\sqrt{V}}(x^*)}\left\{\langle F(y), x_{\text{avg}}^K - y \rangle  + \Psi(x_{\text{avg}}^K) - \Psi(y)\right\} \leq \varepsilon$ in \textbf{Case 2} with $x_{\text{avg}}^K = \frac{1}{K+1}\sum_{k=0}^K x^k$ with probability $\geq 1 - \beta$ \algname{DProx-clipped-SGDA-shift} requires
    \begin{align}
        \textbf{Case 1:}&\quad \widetilde\cO\left(\max\left\{\frac{\ell}{\mu}, \frac{1}{n}\left(\frac{\sigma^2}{\mu^2\varepsilon}\right)^{\frac{\alpha}{2(\alpha-1)}}\right\}\right)& \text{iterations/oracle calls per worker}, \label{eq:d_prox_clipped_SGDA_shift_main_result_QSM}\\
        \textbf{Case 2:}&\quad \widetilde\cO\left(\max\left\{\frac{\ell V}{\varepsilon}, \frac{1}{n}\left(\frac{\sigma \sqrt{V}}{\varepsilon}\right)^{\frac{\alpha}{\alpha-1}}\right\}\right)&  \text{iterations/oracle calls per worker}. \label{eq:d_prox_clipped_SGDA_shift_main_result_mon}
    \end{align}
\end{theorem}

As in the case of minimization, in the single-node case, the derived results coincide with ones known for \algname{clipped-SGD} in the unconstrained case \citep{sadiev2023high} Up to the difference between $V$ and $\|x^0 - x^*\|^2$. In the distributed case, we also observe the benefits of parallelization.

\subsection{DProx-clipped-SEG-shift}

Finally, we propose a distributed version of \algname{clipped-SEG} for composite VIPs (\algname{DProx-clipped-SEG-shift}):
\begin{gather}
    \tx^k = \prox_{\gamma\Psi}\left(x^k - \gamma \tg^k\right),\;\; \tg^k = \frac{1}{n}\sum\limits_{i=1}^n \tg_i^k,\;\; \tg_i^k = \thh_{i}^k + \tilde\Delta_i^k,\;\; \thh_i^{k+1} = \thh_i^k + \nu \tilde \Delta_i^k\label{eq:clipped_SEG_extr_step}\\
    x^{k+1} = \prox_{\gamma\Psi}\left(x^k - \gamma \hg^k\right),\;\; \hg^k = \frac{1}{n}\sum\limits_{i=1}^n \hg_i^k,\;\; \hg_i^k = \hh_{i}^k + \hat\Delta_i^k,\;\; \hh_i^{k+1} = \hh_i^k + \nu \hat \Delta_i^k \label{eq:clipped_SEG_upd_step}
\vspace{-0.3cm}
\end{gather}
where $\tilde\Delta_i^k = \clip(F_{\xi_{1,i}^k}(x^k) - \thh_i^k, \lambda_k)$, $\hat\Delta_i^k = \clip(F_{\xi_{2,i}^k}(\tx^k) - \hh_i^k, \lambda_k)$ and $\xi_{1,1}^k,\ldots, \xi_{1,n}^k, \xi_{2,1}^k,\ldots, \xi_{2,n}^k$ are sampled independently from each other and previous steps. For the proposed method, we derive the following result.
\begin{theorem}[Convergence of \algname{DProx-clipped-SEG-shift}]\label{thm:D_prox_clipped_SEG_main}
    Let $K \geq 1$, $\beta \in (0,1)$, $A = \ln \frac{48n(K+1)}{\beta}$, $V \geq \|x^0 - x^*\|^2 + \frac{409600\gamma^2 A^2}{n^2}\sum_{i=1}^n \|F_i(x^*)\|^2$.\newline
    \textbf{Case 1.} Let Assumptions~\ref{as:bounded_alpha_moment}, \ref{as:L_Lip}, and \ref{as:str_monotonicity} with $\mu > 0$ hold for $Q = B_{3\sqrt{V}}(x^*)$. Assume that $\nu = \gamma\mu$, $0 < \gamma = \cO\left(\min\{\nicefrac{1}{\mu A^2}, \nicefrac{1}{L}, \nicefrac{\sqrt{n}}{L A}, \nicefrac{\ln(B_K)}{\mu(K+1)}\}\right)$, $B_K = \Theta\left(\max\{2, \nicefrac{(K+1)^{\nicefrac{2(\alpha-1)}{\alpha}}\mu^2 n^{\nicefrac{2(\alpha-1)}{\alpha}} V}{\sigma^2 A^{\nicefrac{2(\alpha-1)}{\alpha}}\ln^2(B_K)}\}\right)$, $\lambda_k = \Theta(\nicefrac{n\exp(-\gamma\mu(1+\nicefrac{k}{4}))\sqrt{V}}{\gamma A})$. \newline
    \textbf{Case 2.} Let Assumptions~\ref{as:bounded_alpha_moment}, \ref{as:L_Lip}, and \ref{as:monotonicity} hold for $Q = B_{4n\sqrt{V}}(x^*)$. Assume that $\nu = 0$, $0 < \gamma = \cO(\min\{\nicefrac{1}{L A}, \nicefrac{n^{\nicefrac{(\alpha-1)}{\alpha}}\sqrt{V}}{\sigma K^{\nicefrac{1}{\alpha}}A^{\nicefrac{(\alpha-1)}{\alpha}}}\})$, $\lambda_k = \lambda = \Theta(\nicefrac{n\sqrt{V}}{\gamma A})$.\newline
    Then to guarantee $\|x^K - x^*\|^2 \leq \varepsilon$ in \textbf{Case 1} and $\gap_{\sqrt{V}}(\tx_{\text{avg}}^K) = \max_{y \in B_{\sqrt{V}}(x^*)}\left\{\langle F(y), \tx_{\text{avg}}^K - y \rangle + \Psi(\tx_{\text{avg}}^K) - \Psi(y)\right\} \leq \varepsilon$ in \textbf{Case 2} with $\tx_{\text{avg}}^K = \frac{1}{K+1}\sum_{k=0}^K \tx^k$ with probability $\geq 1 - \beta$ \algname{DProx-clipped-SEG-shift} requires
    \begin{align}
        \textbf{Case 1:}&\quad \widetilde\cO\left(\max\left\{\frac{L}{\mu}, \frac{1}{n}\left(\frac{\sigma^2}{\mu^2\varepsilon}\right)^{\frac{\alpha}{2(\alpha-1)}}\right\}\right)& \text{iterations/oracle calls per worker}, \label{eq:d_prox_clipped_SEG_shift_main_result_QSM}\\
        \textbf{Case 2:}&\quad \widetilde\cO\left(\max\left\{\frac{L V}{\varepsilon}, \frac{1}{n}\left(\frac{\sigma \sqrt{V}}{\varepsilon}\right)^{\frac{\alpha}{\alpha-1}}\right\}\right)&  \text{iterations/oracle calls per worker}. \label{eq:d_prox_clipped_SEG_shift_main_result_mon}
    \end{align}
\end{theorem}

The main properties of the above result are similar to the ones of the result for \algname{DProx-clipped-SGDA-shift}. The only difference is that the methods (\algname{DProx-clipped-SGDA/SEG-shift}) are analyzed for different classes of problems and, thus, complement each other. According to the known lower bounds, our upper bound \eqref{eq:d_prox_clipped_SEG_shift_main_result_QSM} has optimal dependence on $\varepsilon$ up to logarithmic factors.

\clearpage

\section{Missing Proofs for \algname{Prox-clipped-SGD-star}}\label{appendix:prox_clipped_SGD_star}

This section provides the complete formulations of our results for \algname{Prox-clipped-SGD-star} and rigorous proofs. We start with the following result -- a generalization of Lemma E.7 from \citep{sadiev2023high} to the composite distributed problems.
\begin{lemma}\label{lem:QSC_to_QSM}
    Consider differentiable function $f: \R^d \to \R$ having a finite-sum structure \eqref{eq:distributed_problem}. If $f$ satisfies Assumption~\ref{as:QSC} on some set $Q$ with parameter $\mu$ and $D_f(x,x^*) \geq 0$ for all\footnote{For example $D_f(x,x^*) \geq 0$ when $f$ is convex or when $\Psi(x) = 0$. We notice that Assumption~\ref{as:L_smoothness} implies $D_f(x,x^*)\geq 0$ since the right-hand side of \eqref{eq:L_smoothness_cor_2} equals $D_f(x,x^*)$ after averaging.} $x\in Q$, then operator $F(x) = \nabla f(x)$ satisfies Assumption~\ref{as:str_monotonicity} on $Q$ with parameter $\nicefrac{\mu}{2}$. If $f_1,\ldots,f_n$ satisfy Assumption~\ref{as:L_smoothness} and \ref{as:QSC} with $\mu = 0$ on some set $Q$, then operator $F(x) = \nabla f(x)$ satisfies Assumption~\ref{as:star_cocoercivity} on $Q$ with $\ell = 2L$.
\end{lemma}
\begin{proof}
    Let Assumption~\ref{as:QSC} hold on some set $Q$ and $D_f(x,x^*) \geq 0$ for all $x\in Q$. Then, averaging inequalities \eqref{eq:QSC}, we get that for all $x \in Q$
    \begin{equation*}
        f(x^*) \geq f(x) + \langle \nabla f(x), x^* - x \rangle+ \frac{\mu}{2}\|x - x^*\|^2, 
    \end{equation*}
    implying for $F(x) = \nabla f(x)$ that
    \begin{eqnarray*}
        \langle F(x) - F(x^*), x  - x^* \rangle &\geq& D_f(x,x^*) + \frac{\mu}{2}\|x - x^*\|^2 \geq \frac{\mu}{2}\|x - x^*\|^2,
    \end{eqnarray*}
    meaning that Assumption~\ref{as:str_monotonicity} is satisfied with parameter $\nicefrac{\mu}{2}$.

    It remains to show the second part of the lemma. Let Assumptions~\ref{as:L_smoothness} and \ref{as:QSC} with $\mu = 0$ hold on some set $Q$. We need to show that operators $F_i(x) = \nabla f_i(x)$, $i=1,\ldots,n$ satisfy Assumption~\ref{as:star_cocoercivity} on $Q$ with $\ell = 2L$. Guided by \citep[Lemma C.6]{gorbunov2022extragradient} and \citep[Lemma E.7]{sadiev2023high}, we derive
    \begin{eqnarray}
        \left\|x - x^* -\frac{1}{L}(F_i(x) - F_i(x^*))\right\|^2 &=& \|x - x^*\|^2 - \frac{2}{L}\langle x - x^*, F_i(x) - F_i(x^*) \rangle\notag\\
        &&\quad + \frac{1}{L^2}\|F_i(x) - F_i(x^*)\|^2 \label{eq:cjsidcncnvdf}\\
        &=& \|x - x^*\|^2 - \frac{2}{L}\langle x - x^*, \nabla f_i(x) - \nabla f_i(x^*) \rangle\notag\\
        &&\quad + \frac{1}{L^2}\|\nabla f_i(x) - \nabla f_i(x^*)\|^2 \notag\\
        &\overset{\eqref{eq:L_smoothness_cor_2}}{\leq}&  \|x - x^*\|^2 - \frac{2}{L} \langle x - x^*, \nabla f_i(x) \rangle \notag\\
        &&\quad + \frac{2}{L}\left(f_i(x) - f_i(x^*)\right) \notag\\
        &\overset{\eqref{eq:QSC}}{\leq}& \|x - x^*\|^2. \label{eq:vhjdfbhjvdfbjv}
    \end{eqnarray}
    From \eqref{eq:cjsidcncnvdf} and \eqref{eq:vhjdfbhjvdfbjv} we get
    \begin{equation*}
        \|x - x^*\|^2 - \frac{2}{L}\langle x - x^*, F_i(x) - F_i(x^*) \rangle + \frac{1}{L^2}\|F_i(x) - F_i(x^*)\|^2 \leq \|x - x^*\|^2
    \end{equation*}
    that is equivalent to \eqref{eq:star_cocoercivity} with $\ell = 2L$.
\end{proof}

Therefore, for smooth quasi-strongly convex $f$ such that $D_f(x,x^*)\geq 0$ ($n =1$) we can consider operator $F(x) = \nabla f(x)$ and VI formulation instead. In this case, the method is equivalent to \algname{Prox-clipped-SGDA-star}:
\begin{gather*}
    x^{k+1} = \proxkPsi \left(x^k - \gamma \tg^k\right),\;\; \tg^k = F(x^*) + \clip\left(F_{\xi^k}(x^k) - F(x^*), \lambda_k\right)\\
    \hg^k = \clip\left(F_{\xi^k}(x^k) - F(x^*), \lambda_k\right).
\end{gather*}

The following lemma is the main ``optimization'' part of the analysis of \algname{Prox-clipped-SGDA-star}.

\begin{lemma}\label{lem:optimization_lemma_str_mon_SGDA}
    Let $n=1$, Assumptions~\ref{as:str_monotonicity}, \ref{as:star_cocoercivity} hold for $Q = B_{2R}(x^*)$, where $R \geq R_0 \eqdef \|x^0 - x^*\|$, and $0 < \gamma \leq \nicefrac{1}{\ell}$. If $x^k$ lies in $B_{2R}(x^*)$ for all $k = 0,1,\ldots, K$ for some $K\geq 0$, then the iterates produced by \algname{Prox-clipped-SGDA-star} satisfy
    \begin{eqnarray}
        \|x^{K+1} - x^*\|^2 &\leq& (1 - \gamma \mu)^{K+1}\|x^0 - x^*\|^2 + \gamma^2 \sum\limits_{k=0}^K (1-\gamma\mu)^{K-k} \|\omega_k\|^2\notag\\
        &&\quad + 2\gamma \sum\limits_{k=0}^K (1-\gamma\mu)^{K-k}\langle x^k - x^* - \gamma (F(x^k)-F(x^*)), \omega_k \rangle, \label{eq:optimization_lemma_prox_SGDA_str_mon}\\
        \omega_k &\eqdef& F(x^k) - F(x^*)-\hg^k. \label{eq:omega_k_prox_SGDA}
    \end{eqnarray}
\end{lemma}
\begin{proof}
     Using the update rule of \algname{Prox-clipped-SGDA-star}, we obtain
     \begin{eqnarray*}
        \|x^{k+1} - x^*\|^2 &=& \|\proxkPsi\left(x^k - \gamma\tg^k\right) - \proxkPsi\left(x^* - \gamma F(x^*)\right)\|^2\\
        &\le& \|x^k - x^* - \gamma(\tg^k - F(x^*))\|^2\\
        &=& \|x^k - x^*\|^2 -2\gamma \langle x^k - x^*, \hg^k \rangle + \gamma^2\|\hg^k\|^2\\
        &\overset{\eqref{eq:omega_k_prox_SGDA}}{=}& \|x^k - x^*\|^2 - 2\gamma \langle x^k - x^*, F(x^k)-F(x^*) \rangle - 2\gamma^2\langle F(x^k) - F(x^*), \omega_k\rangle\\
        &&\quad  + 2\gamma \langle x^k - x^*, \omega_k \rangle  + \gamma^2\|F(x^k) - F(x^*)\|^2 + \gamma^2\|\omega_k\|^2\\
        &\overset{\eqref{eq:star_cocoercivity}}{\leq}&  \|x^k - x^*\|^2 + 2\gamma \langle x^k - x^*, \omega_k \rangle - 2 \gamma^2 \langle F(x^k) - F(x^*), \omega_k\rangle\\
        &&\quad  - 2\gamma \left(1 - \frac{\gamma \ell}{2}\right) \langle x^k - x^*, F(x^k)  -F(x^*)\rangle + \gamma^2\|\omega_k\|^2\\
        &\overset{\eqref{eq:str_monotonicity}, \gamma \leq \frac{1}{\ell}}{\leq}&\|x^k - x^*\|^2 + 2\gamma \langle x^k - x^* - \gamma (F(x^k) - F(x^*)), \omega_k \rangle\\
        &&\quad  - 2\gamma \mu \left(1 - \frac{\gamma \ell}{2}\right) \|x^k - x^*\|^2 + \gamma^2\|\omega_k\|^2\\
        &\overset{\gamma \leq \frac{1}{\ell}}{\leq}&(1-\gamma \mu)\|x^k - x^*\|^2 + 2\gamma \langle x^k - x^* - \gamma (F(x^k) - F(x^*)), \omega_k \rangle + \gamma^2\|\omega_k\|^2.
    \end{eqnarray*}
    Unrolling the recurrence, we obtain \eqref{eq:optimization_lemma_prox_SGDA_str_mon}.
\end{proof}

\begin{theorem}\label{thm:main_result_prox_clipped_SGDA_star}
Let $n =1$, Assumptions~\ref{as:str_monotonicity}, \ref{as:star_cocoercivity}, hold for $Q = B_{2R}(x^*) = \{x\in\R^d\mid \|x - x^*\| \leq 2R\}$ for any $x \in B_{2R}(x^*)$, where $R \geq \|x^0 - x^*\|$, and
    \begin{eqnarray}
        0< \gamma &\leq& \min\left\{\frac{1}{400 \ell\ln \tfrac{4(K+1)}{\beta}}, \frac{\ln(B_K)}{\mu(K+1)}\right\}, \label{eq:gamma_SGDA_prox_clipped}\\
        B_K &=& \max\left\{2, \frac{(K+1)^{\frac{2\alpha-1}{\alpha}}\mu^2R^2}{4\cdot 10^{\frac{1}{\alpha}}120^{\frac{2(\alpha-1)}{\alpha}}\sigma^2\ln^{\frac{2(\alpha-1)}{\alpha}}\left(\frac{4(K+1)}{\beta}\right)\ln^2(B_K)} \right\}  \label{eq:B_K_SGDA_prox_clipped_1} \\
        &=& \cO\left(\max\left\{2, \frac{K^{\frac{2\alpha-1}{\alpha}}\mu^2R^2}{\sigma^2\ln^{\frac{2(\alpha-1)}{\alpha}}\left(\frac{K}{\beta}\right)\ln^2\left(\max\left\{2, \frac{K^{\frac{2\alpha-1}{\alpha}}\mu^2R^2}{\sigma^2\ln^{\frac{2(\alpha-1)}{\alpha}}\left(\frac{K}{\beta}\right)} \right\}\right)} \right\}\right\},  \label{eq:B_K_SGDA_prox_clipped_2} \\
        \lambda_k &=& \frac{\exp(-\gamma\mu(1 + \nicefrac{k}{2}))R}{120\gamma \ln \tfrac{4(K+1)}{\beta}}, \label{eq:lambda_SGDA_prox_clipped}
    \end{eqnarray}
    for some $K \geq 0$ and $\beta \in (0,1]$ such that $\ln \tfrac{4(K+1)}{\beta} \geq 1$. Then, after $K$ iterations the iterates produced by \algname{Prox-clipped-SGDA-star} with probability at least $1 - \beta$ satisfy 
    \begin{equation}
        \|x^{K+1} - x^*\|^2 \leq 2\exp(-\gamma\mu(K+1))R^2. \label{eq:main_result_str_mon_SGDA}
    \end{equation}
    In particular, when $\gamma$ equals the minimum from \eqref{eq:gamma_SGDA_prox_clipped}, then the iterates produced by \algname{Prox-clipped-SGDA-star} after $K$ iterations with probability at least $1-\beta$ satisfy
    \begin{eqnarray}
       R_K^2 = \cO\left(\max\left\{R^2\exp\left(- \frac{\mu K}{\ell \ln \tfrac{K}{\beta}}\right), \frac{\sigma^2\ln^{\frac{2(\alpha-1)}{\alpha}}\left(\frac{K}{\beta}\right)\ln^2\left(\max\left\{2, \frac{K^{\frac{2\alpha-1}{\alpha}}\mu^2R^2}{\sigma^2\ln^{\frac{2(\alpha-1)}{\alpha}}\left(\frac{K}{\beta}\right)} \right\}\right)}{K^{\frac{2\alpha-1}{\alpha}}\mu^2}\right\}\right),  \label{eq:clipped_SGDA_prox_clipped_case_1}
    \end{eqnarray}
    meaning that to achieve $R_K^2 = \|x^{K} - x^*\|^2 \leq \varepsilon$ with probability at least $1 - \beta$ \algname{Prox-clipped-SGDA-star} requires
    \begin{equation}
        K = \cO\left(\frac{\ell}{\mu}\ln\left(\frac{R^2}{\varepsilon}\right)\ln\left(\frac{\ell}{\mu \beta}\ln\frac{R^2}{\varepsilon}\right), \left(\frac{\sigma^2}{\mu^2\varepsilon}\right)^{\frac{\alpha}{2\alpha-1}}\ln \left(\frac{1}{\beta} \left(\frac{\sigma^2}{\mu^2\varepsilon}\right)^{\frac{\alpha}{2\alpha-1}}\right)\ln^{\frac{\alpha}{\alpha-1}}\left(B_\varepsilon\right)\right)  \label{eq:clipped_SGDA_prox_clipped_case_complexity_appendix}
    \end{equation}
    iterations/oracle calls, where
    \begin{equation*}
        B_\varepsilon = \max\left\{2, \frac{R^2}{\varepsilon \ln \left(\frac{1}{\beta} \left(\frac{\sigma^2}{\mu^2\varepsilon}\right)^{\frac{\alpha}{2\alpha-1}}\right)}\right\}.
    \end{equation*}
\end{theorem}
\begin{proof}
    Let $R_k = \|x^k - x^*\|$ for all $k\geq 0$. Our proof is induction-based: by induction, we show that the iterates of the method stay in some ball around the solution with high probability. To formulate the statement rigorously, we introduce probability event $E_k$ for each $k = 0,1,\ldots, K+1$ as follows: inequalities
    \begin{equation}
        R_t^2 \leq 2 \exp(-\gamma\mu t) R^2 \label{eq:induction_inequality_prox_clipped_SGDA_star}
    \end{equation}
    hold for $t = 0,1,\ldots,k$ simultaneously. We will prove by induction that $\PP\{E_k\} \geq  1 - \nicefrac{k\beta}{(K+1)}$ for all $k = 0,1,\ldots,K+1$. The base of the induction follows immediately by the definition of $R$. Next, assume that for $k = T-1 \leq K$ the statement holds: $\PP\{E_{T-1}\} \geq  1 - \nicefrac{(T-1)\beta}{(K+1)}$. Given this, we need to prove $\PP\{E_{T}\} \geq  1 - \nicefrac{T\beta}{(K+1)}$. Since $R_t^2 \leq 2\exp(-\gamma\mu t) R^2 \leq 2R^2$, we have $x^t \in B_{2R}(x^*)$ for $t = 0,1,\ldots, T-1$, where operator $F$ is $\ell$-star-cocoercive. Thus, $E_{T-1}$ implies
    \begin{eqnarray}
        \|F(x^t) - F(x^*)\| &\leq& \ell\|x^t - x^*\| \overset{\eqref{eq:induction_inequality_prox_clipped_SGDA_star}}{\leq} \sqrt{2}\ell\exp(- \nicefrac{\gamma\mu t}{2})R \overset{\eqref{eq:gamma_SGDA_prox_clipped},\eqref{eq:lambda_SGDA_prox_clipped}}{\leq} \frac{\lambda_t}{2} \label{eq:operator_bound_x_t_SGDA_prox_clipped}
    \end{eqnarray}
    and
    \begin{eqnarray}
        \|\omega_t\|^2 &\leq& 2\|F(x^t)-F(x^*)\|^2 + 2\|\hg^t\|^2\overset{\eqref{eq:operator_bound_x_t_SGDA_prox_clipped}}{\leq} \frac{5}{2}\lambda_t^2 \overset{\eqref{eq:lambda_SGDA_prox_clipped}}{\leq} \frac{\exp(-\gamma\mu t)R^2}{4\gamma^2} \label{eq:omega_bound_x_t_SGDA_prox_clipped}
    \end{eqnarray}
    for all $t = 0, 1, \ldots, T-1$, where we use that $\|a+b\|^2 \leq 2\|a\|^2 + 2\|b\|^2$ holding for all $a,b \in \R^d$. 
    This means that we can apply Lemma~\ref{lem:optimization_lemma_str_mon_SGDA} and $(1 - \gamma\mu)^T \leq \exp(-\gamma\mu T)$: $E_{T-1}$ implies
    \begin{eqnarray}
        R_T^2 &\leq& \exp(-\gamma\mu T)R^2 + 2\gamma \sum\limits_{t=0}^{T-1} (1-\gamma\mu)^{T-1-t} \langle x^t - x^* - \gamma (F(x^t)-F(x^*)), \omega_t \rangle\notag\\
        &&\quad + \gamma^2 \sum\limits_{t=0}^{T-1} (1-\gamma\mu)^{T-1-t}\|\omega_t\|^2. \notag
    \end{eqnarray}
    Before we proceed, we introduce a new notation:
    \begin{gather}
        \eta_t = \begin{cases} \underbrace{x^t - x^* - \gamma (F(x^t)-F(x^*))}_{\hat{\eta}_t},& \text{if } \|\hat{\eta}_t\| \leq \sqrt{2}(1 + \gamma \ell) \exp(- \nicefrac{\gamma\mu t}{2})R,\\ 0,& \text{otherwise}, \end{cases} \label{eq:eta_t_SGDA_str_mon}
    \end{gather}
    for $t = 0, 1, \ldots, T-1$. Random vectors $\{\eta_t\}_{t=0}^T$ are bounded almost surely:
    \begin{equation}
         \|\eta_t\| \leq \sqrt{2}(1 + \gamma \ell)\exp(-\nicefrac{\gamma\mu t}{2})R \label{eq:eta_t_bound_SGDA_prox_clipped} 
    \end{equation}
    for all $t = 0, 1, \ldots, T-1$. We also notice that $E_{T-1}$ implies $\|F(x^t)-F(x^*)\| \leq \sqrt{2}\ell\exp(-\nicefrac{\gamma\mu t}{2})R$ (due to \eqref{eq:operator_bound_x_t_SGDA_prox_clipped}) and
    \begin{eqnarray*}
        \|x^t - x^* - \gamma (F(x^t)-F(x^*))\| &\leq& \|x^t - x^*\| + \gamma \|F(x^t)-F(x^*)\|\\
        &\overset{\eqref{eq:operator_bound_x_t_SGDA_prox_clipped}}{\leq}& \sqrt{2}(1 + \gamma \ell)\exp(-\nicefrac{\gamma\mu t}{2})R
    \end{eqnarray*}
    for $t = 0, 1, \ldots, T-1$. Therefore, $E_{T-1}$ implies  $\eta_t = x^t - x^* - \gamma (F(x^t)-F(x^*))$ for all $t = 0,1,\ldots,T-1$ and from $E_{T-1}$ it follows that
    \begin{eqnarray}
        R_T^2 &\leq& \exp(-\gamma\mu T)R^2 + 2\gamma \sum\limits_{t=0}^{T-1} (1-\gamma\mu)^{T-1-t} \langle \eta_t, \omega_t \rangle\notag\\
        &&\quad + \gamma^2 \sum\limits_{t=0}^{T-1} (1-\gamma\mu)^{T-1-t} \|\omega_t\|^2. \notag
    \end{eqnarray}
    For convenience, we define unbiased and biased parts of $ \omega_t$:
    \begin{gather}
        \omega_t^u \eqdef \EE_{\xi^t}\left[\hg^t\right] - \hg^t,\quad \omega_t^b \eqdef F(x^t) - F(x^*) - \EE_{\xi^t}\left[\hg^t\right], \label{eq:omega_unbias_bias_SGDA_str_mon}
    \end{gather}
    for all $t = 0,\ldots, T-1$. By definition we have  $\omega_t = \omega_t^u + \omega_t^b$ for all $t = 0,\ldots, T-1$. Therefore, $E_{T-1}$ implies
    \begin{eqnarray}
        R_T^2 &\leq& \exp(-\gamma\mu T) R^2 + \underbrace{2\gamma \sum\limits_{t=0}^{T-1} (1-\gamma\mu)^{T-1-t} \langle \eta_t, \omega_t^u \rangle}_{\circledOne} \notag\\ &&\quad + \underbrace{2\gamma \sum\limits_{t=0}^{T-1} (1-\gamma\mu)^{T-1-t} \langle \eta_t, \omega_t^b \rangle}_{\circledTwo} 
        + \underbrace{2\gamma^2 \sum\limits_{t=0}^{T-1} (1-\gamma\mu)^{T-1-t} \EE_{\xi^t}\left[\|\omega^u_t\|^2\right]}_{\circledThree}\notag\\
        &&\quad + \underbrace{2\gamma^2 \sum\limits_{t=0}^{T-1} (1-\gamma\mu)^{T-1-t}\left( \|\omega^u_t\|^2 -  \EE_{\xi^t}\left[\|\omega^u_t\|^2\right]\right)}_{\circledFour}\notag\\
        &&\quad + \underbrace{2\gamma^2 \sum\limits_{t=0}^{T-1} (1-\gamma\mu)^{T-1-t} \|\omega^b_t\|^2}_{\circledFive}. \label{eq:SGDA_prox_clipped_12345_bound}
    \end{eqnarray}
    where we also use inequality $\|a+b\|^2 \leq 2\|a\|^2 + 2\|b\|^2$ holding for all $a,b \in \R^d$ to upper bound $\|\omega_t\|^2$. To derive high-probability bounds for $\circledOne, \circledTwo, \circledThree, \circledFour, \circledFive$ we need to establish several useful inequalities related to $\omega_{i,t}^u, \omega_{i,t}^b$. First, by definition of clipping
    \begin{equation}
        \|\omega_t^u\| \leq 2\lambda_t.\label{eq:omega_magnitude_prox_clipped}
    \end{equation}
    Next, $E_{T-1}$ implies that $\|F(x^t)-F(x^*)\| \leq \nicefrac{\lambda_t}{2}$ for all $t = 0,1, \ldots, T-1$ (see \eqref{eq:operator_bound_x_t_SGDA_prox_clipped}). Therefore, from Lemma~\ref{lem:bias_and_variance_clip} we also have that $E_{T-1}$ implies
    \begin{gather}
        \left\|\omega_t^b\right\| \leq \frac{2^{\alpha}\sigma^{\alpha}}{\lambda_t^{\alpha-1}}, \label{eq:bias_omega_prox_clipped}\\
        \EE_{\xi^t}\left[\left\|\omega_t^b\right\|^2\right] \leq 18 \lambda_t^{2-\alpha} \sigma^{\alpha}, \label{eq:distortion_omega_prox_clipped}\\
        \EE_{\xi^t}\left[\left\|\omega_t^u\right\|^2\right] \leq 18 \lambda_t^{2-\alpha} \sigma^{\alpha}, \label{eq:variance_omega_prox_clipped}
    \end{gather}
    for all $t = 0,1, \ldots, T-1$.

\paragraph{Upper bound for $\circledOne$.} To estimate this sum, we will use Bernstein's inequality. The summands have conditional expectations equal to zero:
    \begin{equation*}
        \EE_{\xi^t}\left[2\gamma (1-\gamma\mu)^{T-1-t} \langle \eta_t, \omega_t^u \rangle\right] = 0.
    \end{equation*}
    Next, the summands are bounded:
    \begin{eqnarray}
        |2\gamma (1-\gamma\mu)^{T-1-t} \langle \eta_t, \omega_t^u \rangle | &\leq& 2\gamma\exp(-\gamma\mu (T - 1 - t)) \|\eta_t\|\cdot \|\omega_t^u\|\notag\\
        &\overset{\eqref{eq:eta_t_bound_SGDA_prox_clipped},\eqref{eq:omega_magnitude_prox_clipped}}{\leq}& 4\sqrt{2}\gamma (1 + \gamma \ell) \exp(-\gamma\mu (T - 1 - \nicefrac{t}{2})) R \lambda_t\notag\\
        &\overset{\eqref{eq:gamma_SGDA_prox_clipped},\eqref{eq:lambda_SGDA_prox_clipped}}{\leq}& \frac{\exp(-\gamma\mu T)R^2}{5\ln\tfrac{4(K+1)}{\beta}} \eqdef c. \label{eq:SGDA_str_mon_technical_1_1}
    \end{eqnarray}
Finally, conditional variances $\sigma_t^2 \eqdef \EE_{\xi^t}\left[4\gamma^2 (1-\gamma\mu)^{2T-2-2t} \langle \eta_t, \omega_t^u \rangle^2\right]$ of the summands are bounded:
    \begin{eqnarray}
        \sigma_t^2 &\leq& \EE_{\xi^t}\left[4\gamma^2\exp(-\gamma\mu (2T - 2 - 2t)) \|\eta_t\|^2\cdot \|\omega_t^u\|^2\right]\notag\\
        &\overset{\eqref{eq:eta_t_bound_SGDA_prox_clipped}}{\leq}& 8\gamma^2 (1 + \gamma \ell)^2 \exp(-\gamma\mu (2T - 2 - t)) R^2 \EE_{\xi^t}\left[\|\omega_t^u\|^2\right]\notag\\
        &\overset{\eqref{eq:gamma_SGDA_prox_clipped}}{\leq}& 10\gamma^2\exp(-\gamma\mu (2T - t))R^2 \EE_{\xi^t}\left[\|\omega_t^u\|^2\right]. \label{eq:SGDA_str_mon_technical_1_2}
    \end{eqnarray}
Applying Bernstein's inequality (Lemma~\ref{lem:Bernstein_ineq}) with $X_t = 2\gamma (1-\gamma\mu)^{T-1-t} \langle \eta_t, \omega_t^u \rangle$, constant $c$ defined in \eqref{eq:SGDA_str_mon_technical_1_1}, $b = \tfrac{1}{5}\exp(-\gamma\mu T) R^2$, $G = \tfrac{\exp(-2 \gamma\mu T) R^4}{150\ln\frac{4(K+1)}{\beta}}$, we get
    \begin{eqnarray*}
        \PP\left\{|\circledOne| > \frac{1}{5}\exp(-\gamma\mu T) R^2 \text{ and } \sum\limits_{t=0}^{T-1}\sigma_t^2 \leq \frac{\exp(- 2\gamma\mu T) R^4}{150\ln\tfrac{4(K+1)}{\beta}}\right\} &\leq& 2\exp\left(- \frac{b^2}{2F + \nicefrac{2cb}{3}}\right)\\
        &=& \frac{\beta}{2(K+1)}.
    \end{eqnarray*}
The above is equivalent to $\PP\{E_{\circledOne}\} \geq 1 - \frac{\beta}{2(K+1)}$ for
    \begin{equation}
        E_{\circledOne} = \left\{\text{either} \quad \sum\limits_{t=0}^{T-1}\sigma_t^2 > \frac{\exp(- 2\gamma\mu T) R^4}{150\ln\tfrac{4(K+1)}{\beta}}\quad \text{or}\quad |\circledOne| \leq \frac{1}{5}\exp(-\gamma\mu T) R^2\right\}. \label{eq:bound_1_SGDA_prox_clipped}
    \end{equation}
Moreover, $E_{T-1}$ implies
    \begin{eqnarray}
        \sum\limits_{t=0}^{T-1}\sigma_t^2 &\overset{\eqref{eq:SGDA_str_mon_technical_1_2}}{\leq}& 10\gamma^2\exp(- 2\gamma\mu T)R^2\sum\limits_{t=0}^{T-1} \frac{\EE_{\xi^t}\left[\|\omega_t^u\|^2\right]}{\exp(-\gamma\mu t)}\notag\\ 
        &\overset{\eqref{eq:variance_omega_prox_clipped}, T \leq K+1}{\leq}& 180\gamma^2\exp(-2\gamma\mu T) R^2 \sigma^2 \sum\limits_{t=0}^{K} \frac{\lambda_t^{2-\alpha}}{\exp(-\gamma\mu t)}\notag\\
        &\overset{\eqref{eq:lambda_SGDA_prox_clipped}}{\leq}&\frac{ 180\gamma^{\alpha}\exp(-2\gamma\mu T) R^{4-\alpha} \sigma^{\alpha} (K+1)\exp(\frac{\gamma \mu \alpha K}{2})}{120^{2-\alpha}\ln^{2-\alpha}\frac{4(K+1)}{\beta}}\notag\\
        &\overset{\eqref{eq:gamma_SGDA_prox_clipped}}{\leq}& \frac{\exp(-2\gamma\mu T)R^4}{150\ln\tfrac{4(K+1)}{\beta}}. \label{eq:bound_1_variances_SGDA_prox_clipped}
    \end{eqnarray}

 \paragraph{Upper bound for $\circledTwo$.} Probability event $E_{T-1}$ implies
    \begin{eqnarray}
        \circledTwo &\leq& 2\gamma \exp(-\gamma\mu (T-1)) \sum\limits_{t=0}^{T-1} \frac{\|\eta_t\|\cdot \|\omega_t^b\|}{\exp(-\gamma\mu t)}\notag\\
        &\overset{\eqref{eq:eta_t_bound_SGDA_prox_clipped}, \eqref{eq:bias_omega_prox_clipped}}{\leq}& 2^{1+\alpha}\sqrt{2} \gamma (1+\gamma \ell) \exp(-\gamma\mu (T-1)) R \sigma^{\alpha} \sum\limits_{t=0}^{T-1} \frac{1}{\lambda_t^{\alpha-1} \exp(-\nicefrac{\gamma\mu t}{2})}\notag\\
        &\overset{\eqref{eq:lambda_SGDA_prox_clipped},  T \leq K+1}{\leq}& \frac{2^{1+\alpha}120^{\alpha-1}\sqrt{2} \gamma^{\alpha}\sigma^{\alpha} R^{2-\alpha} (1+\gamma \ell)\exp(-\gamma\mu (T-1)) (K+1) \exp\left(\frac{\gamma\mu \alpha K}{2} \right)}{\ln^{1-\alpha}\tfrac{4(K+1)}{\beta}} \notag \\
        &\overset{\eqref{eq:gamma_SGDA_prox_clipped}}{\leq}& \frac{1}{5}\exp(-\gamma\mu T) R^2. \label{eq:bound_2_SGDA_prox_clipped}
    \end{eqnarray}

\paragraph{Upper bound for $\circledThree$.} Probability event $E_{T-1}$ implies
    \begin{eqnarray}
        \circledThree &=& 2\gamma^2 \exp(-\gamma\mu (T-1)) \sum\limits_{t=0}^{T-1} \frac{\EE_{\xi^t}\left[\|\omega_t^u\|^2\right]}{\exp(-\gamma\mu t)} \notag\\
        &\overset{\eqref{eq:variance_omega_prox_clipped}}{\leq}& 36\gamma^2\exp(-\gamma\mu (T-1)) \sigma^{\alpha}\sum\limits_{t=0}^{T-1} \frac{\lambda^{2-\alpha}_t}{\exp(-\gamma\mu t)} \notag\\
        &\overset{\eqref{eq:lambda_SGDA_prox_clipped},  T \leq K+1}{\leq}&  \frac{36\gamma^{\alpha} R^{2-\alpha} \exp(-\gamma\mu (T-1)) \sigma^{\alpha} (K+1)\exp(\frac{\gamma\mu\alpha K}{2})}{120^{2-\alpha}\ln^{2-\alpha}\frac{4(K+1)}{\beta}} \notag\\
        &\overset{\eqref{eq:gamma_SGDA_prox_clipped}}{\leq}& \frac{1}{5} \exp(-\gamma\mu T) R^2. \label{eq:bound_3_SGDA_prox_clipped}
    \end{eqnarray}

\paragraph{Upper bound for $\circledFour$.} To estimate this sum, we will use Bernstein's inequality. The summands have conditional expectations equal to zero:
    \begin{equation*}
        2\gamma^2 (1-\gamma\mu)^{T-1-t}\EE_{\xi^t}\left[\|\omega_t^u\|^2  -\EE_{\xi^t}\left[\|\omega_t^u\|^2\right] \right] = 0.
    \end{equation*}
    Next, the summands are bounded:
    \begin{eqnarray}
        2\gamma^2 (1-\gamma\mu)^{T-1-t}\left| \|\omega_t^u\|^2  -\EE_{\xi^t}\left[\|\omega_t^u\|^2\right] \right| 
        &\overset{\eqref{eq:omega_magnitude_prox_clipped}}{\leq}& \frac{16\gamma^2 \exp(-\gamma\mu T) \lambda_t^2}{\exp(-\gamma\mu (t+1))}\notag\\
        &\overset{\eqref{eq:lambda_SGDA_prox_clipped}}{\leq}& \frac{\exp(-\gamma\mu T)R^2}{5\ln\tfrac{4(K+1)}{\beta}}\notag\\
        &\eqdef& c. \label{eq:SGDA_prox_clipped_technical_4_1}
    \end{eqnarray}
    Finally, conditional variances
    \begin{equation*}
        \widetilde\sigma_t^2 \eqdef \EE_{\xi^t}\left[4\gamma^4 (1-\gamma\mu)^{2T-2-2t} \left|\|\omega_t^u\|^2  -\EE_{\xi^t}\left[\|\omega_t^u\|^2\right] \right|^2 \right]
    \end{equation*}
    of the summands are bounded:
    \begin{eqnarray}
        \widetilde\sigma_t^2 &\overset{\eqref{eq:SGDA_prox_clipped_technical_4_1}}{\leq}& \frac{2\gamma^2\exp(-2\gamma\mu T)R^2}{5\exp(-\gamma\mu (1+t))\ln\tfrac{4(K+1)}{\beta}} \EE_{\xi^t}\left[ \left|\|\omega_t^u\|^2  -\EE_{\xi^t}\left[\|\omega_t^u\|^2\right] \right|\right]\notag\\
        &\leq& \frac{4\gamma^2\exp(-2\gamma\mu T)R^2}{5\exp(-\gamma\mu (1+t))\ln\tfrac{4(K+1)}{\beta}} \EE_{\xi^t}\left[\|\omega_t^u\|^2\right]. \label{eq:SGDA_str_mon_technical_4_2}
    \end{eqnarray}
    Applying Bernstein's inequality (Lemma~\ref{lem:Bernstein_ineq}) with $X_t = 2\gamma^2 (1-\gamma\mu)^{T-1-t}\left( \|\omega_t^u\|^2 -\EE_{\xi^t}\left[\|\omega_t^u\|^2\right]\right)$, constant $c$ defined in \eqref{eq:SGDA_prox_clipped_technical_4_1}, $b = \tfrac{1}{5}\exp(-\gamma\mu T) R^2$, $G = \tfrac{\exp(-2 \gamma\mu T) R^4}{150\ln\frac{4(K+1)}{\beta}}$, we get:
    \begin{eqnarray*}
        \PP\left\{|\circledFour| > \frac{1}{5}\exp(-\gamma\mu T) R^2 \text{ and } \sum\limits_{l=0}^{T-1}\widetilde\sigma_t^2 \leq \frac{\exp(-2\gamma\mu T) R^4}{150\ln\frac{4(K+1)}{\beta}}\right\} &\leq& 2\exp\left(- \frac{b^2}{2G + \nicefrac{2cb}{3}}\right)\\
        &=& \frac{\beta}{2(K+1)}.
    \end{eqnarray*}
    The above is equivalent to $\PP\{E_{\circledFour}\} \geq 1 - \frac{\beta}{2(K+1)}$ for
    \begin{equation}
        E_{\circledFour} = \left\{\text{either} \quad \sum\limits_{t=0}^{T-1}\widetilde\sigma_t^2 > \frac{\exp(-2\gamma\mu T) R^4}{150\ln\tfrac{4(K+1)}{\beta}}\quad \text{or}\quad |\circledFour| \leq \frac{1}{5}\exp(-\gamma\mu T) R^2\right\}. \label{eq:bound_4_SGDA_prox_clipped}
    \end{equation}
    Moreover, $E_{T-1}$ implies
    \begin{eqnarray}
        \sum\limits_{l=0}^{T-1}\widetilde\sigma_t^2 &\overset{\eqref{eq:SGDA_str_mon_technical_4_2}}{\leq}& \frac{4\gamma^2\exp(-\gamma\mu (2T-1))R^2}{5\ln\tfrac{4(K+1)}{\beta}} \sum\limits_{t=0}^{T-1} \frac{\EE_{\xi^t}\left[\|\omega_l^u\|^2\right]}{\exp(-\gamma\mu t)}\notag\\ &\overset{\eqref{eq:variance_omega_prox_clipped}, T \leq K+1}{\leq}& \frac{72\gamma^2\exp(-\gamma\mu (2T-1)) R^2 \sigma^{\alpha}}{5\ln\tfrac{4(K+1)}{\beta}} \sum\limits_{t=0}^{K} \frac{\lambda_t^{2-\alpha}}{\exp(-\gamma\mu t)}\notag\\
        &\overset{\eqref{eq:lambda_SGDA_prox_clipped}}{\leq}& \frac{72\gamma^{\alpha}\exp(-\gamma\mu (2T-1)) R^{4-\alpha} \sigma^{\alpha} (K+1)\exp(\frac{\gamma\mu\alpha K}{2})}{5\cdot120^{2-\alpha}\ln^{3-\alpha}\tfrac{4(K+1)}{\beta}} \notag\\
        &\overset{\eqref{eq:gamma_SGDA_prox_clipped}}{\leq}& \frac{\exp(-2\gamma\mu T)R^4}{150\ln\tfrac{4(K+1)}{\beta}}. \label{eq:bound_4_variances_SGDA_prox_clipped}
    \end{eqnarray}

\paragraph{Upper bound for $\circledFive$.} Probability event $E_{T-1}$ implies
    \begin{eqnarray}
        \circledFive &=&  2\gamma^2 \sum\limits_{t=0}^{T-1} \exp(-\gamma\mu (T-1-t)) \|\omega_t^b\|^2\notag\\
        &\overset{\eqref{eq:bias_omega_prox_clipped}}{\leq}& 2\cdot 2^{2 \alpha} \gamma^2 \sigma^{2 \alpha}  \exp(-\gamma\mu (T-1)) \sum\limits_{t=0}^{T-1} \frac{1}{\lambda_t^{2\alpha-2} \exp(-\gamma\mu t)} \notag\\
        &\overset{\eqref{eq:lambda_SGDA_prox_clipped}, T \leq K+1}{\leq}& \frac{2\cdot 2^{2\alpha}120^{2\alpha-2}\gamma^{2\alpha} \sigma^{2\alpha} \exp(-\gamma\mu (T-3))  \ln^{2\alpha-2}\tfrac{4(K+1)}{\beta}}{R^{2\alpha-2}} \sum\limits_{t=0}^{K} \exp\left(\gamma\mu\alpha t\right)\notag\\
        &\leq& \frac{2\cdot 2^{2\alpha}120^{2\alpha-2}\gamma^{2\alpha} \sigma^{2\alpha} \exp(-\gamma\mu (T-3))  \ln^{2\alpha-2}\tfrac{4(K+1)}{\beta}(K+1)\exp(\gamma\mu\alpha K)}{R^{2\alpha-2}}\notag\\
        &\overset{\eqref{eq:gamma_SGDA_prox_clipped}}{\leq}& \frac{1}{5}\exp(-\gamma\mu T) R^2. \label{eq:bound_5_SGDA_prox_clipped}
    \end{eqnarray}

    That is, we derive the upper bounds for  $\circledOne, \circledTwo, \circledThree, \circledFour, \circledFive$. More precisely, $E_{T-1}$ implies
    \begin{gather*}
        R_T^2 \overset{\eqref{eq:SGDA_prox_clipped_12345_bound}}{\leq} \exp(-\gamma\mu T) R^2 + \circledOne + \circledTwo + \circledThree + \circledFour + \circledFive ,\\
        \circledTwo \overset{\eqref{eq:bound_2_SGDA_prox_clipped}}{\leq} \frac{1}{5}\exp(-\gamma\mu T)R^2,\quad \circledThree \overset{\eqref{eq:bound_3_SGDA_prox_clipped}}{\leq} \frac{1}{5}\exp(-\gamma\mu T)R^2,\quad \circledFive \overset{\eqref{eq:bound_5_SGDA_prox_clipped}}{\leq} \frac{1}{5}\exp(-\gamma\mu T)R^2,\\
        \sum\limits_{t=0}^{T-1}\sigma_t^2 \overset{\eqref{eq:bound_1_variances_SGDA_prox_clipped}}{\leq}  \frac{\exp(-2\gamma\mu T)R^4}{150\ln\tfrac{4(K+1)}{\beta}},\quad \sum\limits_{t=0}^{T-1}\widetilde\sigma_t^2 \overset{\eqref{eq:bound_4_variances_SGDA_prox_clipped}}{\leq} \frac{\exp(-2\gamma\mu T)R^4}{150\ln\tfrac{4(K+1)}{\beta}}.
    \end{gather*}
     In addition, we also establish (see \eqref{eq:bound_1_SGDA_prox_clipped}, \eqref{eq:bound_4_SGDA_prox_clipped} and our induction assumption)
     \begin{gather*}
        \PP\{E_{T-1}\} \geq 1 - \frac{(T-1)\beta}{K+1},\\
        \PP\{E_{\circledOne}\} \geq 1 - \frac{\beta}{2(K+1)}, \quad \PP\{E_{\circledFour}\} \geq 1 - \frac{\beta}{2(K+1)}.
    \end{gather*}
    where
    \begin{eqnarray}
        E_{\circledOne}&=&  \left\{\text{either} \quad \sum\limits_{t=0}^{T-1}\sigma_t^2 > \frac{\exp(- 2\gamma\mu T) R^4}{150\ln\tfrac{4(K+1)}{\beta}}\quad \text{or}\quad |\circledOne| \leq \frac{1}{5}\exp(-\gamma\mu T) R^2\right\},\notag\\
        E_{\circledFour}&=& \left\{\text{either} \quad \sum\limits_{t=0}^{T-1}\widetilde\sigma_t^2 > \frac{\exp(-2\gamma\mu T) R^4}{150\ln\tfrac{4(K+1)}{\beta}}\quad \text{or}\quad |\circledFour| \leq \frac{1}{5}\exp(-\gamma\mu T) R^2\right\}.\notag
    \end{eqnarray}
    Therefore, probability event $E_{T-1} \cap E_{\circledOne} \cap E_{\circledFour} $ implies
    \begin{eqnarray*}
        R_T^2 &\overset{\eqref{eq:SGDA_prox_clipped_12345_bound}}{\leq}& \exp(-\gamma\mu T) R^2 + \circledOne + \circledTwo + \circledThree + \circledFour + \circledFive\\
        &\leq& 2\exp(-\gamma\mu T) R^2,
    \end{eqnarray*}
    which is equivalent to \eqref{eq:induction_inequality_prox_clipped_SGDA_star} for $t = T$. Moreover,
    \begin{equation}
        \PP\{E_T\} \geq \PP\{E_{T-1} \cap E_{\circledOne} \cap E_{\circledFour} \} = 1 - \PP\{\overline{E}_{T-1} \cup \overline{E}_{\circledOne} \cup \overline{E}_{\circledFour} \} \geq 1 - \frac{T\beta}{K+1}. \notag
    \end{equation}
    In other words, we showed that $\PP\{E_k\} \geq 1 - \nicefrac{k\beta}{(K+1)}$ for all $k = 0,1,\ldots,K+1$. For $k = K+1$ we have that with probability at least $1 - \beta$
    \begin{equation}
        \|x^{K+1} - x^*\|^2 \leq 2\exp(-\gamma\mu (K+1))R^2. \notag
    \end{equation}
    
    Finally, if 
    \begin{eqnarray*}
        \gamma &=& \min\left\{\frac{1}{400 \ell \ln \tfrac{4(K+1)}{\beta}}, \frac{\ln(B_K)}{\mu(K+1)}\right\}, \notag\\
        B_K &=& \max\left\{2, \frac{(K+1)^{\frac{2\alpha-1}{\alpha}}\mu^2R^2}{4\cdot 10^{\frac{1}{\alpha}}120^{\frac{2(\alpha-1)}{\alpha}}\sigma^2\ln^{\frac{2(\alpha-1)}{\alpha}}\left(\frac{4(K+1)}{\beta}\right)\ln^2(B_K)} \right\}  \\
        &=& \cO\left(\max\left\{2, \frac{K^{\frac{2\alpha-1}{\alpha}}\mu^2R^2}{\sigma^2\ln^{\frac{2(\alpha-1)}{\alpha}}\left(\frac{K}{\beta}\right)\ln^2\left(\max\left\{2, \frac{K^{\frac{2\alpha-1}{\alpha}}\mu^2R^2}{\sigma^2\ln^{\frac{2(\alpha-1)}{\alpha}}\left(\frac{K}{\beta}\right)} \right\}\right)} \right\}\right\}
    \end{eqnarray*}
    then with probability at least $1-\beta$
    \begin{eqnarray*}
        \|x^{K+1} - x^*\|^2 &\leq& 2\exp(-\gamma\mu (K+1))R^2\\
        &=& 2R^2\max\left\{\exp\left(-\frac{\mu(K+1)}{400 \ell \ln \tfrac{4(K+1)}{\beta}}\right), \frac{1}{B_K} \right\}\\
        &=& \cO\left(\max\left\{R^2\exp\left(- \frac{\mu K}{\ell \ln \tfrac{K}{\beta}}\right), \frac{\sigma^2\left(\frac{K}{\beta}\right)\ln^2\left(\max\left\{2, \frac{K^{\frac{2\alpha-1}{\alpha}}\mu^2R^2}{\sigma^2\ln^{\frac{2(\alpha-1)}{\alpha}}\left(\frac{K}{\beta}\right)} \right\}\right)}{\ln^{\frac{2(1-\alpha)}{\alpha}}K^{\frac{2\alpha-1}{\alpha}}\mu^2}\right\}\right).
    \end{eqnarray*}
    To get $\|x^{K+1} - x^*\|^2 \leq \varepsilon$ with probability at least $1-\beta$, $K$ should be
    \begin{equation*}
         K = \cO\left(\frac{\ell}{\mu}\ln\left(\frac{R^2}{\varepsilon}\right)\ln\left(\frac{\ell}{\mu \beta}\ln\frac{R^2}{\varepsilon}\right), \left(\frac{\sigma^2}{\mu^2\varepsilon}\right)^{\frac{\alpha}{2\alpha-1}}\ln \left(\frac{1}{\beta} \left(\frac{\sigma^2}{\mu^2\varepsilon}\right)^{\frac{\alpha}{2\alpha-1}}\right)\ln^{\frac{\alpha}{\alpha-1}}\left(B_\varepsilon\right)\right),
    \end{equation*}
    where
    \begin{equation*}
        B_\varepsilon = \max\left\{2, \frac{R^2}{\varepsilon \ln \left(\frac{1}{\beta} \left(\frac{\sigma^2}{\mu^2\varepsilon}\right)^{\frac{\alpha}{2\alpha-1}}\right)}\right\}.
    \end{equation*}
\end{proof}

\clearpage

\section{Missing Proofs for \algname{DProx-clipped-SGD-shift}}\label{appendix:dprox_clipped_SGD_shift}

In this section, we give the complete formulations of our results for \algname{DProx-clipped-SGD-shift} and rigorous proofs. For the readers' convenience, the method's update rule is repeated below:
\begin{gather*}
    x^{k+1} = \proxkPsi \left(x^k - \gamma \tg^k\right),\;\; \text{where}\;\; \tg^k = \frac{1}{n}\sum\limits_{i=1}^n \tg_i^k,\;\; \tg_i^k = h_i^k + \hat\Delta_i^k, \\
    h_i^{k+1} = h_i^k + \nu \hat \Delta_i^k,\quad \hat\Delta_i^k = \clip\left(\nabla f_{\xi_i^k}(x^k) - h_i^k, \lambda_k\right). 
\end{gather*}

\begin{lemma}\label{lem:main_opt_lemma_prox_clipped_SGD_convex}
    Let Assumptions~\ref{as:L_smoothness} and \ref{as:str_cvx} with $\mu = 0$ hold on $Q = B_{3n\sqrt{V}}(x^*)$, where $V \geq \|x^0 - x^*\|^2 + \frac{36864\gamma^2 \ln^2\frac{48n(K+1)}{\beta}}{n^2}\sum\limits_{i=1}^n\|\nabla f_i(x^*)\|^2$, and let stepsize  $\gamma$ satisfy $\gamma \leq \frac{1}{L}$. If $x^{k} \in Q$ for all $k = 0,1,\ldots,K+1$, $K \ge 0$, then after $K$ iterations of \algname{DProx-clipped-SGD-shift} we have
    \begin{eqnarray}
       2\gamma\left(\Phi(\overline{x}^{K+1}) -\Phi(x^*)\right) &\leq& \frac{\|x^0 - x^*\|^2 - \|x^{K+1} - x^*\|^2}{K+1}\notag\\
       &&\quad - \frac{2\gamma}{K+1}\sum\limits_{k=0}^{K} \langle \omega_k, \Hat{x}^k - x^* \rangle  + \frac{2\gamma^2}{K+1}\sum\limits_{k=0}^{K} \|\omega_k\|^2 ,\label{eq:main_opt_lemma_clipped_SGD_star_convex}\\
       \overline{x}^{K+1} &\eqdef& \frac{1}{K+1}\sum\limits_{k=0}^{K} x^{k+1},
       \label{eq:x_avg_prox_clipped_SGD}\\
       \Hat{x}^k &\eqdef& \proxkPsi \left(x^k - \gamma \nabla f(x^k)\right),
       \label{eq:x_hat_prox_clipped_SGD}\\
        \omega_k &\eqdef& \nabla f(x^k) - \tg^k. \label{eq:omega_k_prox_clipped_SGD}
    \end{eqnarray}
\end{lemma}
\begin{proof}
    Using Lemma~C.2 from \citep{khaled2020unified} with $p=x^{k+1}, \ y = x^k - \gamma \tg^k, \ x=x^k$, we derive for all $k = 0,1, \ldots, K$ that
    \begin{eqnarray*}
        2\gamma\left(\Phi(x^{k+1}) - \Phi(x^*)\right) &\leq& \|x^{k} - x^*\|^2 - \|x^{k+1} - x^*\|^2 - 2\gamma \langle \tg^k - \nabla f(x^k), x^{k+1}-x^*\rangle.
    \end{eqnarray*}
    Next, we obtain the following inequality
    \begin{eqnarray*}
        -2\gamma\langle \tg^k - \nabla f(x^k), x^{k+1} - x^*\rangle &=& -2\gamma \langle \tg^k - \nabla f(x^k), \Hat{x}^k - x^*\rangle +2\gamma \langle \tg^k - \nabla f(x^k), \Hat{x}^k - x^{k+1}\rangle\\
        &\overset{\eqref{eq:omega_k_prox_clipped_SGD}}{\leq}& - 2\gamma \langle \omega_k, \Hat{x}^k - x^*\rangle + 2\gamma \|\tg^k - \nabla f(x^k)\|\cdot\|\Hat{x}^k - x^{k+1}\|\\
        &\overset{\eqref{eq:x_hat_prox_clipped_SGD}}{=}& -2\gamma \langle \omega_k, \Hat{x}^k - x^*\rangle + 2\gamma \|\tg^k - \nabla f(x^k)\|\\
        && \cdot\|\proxkPsi \left(x^k - \gamma \nabla f(x^k)\right) - \proxkPsi \left(x^k - \gamma \tg^k\right)\|\\
        &\overset{\eqref{eq:omega_k_prox_clipped_SGD}}{\leq}& -2\gamma \langle \omega_k, \Hat{x}^k - x^*\rangle + 2 \gamma^2 \|\omega_k\|^2.
    \end{eqnarray*}
    Putting all together we get 
    \begin{eqnarray*}
        2\gamma\left(\Phi(x^{k+1}) - \Phi(x^*)\right) &\leq& \|x^{k} - x^*\|^2 - \|x^{k+1} - x^*\|^2 - 2\gamma \langle \omega_k, \Hat{x}^k - x^*\rangle + 2 \gamma^2 \|\omega_k\|^2.
    \end{eqnarray*}
    Summing up the above inequalities for $k = 0,1,\ldots, K$, we get
    \begin{eqnarray*}
        \frac{2\gamma}{K+1}\sum\limits_{k=0}^{K}\left(\Phi(x^{k+1}) - \Phi(x^*)\right) &\leq& \frac{1}{K+1}\sum\limits_{k=0}^{K}\left(\|x^k - x^*\|^2 - \|x^{k+1} - x^*\|^2\right)\\
        &&\quad - \frac{2\gamma}{K+1}\sum\limits_{k=0}^{K} \langle \omega_k, \Hat{x}^k - x^* \rangle + \frac{2\gamma^2}{K+1}\sum\limits_{k=0}^{K} \|\omega_k\|^2\\
        &=& \frac{\|x^0 - x^*\|^2 - \|x^{K+1} - x^*\|^2}{K+1} - \frac{2\gamma}{K+1}\sum\limits_{k=0}^{K} \langle \omega_k, \Hat{x}^k - x^* \rangle \\
        &&\quad + \frac{2\gamma^2}{K+1}\sum\limits_{k=0}^{K} \|\omega_k\|^2.  
    \end{eqnarray*}
    Finally, we use the definition of $\overline{x}^{K}$ and Jensen's inequality and get the result.
\end{proof}

\begin{theorem}\label{thm:prox_clipped_SGD_convex_case}
    Let Assumptions~\ref{as:L_smoothness} and \ref{as:str_cvx} with $\mu = 0$ hold on $Q = B_{3n\sqrt{V}}(x^*)$, where $V \geq \|x^0 - x^*\|^2 + \frac{36864\gamma^2 \ln^2\frac{48n(K+1)}{\beta}}{n^2}\sum\limits_{i=1}^n\|\nabla f_i(x^*)\|^2$, and $\nu = 0$, $h_1^0 = \ldots = h_n^0 = 0$,
    \begin{gather}
        \gamma \leq \min\left\{\frac{1}{360L \ln \frac{48n(K+1)}{\beta}}, \revision{\frac{R\sqrt{n}}{192 A \zeta_*},} \frac{\sqrt{V}n^{\frac{\alpha-1}{\alpha}}}{27^{\frac{1}{\alpha}}\cdot 48 \sigma  K^{\frac{1}{\alpha}}\left(\ln \frac{48n(K+1)}{\beta}\right)^{\frac{\alpha-1}{\alpha}}}\right\},\label{eq:prox_clipped_SGD_step_size_cvx}\\
        \lambda_{k} = \lambda = \frac{n\sqrt{V}}{48 \gamma\ln\frac{48n(K+1)}{\beta}}, \label{eq:prox_clipped_SGD_clipping_level_cvx}
    \end{gather}
    for some \revision{$\zeta_* \geq \sqrt{\frac{1}{n}\sum_{i=1}^n \|\nabla f_i(x^*)\|^2}$,} $K+1 > 0$ and $\beta \in (0,1]$. Then, after $K+1$ iterations of \algname{DProx-clipped-SGD-shift} the iterates with probability at least $1 - \beta$ satisfy
    \begin{equation}
        \Phi(\overline{x}^{K+1}) - \Phi(x^*) \leq \frac{V}{\gamma(K+1)} \quad \text{and} \quad  \{x^k\}_{k=0}^{K+1} \subseteq B_{3n\sqrt{V}}(x^*). \label{eq:prox_clipped_SGD_convex_case}
    \end{equation}
    In particular, we have \revision{$V \leq 2R^2$}, and when $\gamma$ equals the minimum from \eqref{eq:prox_clipped_SGD_step_size_cvx}, then the iterates produced by \algname{DProx-clipped-SGD-shift} after $K+1$ iterations with probability at least $1-\beta$ satisfy
    \begin{equation}
        \Phi(\overline{x}^{K+1}) - \Phi(x^*) = \cO\left(\max\left\{\frac{L\revision{R^2} \ln \frac{nK}{\beta}}{K}, \revision{\frac{R\zeta_* \ln \frac{nK}{\beta}}{\sqrt{n}K}}, \frac{\sigma \revision{R}\ln^{\frac{\alpha-1}{\alpha}}\frac{nK}{\beta}}{n^{\frac{\alpha-1}{\alpha}}K^{\frac{\alpha-1}{\alpha}}}\right\}\right), \label{eq:prox_clipped_SGD_convex_case_2}
    \end{equation}
    meaning that to achieve $\Phi(\overline{x}^{K+1}) - \Phi(x^*) \leq \varepsilon$ with probability at least $1 - \beta$ \algname{DProx-clipped-SGD-shift} requires
    \begin{equation}
        K = \cO\left(\max\left\{\frac{L\revision{R^2}}{\varepsilon}\revision{\ln\frac{nL\revision{R^2}}{\beta\varepsilon}}, \revision{\frac{R\zeta_*}{\sqrt{n}\varepsilon}}\revision{\ln\frac{\sqrt{n}R\zeta_*}{\beta\varepsilon}}, \left(\frac{\sigma \sqrt{V}}{\varepsilon n^{\frac{\alpha-1}{\alpha}}}\right)^{\frac{\alpha}{\alpha-1}} \ln \left(\frac{1}{\beta} \left(\frac{\sigma \sqrt{V}}{\varepsilon}\right)^{\frac{\alpha}{\alpha-1}}\right) \right\}\right) \label{eq:prox_clipped_SGD_convex_case_complexity}
    \end{equation}
    \revision{iterations/oracle calls.}
\end{theorem}
\begin{proof}
 The key idea behind the proof is similar to the one used in \citep{gorbunov2022clipped, sadiev2023high}: we prove by induction that the iterates do not leave some ball and the sums decrease as $\nicefrac{1}{K+1}$. To formulate the statement rigorously, we introduce probability event $E_k$ for each $k = 0,1,\ldots, K+1$ as follows: inequalities
    \begin{eqnarray}
        \underbrace{\|x^0 - x^*\|^2 - 2\gamma\sum\limits_{l=0}^{t-1} \langle \omega_l, \Hat{x}^l - x^* \rangle  + 2\gamma^2\sum\limits_{l=0}^{t-1} \|\omega_l\|^2}_{A_t} \leq 2V, \label{eq:prox_clipped_SGD_convex_induction_inequality_1}
    \end{eqnarray}
    \begin{gather}
        \left\|\frac{\gamma}{n}\sum\limits_{i=1}^{r-1}\omega_{i,t-1}^u\right\| \leq \frac{\sqrt{V}}{2}
    \end{gather}
hold for $t = 0,1,\ldots,k$ and $r = 1, 2, \ldots, n$ simultaneously, where
    \begin{gather}
        \omega_l = \omega_l^u + \omega_l^b, \label{eq:prox_clipped_SGD_technical_4_full_theta_omega}\\
        \omega_l^u \eqdef \frac{1}{n}\sum\limits_{i=1}^n \omega_{i,l}^u,\quad \omega_l^b \eqdef \frac{1}{n}\sum\limits_{i=1}^n \omega_{i,l}^b, \\
        \omega_{i,l}^u \eqdef \EE_{\xi_{i}^l}\left[\tg_{i}^l\right] - \tg_{i}^l,\quad \omega_{i,l}^b \eqdef \nabla f(x^l) - \EE_{\xi_{i}^l}\left[\tg_{i}^l\right] \quad \forall\; i\in [n]. \label{eq:prox_clipped_SGD_convex_theta_u_b}
    \end{gather}
    We will prove by induction that $\PP\{E_k\} \geq  1 - \nicefrac{k\beta}{(K+1)}$ for all $k = 0,1,\ldots,K+1$. The base of induction follows immediately: $\|x^0 - x^*\|^2 \leq V < 2V$ and for $k = 0$ we have  $\|\frac{\gamma}{n}\sum_{i=1}^{r-1} \omega_{i,k-1}^u\| = 0$ since $\omega_{i,-1}^u = 0$. Next, we assume that the statement holds for $k = T-1 \leq K$, i.e., $\PP\{E_{T-1}\} \geq  1 - \nicefrac{(T-1)\beta}{(K+1)}$. Let us show that it also holds for $k = T$, i.e., $\PP\{E_{T}\} \geq  1 - \nicefrac{T\beta}{(K+1)}$.
    
    To proceed, we need to show that $E_{T-1}$ implies $\|x^t - x^*\| \leq 3n\sqrt{V}$ for all $t = 0,1,\ldots,T$. First, for $t = 0,1,\ldots,T-1$ probability event $E_{T-1}$ implies (in view of, $\Phi(\overline{x}^t) - \Phi(x^*) \geq 0$)
    \begin{equation}
        \|x^t - x^*\|^2 \overset{\eqref{eq:main_opt_lemma_clipped_SGD_star_convex}}{\leq} A_t \overset{\eqref{eq:prox_clipped_SGD_convex_induction_inequality_1}}{\leq} 2V. \label{eq:jsknfdbvufbvudfu}
    \end{equation}
    Next, by definition of $V$ we have
    \begin{eqnarray}
        \|\nabla f(x^*)\| &=& \sqrt{\|\nabla f(x^*)\|^2} \leq \sqrt{\sum\limits_{i=1}^n \|\nabla f_i(x^*)\|^2} \leq \frac{n\sqrt{V}}{192\gamma\ln\frac{48n(K+1)}{\beta}}.  \label{eq:msodfnvidvidbsi}
    \end{eqnarray}
    Then, for $t = T$ we have that $E_{T-1}$ implies
    \begin{eqnarray*}
        \|x^T - x^*\| &=& \|\proxkPsi (x^k - \gamma \tg^k) - \proxkPsi (x^* - \gamma \nabla f(x^*))\|\\
        &\leq& \|x^k - \gamma \tg^k - x^* + \gamma \nabla f(x^*)\| \leq \|x^k - x^*\| + \gamma\|\tg^k\| + \gamma\|\nabla f(x^*)\|\\
        &\overset{\eqref{eq:jsknfdbvufbvudfu},\eqref{eq:msodfnvidvidbsi}}{\leq}& \left(\sqrt{2} + \frac{n}{192\ln\frac{48n(K+1)}{\beta}}\right)\sqrt{V} + \gamma\lambda \overset{\eqref{eq:prox_clipped_SGD_clipping_level_cvx}}{\leq} 3n\sqrt{V}.
    \end{eqnarray*}
    
    
    This means that $E_{T-1}$ implies $x^t \in B_{3n\sqrt{V}}(x^*)$ for $t = 0,1,\ldots, T$ and we can apply Lemma~\ref{lem:main_opt_lemma_prox_clipped_SGD_convex}: $E_{T-1}$ implies
     \begin{eqnarray}
        2\gamma \left(\Phi(\overline{x}^{T}) -\Phi(x^*)\right) &\leq& \frac{\|x^0 - x^*\|^2 - \|x^{T} - x^*\|^2}{T}\notag\\
       &&\quad - \frac{2\gamma}{T}\sum\limits_{l=0}^{T-1} \langle \omega_l, \Hat{x}^l - x^* \rangle  + \frac{2\gamma^2}{T}\sum\limits_{l=0}^{T-1} \|\omega_l\|^2 \notag\\
       &\leq& \frac{A_T}{T}.\label{eq:prox_clipped_SGD_convex_technical_1}
    \end{eqnarray}
    Before we proceed, we introduce a new notation:
    \begin{equation}
        \eta_t = \begin{cases} \Hat{x}^t - x^*,& \text{if } \|\Hat{x}^t - x^*\| \leq 2\sqrt{V},\\ 0,&\text{otherwise}, \end{cases} \notag
    \end{equation}
    for all $t = 0, 1, \ldots, T-1$. Random vectors $\{\eta_t\}_{t=0}^T$ are bounded almost surely:
    \begin{equation}
        \|\eta_t\| \leq 2\sqrt{V}. \label{eq:prox_clipped_SGD_convex_technical_6}
    \end{equation}
    for all $t = 0, 1, \ldots, T-1$. In addition, $E_{T-1}$ implies for all $t = 0, 1, \ldots, T-1$ that
    \begin{eqnarray*}
        \|\Hat{x}^{t} - x^*\| &=& \|\proxkPsi \left(x^{t} - \gamma \nabla f(x^{t})\right) - \proxkPsi \left(x^* - \gamma \nabla f(x^*)\right)\|\\
        &\leq& \|x^{t} - x^* - \gamma(\nabla f(x^{t}) - \nabla f(x^*))\| \\
        &\leq& \|x^{t} - x^*\| + \gamma \|\nabla f(x^{t}) - \nabla f(x^*)\|\\
        &\overset{\eqref{eq:L_smoothness}}{\leq}& (1+L\gamma)\|x^{t} - x^*\| \overset{\eqref{eq:prox_clipped_SGD_step_size_cvx}}{\leq} \frac{361}{360}\|x^{t} - x^*\| \overset{\eqref{eq:jsknfdbvufbvudfu}}{\leq} 2\sqrt{V}.
    \end{eqnarray*}
    meaning that $\eta_t = \Hat{x}^t - x^*$ follows from $E_{T-1}$ for all $t = 0, 1, \ldots, T-1$. Thus, $E_{T-1}$ implies
    \begin{eqnarray}
        A_T &\overset{\eqref{eq:prox_clipped_SGD_convex_induction_inequality_1}}{=}& \|x^0 - x^*\|^2 - 2\gamma\sum\limits_{l=0}^{T-1} \langle \omega_l, \Hat{x}^l - x^* \rangle  + 2\gamma^2\sum\limits_{l=0}^{T-1} \|\omega_l\|^2\notag\\
        &\leq& V - 2\gamma\sum\limits_{l=0}^{T-1} \langle \omega_l, \eta_l \rangle  + 2\gamma^2\sum\limits_{l=0}^{T-1} \|\omega_l\|^2.\label{eq:prox_clipped_SGD_convex_technical_2}
    \end{eqnarray}
    Using the notation from \eqref{eq:prox_clipped_SGD_technical_4_full_theta_omega}-\eqref{eq:prox_clipped_SGD_convex_theta_u_b}, we can rewrite $\|\omega_l\|^2$ as
    \begin{eqnarray}
        \|\omega_l\|^2 &\leq&  2\|\omega_l^u\|^2 + 2\|\omega_l^b\|^2 = \frac{2}{n^2}\left\|\sum\limits_{i=1}^n \omega_{i,l}^u\right\|^2 + 2\|\omega_l^b\|^2\notag \\
        &=& \frac{2}{n^2}\sum\limits_{i=1}^n\|\omega_{i,l}^u\|^2 + \frac{4}{n^2}\sum\limits_{j=2}^n\left\langle \sum\limits_{i=1}^{j-1} \omega_{i,l}^u, \omega_{j,l}^u \right\rangle + 2\|\omega_l^b\|^2 .\label{eq:omega_bound_x_t_prox_clipped_SGD_convex}
    \end{eqnarray}
    Putting all together, we obtain that $E_{T-1}$ implies
    \begin{eqnarray}
        A_T 
        &\leq& V \underbrace{-\frac{2\gamma}{n}\sum\limits_{l=0}^{T-1}\sum\limits_{i=1}^n\langle \omega_{i,l}^u, \eta_l\rangle}_{\circledOne}  \underbrace{-2\gamma\sum\limits_{l=0}^{T-1}\langle \omega_l^b, \eta_l\rangle}_{\circledTwo} + \underbrace{\frac{4\gamma^2}{n^2}\sum\limits_{l=0}^{T-1}\sum\limits_{i=1}^n\left(\left\|\omega_{i,l}^u\right\|^2 - \EE_{\xi^l_i}\left[\left\|\omega_{i,l}^u\right\|^2\right]\right)}_{\circledThree}\notag\\
        &&\quad + \underbrace{\frac{4\gamma^2}{n^2}\sum\limits_{l=0}^{T-1}\sum\limits_{i=1}^n\EE_{\xi^l_i}\left[\left\|\omega_{i,l}^u\right\|^2\right]}_{\circledFour} + \underbrace{4\gamma^2\sum\limits_{l=0}^{T-1}\|\omega_{l}^b\|^2}_{\circledFive} + \underbrace{\frac{8\gamma^2}{n^2}\sum\limits_{l=0}^{T-1}\sum\limits_{j=2}^n\left\langle \sum\limits_{i=1}^{j-1}\omega_{i,l}^u, \omega_{j,l}^u\right\rangle}_{\circledSix}. \label{eq:prox_clipped_SGD_convex_technical_7}
    \end{eqnarray}
    To finish the proof, it remains to estimate $\circledOne, \circledTwo, \circledThree, \circledFour, \circledFive, \circledSix$ with high probability. More precisely, the goal is to prove that $\circledOne + \circledTwo + \circledThree + \circledFour + \circledFive + \circledSix \leq V$ with high probability. Before we proceed, we need to derive several useful inequalities related to $ \omega_{i,l}^u, \omega_{l}^b$. First of all, we have
    \begin{equation}
       \|\omega_{i,l}^u\| \leq 2\lambda \label{eq:prox_clipped_SGD_convex_norm_theta_u_bound}
    \end{equation}
    by definition of the clipping operator. Next, probability event $E_{T-1}$ implies 
    \begin{eqnarray}
        \|\nabla f_i(x^{t})\| &\leq& \|\nabla f_i(x^t) - \nabla f_i(x^*)\| + \|\nabla f_i(x^*)\|\overset{\eqref{eq:L_smoothness}}{\leq} L\|x^t - x^*\| + \sqrt{\sum\limits_{i=1}^n\|\nabla f_i(x^*)\|^2}\notag\\ 
        &\leq& \sqrt{2}L\sqrt{V} + \frac{n\sqrt{V}}{192\gamma\ln\frac{48n(K+1)}{\beta}} \leq \frac{n\sqrt{V}}{96\gamma\ln\frac{48n(K+1)}{\beta}} \leq \frac{\lambda}{2}.\label{eq:prox_clipped_SGD_convex_technical_4}
    \end{eqnarray}for $t = 0,1,\ldots,T-1$ and $i\in [n]$. Therefore, Lemma~\ref{lem:bias_and_variance_clip} and $E_{T-1}$ imply
    \begin{gather}
        \left\|\omega_l^b\right\| \leq \frac{1}{n}\sum\limits_{i=1}^n \|\omega_{i,l}^b\| \leq \frac{2^\alpha\sigma^\alpha}{\lambda^{\alpha-1}}, \label{eq:prox_clipped_SGD_convex_norm_theta_b_bound}\\
        \EE_{\xi_{i}^l}\left[\left\|\omega_{i,l}^u\right\|^2\right] \leq 18 \lambda^{2-\alpha}\sigma^\alpha, \label{eq:prox_clipped_SGD_convex_second_moment_theta_u_bound}
    \end{gather}
    for all $l = 0,1, \ldots, T-1$ and $i\in [n]$.

\textbf{Upper bound for $\circledOne$.} 
To estimate this sum, we will use Bernstein's inequality. The summands have conditional expectations equal to zero:
    \begin{equation}
        \EE_{\xi^l_i}\left[-\frac{2\gamma}{n}\langle\omega_{i,l}^u, \eta_l\rangle\right] = -\frac{2\gamma}{n} \left\langle \eta_l, \EE_{\xi_{i}^l}[\omega_{i,l}^u] \right\rangle =  0. \notag
    \end{equation}
    Moreover, for all $l = 0,\ldots, T-1$ random vectors $\{\omega_{i,l}^u\}_{i=1}^n$ are independent. Thus, sequence $\left\{-\frac{2\gamma}{n} \langle \eta_l, \omega_{i,l}^u \rangle\right\}_{l,i = 0,1}^{T-1,n}$ is a martingale difference sequence. Next, the summands are bounded:
    \begin{equation}
        \left|\frac{2\gamma}{n}\left\la \omega_{i,l}^u, \eta_l\right\ra\right| \leq \frac{2\gamma}{n} \|\omega_{i,l}^u\| \cdot \|\eta_l\| \overset{\eqref{eq:prox_clipped_SGD_convex_technical_6},\eqref{eq:prox_clipped_SGD_convex_norm_theta_u_bound}}{\leq} \frac{8\gamma \lambda \sqrt{V}}{n}\overset{\eqref{eq:prox_clipped_SGD_clipping_level_cvx}}{=} \frac{V}{6\ln\frac{48n(K+1)}{\beta}} \eqdef c. \label{eq:prox_clipped_SGD_convex_technical_8} 
    \end{equation}
    Finally, conditional variances $\sigma_{i,l}^2 \eqdef \EE_{\xi^l_i}[\frac{4\gamma^2}{n^2}\langle\omega_{i,l}^u, \eta_l\rangle^2]$ of the summands are bounded:
    \begin{equation}
        \sigma_{i,t}^2 \leq \EE_{\xi^t_i}\left[\frac{4\gamma^2}{n^2}\|\omega_{i,t}^u\|^2\cdot \|\eta_t\|^2\right] \overset{\eqref{eq:prox_clipped_SGD_convex_technical_6}}{\leq} \frac{16\gamma^2 V}{n^2} \EE_{\xi^t_i}\left[\|\omega_{i,t}^u\|^2\right]. \label{eq:prox_clipped_SGD_convex_technical_9}
    \end{equation}
    Applying Bernstein's inequality (Lemma~\ref{lem:Bernstein_ineq}) with $X_{i,l} = -\frac{2\gamma}{n} \langle \eta_l, \omega_{i,l}^u \rangle$, constant $c$ defined in \eqref{eq:prox_clipped_SGD_convex_technical_8}, $b = \frac{V}{6}$, $G = \frac{V^2}{216\ln\frac{48n(K+1)}{\beta}}$, we get
    \begin{equation*}
        \PP\left\{|\circledOne| > \frac{V}{6}\quad \text{and}\quad \sum\limits_{l=0}^{T-1}\sum\limits_{i=1}^n \sigma_{i,l}^2 \leq \frac{V^2}{216\ln\frac{48n(K+1)}{\beta}}\right\} \leq 2\exp\left(- \frac{b^2}{2G + \nicefrac{2cb}{24n}}\right) = \frac{\beta}{24n(K+1)}.
    \end{equation*}
    The above is equivalent to 
    \begin{equation}
        \PP\left\{ E_{\circledOne} \right\} \geq 1 - \frac{\beta}{24n(K+1)},\quad \text{for}\quad E_{\circledOne} = \left\{ \text{either} \quad  \sum\limits_{l=0}^{T-1}\sum\limits_{i=1}^n \sigma_{i,l}^2 > \frac{V^2}{216\ln\frac{48n(K+1)}{\beta}} \quad \text{or}\quad |\circledOne| \leq \frac{V}{6}\right\}. \label{eq:prox_clipped_SGD_convex_sum_1_upper_bound}
    \end{equation}
    Moreover, $E_{T-1}$ implies
    \begin{eqnarray}
        \sum\limits_{l=0}^{T-1}\sum\limits_{i=1}^n \sigma_{i,l}^2 &\overset{\eqref{eq:prox_clipped_SGD_convex_technical_9}}{\leq}& \frac{16\gamma^2 V}{n^2} \sum\limits_{l=0}^{T-1} \sum\limits_{i=1}^n \EE_{\xi^l_i}\left[\|\omega_{i,l}^u\|^2\right] \overset{\eqref{eq:prox_clipped_SGD_convex_second_moment_theta_u_bound}}{\leq} \frac{288\gamma^2 V \sigma^{\alpha}T \lambda^{2-\alpha}}{n}\notag\\
        &\overset{\eqref{eq:prox_clipped_SGD_clipping_level_cvx}}{=}& \frac{ 48^{\alpha}\sqrt{V}^{4-\alpha}\sigma^\alpha T\gamma^{\alpha}}{8 n^{\alpha-1}\ln^{2-\alpha}\frac{48n(K+1)}{\beta}} \overset{\eqref{eq:prox_clipped_SGD_step_size_cvx}}{\leq} \frac{V^2}{216 \ln\frac{48n(K+1)}{\beta}}. \label{eq:prox_clipped_SGD_convex_sum_1_variance_bound}
    \end{eqnarray}
    
 \textbf{Upper bound for $\circledTwo$.} Probability event $E_{T-1}$ implies
    \begin{eqnarray}
        \circledTwo &=& -2\gamma\sum\limits_{l=0}^{T-1}\langle \omega_l^b, \eta_l \rangle \leq 2\gamma\sum\limits_{l=0}^{T-1}\|\omega_{l}^b\|\cdot \|\eta_l\| \overset{\eqref{eq:prox_clipped_SGD_convex_technical_6},\eqref{eq:prox_clipped_SGD_convex_norm_theta_b_bound}}{\leq}  \frac{4\cdot 2^\alpha \gamma \sigma^\alpha T \sqrt{V}}{\lambda^{\alpha-1}} \notag\\ &\overset{\eqref{eq:prox_clipped_SGD_clipping_level_cvx}}{=}& \frac{96^\alpha}{12} \cdot \frac{\sigma^\alpha T \sqrt{V}^{2-\alpha}\gamma^\alpha}{n^{\alpha-1}\ln^{1-\alpha}\frac{48n(K+1)}{\beta}}\overset{\eqref{eq:prox_clipped_SGD_step_size_cvx}}{\leq} \frac{V}{6}. \label{eq:prox_clipped_SGD_convex_sum_2_upper_bound}
\end{eqnarray}

\textbf{Upper bound for $\circledThree$.} To estimate this sum, we will use Bernstein's inequality. The summands have conditional expectations equal to zero:
    \begin{equation}
        \EE_{\xi^l_i}\left[\frac{4\gamma^2}{n^2}\left(\left\|\omega_{i,l}^u\right\|^2 - \EE_{\xi^l_i}\left[\left\|\omega_{i,l}^u\right\|^2\right]\right)\right] = 0. \notag
    \end{equation}
    Moreover, for all $l = 0,\ldots, T-1$ random vectors $\{\omega_{i,l}^u\}_{i=1}^n$ are independent. Thus, sequence $\left\{\frac{4\gamma^2}{n^2}\left(\|\omega_{i,l}^u\|^2 - \EE_{\xi_{i}^l}\left[\|\omega_{i,l}^u\|^2\right]\right)\right\}_{l,i = 0,1}^{T-1,n}$ is a martingale difference sequence. Next, the summands are bounded:
    \begin{eqnarray}
        \left|\frac{4\gamma^2}{n^2}\left(\left\|\omega_{i,l}^u\right\|^2 - \EE_{\xi^l_i}\left[\left\|\omega_{i,l}^u\right\|^2\right]\right)\right| &\leq& \frac{4\gamma^2}{n^2}\left( \|\omega_{i,l}^u\|^2 +   \EE_{\xi^l_i}\left[\left\|\omega_{i,l}^u\right\|^2\right]\right)\notag\\
        &\overset{\eqref{eq:prox_clipped_SGD_convex_norm_theta_u_bound}}{\leq}& \frac{32\gamma^2\lambda^2}{n^2}\overset{\eqref{eq:prox_clipped_SGD_clipping_level_cvx}}{=} \frac{V}{72\ln^2\frac{48n(K+1)}{\beta}} \notag\\
        &\le& \frac{V}{12 \ln\frac{48n(K+1)}{\beta}}\eqdef c. \label{eq:prox_clipped_SGD_convex_technical_10}
    \end{eqnarray}
    Finally, conditional variances  $$\widetilde\sigma_{i,t}^2 \eqdef \EE_{\xi^t_i}\left[\frac{16\gamma^4}{n^4}\left(\left\|\omega_{i,t}^u\right\|^2 - \EE_{\xi^t_i}\left[\left\|\omega_{i,t}^u\right\|^2\right]\right)^2\right]$$
    of the summands are bounded:
    \begin{eqnarray}
        \widetilde\sigma_{i,t}^2 &\overset{\eqref{eq:prox_clipped_SGD_convex_technical_10}}{\leq}& \frac{V}{12 \ln\frac{48n(K+1)}{\beta}} \EE_{\xi^t_i}\left[\frac{4\gamma^2}{n^2}\left|\left\|\omega_{i,t}^u\right\|^2 - \EE_{\xi^t_i}\left[\left\|\omega_{i,t}^u\right\|^2\right]\right|\right] \notag\\
        &\leq& \frac{\gamma^2 V}{3n^2\ln\frac{48n(K+1)}{\beta}} \EE_{\xi^t_i}\left[\|\omega_{i,t}^u\|^2\right]. \label{eq:prox_clipped_SGD_convex_technical_11}
    \end{eqnarray}
    Applying Bernstein's inequality (Lemma~\ref{lem:Bernstein_ineq}) with $X_{i,l} = \frac{4\gamma^2}{n^2}\left(\|\omega_{i,l}^u\|^2  - \EE_{\xi_{i}^l}\left[\|\omega_{i,l}^u\|^2\right]\right)$, constant $c$ defined in \eqref{eq:prox_clipped_SGD_convex_technical_10}, $b = \frac{V}{12}$, $G = \frac{V^2}{864\ln\frac{48n(K+1)}{\beta}}$, we get
    \begin{equation*}
        \PP\left\{|\circledThree| > \frac{V}{12}\quad \text{and}\quad \sum\limits_{l=0}^{T-1}\sum\limits_{i=1}^n \widetilde\sigma_{i,l}^2 \leq \frac{V^2}{864\ln\frac{48n(K+1)}{\beta}}\right\} \leq 2\exp\left(- \frac{b^2}{2G + \nicefrac{2cb}{3}}\right) = \frac{\beta}{24n(K+1)}.
    \end{equation*}
    The above is equivalent to 
    \begin{equation}
        \PP\left\{ E_{\circledThree} \right\} \geq 1 - \frac{\beta}{24n(K+1)},\quad \text{for}\quad E_{\circledThree} = \left\{ \text{either} \quad  \sum\limits_{l=0}^{T-1}\sum\limits_{i=1}^n \widetilde\sigma_{i,l}^2 > \frac{V^2}{864\ln\frac{48n(K+1)}{\beta}} \quad \text{or}\quad |\circledThree| \leq \frac{V}{12}\right\}. \label{eq:prox_clipped_SGD_convex_sum_3_upper_bound}
    \end{equation}
    Moreover, $E_{T-1}$ implies
    \begin{eqnarray}
        \sum\limits_{l=0}^{T-1}\sum\limits_{i=1}^n \widetilde\sigma_{i,l}^2 &\overset{\eqref{eq:prox_clipped_SGD_convex_technical_11}}{\leq}& \frac{\gamma^2V}{3n^2\ln\frac{48n(K+1)}{\beta}} \sum\limits_{l=0}^{T-1}\sum\limits_{i=1}^n  \EE_{\xi^l_i}\left[\|\omega_{i,l}^u\|^2\right] \overset{\eqref{eq:prox_clipped_SGD_convex_second_moment_theta_u_bound}}{\leq}  
        \frac{18\gamma^2V\lambda^{2-\alpha}\sigma^{\alpha}T}{3n\ln\frac{48n(K+1)}{\beta}}\notag\\
        &\overset{\eqref{eq:prox_clipped_SGD_clipping_level_cvx}}{=}&
        \frac{ 48^\alpha\cdot 6}{48^2}\cdot\frac{\sigma^\alpha T \sqrt{V}^{4-\alpha}\gamma^{\alpha}}{n^{\alpha-1}\ln^{3-\alpha}\frac{48n(K+1)}{\beta}} \overset{\eqref{eq:prox_clipped_SGD_step_size_cvx}}{\leq} \frac{V^2}{864 \ln\frac{48n(K+1)}{\beta}}. \label{eq:prox_clipped_SGD_convex_sum_3_variance_bound}
    \end{eqnarray}

\textbf{Upper bound for $\circledFour$.} Probability event $E_{T-1}$ implies
\begin{equation}
        \circledFour = \frac{4\gamma^2}{n^2}\sum\limits_{l=0}^{T-1}\sum\limits_{i=1}^n\EE_{\xi^l_i}\left[\left\|\omega_{i,l}^u\right\|^2\right] \overset{\eqref{eq:prox_clipped_SGD_convex_second_moment_theta_u_bound}}{\leq} \frac{72\gamma^2\lambda^{2-\alpha}\sigma^{\alpha}T}{n} \overset{\eqref{eq:prox_clipped_SGD_clipping_level_cvx}}{=}
        \frac{48^{\alpha}\gamma^{\alpha}\sigma^{\alpha}T\sqrt{V}^{2-\alpha}}{32n^{\alpha-1}\ln^{2-\alpha}\frac{48n(K+1)}{\beta}}\overset{\eqref{eq:prox_clipped_SGD_step_size_cvx}}{\leq} \frac{V}{12}.\label{eq:prox_clipped_SGD_convex_sum_4_upper_bound}
\end{equation}

\textbf{Upper bound for $\circledFive$.} Probability event $E_{T-1}$ implies
\begin{equation}
        \circledFive = 4\gamma^2\sum\limits_{l=0}^{T-1}\left\|\omega_{l}^b\right\|^2 \overset{\eqref{eq:prox_clipped_SGD_convex_norm_theta_b_bound}}{\leq} \frac{4^{(\alpha+1)} \sigma^{2\alpha}T\gamma^2}{\lambda^{2(\alpha-1)}} \overset{\eqref{eq:prox_clipped_SGD_clipping_level_cvx}}{=}   \frac{9216^\alpha}{576}\cdot \frac{\sigma^{2\alpha}T\gamma^{2\alpha}\sqrt{V}^{2(1-\alpha)}}{n^{2(1-\alpha)}\ln^{2(1-\alpha)}\frac{48n(K+1)}{\beta}}\overset{\eqref{eq:prox_clipped_SGD_step_size_cvx}}{\leq} \frac{V}{6}.\label{eq:prox_clipped_SGD_convex_sum_5_upper_bound}
\end{equation}

\paragraph{Upper bounds for $\circledSix$.} This sum requires a more refined analysis. We introduce new vectors:
    \begin{equation}
        \delta_{j}^l = \begin{cases}
            \frac{\gamma}{n}\sum\limits_{i=1}^{j-1} \omega_{i,l}^u,& \text{if } \left\|\frac{\gamma}{n}\sum\limits_{i=1}^{j-1} \omega_{i,l}^u\right\| \leq \frac{\sqrt{V}}{2},\\
            0,& \text{otherwise,}
        \end{cases}\label{eq:SGD_zeta_delta_def}
    \end{equation}
    for all $j \in [n]$ and $l = 0,\ldots, T-1$. Then, by definition
    \begin{equation}
        \|\delta_j^l\| \leq \frac{\sqrt{V}}{2} \label{eq:SGD_zeta_delta_bound}
    \end{equation}
    and
    \begin{eqnarray}
        \circledSix &=& \underbrace{\frac{8\gamma}{n} \sum\limits_{l=0}^{T-1} \sum\limits_{j=2}^n \left\langle \delta_j^l, \omega_{j,l}^u \right\rangle}_{\circledSix'} + \frac{8\gamma}{n} \sum\limits_{l=0}^{T-1} \sum\limits_{j=2}^n \left\langle \frac{\gamma}{n}\sum\limits_{i=1}^{j-1} \omega_{i,l}^u - \delta_j^l, \omega_{j,l}^u \right\rangle. \label{eq:SGD_extra_sums_seven_distributed}
    \end{eqnarray}
    We also note here that $E_{T-1}$ implies
    \begin{eqnarray}
        \frac{8\gamma}{n} \sum\limits_{l=0}^{T-1} \sum\limits_{j=2}^n \left\langle \frac{\gamma}{n}\sum\limits_{i=1}^{j-1} \omega_{i,l}^u - \delta_j^l, \omega_{j,l}^u \right\rangle &=& \frac{8\gamma}{n}  \sum\limits_{j=2}^n \left\langle \frac{\gamma}{n}\sum\limits_{i=1}^{j-1} \omega_{i,T-1}^u - \delta_j^{T-1}, \omega_{j,T-1}^u \right\rangle. \label{eq:SGD_extra_sums_seven_distributed_1}
    \end{eqnarray}
    
\paragraph{Upper bound for $\circledSix'$.} To estimate this sum, we will use Bernstein's inequality. The summands have conditional expectations equal to zero:
    \begin{equation*}
    \EE_{\xi_{j}^l}\left[\frac{8\gamma}{n}\left\langle \delta_j^l, \omega_{j,l}^u \right\rangle\right] = \frac{8\gamma}{n}\left\langle \delta_j^l, \EE_{\xi_{j}^l}[\omega_{j,l}^u] \right\rangle = 0.
    \end{equation*}
    Moreover, for all $l = 0,\ldots, T-1$ random vectors $\{\omega_{i,l}^u\}_{i=1}^n$ are independent. Thus, sequence $\left\{\frac{8\gamma}{n}\left\langle \delta_j^l, \omega_{j,l}^u \right\rangle\right\}_{l,j = 0,2}^{T-1,n}$ is a martingale difference sequence. Next, the summands are bounded:
    \begin{eqnarray}
        \left|\frac{8\gamma}{n}\left\langle \delta_j^l, \omega_{j,l}^u \right\rangle \right| \leq \frac{8\gamma}{n} \left\|\delta_j^l\right\| \cdot \|\omega_{j,l}^u\| \overset{\eqref{eq:SGD_zeta_delta_bound},\eqref{eq:prox_clipped_SGD_convex_norm_theta_u_bound}}{\leq} \frac{8\gamma}{n}\cdot \frac{\sqrt{V}}{2}\cdot 2\lambda = \frac{V}{6\ln\frac{48n(K+1)}{\beta}} \eqdef c. \label{eq:gap_thm_SGD_technical_6_5_1234}
    \end{eqnarray}
    Finally, conditional variances $(\sigma_{j,l}')^2 \eqdef \EE_{\xi_{j}^l}\left[\frac{64\gamma^2}{n^2} \langle \delta_j^l, \omega_{j,l}^u \rangle^2\right]$ of the summands are bounded:
    \begin{equation}
        (\sigma_{j,l}')^2 \leq \EE_{\xi_{j}^l}\left[\frac{64\gamma^2 }{n^2}\|\delta_j^l\|^2\cdot \|\omega_{j,l}^u\|^2\right] \overset{\eqref{eq:SGD_zeta_delta_bound}}{\leq} \frac{16\gamma^2 V}{n^2} \EE_{\xi_{j}^l}\left[\|\omega_{j,l}^u\|^2\right]. \label{eq:gap_thm_SGD_technical_7_1234}
    \end{equation}
    Applying Bernstein's inequality (Lemma~\ref{lem:Bernstein_ineq}) with $X_{i,l} = \frac{8\gamma}{n}\left\langle \delta_j^l, \omega_{j,l}^u \right\rangle$, constant $c$ defined in \eqref{eq:gap_thm_SGD_technical_6_5_1234}, $b = \frac{V}{6}$, $G = \frac{V^2}{216\ln\frac{48n(K+1)}{\beta}}$, we get
    \begin{equation*}
        \PP\left\{|\circledSix'| > \frac{V}{6} \text{ and } \sum\limits_{l=0}^{T-1}\sum\limits_{j=2}^n(\sigma_{i,l}')^2 \leq \frac{V^2}{216\ln\frac{48n(K+1)}{\beta}}\right\} \leq 2\exp\left(- \frac{b^2}{2G + \nicefrac{2cb}{3}}\right) = \frac{\beta}{24n(K+1)}.
    \end{equation*}
    The above is equivalent to 
    \begin{equation}
        \PP\{E_{\circledSix'}\} \geq 1 - \frac{\beta}{24n(K+1)},\; \text{ for }\; E_{\circledSix'} = \left\{\text{either } \sum\limits_{l=0}^{T-1}\sum\limits_{j=2}^n(\sigma_{i,l}')^2 > \frac{V^2}{216\ln\frac{48n(K+1)}{\beta}} \text{ or } |\circledSix'| \leq \frac{V}{6}\right\}. \label{eq:bound_7_gap_SGD}
    \end{equation}
    Moreover, $E_{T-1}$ implies
    \begin{eqnarray}
        \sum\limits_{l=0}^{T-1}\sum\limits_{j=2}^n(\sigma_{j,l}')^2 &\overset{\eqref{eq:gap_thm_SGD_technical_7_1234}}{\leq}& \frac{16\gamma^2 V}{n^2} \sum\limits_{l=0}^{T-1} \sum\limits_{j=2}^n \EE_{\xi_{j}^l}\left[\|\omega_{j,l}^u\|^2\right] \overset{\eqref{eq:prox_clipped_SGD_convex_second_moment_theta_u_bound}, T \leq K+1}{\leq} \frac{288(K+1)\gamma^2 V \lambda^{2-\alpha} \sigma^\alpha}{n} \notag \\
        &\overset{\eqref{eq:prox_clipped_SGD_clipping_level_cvx}}{\leq}& \frac{288 (K+1)\gamma^{\alpha}\sigma^\alpha V^{2-\frac{\alpha}{2}}}{48^{2-\alpha} n^{\alpha-1} \ln^{2-\alpha}\frac{48n(K+1)}{\beta}} \overset{\eqref{eq:prox_clipped_SGD_step_size_cvx}}{\leq} \frac{V^2}{216\ln\frac{48n(K+1)}{\beta}}. \label{eq:bound_7_variances_gap_SGD}
    \end{eqnarray}

That is, we derive the upper bounds for  $\circledOne, \circledTwo, \circledThree, \circledFour, \circledFive, \circledSix$. More precisely, $E_{T-1}$ implies
\begin{gather*}
        A_T \overset{\eqref{eq:prox_clipped_SGD_convex_technical_7}}{\leq} V + \circledOne + \circledTwo + \circledThree + \circledFour + \circledFive + \circledSix,\\
        \circledSix \overset{\eqref{eq:SGD_extra_sums_seven_distributed}}{=} \circledSix' + \frac{8\gamma}{n} \sum\limits_{j=2}^n \left\langle \frac{\gamma}{n}\sum\limits_{i=1}^{j-1} \omega_{i,T-1}^u - \delta_j^{T-1}, \omega_{j,T-1}^u \right\rangle,\\
        \circledTwo \overset{\eqref{eq:prox_clipped_SGD_convex_sum_2_upper_bound}}{\leq} \frac{V}{6},\quad \circledFour \overset{\eqref{eq:prox_clipped_SGD_convex_sum_4_upper_bound}}{\leq} \frac{V}{12}, \quad \circledFive \overset{\eqref{eq:prox_clipped_SGD_convex_sum_5_upper_bound}}{\leq} \frac{V}{6},\\
        \sum\limits_{t=0}^{T-1} \sigma_t^2 \overset{\eqref{eq:prox_clipped_SGD_convex_sum_1_variance_bound}}{\leq} \frac{V^2}{216 \ln\frac{48n(K+1)}{\beta}},\quad \sum\limits_{t=0}^{T-1} \widetilde\sigma_t^2 \overset{\eqref{eq:prox_clipped_SGD_convex_sum_3_variance_bound}}{\leq} \frac{V^2}{864 \ln\frac{48n(K+1)}{\beta}},\\
        \sum\limits_{l=0}^{T-1}\sum\limits_{j=2}^n(\sigma_{j,l}')^2 \overset{\eqref{eq:bound_7_variances_gap_SGD}}{\leq} \frac{V^2}{216\ln\frac{48n(K+1)}{\beta}}.
\end{gather*}
    In addition, we also establish (see \eqref{eq:prox_clipped_SGD_convex_sum_1_upper_bound}, \eqref{eq:prox_clipped_SGD_convex_sum_3_upper_bound}, \eqref{eq:bound_7_gap_SGD} and our induction assumption)
\begin{gather*}
        \PP\{E_{T-1}\} \geq 1 - \frac{(T-1)\beta}{K+1},\\
        \PP\{E_{\circledOne}\} \geq 1 - \frac{\beta}{24n(K+1)}, \quad \PP\{E_{\circledThree}\} \geq 1 - \frac{\beta}{24n(K+1)}, \quad \PP\{E_{\circledSix'}\} \geq 1 - \frac{\beta}{24n(K+1)},
    \end{gather*}
    where
\begin{eqnarray*}
        E_{\circledOne} &=& \left\{ \text{either} \quad  \sum\limits_{l=0}^{T-1} \sigma_{l}^2 > \frac{V^2}{216\ln\frac{48n(K+1)}{\beta}} \quad \text{or}\quad |\circledOne| \leq \frac{V}{6}\right\},\\
        E_{\circledThree} &=& \left\{ \text{either} \quad  \sum\limits_{l=0}^{T-1} \widetilde\sigma_{l}^2 > \frac{V^2}{864\ln\frac{48n(K+1)}{\beta}} \quad \text{or}\quad |\circledThree| \leq \frac{V}{12}\right\}, \\
        E_{\circledSix'} &=& \left\{ \text{either} \quad  \sum\limits_{t=0}^{T-1} \sum\limits^n_{j=2}(\sigma_{j,l}')^2  > \frac{V^2}{216\ln\frac{48n(K+1)}{\beta}} \quad \text{or}\quad |\circledSix'| \leq \frac{V}{6}\right\}.
\end{eqnarray*}
     Therefore, probability event $E_{T-1} \cap E_{\circledOne} \cap E_{\circledThree} \cap E_{\circledSix'}$ implies

     \begin{eqnarray}
        A_T &\leq& V + \frac{V}{6} + \frac{V}{6} + \frac{V}{12} + \frac{V}{12} + \frac{V}{6} + \frac{V}{6}\notag\\
        &&\quad + \frac{8\gamma}{n} \sum\limits_{j=2}^n \left\langle \frac{\gamma}{n}\sum\limits_{i=1}^{j-1} \omega_{i,T-1}^u - \delta_j^{T-1}, \omega_{j,T-1}^u \right\rangle \notag \\
        &\leq& 2V + \frac{8\gamma}{n} \sum\limits_{j=2}^n \left\langle \frac{\gamma}{n}\sum\limits_{i=1}^{j-1} \omega_{i,T-1}^u - \delta_j^{T-1}, \omega_{j,T-1}^u \right\rangle \label{eq:gap_thm_SGD_technical_11}
    \end{eqnarray}
    for $t = T$. 

    In the final part of the proof, we will show that $\frac{\gamma}{n}\sum\limits_{i=1}^{j-1} \omega_{i,T-1}^u = \delta_j^{T-1}$ with high probability. In particular, we consider probability event $\widetilde{E}_{T-1,j}$ defined as follows: inequalities
    \begin{equation}
         \left\|\frac{\gamma}{n}\sum\limits_{i=1}^{r-1} \omega_{i,T-1}^u\right\| \leq \frac{\sqrt{V}}{2} \notag
    \end{equation}
    hold for $r = 2, \ldots, j$ simultaneously. We want to show that $\PP\{E_{T-1} \cap \widetilde{E}_{T-1,j}\} \geq 1 - \frac{(T-1)\beta}{K+1} - \frac{j\beta}{8n(K+1)}$ for all $j = 2, \ldots, n$. For $j = 2$ the statement is trivial since
    \begin{eqnarray*}
         \left\|\frac{\gamma}{n} \omega_{1,T-1}^u\right\| \overset{\eqref{eq:prox_clipped_SGD_convex_norm_theta_u_bound}}{\leq} \frac{2\gamma\lambda}{n} \leq \frac{\sqrt{V}}{2}.
    \end{eqnarray*}
    Next, we assume that the statement holds for some $j = m-1 < n$, i.e.,  $\PP\{E_{T-1}\cap\widetilde{E}_{T-1,m-1}\} \geq 1 - \frac{(T-1)\beta}{K+1} - \frac{(m-1)\beta}{8n(K+1)}$. Our goal is to prove that $\PP\{E_{T-1}\cap\widetilde{E}_{T-1,m}\} \geq 1 - \frac{(T-1)\beta}{K+1} - \frac{m\beta}{8n(K+1)}$. We have
    \begin{eqnarray}
        \left\|\frac{\gamma}{n}\sum\limits_{i=1}^{m-1} \omega_{i,T-1}^u\right\| &=& \sqrt{\frac{\gamma^2}{n^2}\left\|\sum\limits_{i=1}^{m-1} \omega_{i,T-1}^u\right\|^2} \notag\\
        &=& \sqrt{\frac{\gamma^2}{n^2}\sum\limits_{i=1}^{m-1}\|\omega_{i,T-1}^u\|^2 + \frac{2\gamma}{n}\sum\limits_{i=1}^{m-1}\left\langle \frac{\gamma}{n}\sum\limits_{r=1}^{i-1}\omega_{r,T-1}^u , \omega_{i,T-1}^u \right\rangle} \notag\\
        &\leq& \sqrt{\frac{\gamma^2}{n^2}\sum\limits_{l=0}^{T-1}\sum\limits_{i=1}^{m-1}\|\omega_{i,l}^u\|^2 + \frac{2\gamma}{n}\sum\limits_{i=1}^{m-1}\left\langle \frac{\gamma}{n}\sum\limits_{r=1}^{i-1}\omega_{r,T-1}^u , \omega_{i,T-1}^u \right\rangle}. \notag
    \end{eqnarray}

    Next, we introduce a new notation:
    \begin{gather*}
        \rho_{i,T-1}' = \begin{cases}
            \frac{\gamma}{n}\sum\limits_{r=1}^{i-1}\omega_{r,T-1}^u,& \text{if } \left\|\frac{\gamma}{n}\sum\limits_{r=1}^{i-1}\omega_{r,T-1}^u\right\| \leq \frac{\sqrt{V}}{2},\\
            0,& \text{otherwise}
        \end{cases}
    \end{gather*}
    for $i = 1,\ldots,m-1$. By definition, we have
    \begin{equation}
        \|\rho_{i,T-1}'\| \leq \frac{\sqrt{V}}{2} \label{eq:SGD_bound_rho}
    \end{equation}
    for $i = 1,\ldots,m-1$. Moreover, $\widetilde{E}_{T-1,m-1}$ implies $\rho_{i,T-1}' = \frac{\gamma}{n}\sum\limits_{r=1}^{i-1}\omega_{r,T-1}^u$ for $i = 1,\ldots,m-1$ and
    \begin{eqnarray}
        \left\|\frac{\gamma}{n}\sum\limits_{i=1}^{m-1} \omega_{i,l}^u\right\| &\leq& \sqrt{\circledThree + \circledFour + \circledSeven}, \notag
    \end{eqnarray}
    where
    \begin{gather*}
         \circledSeven =  \frac{2\gamma}{n}\sum\limits_{i=1}^{m-1}\left\langle \rho_{i,T-1}' , \omega_{i,T-1}^u \right\rangle.
    \end{gather*}
    It remains to estimate $\circledSeven$.

     \textbf{Upper bound for $\circledSeven$.} To estimate this sum, we will use Bernstein's inequality. The summands have conditional expectations equal to zero:
    \begin{equation*}
        \EE_{\xi_{i}^{T-1}}\left[\frac{2\gamma}{n} \langle \rho_{i,T-1}', \omega_{i,T-1}^u \rangle\right] = \frac{2\gamma}{n} \left\langle \rho_{i,T-1}', \EE_{\xi_{i}^{T-1}}[\omega_{i,T-1}^u] \right\rangle = 0,
    \end{equation*}
    since random vectors $\{\omega_{i,T-1}^u\}_{i=1}^n$ are independent. Thus, sequence $\left\{\frac{2\gamma}{n} \langle \rho_{i,T-1}', \omega_{i,T-1}^u \rangle\right\}_{i = 1}^{m-1}$ is a martingale difference sequence. Next, the summands are bounded:
    \begin{eqnarray}
        \left|\frac{2\gamma}{n}\langle \rho_{i,T-1}', \omega_{i,T-1}^u \rangle \right| \leq \frac{2\gamma}{n} \|\rho_{i,T-1}'\|\cdot \|\omega_{i,T-1}^u\| \overset{\eqref{eq:SGD_bound_rho},\eqref{eq:prox_clipped_SGD_convex_norm_theta_u_bound}}{\leq} \frac{\gamma}{n} \sqrt{V}\lambda \overset{\eqref{eq:prox_clipped_SGD_clipping_level_cvx}}{=} \frac{V}{24\ln\tfrac{48n(K+1)}{\beta}} \eqdef c. \label{eq:gap_thm_SGD_technical_6_5_10'}
    \end{eqnarray}
    Finally, conditional variances $(\widetilde\sigma_{i,T-1}')^2 \eqdef \EE_{\xi_{i}^{T-1}}\left[\frac{4\gamma^2}{n^2} \langle \rho_{i,T-1}', \omega_{i,T-1}^u \rangle^2\right]$ of the summands are bounded:
    \begin{equation}
        (\widetilde\sigma_{i,T-1}')^2 \leq \EE_{\xi_{i}^{T-1}}\left[\frac{4\gamma^2 }{n^2}\|\rho_{i,T-1}'\|^2\cdot \|\omega_{i,T-1}^u\|^2\right] \overset{\eqref{eq:SGD_bound_rho}}{\leq} \frac{\gamma^2 V}{n^2} \EE_{\xi_{i}^{T-1}}\left[\|\omega_{i,T-1}^u\|^2\right]. \label{eq:gap_thm_SGD_technical_7_10'}
    \end{equation}
    Applying Bernstein's inequality (Lemma~\ref{lem:Bernstein_ineq}) with $X_{i} = \frac{2\gamma}{n} \langle \rho_{i,T-1}', \omega_{i,T-1}^u \rangle$, constant $c$ defined in \eqref{eq:gap_thm_SGD_technical_6_5_10'}, $b = \frac{V}{24}$, $G = \frac{V^2}{3456\ln\frac{48n(K+1)}{\beta}}$, we get
    \begin{equation*}
        \PP\left\{|\circledSeven| > \frac{V}{24} \text{ and } \sum\limits_{i=1}^{m-1}(\widetilde\sigma_{i,T-1}')^2 \leq \frac{V^2}{3456\ln\frac{48n(K+1)}{\beta}}\right\} \leq 2\exp\left(- \frac{b^2}{2G + \nicefrac{2cb}{3}}\right) = \frac{\beta}{24n(K+1)}.
    \end{equation*}
    The above is equivalent to 
    \begin{equation}
        \PP\{E_{\circledSeven}\} \geq 1 - \frac{\beta}{24n(K+1)},\; \text{ for }\; E_{\circledSeven} = \left\{\text{either } \sum\limits_{i=1}^{m-1}(\widetilde\sigma_{i,T-1}')^2 > \frac{V^2}{3456\ln\frac{48n(K+1)}{\beta}} \text{ or } |\circledSeven| \leq \frac{V}{24}\right\}. \label{eq:bound_10_gap_SGD}
    \end{equation}
    Moreover, $E_{T-1}$ implies
    \begin{eqnarray}
        \sum\limits_{i=1}^{m-1}(\widetilde\sigma_{i,T-1}')^2 &\overset{\eqref{eq:gap_thm_SGD_technical_7_10'}}{\leq}& \frac{\gamma^2 V}{n^2} \sum\limits_{i=1}^n \EE_{\xi_{i}^{T-1}}\left[\|\omega_{i,T-1}^u\|^2\right] \overset{\eqref{eq:prox_clipped_SGD_convex_second_moment_theta_u_bound}}{\leq} \frac{18\gamma^2 V \lambda^{2-\alpha} \sigma^\alpha}{n} \notag \\
        &\overset{\eqref{eq:prox_clipped_SGD_clipping_level_cvx}}{\leq}& \frac{18 \gamma^{\alpha}\sigma^\alpha V^{2-\frac{\alpha}{2}}}{48^{2-\alpha} n^{\alpha-1} \ln^{2-\alpha}\frac{48n(K+1)}{\beta}} \overset{\eqref{eq:prox_clipped_SGD_step_size_cvx}}{\leq} \frac{V^2}{3456\ln\tfrac{48n(K+1)}{\beta}}. \label{eq:bound_10_variances_gap_SGD}
    \end{eqnarray}

    Putting all together we get that $E_{T-1}\cap \widetilde{E}_{T-1,m-1}$ implies
    \begin{gather*}
         \left\|\frac{\gamma}{n}\sum\limits_{i=1}^{m-1} \omega_{i,T-1}^u\right\| \leq \sqrt{\circledThree + \circledFour + \circledSeven},\quad \circledFour \overset{\eqref{eq:prox_clipped_SGD_convex_sum_4_upper_bound}}{\leq} \frac{V}{6},\\
        \sum\limits_{l=0}^{T-1}\sum\limits_{i=1}^n\widetilde\sigma_{i,l}^2 \overset{\eqref{eq:prox_clipped_SGD_convex_sum_3_variance_bound}}{\leq} \frac{V^2}{216\ln\frac{48n(K+1)}{\beta}},\quad 
        \sum\limits_{i=1}^{m-1}(\widetilde\sigma_{i,T-1}')^2 \leq \frac{V^2}{3456\ln\tfrac{48n(K+1)}{\beta}}.
    \end{gather*}
    In addition, we also establish (see \eqref{eq:prox_clipped_SGD_convex_sum_3_upper_bound}, \eqref{eq:bound_10_gap_SGD} and our induction assumption)
    \begin{gather*}
        \PP\{E_{T-1}\cap \widetilde{E}_{T-1,m-1}\} \geq 1 - \frac{(T-1)\beta}{K+1} - \frac{(m-1)\beta}{8n(K+1)},\\
        \PP\{E_{\circledThree}\} \geq 1 - \frac{\beta}{24n(K+1)}, \quad \PP\{E_{\circledSeven}\} \geq 1 - \frac{\beta}{24n(K+1)}
    \end{gather*}
    where
    \begin{eqnarray}
        E_{\circledThree} &=& \left\{\text{either } \sum\limits_{l=0}^{T-1}\sum\limits_{i=1}^n\widetilde\sigma_{i,l}^2 > \frac{V^2}{864\ln\frac{48n(K+1)}{\beta}} \text{ or } |\circledThree| \leq \frac{V}{12}\right\}, \notag\\
        E_{\circledSeven} &=& \left\{\text{either } \sum\limits_{i=1}^{m-1}(\widetilde\sigma_{i,T-1}')^2 > \frac{V^2}{3456\ln\frac{48n(K+1)}{\beta}} \text{ or } |\circledSeven| \leq \frac{V}{24}\right\} \notag
    \end{eqnarray}
    Therefore, probability event $E_{T-1} \cap \widetilde{E}_{T-1,m-1} \cap E_{\circledThree} \cap E_{\circledSeven}$ implies
    \begin{gather*}
        \left\|\frac{\gamma}{n}\sum\limits_{i=1}^{m-1} \omega_{i,T-1}^u\right\| \leq \sqrt{\frac{V}{12} + \frac{V}{12} + \frac{V}{24}} \leq \frac{\sqrt{V}}{2}.
    \end{gather*}
     This implies $\widetilde E_{T-1,m}$ and
    \begin{eqnarray*}
        \PP\{E_{T-1} \cap \widetilde{E}_{T-1,m}\} &\geq& \PP\{E_{T-1} \cap \widetilde{E}_{T-1,m-1} \cap E_{\circledThree} \cap E_{\circledSeven} \} \\
        &=& 1 - \PP\left\{\overline{E_{T-1} \cap \widetilde{E}_{T-1,m-1}} \cup \overline{E}_{\circledThree} \cup \overline{E}_{\circledSeven} \right\}\\
        &\geq& 1 - \frac{(T-1)\beta}{K+1} - \frac{m\beta}{8n(K+1)}.
    \end{eqnarray*}
    Therefore, for all $m = 2,\ldots,n$ the statement holds and, in particular, $\PP\{E_{T-1} \cap \widetilde{E}_{T-1, n}\} \geq 1 - \frac{(T-1)\beta}{K+1} - \frac{\beta}{8(K+1)}$. Taking into account \eqref{eq:gap_thm_SGD_technical_11}, we conclude that $E_{T-1} \cap \widetilde{E}_{T-1, n} \cap E_{\circledOne} \cap E_{\circledThree} \cap E_{\circledSix'}$ implies
    \begin{equation*}
        A_T \leq 2V
    \end{equation*}
    that is equivalent to \eqref{eq:prox_clipped_SGD_convex_induction_inequality_1} for $t = T$. Moreover,
    \begin{eqnarray*}
        \PP\left\{E_T\right\} &\geq& \PP\left\{E_{T-1} \cap \widetilde{E}_{T-1,n} \cap E_{\circledOne} \cap E_{\circledThree} \cap E_{\circledSix'}\right\} \\
        &=& 1 - \PP\left\{\overline{E_{T-1} \cap \widetilde{E}_{n}} \cup \overline{E}_{\circledOne} \cup \overline{E}_{\circledThree} \cup \overline{E}_{\circledSix'}\right\}\\
        &=& 1 - \frac{(T-1)\beta}{K+1} - \frac{\beta}{8(K+1)} - 3\cdot \frac{\beta}{24n(K+1)} = 1 - \frac{T\beta}{K+1}.
    \end{eqnarray*}
     In other words, we showed that $\PP\{E_k\} \geq 1 - \nicefrac{k\beta}{(K+1)}$ for all $k = 0,1,\ldots,K+1$. For $k = K+1$ we have that with probability at least $1 - \beta$
\begin{equation*}
        \Phi(\overline{x}^{K+1}) - \Phi(x^*) \overset{\eqref{eq:prox_clipped_SGD_convex_technical_1},\eqref{eq:prox_clipped_SGD_convex_induction_inequality_1}}{\leq}\frac{V}{\gamma(K+1)}.
\end{equation*}
Finally, if
\begin{equation*}
        \gamma \leq \min\left\{\frac{1}{360 L \ln \frac{48n(K+1)}{\beta}}, \; \frac{n^{\frac{\alpha-1}{\alpha}}\sqrt{V}}{27^{\frac{1}{\alpha}}\cdot 48 \sigma  K^{\frac{1}{\alpha}}\left(\ln \frac{48n(K+1)}{\beta}\right)^{\frac{\alpha-1}{\alpha}}}\right\},
\end{equation*}
then with probability at least $1-\beta$
\begin{eqnarray*}
        \Phi(\overline{x}^{K+1}) - \Phi(x^*) &\leq& \frac{V}{\gamma(K+1)}\\
        &=& \max\left\{\frac{360 LV \ln \frac{48n(K+1)}{\beta}}{K+1}, \frac{48\cdot 27^{\frac{1}{\alpha}}\sigma \sqrt{V}K^{\frac{1}{\alpha}}\left(\ln\frac{48n(K+1)}{\beta}\right)^{\frac{
        \alpha-1
        }{\alpha}}}{n^{\frac{\alpha-1}{\alpha}}(K+1)}\right\}\notag\\
        &=& \cO\left(\max\left\{\frac{L V \ln \frac{nK}{\beta}}{K}, \frac{ \sigma \sqrt{V}\ln^{\frac{\alpha-1}{\alpha}}\frac{nK}{\beta}}{n^{\frac{\alpha-1}{\alpha}}K^{\frac{\alpha-1}{\alpha}}}\right\}\right).
\end{eqnarray*}
To get $\Phi(\overline{x}^{K+1}) - \Phi(x^*) \leq \varepsilon$ with probability at least $1-\beta$ it is sufficient to choose $K$ such that both terms in the maximum above are $\cO(\varepsilon)$. This leads to
\begin{equation*}
         K = \cO\left(\max\left\{\frac{LV}{\varepsilon} \ln\frac{L V}{\varepsilon\beta}, \left(\frac{\sigma \sqrt{V}}{\varepsilon n^{\frac{\alpha-1}{\alpha}}}\right)^{\frac{\alpha}{\alpha-1}} \ln \left(\frac{1}{\beta} \left(\frac{\sigma \sqrt{V}}{\varepsilon }\right)^{\frac{\alpha}{\alpha-1}}\right) \right\}\right),
\end{equation*}
which concludes the proof.
\end{proof}

In view of Lemma~\ref{lem:QSC_to_QSM}, the result in the quasi-strongly convex case for \algname{DProx-clipped-SGD-shift} follows from our result for \algname{DProx-clipped-SGDA-shift}.

\clearpage

\section{Missing Proofs for \algname{DProx-clipped-SSTM-shift}}\label{appendix:dprox_clipped_SSTM_shift}

In this section, we provide the complete formulations of our results for \algname{DProx-clipped-SSTM-shift} and proofs. For the readers' convenience, the method's update rule is repeated below: $x^0 = y^0 = z^0$, $A_0 = \alpha_0 = 0$, $\alpha_{k+1} = \frac{k+2}{2aL}$, $A_{k+1} = A_k + \alpha_{k+1}$ and
\begin{gather*}
    x^{k+1} = \frac{A_k y^k + \alpha_{k+1}z^k}{A_{k+1}},\quad z^{k+1} = \prox_{\alpha_{k+1}\Psi}\left(z^k - \alpha_{k+1}\tg(x^{k+1})\right), \\
    \tg(x^{k+1}) = \frac{1}{n}\sum\limits_{i=1}^n \tg_i(x^{k+1}),\;\; \tg_i(x^{k+1}) = h_i^k + \hat\Delta_i^k, \\
    h_i^{k+1} = h_i^k + \nu_k \hat \Delta_i^k,\quad \hat\Delta_i^k = \clip\left(\nabla f_{\xi_i^k}(x^{k+1}) - h_i^k, \lambda_k\right), \\
    y^{k+1} = \frac{A_k y^k + \alpha_{k+1}z^{k+1}}{A_{k+1}} 
\end{gather*}
where $\xi_1^k,\ldots, \xi_n^k$ are sampled independently from each other and previous steps.

\subsection{Convex case}
The following lemma is the main ``optimization'' part of the analysis of \algname{DProx-clipped-SSTM-shift}.
\begin{lemma}
\label{lem:SSTM_optimization_lemma_convex}
    Let Assumptions~\ref{as:bounded_alpha_moment}, \ref{as:L_smoothness} and \ref{as:str_cvx}($\mu = 0$) hold on $Q = B_{5n\sqrt{M}}(x^*)$, where $M \geq \|x^0-x^*\|^2+C^2\alpha^2_{K_0+1}\frac{1}{n}\sum^n_{i=1}\|\nabla f_i(x^*)\|^2$, where $C >0$, and $a \geq 0$. Let $x^{k},~y^k,~z^k$ lie in $B_{5n\sqrt{M}}(x^*)$ for all $k = 0,1,\dots, K$ for some $K\geq 0$. Additionally, let parameters of \algname{DProx-clipped-SSTM-shift} satisfy
    \begin{equation}
        \label{eq:a_K0_D_prox_clipped_SSTM_shift_lemma}
        a \geq \max\left\{2, \frac{7}{6}C^2\right\},\quad K_0 = \left\lceil\frac{3}{2}C^2n\right\rceil;
    \end{equation}
    \begin{equation}
         \label{eq:nu_k_D_prox_clipped_SSTM_shift_lemma}
         \nu_k = \begin{cases}
                    \frac{(k+2)^2}{C^2(K_0+2)^2n} ,& \text{if $k < K_0$};\\
                    \frac{2k+5}{(k+3)^2} ,& \text{if $k \geq K_0$};
    \end{cases},     
    \end{equation}
    
    then the iterates produced by \algname{DProx-clipped-SSTM-shift} satisfy 
    \begin{eqnarray}
        A_K(\Phi(y^K)-\Phi(x^*)) &\leq& \frac{1}{2}M_0-\frac{1}{2}M_K + \sum^{K-1}_{k=0}\alpha_{k+1}\la\omega_{k+1}, x^*-z^k\ra  + \sum^{K-1}_{k=0}\alpha^2_{k+1}\|\omega_{k+1}\|^2\notag\\
        && + \sum^{K-1}_{k=0}\sum^{n}_{i=1}\frac{\alpha^2_{k+1}}{n^2}\|\omega_{i,k+1}\|^2,
    \end{eqnarray}
    where Lyapunov function $M_k$ is defined as follows 
    \begin{equation}
        \label{eq:Lyapunov_function_D_prox_clipped_SSTM_shift}
        M_k = \|z^k-x^*\|^2 +C^2\widetilde\alpha^2_{k+1}\frac{1}{n}\sum^n_{i=1}\|h^k_i-h^*_i\|^2,
    \end{equation}
    where 
    \begin{equation}
        \widetilde\alpha_{k+1} = \begin{cases}
                    \alpha_{K_0+1}& \text{if $k < K_0$};\\
                    \alpha_{k+1} & \text{if $k \geq K_0$};
    \end{cases},     
    \end{equation}
     and $\omega_{k+1}$ is defined as follows 
    \begin{equation}
    \label{eq:omega_k+1_def_dprox_clipped_SSTM_shift}
        \omega_{i,k+1} \eqdef \tg_i(x^{k+1}) - \nabla f_i(x^{k+1}),\quad \omega_{k+1} \eqdef \frac{1}{n}\sum^n_{i=1}\omega_{i,k+1}.
    \end{equation}
\end{lemma}
\begin{proof}
  By optimality condition for the problem \eqref{eq:D_prox_clipped_SSTM_shift_x_z}, we have for any $z\in B_{3\sqrt{M}}(x^*)$
\begin{eqnarray*}
    \alpha_{k+1}\left\la \tg(x^{k+1}), z^k-z\right\ra &\leq& \alpha_{k+1}(\Psi(z)-\Psi(z^{k+1})) +\alpha_{k+1}\left\la \tg(x^{k+1}), z^k-z^{k+1}\right\ra \\
    &&+ \frac{1}{2}\|z^k-z\|^2 - \frac{1}{2}\|z^{k+1}-z\|^2 - \frac{1}{2}\|z^{k+1}-z^k\|^2  \\
    &\overset{\eqref{eq:omega_k+1_def_dprox_clipped_SSTM_shift}}{\leq}& \alpha_{k+1}(\Psi(z)-\Psi(z^{k+1})) +\alpha_{k+1}\left\la \omega_{k+1}, z^k-z^{k+1}\right\ra \\ && +\alpha_{k+1}\left\la \nabla f(x^{k+1}), z^k-z^{k+1}\right\ra\\
    && + \frac{1}{2}\|z^k-z\|^2 - \frac{1}{2}\|z^{k+1}-z\|^2 - \frac{1}{2}\|z^{k+1}-z^k\|^2 
\end{eqnarray*}
Using $A_{k+1}(y^{k+1}-x^{k+1})= \alpha_{k+1}(z^{k+1}-z^k)$, we get 
\begin{eqnarray*}
    \alpha_{k+1}\left\la \tg(x^{k+1}), z^k-z\right\ra &\leq& \alpha_{k+1}(\Psi(z)-\Psi(z^{k+1})) +\alpha_{k+1}\left\la \omega_{k+1}, z^k-z^{k+1}\right\ra \\ && +A_{k+1}\left\la \nabla f(x^{k+1}), x^{k+1}-y^{k+1}\right\ra\\
    && + \frac{1}{2}\|z^k-z\|^2 - \frac{1}{2}\|z^{k+1}-z\|^2 - \frac{1}{2}\|z^{k+1}-z^k\|^2 \\
    &\overset{(\ast)}{\leq}& \alpha_{k+1}(\Psi(z)-\Psi(z^{k+1})) +\alpha_{k+1}\left\la \omega_{k+1}, z^k-z^{k+1}\right\ra \\ && +A_{k+1}\left(f(x^{k+1})-f(y^{k+1}) + \frac{L}{2}\|y^{k+1}-x^{k+1}\|^2\right)\\
    && + \frac{1}{2}\|z^k-z\|^2 - \frac{1}{2}\|z^{k+1}-z\|^2 - \frac{1}{2}\|z^{k+1}-z^k\|^2 \\
    &=& \alpha_{k+1}(\Psi(z)-\Psi(z^{k+1})) +\alpha_{k+1}\left\la \omega_{k+1}, z^k-z^{k+1}\right\ra \\ && +A_{k+1}\left(f(x^{k+1})-f(y^{k+1})\right) + \frac{\alpha_{k+1}^2L}{2A_{k+1}}\|z^{k+1}-z^{k}\|^2\\
    && + \frac{1}{2}\|z^k-z\|^2 - \frac{1}{2}\|z^{k+1}-z\|^2 - \frac{1}{2}\|z^{k+1}-z^k\|^2 \\
    &=& \alpha_{k+1}(\Psi(z)-\Psi(z^{k+1})) +\alpha_{k+1}\left\la \omega_{k+1}, z^k-z^{k+1}\right\ra \\ && +A_{k+1}\left(f(x^{k+1})-f(y^{k+1})\right) + \frac{1}{2}\|z^k-z\|^2 - \frac{1}{2}\|z^{k+1}-z\|^2\\
    &&  - \frac{1}{2}\left(1- \frac{\alpha_{k+1}^2L}{A_{k+1}}\right)\|z^{k+1}-z^k\|^2 
\end{eqnarray*}
where in $(\ast)$ $L$-smoothness of $f$ was used. 
Using Young's inequality, we have 
\begin{eqnarray*}
    \alpha_{k+1}\left\la \tg(x^{k+1}), z^k-z\right\ra &\leq&  \alpha_{k+1}(\Psi(z)-\Psi(z^{k+1})) +\alpha_{k+1}\frac{D}{2}\|\omega_{k+1}\|^2+ \frac{\alpha_{k+1}}{2D} \|z^k-z^{k+1}\|^2 \\ 
    && +A_{k+1}\left(f(x^{k+1})-f(y^{k+1})\right) + \frac{1}{2}\|z^k-z\|^2 - \frac{1}{2}\|z^{k+1}-z\|^2\\
    &&  - \frac{1}{2}\left(1- \frac{\alpha_{k+1}^2L}{A_{k+1}}\right)\|z^{k+1}-z^k\|^2 \\
    &\overset{D= 2\alpha_{k+1}}{=}&  \alpha_{k+1}(\Psi(z)-\Psi(z^{k+1})) +\alpha^2_{k+1}\|\omega_{k+1}\|^2+ \frac{1}{4} \|z^k-z^{k+1}\|^2 \\ 
    && +A_{k+1}\left(f(x^{k+1})-f(y^{k+1})\right) + \frac{1}{2}\|z^k-z\|^2 - \frac{1}{2}\|z^{k+1}-z\|^2\\
    &&  - \frac{1}{2}\left(1- \frac{\alpha_{k+1}^2L}{A_{k+1}}\right)\|z^{k+1}-z^k\|^2 
\end{eqnarray*}
Now, by  $a \geq 2$, we have $\frac{1}{2}- \frac{\alpha_{k+1}^2L}{A_{k+1}} \geq 0$ and 
\begin{eqnarray}
\label{eq:hdfjdofodpf}
    \alpha_{k+1}\left\la \tg(x^{k+1}), z^k-z\right\ra &\leq& \alpha_{k+1}(\Psi(z)-\Psi(z^{k+1})) +\alpha^2_{k+1}\|\omega_{k+1}\|^2 \notag\\ 
    && +A_{k+1}\left(f(x^{k+1})-f(y^{k+1})\right) + \frac{1}{2}\|z^k-z\|^2 - \frac{1}{2}\|z^{k+1}-z\|^2\\
    &&  - \frac{1}{2}\left(\frac{1}{2}- \frac{\alpha_{k+1}^2L}{A_{k+1}}\right)\|z^{k+1}-z^k\|^2 \notag\\
    &\leq& \alpha_{k+1}(\Psi(z)-\Psi(z^{k+1})) +\alpha^2_{k+1}\|\omega_{k+1}\|^2 +A_{k+1}\left(f(x^{k+1})-f(y^{k+1})\right) \notag\\
    &&+ \frac{1}{2}\|z^k-z\|^2 - \frac{1}{2}\|z^{k+1}-z\|^2.
\end{eqnarray}

To continue the proof, we have to mention that 
\begin{eqnarray}
\label{eq:dfjfimo}
    \left\la \tg(x^{k+1}), y^k-x^{k+1}\right\ra &\overset{\eqref{eq:omega_k+1_def_dprox_clipped_SSTM_shift}}{=}&
    \left\la \nabla f(x^{k+1}), y^k-x^{k+1}\right\ra + \left\la \omega_{k+1}, y^k-x^{k+1}\right\ra \notag\\
    &\leq&  f(y^k) - f(x^{k+1}) + \left\la \omega_{k+1}, y^k-x^{k+1}\right\ra,
\end{eqnarray}
where in the last inequality we used convexity of $f$. Also, by convexity of $\Psi$ and definition of $y^{k+1}$, we have 
\begin{eqnarray}
\label{eq:hfjfk}
    \Psi(y^{k+1}) &=& \Psi\left(\frac{A_k}{A_{k+1}}y^{k} + \frac{\alpha_{k+1}}{A_{k+1}} z^{k+1}\right) \leq \frac{A_k}{A_{k+1}}\Psi(y^{k}) + \frac{\alpha_{k+1}}{A_{k+1}} \Psi(z^{k+1});\notag\\
    -\alpha_{k+1}\Psi(z^{k+1})&\leq& - A_{k+1}\Psi(y^{k+1}) + A_k \Psi(y^k). 
\end{eqnarray}

Thus, we acquire 
\begin{eqnarray*}
    \alpha_{k+1} \left\la \tg(x^{k+1}), x^{k+1}- z\right\ra  &= & \alpha_{k+1} \left\la \tg(x^{k+1}), x^{k+1}- z^k\right\ra + \alpha_{k+1} \left\la \tg(x^{k+1}), z^k-z\right\ra \\
    &=& A_k \left\la \tg(x^{k+1}), y^k - x^{k+1}\right\ra + \alpha_{k+1} \left\la \tg(x^{k+1}), z^k-z\right\ra
\end{eqnarray*}
where the last equation is true due to that $\alpha_{k+1}(x^{k+1}-z^k) = A_k(y^k-x^{k+1})$. By \eqref{eq:hdfjdofodpf}, \eqref{eq:dfjfimo}, we get 
\begin{eqnarray*}
    \alpha_{k+1} \left\la \tg(x^{k+1}), x^{k+1}- z\right\ra &\leq& A_k( f(y^k) - f(x^{k+1})) + A_k\left\la \omega_{k+1}, y^k-x^{k+1}\right\ra\\
    &&+ \alpha_{k+1}(\Psi(z)-\Psi(z^{k+1})) +A_{k+1}\left(f(x^{k+1})-f(y^{k+1})\right) \notag\\
    && +\alpha^2_{k+1}\|\omega_{k+1}\|^2 + \frac{1}{2}\|z^k-z\|^2 - \frac{1}{2}\|z^{k+1}-z\|^2\\
    &\overset{\eqref{eq:hfjfk}}{\leq}& A_k( f(y^k) - f(x^{k+1})) + A_k\left\la \omega_{k+1}, y^k-x^{k+1}\right\ra\\
    &&+ \alpha_{k+1}\Psi(z)- A_{k+1}\Psi(y^{k+1}) + A_k \Psi(y^k) +A_{k+1}\left(f(x^{k+1})-f(y^{k+1})\right) \notag\\
    && +\alpha^2_{k+1}\|\omega_{k+1}\|^2 + \frac{1}{2}\|z^k-z\|^2 - \frac{1}{2}\|z^{k+1}-z\|^2.
\end{eqnarray*}
By definition of function $\Phi(\cdot)$ \eqref{eq:min_problem}, we have 
\begin{eqnarray*}
    \alpha_{k+1} \left\la \tg(x^{k+1}), x^{k+1}- z\right\ra &\leq&  A_k \Phi(y^k) - A_{k+1} \Phi(y^{k+1})  + A_k\left\la \omega_{k+1}, y^k-x^{k+1}\right\ra\\
    &&+ \alpha_{k+1}\Psi(z) + (A_{k+1} - A_k)f(x^{k+1}) \\
    && +\alpha^2_{k+1}\|\omega_{k+1}\|^2 + \frac{1}{2}\|z^k-z\|^2 - \frac{1}{2}\|z^{k+1}-z\|^2\\
    &\overset{(\ast\ast)}{=}& A_k \Phi(y^k) - A_{k+1} \Phi(y^{k+1}) + \alpha_{k+1}\left\la \omega_{k+1},x^{k+1} - z^k\right\ra\\
    &&+ \alpha_{k+1}\Psi(z) +  \alpha_{k+1}f(x^{k+1})\\
    && +\alpha^2_{k+1}\|\omega_{k+1}\|^2 + \frac{1}{2}\|z^k-z\|^2 - \frac{1}{2}\|z^{k+1}-z\|^2
\end{eqnarray*}
where in $(\ast\ast)$ we used $\alpha_{k+1}(x^{k+1}-z^k) = A_k(y^k-x^{k+1})$ and $A_{k+1} = A_k + \alpha_{k+1}$. Making a small rearrangement, we derive
\begin{eqnarray}
\label{eq:uyerueio}
    A_{k+1} \Phi(y^{k+1}) - A_k \Phi(y^k) &\leq&   \frac{1}{2}\|z^k-z\|^2 - \frac{1}{2}\|z^{k+1}-z\|^2\ + \alpha_{k+1}\Psi(z)  \notag\\
    &&  + \alpha_{k+1}f(x^{k+1}) + \alpha_{k+1} \left\la \tg(x^{k+1}),z - x^{k+1}\right\ra \notag\\
    &&  +  \alpha_{k+1}\left\la \omega_{k+1},x^{k+1} - z^k\right\ra +\alpha^2_{k+1}\|\omega_{k+1}\|^2 \notag\\
    &\overset{\eqref{eq:omega_k+1_def_dprox_clipped_SSTM_shift}}{=}& \frac{1}{2}\|z^k-z\|^2 - \frac{1}{2}\|z^{k+1}-z\|^2\ + \alpha_{k+1}\Psi(z)  \notag\\
    &&  + \alpha_{k+1}\frac{1}{n}\sum^n_{i=1}f_i(x^{k+1}) + \alpha_{k+1} \left\la \frac{1}{n}\sum^n_{i=1}\nabla f_i(x^{k+1}),z - x^{k+1}\right\ra \notag\\
    &&  + \alpha_{k+1} \left\la \omega_{k+1} ,z - x^{k+1}\right\ra +  \alpha_{k+1}\left\la \omega_{k+1},x^{k+1} - z^k\right\ra +\alpha^2_{k+1}\|\omega_{k+1}\|^2 \notag\\
    &\leq& \frac{1}{2}\|z^k-z\|^2 - \frac{1}{2}\|z^{k+1}-z\|^2 + \alpha_{k+1}\Psi(z)  + \alpha_{k+1} f(z) \notag\\
    &&  + \alpha_{k+1} \left\la \omega_{k+1} ,z  - z^k\right\ra +\alpha^2_{k+1}\|\omega_{k+1}\|^2 \notag\\
    && -\frac{\alpha_{k+1}}{2Ln}\sum^n_{i=1}\|\nabla f_i(x^{k+1}) - \nabla f_i(z)\|^2, 
\end{eqnarray}
where in the last inequality we used $L$-smoothness and convexity of each $f_i$. Now we consider the sequences of $h^k_i$, produced by the method, for any $i \in [n]$. Denoting $h^*_i = \nabla f_i(x^*)$ and , we have 
\begin{eqnarray}
\label{eq:djkfdjksd}
    \|h^{k+1}_i-h^*_i\|^2 &\overset{\eqref{eq:D_prox_clipped_SSTM_shift_estimator_2}}{=}& \|h^{k}_i-h^*_i\|^2 + 2\nu_k \left\la \hat\Delta^k_i, h^k_i-h^*_i\right\ra + \nu^2_k\|\hat\Delta^k_i\|^2\notag\\
    &=& \|h^{k}_i-h^*_i\|^2 + 2\nu_k \left\la \tg_i(x^{k+1}) -h^k_i, h^k_i-h^*_i\right\ra + \nu^2_k\|\tg_i(x^{k+1}) -h^k_i\|^2\notag\\
    &\overset{\nu_k \leq 1}{\leq}& \|h^{k}_i-h^*_i\|^2 + 2\nu_k \left\la \tg_i(x^{k+1})-h^k_i, h^k_i-h^*_i\right\ra + \nu_k\|\tg_i(x^{k+1})-h^k_i\|^2\notag\\
    &=&\|h^{k}_i-h^*_i\|^2 + \nu_k \left\la \tg_i(x^{k+1})-h^k_i,  \tg_i(x^{k+1}) + h^k_i- 2h^*_i\right\ra\notag\\
    &\leq& (1-\nu_k)\|h^{k}_i-h^*_i\|^2 + \nu_k\|\tg_i(x^{k+1})-h^*_i\|^2 \notag\\
    &\leq& (1-\nu_k)\|h^k_i-h^*_i\|^2 + 2\nu_k\|\tg_i(x^{k+1})-\nabla f_i(x^{k+1})\|^2 + 2\nu_k\|\nabla f_i(x^{k+1})-h^*_i\|^2 \notag\\
    &\overset{\eqref{eq:omega_k+1_def_dprox_clipped_SSTM_shift}}{=}& (1-\nu_k)\|h^{k}_i-h^*_i\|^2 + 2\nu_k\|\omega_{i,k+1}\|^2 + 2\nu_k\|\nabla f_i(x^{k+1})-\nabla f_i(x^{\star})\|^2.
\end{eqnarray}
Summing up \eqref{eq:djkfdjksd} by $i$ from $1$ to $n$, we obtain 
\begin{eqnarray}
\label{eq:jdjldsdslocxnxcmn}
    \frac{1}{n}\sum^n_{i=1}\|h^{k+1}_i-h^*_i\|^2 &\leq& (1-\nu_k)\frac{1}{n}\sum^n_{i=1}\|h^{k}_i-h^*_i\|^2 + \frac{2\nu_k}{n}\sum^n_{i=1}\|\omega_{i,k+1}\|^2\notag\\
    && + \frac{2\nu_k}{n}\sum^n_{i=1}\|\nabla f_i(x^{k+1})-\nabla f_i(x^{\star})\|^2.
\end{eqnarray}

Combining inequality \eqref{eq:uyerueio}, where we take $z = x^*$, and inequality \eqref{eq:jdjldsdslocxnxcmn} multiplied by $\frac{1}{2}C^2\widetilde\alpha^2_{k+2}$, we get 
\begin{eqnarray*}
    A_{k+1}\left( \Phi(y^{k+1})-\Phi(x^*)\right) &\leq& A_{k}\left( \Phi(y^{k})-\Phi(x^*)\right) + \frac{1}{2}\|z^k-x^*\|^2 + \frac{1}{2}C^2\widetilde\alpha^2_{k+1}\frac{1}{n}\sum^n_{i=1} \|h^k_i-h^*_i\|^2 \\
    && - \frac{1}{2}\|z^{k+1}-x^*\|^2 - \frac{1}{2}C^2\widetilde\alpha^2_{k+2}\frac{1}{n}\sum^n_{i=1} \|h^{k+1}_i-h^*_i\|^2\\
    &&+ \frac{1}{2}(1-\nu_k)C^2\widetilde\alpha^2_{k+2}\frac{1}{n}\sum^n_{i=1} \|h^k_i-h^*_i\|^2 - \frac{1}{2}C^2\widetilde\alpha^2_{k+1}\frac{1}{n}\sum^n_{i=1} \|h^k_i-h^*_i\|^2\\
    && + \alpha_{k+1} \left\la \omega_{k+1} , x^*- z^k\right\ra +\alpha^2_{k+1}\|\omega_{k+1}\|^2 + \frac{1}{2}C^2\widetilde\alpha^2_{k+2}\frac{2\nu_k}{n}\sum^n_{i=1}\|\omega_{i,k+1}\|^2\\
    && - \left(\frac{\alpha_{k+1}}{2Ln} - \frac{1}{n}\nu_kC^2\widetilde\alpha^2_{k+2}\right)\sum^n_{i=1}\|\nabla f_i(x^{k+1}) - \nabla f_i(z)\|^2.
\end{eqnarray*}
By the selection of parameters \eqref{eq:a_K0_D_prox_clipped_SSTM_shift_lemma}, \eqref{eq:nu_k_D_prox_clipped_SSTM_shift_lemma} and definition of Lyapunov function $M_k$ \eqref{eq:Lyapunov_function_D_prox_clipped_SSTM_shift}, we have 
\begin{eqnarray*}
    A_{k+1}\left( \Phi(y^{k+1})-\Phi(x^*)\right) &\leq& A_{k}\left( \Phi(y^{k})-\Phi(x^*)\right) + \frac{1}{2}M_k - \frac{1}{2}M_{k+1}\\
    && + \alpha_{k+1} \left\la \omega_{k+1} , x^*- z^k\right\ra +\alpha^2_{k+1}\|\omega_{k+1}\|^2 + \frac{\alpha^2_{k+1}}{n^2}\sum^n_{i=1}\|\omega_{i,k+1}\|^2.
\end{eqnarray*}
Summing up the previous inequality by $k$ from $0$ to $K-1$, we finish the proof. 

\end{proof}

\begin{theorem}
\label{thm:D_prox_clipped_SSTM_shift_full}
    Let Assumptions~\ref{as:bounded_alpha_moment}, \ref{as:L_smoothness} and \ref{as:str_cvx}($\mu = 0$) hold on $Q = B_{5n\sqrt{M}}(x^*)$, where $M \geq \|x^0-x^*\|^2+C^2\alpha^2_{K_0+1}\frac{1}{n}\sum^n_{i=1}\|\nabla f_i(x^*)\|^2$, where $C =\frac{864}{n}\ln \frac{10nK}{\beta}$, and $a \geq 0$, and 
    \begin{gather}
        a \geq \max\left\{2, \frac{8\cdot 3^5\cdot 72^4}{n} \ln^4\frac{10nK}{\beta}, \frac{ 18\cdot 6^5\sigma K^{\frac{1}{\alpha}}(K+1)}{\sqrt{M}Ln^{\frac{\alpha-1}{\alpha}}} \ln^{\frac{\alpha-1}{\alpha}}\frac{10nK}{\beta} \right\},\label{eq:Dprox_clipped_SSTM_shift_parameter_a}\\
        \lambda_{k} = \frac{n\sqrt{M}}{72\widetilde\alpha_{k+1}\ln\frac{10nK}{\beta}}, \label{eq:Dprox_clipped_SSTM_shift_clipping_level}
    \end{gather}
    for some $K \geq K_0 =  \left\lceil\frac{3}{2}C^2n\right\rceil > 0$ and $\beta \in (0,1]$ such that $\ln\frac{10nK}{\beta} \geq 1$. Then, after $K$ iterations of~\algname{DProx-clipped-SSTM-shift} the following inequality holds with probability at least $1 - \beta$ 
    \begin{equation}
        \Phi(y^K) - \Phi(x^*) \leq \frac{6aL M}{K(K+3)} \quad \text{and} \quad  \{x^k\}_{k=0}^{K+1}, \{z^k\}_{k=0}^{K}, \{y^k\}_{k=0}^K \subseteq B_{2\sqrt{M}}(x^*). \label{eq:Dprox_clipped_SSTM_shift}
    \end{equation}
    In particular, when parameter $a$ equals the maximum from \eqref{eq:Dprox_clipped_SSTM_shift_parameter_a}, then after $K$ iterations of \algname{DProx-clipped-SSTM-shift}, we have with probability at least $1-\beta$
    \begin{equation}
        \Phi(y^K) - \Phi(x^*) = \cO\left(\max\left\{\frac{LM}{K^2}, \frac{LM\ln^4\frac{nK}{\beta}}{nK^2},\frac{\sigma \sqrt{M} \ln^{\frac{\alpha-1}{\alpha}}\frac{nK}{\beta}}{n^{\frac{\alpha-1}{\alpha}}K^{\frac{\alpha-1}{\alpha}}}\right\}\right), \label{eq:Dprox_clipped_SSTM_shift_2}
    \end{equation}
    i.e. achieve $\Phi(y^K) - \Phi(x^*) \leq \varepsilon$ with probability at least $1 - \beta$ \algname{DProx-clipped-SSTM-shift} requires
    \begin{equation}
        K = \cO\left(\max\left\{\sqrt{\frac{LM}{\varepsilon}}, \sqrt{\frac{LM}{\varepsilon n}}\ln^2\frac{nLM}{\varepsilon\beta}, \frac{1}{n}\left(\frac{\sigma \sqrt{M}}{\varepsilon}\right)^{\frac{\alpha}{\alpha-1}}\ln\frac{\sigma \sqrt{M}}{\varepsilon\beta}\right\}\right) \label{eq:Dprox_clipped_SSTM_shift_complexity}
    \end{equation}
    iterations/oracle calls per worker.
\end{theorem}
\begin{proof}
    The key idea behind the proof is similar to the one used in \citep{gorbunov2021near,sadiev2023high}. We prove by induction that the iterates do not leave some ball and $\Phi(y^K)-\Phi(y^*)$ decreases as $\sim \nicefrac{1}{K(K+3)}$
    
    Firstly, we denote $R_k = \|z^k - x^*\|$, $\widetilde{R}_0 = R_0$, $\widetilde{R}_{k+1} = \max\{\widetilde{R}_k, R_{k+1}\}$ for all $k\geq 0$, and now we show by induction that for all $k\geq 0$ the iterates $x^{k+1}, z^k, y^k$ lie in $B_{\widetilde{R}_k}(x^*)$. The induction base is trivial since $y^0 = z^0$, $\widetilde{R}_0 = R_0$, and $x^1 = \tfrac{A_0 y^0 + \alpha_1 z^0}{A_1} = z^0$. Next, we assume this statement is true for some  $l \geq 1$: $x^{l}, z^{l-1}, y^{l-1} \in B_{\widetilde{R}_{l-1}}(x^*)$. According to definitions of $R_l$ and $\widetilde{R}_l$, we obtain $z^l \in B_{R_{l}}(x^*)\subseteq B_{\widetilde{R}_{l}}(x^*)$. Due to that $y^l$ is a convex combination of $y^{l-1}\in B_{\widetilde{R}_{l-1}}(x^*)\subseteq B_{\widetilde{R}_{l}}(x^*)$, $z^l\in B_{\widetilde{R}_{l}}(x^*)$ and $B_{\widetilde{R}_{l}}(x^*)$ is a convex set, we have that $y^l \in B_{\widetilde{R}_{l}}(x^*)$. Finally, since $x^{l+1}$ is a convex combination of $y^l$ and $z^l$, we conclude $x^{l+1}$ lies in $B_{\widetilde{R}_{l}}(x^*)$ as well.

    Now to formulate the statement rigorously, we introduce probability event $E_k$ for each for each $k = 0,\ldots, K$ as follows: inequalities
    \begin{gather}
         \underbrace{2\sum\limits^{t-1}_{l=0}\alpha_{l+1}\la\omega_{l+1}, x^*-z^l\ra  + 2\sum\limits^{t-1}_{l=0}\alpha^2_{l+1}\|\omega_{l+1}\|^2 + 2\sum\limits^{t-1}_{l=0}\sum^{n}_{i=1}\frac{\alpha^2_{l+1}}{n^2}\|\omega_{i,l+1}\|^2}_{B_t} \leq M, \label{eq:SSTM_induction_inequality_1}\\
        R_t\leq  \sqrt{M_t} \leq 2\sqrt{M} \label{eq:SSTM_induction_inequality_2},\\
        \left\|\frac{\alpha_{t}}{n}\sum^r_{i=1}\omega^u_{i,t}\right\| \leq \frac{M}{2}\label{eq:SSTM_induction_inequality_3}
    \end{gather}
    hold for  $t = 0,1,\ldots, k$ and $r=1,2\dots,n$ simultaneously, where 
     \begin{gather}
        \omega_{l+1} = \omega_{l+1}^u +\omega^b_{l+1}, \label{eq:SSTM_omega_full}\\
        \omega_{l+1}^u \eqdef \frac{1}{n}\sum_{i=1}^n \omega_{i,l+1}^u, \quad \omega_{l+1}^b \eqdef \frac{1}{n}\sum_{i=1}^n \omega_{i,l+1}^b, \label{eq:SSTM_omega_u_omega_b}\\
        \omega_{i,l+1}^u \eqdef \tg_i(x^{l+1})  - \EE_{\xi^l_i}\left[\tg_i(x^{l+1}) \right],\quad \omega_{i,k+1}^b \eqdef \EE_{\xi^k_i}\left[\tg_i(x^{k+1}) \right] - \nabla f_i(x^{k+1}),~~\forall i \in [n]. \label{eq:SSTM_omega_i}
    \end{gather}

     We want to show via induction $\widetilde{R}_{l} \leq 5n\sqrt{M}$ with high probability, which allows us to apply the result of Lemma~\ref{lem:SSTM_optimization_lemma_convex} and  Bernstein's inequality to estimate the stochastic part of the upper-bound. After that, we will prove by induction that $\PP\{E_k\} \geq 1 - \nicefrac{k\beta}{K}$ for all $k = 0,1,\ldots, K$. The base induction follows immediately: the left-hand side of \eqref{eq:SSTM_induction_inequality_1} equals zero and $M \geq M_0$ by definition, and for $k=0$ we have  $\left\|\frac{\alpha_{0}}{n}\sum\limits^r_{i=1}\omega^u_{i,0}\right\| = 0$, since $\alpha_0 = 0$. Next we assume that the statement holds for some $k = T - 1 \leq K-1$: $\PP\{E_{T-1}\} \geq 1 - \nicefrac{(T-1)\beta}{K}$. Let us show that $\PP\{E_{T}\} \geq 1 - \nicefrac{T\beta}{K}$. 
     
     To proceed, we need to show that probability event $E_{T-1}$ implies that $\widetilde{R}_t \leq 2\sqrt{M}$ for all $t = 0,1,\ldots, T$. The base is already proven. Next we assume that $\widetilde{R}_{t} \leq 2\sqrt{M}$ for all $t = 0,1,\ldots, t'$ for some $t' < T$. Then for all $t = 0,1,\dots, t'$
    \begin{eqnarray*}
        \|z^{t} - x^{\star}\|&=& \|\operatorname{prox}_{\alpha_{t}\Psi}\left(z^{t-1} - \alpha_{t}\tg (x^{t})\right) - \operatorname{prox}_{\alpha_{t}\Psi}\left(x^{\star} - \alpha_{t}\nabla f(x^{\star})\right)\| \\
        &\leq&\|z^{t-1} - x^* -\alpha_{t}\left(\tg(x^{t})- \nabla f(x^{\star})\right)\|\\
        &\leq&\|z^{t-1} - x^*\| + \alpha_{t}\|\tg(x^t)-h^{t-1}\| + \alpha_{t}\|h^{t-1}- h^*\| \\
        &\leq& \left(1 + \frac{1}{C}\right)\sqrt{\|z^{t-1} - x^*\|^2 + C^2\widetilde\alpha^2_t\frac{1}{n}\sum^n_{i=1}\|h^{t-1}_i- h^*_i\|^2}+ \alpha_{t}\lambda_{t-1}\\
        &\leq& 2\sqrt{M_{t-1}} + \alpha_{t}\lambda_{t-1} \overset{\eqref{eq:Dprox_clipped_SSTM_shift_clipping_level},\eqref{eq:SSTM_induction_inequality_2}}{\leq} 4\sqrt{M} + n\sqrt{M}\leq 5n\sqrt{M} .\notag
    \end{eqnarray*}
    
    This means that $x^t, z^t, y^t \in B_{5n\sqrt{M}}(x^*)$ for $t=0,1,\dots, t'$and we can apply Lemma~\ref{lem:SSTM_optimization_lemma_convex}: $E_{T-1}$ implies
    \begin{eqnarray}
        A_{t'}\left(\Phi(y^{t'}) - \Phi(x^*)\right) &\leq& \frac{1}{2}M_0 - \frac{1}{2}M_{t'} + \sum^{t'-1}_{l=0}\alpha_{l+1}\la\omega_{l+1}, x^*-z^l\ra  + \sum^{t'-1}_{l=0}\alpha^2_{l+1}\|\omega_{l+1}\|^2\notag\\
        && + \sum^{t'-1}_{k=0}\sum^{n}_{i=1}\frac{\alpha^2_{l+1}}{n^2}\|\omega_{i,k+1}\|^2\notag\\
        &\leq& \frac{1}{2}M_0 -\frac{1}{2}M_{t'} +B_{t'} \leq \frac{3}{2}M \label{eq:D_prox_clipped_SSTM_shift_technical_1}
    \end{eqnarray}
    that gives
    \begin{equation*}
        M_{t'} \leq M_0 +M \leq 2M. 
    \end{equation*}
    That is, we showed that $E_{T-1}$ implies $x^t, z^t, y^t \in B_{2\sqrt{M}}(x^*)$ and 
    \begin{eqnarray}
        \Phi(y^{t}) - \Phi(x^*) \overset{\eqref{eq:SSTM_induction_inequality_1}, \eqref{eq:D_prox_clipped_SSTM_shift_technical_1}}{\leq} \frac{\frac{1}{2}M_0 - \frac{1}{2}M_{t} + M}{A_{t}} \leq \frac{3M}{2A_{t}} = \frac{6aLM}{t(t+3)}. \label{eq:D_prox_clipped_SSTM_shift_technical_1_1}
    \end{eqnarray}
    for all $t =0,1,\dots,T$.  Before we proceed, we introduce a new notation:
    \begin{equation*}
        \eta_t = \begin{cases} x^* - z^t,& \text{if } \|x^* - z^t\| \leq 2\sqrt{M},\\ 0,& \text{otherwise,} \end{cases}
    \end{equation*}
    for all $t = 0, 1, \ldots, T$. Random vectors $\{\eta_t\}_{t=0}^T$ are bounded almost surely:
     \begin{equation}
        \|\eta_t\| \leq 2\sqrt{M}  \label{eq:gap_thm_SSTM_technical_3}
    \end{equation}
    for all $t = 0, 1, \ldots, T$. In addition, $\eta_t = x^* - z^t$ follows from $E_{T-1}$ for all $t=0,1,\dots, T$ and, thus, $E_{T-1}$ implies 
    \begin{eqnarray}
        B_T &=& 2 \sum^{T-1}_{k=0}\alpha_{k+1}\la\omega_{k+1}, x^*-z^k\ra  + 2\sum^{T-1}_{k=0}\alpha^2_{k+1}\|\omega_{k+1}\|^2+ 2\sum^{T-1}_{k=0}\sum^{n}_{i=1}\frac{\alpha^2_{k+1}}{n^2}\|\omega_{i,k+1}\|^2 \notag\\
        &=& 2 \sum^{T-1}_{k=0}\alpha_{k+1}\la\omega_{k+1}, \eta_k\ra  + 2\sum^{T-1}_{k=0}\alpha^2_{k+1}\|\omega_{k+1}\|^2+ 2\sum^{T-1}_{k=0}\sum^{n}_{i=1}\frac{\alpha^2_{k+1}}{n^2}\|\omega_{i,k+1}\|^2 . \label{eq:D_prox_clipped_SSTM_shift_technical_2}
    \end{eqnarray}
    Using the notation from \eqref{eq:SSTM_omega_full}-\eqref{eq:SSTM_omega_i}, we can rewrite $\|\omega_{k+1}\|^2$ and $\|\omega_{i,k+1}\|^2$ as
    \begin{eqnarray}
        \|\omega_{k+1}\|^2 &\leq& \frac{2}{n^2}\sum\limits_{i=1}^n\|\omega_{i,k+1}^u\|^2 + \frac{4}{n^2}\sum\limits_{j=2}^n\left\langle \sum\limits_{i=1}^{j-1} \omega_{i,k+1}^u, \omega_{j,k+1}^u \right\rangle + 2\|\omega_{k+1}^b\|^2. \label{eq:SSTM_sum_omega_representation}
    \end{eqnarray}
    Putting all together, we obtain that $E_{T-1}$  implies
    \begin{eqnarray}
        B_T 
        &\leq& \underbrace{2\sum\limits_{k=0}^{T-1}\sum^n_{i=1}\frac{\alpha_{k+1}}{n}\left\la \omega_{i,k+1}^u, \eta_k\right\ra}_{\circledOne} + \underbrace{2\sum\limits_{k=0}^{T-1}\sum^n_{i=1}\frac{\alpha_{k+1}}{n}\left\la \omega_{i,k+1}^b, \eta_k\right\ra}_{\circledTwo} \notag\\
        &&\quad+ \underbrace{8\sum\limits_{k=0}^{T-1}\sum^n_{i=1}\frac{\alpha_{k+1}^2}{n^2}\left(\left\|\omega_{i,k+1}^u\right\|^2 - \EE_{\xi^k_i}\left[\left\|\omega_{i,k+1}^u\right\|^2\right]\right)}_{\circledThree}\notag\\
        &&\quad + \underbrace{8\sum\limits_{k=0}^{T-1}\sum^n_{i=1}\frac{\alpha_{k+1}^2}{n^2}\EE_{\xi^k}\left[\left\|\omega_{i,k+1}^u\right\|^2\right]}_{\circledFour} + \underbrace{8\sum\limits_{k=0}^{T-1}\sum^n_{i=1}\frac{\alpha_{k+1}^2}{n}\left\|\omega_{i,k+1}^b\right\|^2}_{\circledFive}\notag\\
        &&\quad +\underbrace{8\sum\limits_{k=0}^{T-1}\sum^n_{j=2}\frac{\alpha_{k+1}^2}{n^2}\left\la\sum^{j-1}_{i=1}\omega^u_{i,k+1}, \omega^u_{j, k+1}\right\ra}_{\circledSix}. \label{eq:D_prox_clipped_SSTM_shift_technical_7}
    \end{eqnarray}

    To finish the proof, it remains to estimate  $\circledOne, \circledTwo, \circledThree, \circledFour, \circledFive, \circledSix$ with high probability. More precisely, the goal to prove  that $\circledOne + \circledTwo + \circledThree + \circledFour + \circledFive + \circledSix  \leq M$ with high probability.
    Before we proceed, we need to derive several useful inequalities related to $\omega^u_{i,k+1},~\omega^b_{i,k+1}$. First of all, we have 
    \begin{equation}
        \|\omega_{i,k+1}^u\| \leq 2\lambda_{k}. \label{eq:D_prox_clipped_SSTM_shift_norm_omega_u_bound}
    \end{equation}
    by definition of the clipping operator. Next, probability event $E_{T-1}$ implies that for $t = 0$ we have $x^1 = x^0$ and
    \begin{eqnarray}
        \|\nabla f_i(x^{1}) - h^0_i\| &\leq& \|\nabla f_i(x^0)- \nabla f_i(x^{\star})\| + \|h^0_i -h^*_i\| \notag\\
        &\overset{\text{\small smooth}}{\leq}& L\|x^0 - x^*\| + \frac{\sqrt{n}}{C\widetilde\alpha_1}\sqrt{C^2\widetilde\alpha^2_1\frac{1}{n}\sum^n_{i=1}\|h^0_i-h^*_i\|^2}\notag\\
        &\leq& \left(\frac{2(K_0+2)}{a\widetilde\alpha_1}+\frac{\sqrt{n}}{C\widetilde\alpha_1}\right) \sqrt{M} \notag\\
        &\overset{\eqref{eq:Dprox_clipped_SSTM_shift_parameter_a},\eqref{eq:Dprox_clipped_SSTM_shift_clipping_level} }{\leq}& \frac{\lambda_0}{2}. \label{eq:D_prox_clipped_SSTM_shift_technical_3}
    \end{eqnarray}
     Next, for $t = 1,\ldots, T-1$ event $E_{T-1}$ implies
    \begin{eqnarray}
        \|\nabla f_i(x^{t+1})-h^t_i\| &\leq& \|\nabla f_i(x^{t+1}) - \nabla f_i(y^t)\| + \|\nabla f_i(y^t) - \nabla f_i(x^{\star})\| + \|h^t_i - h^*_i\| \notag\\
        &\leq& L\|x^{t+1} - y^t\| + \sqrt{2L\left(f_i(y^t) - f_i(x^*) - \la \nabla f_i(x^{\star}), y^t-x^{\star}\ra\right)} \notag\\
        &\overset{(\ast)}{\leq}& L\|x^{t+1} - y^t\| + \sqrt{2nL\left(\Phi(y^t) - \Phi(x^*)\right)} + \sqrt{\sum^n_{i=1}\|h^t_i - h^*_i\|^2}\\
        &\overset{\eqref{eq:D_prox_clipped_SSTM_shift_technical_1_1}}{\leq}& \frac{L\alpha_{t+1}}{A_t}\|x^{t+1} - z^t\| + \sqrt{\frac{12anL^2M}{t(t+3)}} + \frac{\sqrt{n}}{C\widetilde{\alpha}_{t+1}}\sqrt{\frac{1}{n}\sum^n_{i=1}\|h^t_i - h^*_i\|^2}\notag\\
        &\leq& \frac{4L\sqrt{M}\alpha_{t+1}}{A_t} + \sqrt{\frac{12anL^2M}{t(t+3)}} + \frac{\sqrt{n}}{C\widetilde{\alpha}_{t+1}}\sqrt{M_t} \notag\\
        &\overset{\eqref{eq:Dprox_clipped_SSTM_shift_clipping_level}}{\leq}&  \frac{\lambda_t}{2}\left(\frac{8\cdot 72L\alpha_{t+1}\widetilde\alpha_{t+1}\ln\frac{10nK}{\beta}}{nA_t} + 2\sqrt{\frac{12\cdot 72^2 aL^2\widetilde\alpha_{t+1}^2\ln^2\frac{10nK}{\beta}}{nt(t+3)}}\right)\notag\\
        &&+ \frac{\lambda_t}{2}\cdot \frac{288}{C\sqrt{n}}\ln\frac{10nK}{\beta}\notag\\
        &\leq& \frac{\lambda_t}{2}\cdot\frac{576aL^2\max\left\{K_0+2, t+2\right\}(t+2)\ln\frac{10nK}{\beta}}{a^2L^2t(t+3)n} \notag\\
        &&+ \frac{\lambda_t}{2}\cdot\sqrt{\frac{12aL^2\max\left\{(K_0+2)^2, (t+2)^2\right\}72^2\ln^2\frac{10nK}{\beta}}{na^2L^2t(t+3)}}\notag\\
        &&+ \frac{\lambda_t}{2}\cdot \frac{288}{C\sqrt{n}}\ln\frac{10nK}{\beta} \notag\\
        &\leq& \frac{\lambda_t}{2} \cdot \frac{9}{a} \max\left\{(K_0+2), 2\right\}\frac{72}{n} \ln\frac{10nK}{\beta}\notag\\
        &&+\frac{\lambda_t}{2}\sqrt{\frac{3}{a}\max\left\{(K_0+2)^2,9\right\}\frac{72^2}{n}\ln^2\frac{10nK}{\beta}}\notag\\
        &&+ \frac{\lambda_t}{2}\cdot \frac{288}{C\sqrt{n}}\ln\frac{10nK}{\beta} \overset{\eqref{eq:Dprox_clipped_SSTM_shift_parameter_a}}{\leq} \frac{\lambda_t}{2}, \label{eq:D_prox_clipped_SSTM_shift_technical_4}
    \end{eqnarray}
    where in $(\ast)$ we use $- \la \frac{1}{n}\sum\limits^n_{i=1}\nabla f_i(x^{\star}), y^t-x^{\star}\ra \leq \Psi(y^t) - \Psi(x^{\star})$, and  in the last row we use $\frac{(t+2)^2}{t(t+3)} \leq \frac{9}{4}$ for all $t \geq 1$ and $C \geq 12 \cdot 72 \ln\frac{10nK}{\beta}$.

    Therefore, Lemma~\ref{lem:bias_and_variance_clip} and  $E_{T-1}$ imply
    \begin{eqnarray}
        \|\omega_{i,k+1}^b\| &\leq& \frac{2^\alpha\sigma^\alpha}{\lambda_{k}^{\alpha-1}}, \label{eq:D_prox_clipped_SSTM_shift_norm_omega_b_bound} \\
        \EE_{\xi^k_i}\left[\|\omega_{i,k+1}^u\|^2\right] &\leq& 18 \lambda_k^{2-\alpha}\sigma^\alpha. \label{eq:D_prox_clipped_SSTM_shift_second_moment_omega_u_bound}
    \end{eqnarray}

    \paragraph{Upper bound for $\circledOne$.} To estimate this sum, we will use Bernstein's inequality. The summands have conditional equal to zero, since  $\EE_{\xi^k_i}[\omega_{i,k+1}^u] = 0$:
    \begin{equation}
        \EE_{\xi^k_i}\left[\frac{\alpha_{k+1}}{n}\left\la \omega_{i,k+1}^u, \eta_k\right\ra\right] = 0. \notag
    \end{equation}
    Moreover, for all $k = 0,\dots, T-1$ random vectors $\left\{ \omega_{i,k+1}^u\ra\right\}_{k=0}^{T-1}$ are independent. Thus, sequence $\left\{2\frac{\alpha_{k+1}}{n}\left\la \omega_{i,k+1}^u, \eta_k\right\ra\right\}_{k=0}^{T-1}$ is a martingale difference sequence. Next, the summands are bounded:
    \begin{eqnarray}
        \left|2\frac{\alpha_{k+1}}{n}\left\la \omega_{i,k+1}^u, \eta_k\right\ra\right| &\leq& 2\frac{\alpha_{k+1}}{n} \|\omega_{i,k+1}^u\| \cdot \|\eta_k\| \overset{\eqref{eq:D_prox_clipped_SSTM_shift_norm_omega_u_bound}}{\leq} 4\frac{\alpha_{k+1}}{n}\lambda_k \sqrt{M}\notag\\
        &\overset{\eqref{eq:Dprox_clipped_SSTM_shift_clipping_level}}{=}& \frac{4n\alpha_{k+1}\sqrt{M}}{72n\widetilde\alpha_{k+1}\ln\frac{10nK}{\beta}}\leq \frac{\sqrt{M}}{6\ln\frac{10nK}{\beta}} \eqdef c. \label{eq:D_prox_clipped_SSTM_shift_technical_8} 
    \end{eqnarray}
    
    Finally, conditional variances $\sigma_{i,k}^2 \eqdef \EE_{\xi^k_i}\left[4\frac{\alpha_{k+1}^2}{n^2}\left\la \omega_{i,k+1}^u, \eta_k\right\ra^2\right]$ of the summands are bounded::
    \begin{equation}
        \sigma_{i,k}^2 \leq \EE_{\xi^k_i}\left[4\frac{\alpha_{k+1}^2}{n^2}\|\omega_{i,k+1}^u\|^2\cdot \|\eta_k\|^2\right] \leq  16\frac{\alpha_{k+1}^2}{n^2}M \EE_{\xi^t}\left[\|\omega_{i,k+1}^u\|^2\right]. \label{eq:D_prox_clipped_SSTM_shift_technical_9}
    \end{equation}
    Applying Bernstein's inequality (Lemma~\ref{lem:Bernstein_ineq}) with $X_{i,k} = 2\frac{\alpha_{k+1}}{n}\left\la \omega_{i,k+1}^u, \eta_k\right\ra$, parameter $c$ as in \eqref{eq:D_prox_clipped_SSTM_shift_technical_8}, $b = \frac{M}{6}$, $G = \frac{M^2}{6^3\ln\frac{10nK}{\beta}}$:
    \begin{equation*}
        \PP\left\{|\circledOne| > \frac{M}{6}\quad \text{and}\quad \sum\limits_{k=0}^{T-1}\sum^n_{i=1} \sigma_{i,k}^2 \leq \frac{M^2}{6^3\ln\frac{10nK}{\beta}}\right\} \leq 2\exp\left(- \frac{b^2}{2G + \nicefrac{2cb}{3}}\right) = \frac{\beta}{5nK}.
    \end{equation*}
    The above is equivalent to 
    \begin{equation}
        \PP\left\{ E_{\circledOne} \right\} \geq 1 - \frac{\beta}{5nK},\quad \text{for}\quad E_{\circledOne} = \left\{ \text{either} \quad  \sum\limits_{k=0}^{T-1}\sum^n_{i=1} \sigma_{i,k}^2 > \frac{M^2}{6^3\ln\frac{10nK}{\beta}} \quad \text{or}\quad |\circledOne| \leq \frac{M}{6}\right\}. \label{eq:D_prox_clipped_SSTM_shift_sum_1_upper_bound}
    \end{equation}
    Moreover, $E_{T-1}$ implies that
    \begin{eqnarray}
        \sum\limits_{k=0}^{T-1}\sum^n_{i=1} \sigma_{i, k}^2 &\overset{\eqref{eq:D_prox_clipped_SSTM_shift_technical_9}}{\leq}& 16 M \sum\limits_{k=0}^{T-1}\sum^n_{i=1} \frac{\alpha_{k+1}^2}{n^2} \EE_{\xi^t}\left[\|\omega_{i,k+1}^u\|^2\right] \overset{\eqref{eq:D_prox_clipped_SSTM_shift_second_moment_omega_u_bound}}{\leq} 288\sigma^\alpha M \sum\limits_{k=0}^{T-1}\sum^n_{i=1} \frac{\alpha_{k+1}^2}{n^2}\lambda_k^{2-\alpha}\notag\\
        &\overset{\eqref{eq:Dprox_clipped_SSTM_shift_clipping_level}}{\leq}& \frac{288\sigma^\alpha M^{2-\nicefrac{\alpha}{2}}}{72^{2-\alpha}\ln^{2-\alpha}\frac{10nK}{\beta}}\sum\limits_{k=0}^{T-1} \frac{\alpha_{k+1}^{2}}{n^{\alpha-1}\widetilde\alpha_{k+1}^{2-\alpha}}  \leq \frac{288\sigma^\alpha M^{2-\nicefrac{\alpha}{2}}}{72^{2-\alpha}\ln^{2-\alpha}\frac{10nK}{\beta}}\sum\limits_{k=0}^{T-1} \frac{\alpha_{k+1}^{\alpha}}{n^{\alpha-1}}\notag\\
        &\leq& \frac{288\sigma^\alpha M^{2-\nicefrac\alpha2}}{n^{\alpha-1}72^{2-\alpha}\cdot 2^\alpha a^\alpha L^\alpha\ln^{2-\alpha}\frac{10nK}{\beta}}\sum\limits_{k=0}^{T-1} (k+2)^\alpha \notag\\
        &\leq& \frac{1}{a^\alpha} \cdot \frac{144\sigma^\alpha M^{2- \nicefrac\alpha2} T(T+1)^\alpha}{n^{\alpha-1} L^\alpha\ln^{2-\alpha}\frac{10nK}{\beta}} \overset{\eqref{eq:Dprox_clipped_SSTM_shift_parameter_a}}{\leq} \frac{M^2}{6^3 \ln\frac{10nK}{\beta}}. \label{eq:prox_clipped_SSTM_shift_sum_1_variance_bound}
    \end{eqnarray}

     \paragraph{Upper bound for $\circledTwo$.} Probability event $E_{T-1}$ implies
    \begin{eqnarray}
        \circledTwo &\leq& 2\sum\limits_{k=0}^{T-1}\sum^n_{i=1}\frac{\alpha_{k+1}}{n}\|\omega_{i,k+1}^b\|\cdot \|\eta_k\| \overset{\eqref{eq:D_prox_clipped_SSTM_shift_norm_omega_b_bound}}{\leq} 4\sqrt{M}\cdot 2^\alpha\sigma^\alpha \sum\limits_{k=0}^{T-1}\frac{\alpha_{k+1}}{\lambda_{k}^{\alpha-1}}\notag\\ &\overset{\eqref{eq:Dprox_clipped_SSTM_shift_clipping_level}}{\leq}& \frac{16\cdot 72^{\alpha-1}M^{1-\nicefrac{\alpha}{2}}\sigma^\alpha }{n^{\alpha-1}}\ln^{\alpha-1}\frac{10nK}{\beta} \sum\limits_{k=0}^{T-1}\max\left\{\alpha_{k+1}\widetilde\alpha_{k+1}^{\alpha-1},\alpha_{k+1}^{\alpha}\right\}\notag\\
        &\leq& \frac{16\cdot 72^{\alpha-1}\sigma^\alpha M^{1 - \nicefrac\alpha2}}{2^\alpha a^{\alpha} L^\alpha}\ln^{\alpha-1}\frac{10nK}{\beta}\sum\limits_{t=0}^{T-1} \max\left\{(K_0+2)(k+2)^{\alpha-1},(k+2)^\alpha\right\} \notag\\
        &\overset{T, K_0\leq K}{\leq}& \frac{1}{a^\alpha} \cdot \frac{12\cdot 16\cdot 72^{\alpha-1}\sigma^\alpha M^{1 - \nicefrac\alpha2} K(K+1)^\alpha}{4^{\alpha}L^\alpha} \ln^{\alpha-1}\frac{10nK}{\beta} \notag\\
        &\overset{\eqref{eq:Dprox_clipped_SSTM_shift_parameter_a}}{\leq}& \frac{M}{6}. \label{eq:D_prox_clipped_SSTM_shift_sum_2_upper_bound}
    \end{eqnarray}

    \paragraph{Upper bound for $\circledThree$.} To estimate this sum, we will use Bernstein's inequality. The summands have conditional expectations equal to zero: 
    \begin{equation}
        \frac{8\alpha_{k+1}^2}{n^2}\EE_{\xi^k_i}\left[\left\|\omega_{i,k+1}^u\right\|^2 - \EE_{\xi^k_i}\left[\left\|\omega_{i,k+1}^u\right\|^2\right]\right] = 0. \notag
    \end{equation}
    Moreover, for all $k = 0,\dots, T-1$ random vectors $\left\{\omega^u_{i,k+1}\right\}^n_{i=1}$ are independent. Thus, sequence $\left\{ \frac{8\alpha_{k+1}^2}{n^2}\left(\left\|\omega_{i,k+1}^u\right\|^2 - \EE_{\xi^k_i}\left[\left\|\omega_{i,k+1}^u\right\|^2\right]\right) \right\}_{k,i = 0,1}^{T-1, n}$ is a martingale difference sequence. Next, the summands are bounded: 
    \begin{eqnarray}
        \left|\frac{8\alpha_{k+1}^2}{n^2}\left(\left\|\omega_{i,k+1}^u\right\|^2 - \EE_{\xi^k_i}\left[\left\|\omega_{i,k+1}^u\right\|^2\right]\right)\right| &\leq& \frac{8\alpha_{k+1}^2}{n^2}\left( \|\omega_{i,k+1}^u\|^2 +   \EE_{\xi^k_i}\left[\left\|\omega_{i,k+1}^u\right\|^2\right]\right)\notag\\
        &\overset{\eqref{eq:D_prox_clipped_SSTM_shift_norm_omega_u_bound}}{\leq}& \frac{64\alpha_{k+1}^2\lambda_k^2}{n^2} \overset{\eqref{eq:Dprox_clipped_SSTM_shift_clipping_level}}{\leq} \frac{M}{9\ln\frac{10nK}{\beta}} \eqdef c. \label{eq:D_prox_clipped_SSTM_shift_technical_10}
    \end{eqnarray}

   Finally, conditional variances $$\widetilde\sigma_{i,k}^2 \eqdef \EE_{\xi^k_i}\left[\frac{64 \alpha_{k+1}^4}{n^4}\left(\left\|\theta_{t+1}^u\right\|^2 - \EE_{\xi^k_i}\left[\left\|\omega_{k+1}^u\right\|^2\right]\right)^2\right]$$
   of the summands are bounded: 
    \begin{eqnarray}
        \widetilde\sigma_{i,k}^2 &\overset{\eqref{eq:D_prox_clipped_SSTM_shift_technical_10}}{\leq}& \frac{8 \alpha_{k+1}^2 M}{9 n^2 \ln\frac{10nK}{\beta}} \EE_{\xi^k_i}\left[\left|\left\|\omega_{i,k+1}^u\right\|^2 - \EE_{\xi^k_i}\left[\left\|\omega_{i,k+1}^u\right\|^2\right]\right|\right]\notag\\ 
        &\leq& \frac{16\alpha_{k+1}^2M}{9n^2} \EE_{\xi^k_i}\left[\|\omega_{i,k+1}^u\|^2\right]. \label{eq:D_prox_clipped_SSTM_shift_technical_11}
    \end{eqnarray}
    
    Applying Bernstein's inequality (Lemma~\ref{lem:Bernstein_ineq}) with $\widetilde X_{i,k} = \frac{8\alpha_{k+1}^2}{n^2}\left(\left\|\omega_{k+1}^u\right\|^2 - \EE_{\xi^k_i}\left[\left\|\omega_{i,k+1}^u\right\|^2\right]\right)$, parameter $c$ defined in \eqref{eq:D_prox_clipped_SSTM_shift_technical_10}, $b = \frac{M}{9}$, $G = \frac{M^2}{6\cdot9^2\ln\frac{10nK}{\beta}}$:
    \begin{equation*}
        \PP\left\{|\circledThree| > \frac{M}{9}\quad \text{and}\quad \sum\limits_{k=0}^{T-1}\sum^n_{i=1} \widetilde\sigma_{i,k}^2 \leq \frac{M^2}{6\cdot 9^2\ln\frac{10nK}{\beta}}\right\} \leq 2\exp\left(- \frac{b^2}{2G + \nicefrac{2cb}{3}}\right) = \frac{\beta}{5nK}.
    \end{equation*}
    The above is equivalent to
    \begin{equation}
        \PP\left\{ E_{\circledThree} \right\} \geq 1 - \frac{\beta}{5nK},\quad \text{for}\quad E_{\circledThree} = \left\{ \text{either} \quad  \sum\limits_{k=0}^{T-1}\sum^n_{i=1} \widetilde\sigma_{i,k}^2 > \frac{M^2}{6\cdot 9^2\ln\frac{10nK}{\beta}} \quad \text{or}\quad |\circledThree| \leq \frac{M}{9}\right\}. \label{eq:D_prox_clipped_SSTM_shift_sum_3_upper_bound}
    \end{equation}
    Moreover, $E_{T-1}$ implies 
    \begin{eqnarray}
        \sum\limits_{k=0}^{T-1}\sum^n_{i=1} \widetilde\sigma_{i,k}^2 &\overset{\eqref{eq:D_prox_clipped_SSTM_shift_technical_11}}{\leq}& \frac{16}{9}M \sum\limits_{k=0}^{T-1} \sum^n_{i=1}\frac{\alpha_{k+1}^2}{n^2} \EE_{\xi^k_i}\left[\|\omega_{k+1}^u\|^2\right] \overset{\eqref{eq:prox_clipped_SSTM_shift_sum_1_variance_bound}}{\leq} \frac{M^2}{6 \cdot 9^2\ln\frac{10nK}{\beta}}. \label{eq:D_prox_clipped_SSTM_shift_sum_3_variance_bound}
    \end{eqnarray}

    \paragraph{Upper bound for $\circledFour$.} Probability event $E_{T-1}$ implies
    \begin{eqnarray}
        \circledFour &=& 8\sum\limits_{k=0}^{T-1}\sum^n_{i=1}\frac{\alpha_{k+1}^2}{n^2}\EE_{\xi^k_i}\left[\left\|\omega_{i,k+1}^u\right\|^2\right] \leq \frac{1}{M}\cdot 8M\sum\limits_{k=0}^{T-1}\sum^n_{i=1}\frac{\alpha_{k+1}^2}{n^2}\EE_{\xi^t}\left[\left\|\omega_{i,k+1}^u\right\|^2\right]\notag\\ &\overset{\eqref{eq:prox_clipped_SSTM_shift_sum_1_variance_bound}}{\leq}& \frac{M}{6^3\ln\frac{10nK}{\beta}} \leq \frac{M}{9}. \label{eq:D_prox_clipped_SSTM_shift_sum_4_upper_bound}
    \end{eqnarray}

    \paragraph{Upper bound for $\circledFive$.} Probability event $E_{T-1}$ implies
    \begin{eqnarray}
        \circledFive &=& 8\sum\limits_{k=0}^{T-1}\sum^n_{i=1}\frac{\alpha_{k+1}^2}{n}\left\|\omega_{i,k+1}^b\right\|^2 \leq 2^{2\alpha + 3}\sigma^{2\alpha} \sum\limits_{k=0}^{T-1}\frac{\alpha_{k+1}^2}{\lambda_{k}^{2\alpha - 2}} \notag\\
        &\overset{\eqref{eq:Dprox_clipped_SSTM_shift_clipping_level}}{=}& \frac{2^{2\alpha + 3}\cdot 72^{2\alpha-2}\sigma^{2\alpha} \ln^{2\alpha-2}\frac{10nK}{\beta}}{n^{2\alpha-2}M^{\alpha-1}} \sum\limits_{k=0}^{T-1}\max\left\{\alpha^{2}_{k+1}\widetilde\alpha_{k+1}^{2\alpha-2},\alpha_{k+1}^{2}\right\} \notag\\
        &=&   \frac{2^{2\alpha + 3}\cdot 72^{2\alpha-2}\sigma^{2\alpha} \ln^{2\alpha-2}\frac{10nK}{\beta}}{2^{2\alpha}a^{2\alpha} n^{2\alpha-2} L^{2\alpha} M^{\alpha-1}} \sum\limits_{k=0}^{T-1}\max\left\{(k+2)^{2},(K_0+2)^{2\alpha-2}(k+2)^{2}\right\} \notag\\
        &\leq& \frac{1}{a^{2\alpha}} \cdot \frac{8\cdot72^{2\alpha-2} \sigma^{2\alpha}K(K+1)^{2\alpha}\ln^{2\alpha-2}\frac{10nK}{\beta}}{n^{2\alpha-2}L^{2\alpha}M^{\alpha-1}} \overset{\eqref{eq:Dprox_clipped_SSTM_shift_parameter_a}}{\leq} \frac{M}{6}.\label{eq:D_prox_clipped_SSTM_shift_sum_5_upper_bound}
    \end{eqnarray}

    \paragraph{Upper bound for $\circledSix$.} This sum requires more refined analysis. We introduce a new vector: 
    \begin{equation}
        \chi_{j}^k = \begin{cases}
            \frac{\alpha_{k+1}}{n}\sum\limits_{i=1}^{j-1} \omega_{i,k+1}^u,& \text{if } \left\|\frac{\alpha_{k+1}}{n} \sum\limits_{i=1}^{j-1} \omega_{i,k+1}^u\right\| \leq \frac{\sqrt{M}}{2},\\
            0,& \text{otherwise,}
        \end{cases}\label{eq:prox_clipped_SSTM_chi_def}
    \end{equation}
    Then, by definition
    \begin{equation}
         \|\chi_j^k\| \leq \frac{\sqrt{V}}{2} \label{eq:prox_clipped_SGDA_chi_bound}
    \end{equation}
    and
    \begin{eqnarray}
        \circledSix &=& \underbrace{8 \sum\limits_{k=0}^{T-1} \sum\limits_{j=2}^n \frac{\alpha_{k+1}}{n}\left\langle \chi_j^k, \omega_{j,k+1}^u \right\rangle}_{\circledSix'} + 8\sum\limits_{k=0}^{T-1} \sum\limits_{j=2}^n \left\langle \frac{\alpha_{k+1}}{n}\sum\limits_{i=1}^{j-1} \omega_{i,k+1}^u - \chi_j^k, \omega_{j,k+1}^u \right\rangle. \label{eq:prox_clipped_SSTM_extra_sums_six_distributed}
    \end{eqnarray}

    We also note here that $E_{T-1}$ implies
    \begin{equation}
        8\sum\limits_{k=0}^{T-1} \sum\limits_{j=2}^n \left\langle \frac{\alpha_{k+1}}{n}\sum\limits_{i=1}^{j-1} \omega_{i,k+1}^u - \chi_j^k, \omega_{j,k+1}^u \right\rangle = 8\sum\limits_{j=2}^n \left\langle \frac{\alpha_{T}}{n}\sum\limits_{i=1}^{j-1} \omega_{i,T}^u - \chi_j^{T-1}, \omega_{j,T}^u \right\rangle. \label{eq:prox_SSTM_extra_sum_six_bound}
    \end{equation}

    \paragraph{Upper bound for $\circledSix'$.} To estimate this sum, we will use Bernstein's inequality. The summands have conditional expectations equal to zero:
    \begin{equation}
        \EE_{\xi^k_j}\left[\frac{8\alpha_{k+1}}{n}\langle\chi_j^k, \omega^u_{j,k+1} \rangle \right] = \frac{8\alpha_{k+1}}{n}\la \chi_j^k, \EE_{\xi^k_j}\left[\omega^u_{j,k+1}\right]\ra = 0\notag.
    \end{equation}

    Moreover, for all $k = 0,\ldots, T-1$ random vectors $\{\omega_{i,l}^u\}_{i=1}^n$ are independent. Thus, sequence $\left\{\frac{8\alpha_{k+1}}{n}\langle\chi_j^k, \omega^u_{j,k+1} \rangle\right\}_{k,j = 0,2}^{T-1,n}$ is a martingale difference sequence. Next, the summands are bounded:
    \begin{eqnarray}
        \left|\frac{8\alpha_{k+1}}{n}\la \chi^k_j, \omega^u_{j,k+1}\ra\right| \leq \frac{8\alpha_{k+1}}{n} \|\chi^k_j\|\|\omega^u_{j,k+1}\| \overset{\eqref{eq:D_prox_clipped_SSTM_shift_norm_omega_u_bound},\eqref{eq:prox_clipped_SGDA_chi_bound}}{\leq} \frac{8\alpha_{k+1}}{n} \cdot \frac{\sqrt{M}}{2}\cdot 2\lambda_{k} \leq \frac{M}{6 \ln \frac{10nK}{\beta}} \eqdef c. \label{eq:prox_clipped_SSTM_shift_technical_6_1}
    \end{eqnarray}

    Finally, conditional variances $$\hat{\sigma}_{j,k}^2 \eqdef \EE_{\xi^k_j}\left[\frac{64\alpha^2_{k+1}}{n^2} \la \chi^j_k, \omega_{j,k+1}^u\ra^2\right]$$ of summands are bounded: 
    \begin{equation}
        \hat{\sigma}_{j,k}^2 \leq \frac{64\alpha^2_{k+1}}{n^2}\EE_{\xi^k_j}\left[\|\chi^k_j\|^2\|\omega^u_{k+1}\|^2\right] \leq \frac{16\alpha_{k+1}^2 M}{n^2}\EE_{\xi^k_{j}}\left[\|\omega_{j,k+1}^u\|^2\right].
        \label{eq:prox_clipped_SSTM_shift_technical_6_2}
    \end{equation}

    Applying Bernstein's inequality (Lemma~\ref{lem:Bernstein_ineq}) with $X_{j,k} =\frac{8\alpha_{k+1}}{n}\left\la \chi^k_j, \omega^u_{j,k+1}\right\ra$, constant $c$ defined in \eqref{eq:prox_clipped_SSTM_shift_technical_6_1}, $b = \frac{M}{6}$, $G = \frac{M^2}{6^3\ln\frac{10nK}{\beta}}$, we get
    \begin{equation*}
        \PP\left\{|\circledSix'| > \frac{M}{6} \text{ and } \sum\limits_{k=0}^{T-1}\sum\limits_{j=2}^n\hat{\sigma}_{i,k}^2 \leq \frac{M^2}{6^3\ln\frac{10nK}{\beta}}\right\} \leq 2\exp\left(- \frac{b^2}{2G + \nicefrac{2cb}{3}}\right) = \frac{\beta}{5nK}.
    \end{equation*}
    The above is equivalent to 
    \begin{equation}
        \PP\{E_{\circledSix'}\} \geq 1 - \frac{\beta}{5nK},\; \text{ for }\; E_{\circledSix'} = \left\{\text{either} ~~ \sum\limits_{k=0}^{T-1}\sum\limits_{j=2}^n\hat{\sigma}_{i,l}^2 > \frac{M^2}{6^3\ln\frac{10nK}{\beta}} \quad\text{ or } ~~|\circledSix'| ~~ \frac{M}{6}\right\}. \label{eq:prox_clipped_SSTM_shift_technical_6_3}
    \end{equation}
    
    Moreover, $E_{T-1}$ implies 
    \begin{equation}
    \label{eq:prox_clipped_SSTM_shift_sum_6_prime_variance_bound}
        \sum^{T-1}_{k=0}\sum^{n}_{j=2}\hat{\sigma}^2_{i,k} \overset{\eqref{eq:prox_clipped_SSTM_shift_technical_6_2}}{\leq} 16 M\sum^{T-1}_{k=0}\sum^{n}_{j=1}\frac{\alpha^2_{k+1}}{n^2} \EE_{\xi^k_{j}}\left[\|\omega_{j,k+1}^u\|^2\right] \overset{\eqref{eq:prox_clipped_SSTM_shift_sum_1_variance_bound}}{\leq} \frac{M^2}{6^3\ln\frac{10nK}{\beta}}
    \end{equation}

    That is, we derive the upper bounds for  $ \circledOne, \circledTwo, \circledThree, \circledFour, \circledFive, \circledSix$. More precisely, $E_{T-1}$ implies

    \begin{gather*}
        B_T \overset{\eqref{eq:D_prox_clipped_SSTM_shift_technical_7}}{\leq} \circledOne + \circledTwo + \circledThree + \circledFour + \circledFive + \circledSix,\\
        \circledSix \overset{\eqref{eq:prox_clipped_SSTM_extra_sums_six_distributed}, \eqref{eq:prox_SSTM_extra_sum_six_bound}}{=} \circledSix' + 8 \sum\limits_{j=2}^n \left\langle \frac{\alpha_{T}}{n}\sum\limits_{i=1}^{j-1} \omega_{i,T}^u - \chi_j^{T-1}, \omega_{j,T}^u \right\rangle,\\
        \circledTwo \overset{\eqref{eq:D_prox_clipped_SSTM_shift_sum_2_upper_bound}}{\leq} \frac{M}{6},\quad \circledFour \overset{\eqref{eq:D_prox_clipped_SSTM_shift_sum_4_upper_bound}}{\leq} \frac{M}{9},\quad \circledFive \overset{\eqref{eq:D_prox_clipped_SSTM_shift_sum_5_upper_bound}}{\leq} \frac{M}{6},\\
        \sum\limits_{k=0}^{T-1}\sum^n_{i=1} \sigma_{i, k}^2 \overset{\eqref{eq:prox_clipped_SSTM_shift_sum_1_variance_bound}}{\leq} \frac{M^2}{6^3\ln\frac{10nK}{\beta}},\quad  \sum\limits_{k=0}^{T-1}\sum^n_{i=1} \widetilde{\sigma}_{i, k}^2 \overset{\eqref{eq:D_prox_clipped_SSTM_shift_sum_3_variance_bound}}{\leq} \frac{M^2}{6\cdot 9^2\ln \frac{10nK}{\beta}},\\
        \sum\limits_{k=0}^{T-1}\sum\limits_{j=2}^n\hat{\sigma}_{i,k}^2 \overset{\eqref{eq:prox_clipped_SSTM_shift_sum_6_prime_variance_bound}}{\leq} \frac{M^2}{6^3\ln\frac{10nK}{\beta}}.
    \end{gather*}

    In addition, we also establish (see \eqref{eq:D_prox_clipped_SSTM_shift_sum_1_upper_bound}, \eqref{eq:D_prox_clipped_SSTM_shift_sum_3_upper_bound}, \eqref{eq:prox_clipped_SSTM_shift_technical_6_3} and our induction assumption):
    \begin{gather*}
        \PP\{E_{T-1}\} \geq 1 - \frac{(T-1)\beta}{K},\\
        \PP\{E_{\circledOne}\} \geq 1 - \frac{\beta}{5nK}, \quad \PP\{E_{\circledThree}\} \geq 1 - \frac{\beta}{5nK},\quad  \PP\{E_{\circledSix'}\} \geq 1 - \frac{\beta}{5nK},
    \end{gather*}
    where 
    \begin{eqnarray}
        E_{\circledOne} &=& \left\{ \text{either} \quad  \sum\limits_{k=0}^{T-1}\sum^n_{i=1} \sigma_{i,k}^2 > \frac{M^2}{6^3\ln\frac{10nK}{\beta}} \quad \text{or}\quad |\circledOne| \leq \frac{M}{6}\right\}. \notag\\
        E_{\circledThree} &=& \left\{ \text{either} \quad  \sum\limits_{k=0}^{T-1}\sum^n_{i=1} \widetilde\sigma_{i,k}^2 > \frac{M^2}{6\cdot 9^2\ln\frac{10nK}{\beta}} \quad \text{or}\quad |\circledThree| \leq \frac{M}{9}\right\}, \notag\\
        E_{\circledSix'} &=& \left\{\text{either} \quad\sum\limits_{k=0}^{T-1}\sum\limits_{j=2}^n\hat{\sigma}_{i,l}^2 > \frac{M^2}{6^3\ln\frac{10nK}{\beta}}\quad \text{ or } \quad|\circledSix'| \leq \frac{M}{6}\right\}. \notag 
    \end{eqnarray}

    Therefore, probability event $E_{T-1} \cap E_{\circledOne} \cap E_{\circledThree} \cap E_{\circledSix'}$ implies 
    \begin{eqnarray}
        B_T &\leq& \frac{M}{6} + \frac{M}{6} + \frac{M}{9} + \frac{M}{9}+ \frac{M}{6}+ \frac{M}{6}  \notag\\
        &&+ 8\sum\limits_{k=0}^{T-1} \sum\limits_{j=2}^n \left\langle \frac{\alpha_{k+1}}{n}\sum\limits_{i=1}^{j-1} \omega_{i,k+1}^u - \chi_j^k, \omega_{j,k+1}^u \right\rangle \notag\\
        &\leq& M +  8 \sum\limits_{j=2}^n \left\langle \frac{\alpha_{T}}{n}\sum\limits_{i=1}^{j-1} \omega_{i,T}^u - \chi_j^{T-1}, \omega_{j,T}^u \right\rangle. \label{eq:prox_clipped_SSTM_technical_8}
    \end{eqnarray}

    In the final part of the proof, we will show that $\frac{\alpha_{k+1}}{n}\sum\limits_{i=1}^{j-1} \omega_{i,k+1}^u = \chi_j^k$ with high probability. In particular, we consider probability event $\widetilde{E}_{T-1,j}$ defined as follows: inequalities
    \begin{equation}
       \left\|\frac{\alpha_{T}}{n}\sum\limits_{i=1}^{r-1} \omega_{i,T}^u\right\| \leq \frac{\sqrt{M}}{2}
    \end{equation}
    hold for $r = 2, \ldots, j$ simultaneously. We want to show that $\PP\{E_{T-1} \cap \widetilde{E}_{T-1,j}\} \geq 1 - \frac{(T-1)\beta}{K} - \frac{2j\beta}{5nK}$ for all $j = 2, \ldots, n$. For $j = 2$ the statement is trivial since
    \begin{eqnarray*}
         \left\|\frac{\alpha_{T}}{n} \omega_{1,T}^u\right\| \overset{\eqref{eq:D_prox_clipped_SSTM_shift_norm_omega_u_bound}}{\leq} \frac{2\alpha_{T}\lambda_{T-1}}{n} \leq \frac{\sqrt{M}}{2}.
    \end{eqnarray*}
    
    Next, we assume that the statement holds for some $j = m-1 < n$, i.e.,  $\PP\{E_{T-1}\cap\widetilde{E}_{T-1,m-1}\} \geq 1 - \frac{(T-1)\beta}{K+1} - \frac{2(m-1)\beta}{5n(K+1)}$. Our goal is to prove that $\PP\{E_{T-1}\cap\widetilde{E}_{T-1,m}\} \geq 1 - \frac{(T-1)\beta}{K} - \frac{2m\beta}{5nK}$. First, we consider $\left\|\frac{\alpha_{T}}{n}\sum\limits_{i=1}^{m-1} \omega_{i,T}^u\right\|$: 
    \begin{eqnarray}
        \left\|\frac{\alpha_{T}}{n}\sum\limits_{i=1}^{m-1} \omega_{i,T}^u\right\| &=& \sqrt{\frac{\alpha^2_{T}}{n^2}\left\|\sum\limits_{i=1}^{m-1} \omega_{i,T}^u\right\|^2} \notag\\
        &=& \sqrt{\frac{\alpha^2_{T}}{n^2}\sum\limits_{i=1}^{m-1}\|\omega_{i,T}^u\|^2 + \frac{2\alpha_{T}}{n}\sum\limits_{i=1}^{m-1}\left\langle \frac{\alpha_{T}}{n}\sum\limits_{r=1}^{i-1}\omega_{r,T}^u , \omega_{i,T}^u \right\rangle} \notag\\
        &\leq& \sqrt{\sum\limits_{k=0}^{T-1}\sum\limits_{i=1}^{m-1}\frac{\alpha_{k+1}^2}{n^2}\|\omega_{i,k+1}^u\|^2 + \frac{2\alpha_{T}}{n}\sum\limits_{i=1}^{m-1}\left\langle \frac{\alpha_{T}}{n}\sum\limits_{r=1}^{i-1}\omega_{r,T}^u , \omega_{i,T}^u \right\rangle}. \notag
    \end{eqnarray}
    Next, we introduce a new notation:
    \begin{equation*}
        \rho_{i,T-1} = \begin{cases}
            \frac{\alpha_{T}}{n}\sum\limits_{r=1}^{i-1}\omega_{r,T}^u,& \text{if } \left\|\frac{\alpha_{T}}{n}\sum\limits_{r=1}^{i-1}\omega_{r,T}^u\right\| \leq \frac{\sqrt{M}}{2},\\
            0,& \text{otherwise}
        \end{cases}
    \end{equation*}
    for $i = 1,\ldots,m-1$ a. By definition, we have
    \begin{equation}
        \|\rho_{i,T-1}\| \leq \frac{\sqrt{M}}{2} \label{eq:prox_clipped_SSTM_bound_rho}
    \end{equation}
    for $i = 1,\ldots,m-1$. Moreover, $\widetilde{E}_{m-1}$ implies $\rho_{i,T-1} = \frac{\alpha_{T}}{n}\sum\limits_{r=1}^{i-1}\omega_{r,T}^u$ for $i = 1,\ldots,m-1$ and
    \begin{eqnarray}
        \left\|\frac{\alpha_{T}}{n}\sum\limits_{i=1}^{m-1} \omega_{i,T}^u\right\| &\leq& \sqrt{\circledThree + \circledFour + \circledSeven}, \notag
    \end{eqnarray}
    where
    \begin{gather*}
        \circledSeven =  \frac{2\alpha_{T}}{n}\sum\limits_{i=1}^{m-1}\left\langle \rho_{i,T-1} , \omega_{i,T}^u \right\rangle.
    \end{gather*}
    It remains to estimate $\circledSeven$.

    \paragraph{Upper bound for $\circledSeven$}. To estimate this sum, we will use Bernstein's inequality. The summands have conditional expectations equal to zero:
    \begin{equation*}
        \EE_{\xi_{i}^{T-1}}\left[\frac{2\alpha_{T}}{n} \langle \rho_{i,T-1}, \omega_{i,T}^u \rangle\right] = \frac{2\alpha_{T}}{n} \left\langle \rho_{i,T-1}, \EE_{\xi_{i}^{T-1}}[\omega_{i,T}^u] \right\rangle = 0,
    \end{equation*}
    since random vectors $\{\omega_{i,T}^u\}_{i=1}^n$ are independent. Thus, sequence $\left\{\frac{2\alpha_{T}}{n} \langle \rho_{i,T-1}, \omega_{i,T}^u \rangle\right\}_{i = 1}^{m-1}$ is a martingale difference sequence. Next, the summands are bounded:
    \begin{eqnarray}
        \left|\frac{2\alpha_{T}}{n}\langle \rho_{i,T-1}, \omega_{i,T}^u \rangle \right| \leq \frac{2\alpha_T}{n} \|\rho_{i,T-1}\|\cdot \|\omega_{i,T}^u\| \overset{\eqref{eq:D_prox_clipped_SSTM_shift_norm_omega_u_bound}, \eqref{eq:prox_clipped_SSTM_bound_rho}}{\leq} \frac{2\alpha_T}{n} \sqrt{M}\lambda_{T-1} \overset{\eqref{eq:Dprox_clipped_SSTM_shift_clipping_level}}{=} \frac{M}{36\ln\tfrac{10nK}{\beta}} \eqdef c. \label{eq:prox_SSTM_technical_7_1}
    \end{eqnarray}
    Finally, conditional variances $\Bar{\sigma}_{i,T-1}^2 \eqdef \EE_{\xi_{i}^{T-1}}\left[\frac{4\alpha_{T}^2}{n^2} \langle \rho_{i,T-1}, \omega_{i,T}^u \rangle^2\right]$ of the summands are bounded:
    \begin{equation}
        \Bar{\sigma}_{i,T-1}^2 \leq \EE_{\xi_{i}^{T-1}}\left[\frac{4\alpha^2_T }{n^2}\|\rho_{i,T-1}\|^2\cdot \|\omega_{i,T}^u\|^2\right] \overset{\eqref{eq:prox_clipped_SSTM_bound_rho}}{\leq} \frac{\alpha^2_T M}{n^2} \EE_{\xi_{i}^{T-1}}\left[\|\omega_{i,T}^u\|^2\right]. \label{eq:prox_SSTM_technical_7_2}
    \end{equation}

    Applying Bernstein's inequality (Lemma~\ref{lem:Bernstein_ineq}) with $X_{i} = \frac{2\alpha_T}{n} \langle \rho_{i,T-1}, \omega_{i,T}^u \rangle$, constant $c$ defined in \eqref{eq:prox_SSTM_technical_7_1}, $b = \frac{V}{36}$, $G = \frac{M^2}{6^5\ln\frac{10nK}{\beta}}$, we get
    \begin{equation*}
        \PP\left\{|\circledSeven| > \frac{M}{36} \text{ and } \sum\limits_{i=1}^{m-1}\bar{\sigma}_{i,T-1}^2 \leq \frac{M^2}{6^5\ln\frac{10nK}{\beta}}\right\} \leq 2\exp\left(- \frac{b^2}{2G + \nicefrac{2cb}{3}}\right) = \frac{\beta}{5nK}.
    \end{equation*}
    The above is equivalent to 
    \begin{equation}
        \PP\{E_{\circledSeven}\} \geq 1 - \frac{\beta}{5nK},\; \text{ for }\; E_{\circledSeven} = \left\{\text{either } \sum\limits_{i=1}^{m-1}\bar{\sigma}_{i,T-1}^2 > \frac{M^2}{6^5\ln\frac{10nK}{\beta}} \text{ or } |\circledSeven| \leq \frac{M}{36}\right\}. \label{eq:prox_SSTM_technical_7_3}
    \end{equation}

    Moreover, $E_{T-1}$ implies 
    \begin{eqnarray}
        \sum^{m-1}_{i=1}\bar{\sigma}_{i,T-1}^2 &\overset{\eqref{eq:prox_SSTM_technical_7_2}}{\leq}& \sum^{m-1}_{i=1} \frac{\alpha^2_T M}{n^2} \EE_{\xi_{i}^{T-1}}\left[\|\omega_{i,T}^u\|^2\right] \overset{\eqref{eq:D_prox_clipped_SSTM_shift_second_moment_omega_u_bound}}{\leq} 18 (m-1)\frac{\alpha^2_T}{n^2} M \sigma^{\alpha}\lambda_{T-1}^{2-\alpha} \notag\\
        &\overset{\eqref{eq:Dprox_clipped_SSTM_shift_clipping_level}}{\leq}& 18 (m-1) \frac{\alpha^2_T}{n^2} \cdot \frac{n^{2-\alpha} M^{2-\nicefrac{\alpha}{2}}}{72^{2-\alpha}\widetilde{\alpha}_T^{2-\alpha}} \ln^{\alpha-2}\frac{10nK}{\beta} \notag\\
        &\overset{m\leq n, T \leq K}{\leq}& 18\cdot 6^5 \frac{\alpha_K^2}{n^{\alpha-1}} \left(\frac{\sigma}{\sqrt{M}}\right)^{\alpha}\ln^{\alpha-1}\frac{10nK}{\beta} \cdot \frac{M^2}{6^5\ln\frac{10nK}{\beta}}\notag\\
        &\leq& \frac{1}{a^{\alpha}} \frac{18\cdot 6^5}{2^{\alpha}} \frac{1}{n^{\alpha-1}}\left(\frac{\sigma}{L\sqrt{M}}\right)^{\alpha} (K+1)^{\alpha} \ln^{\alpha-1}\frac{10nK}{\beta} \cdot \frac{M^2}{6^5\ln\frac{10nK}{\beta}}\notag\\
        &\overset{\eqref{eq:Dprox_clipped_SSTM_shift_parameter_a}}{\leq}& \frac{M^2}{6^5\ln\frac{10nK}{\beta}}. \label{eq:prox_clipped_SSTM_shift_sum_7_variance_bound}
    \end{eqnarray}

    Putting all together we get that $E_{T-1}\cap \widetilde{E}_{T-1,m-1}$ implies
    \begin{gather*}
        \left\|\frac{\alpha_{T}}{n}\sum\limits_{i=1}^{m-1} \omega_{i,T}^u\right\| \leq \sqrt{\circledThree + \circledFour + \circledSeven},\quad \circledFour \overset{\eqref{eq:D_prox_clipped_SSTM_shift_sum_4_upper_bound}}{\leq} \frac{M}{9}, \\
        \sum\limits_{k=0}^{T-1}\sum^n_{i=1} \widetilde{\sigma}_{i, k}^2 \overset{\eqref{eq:D_prox_clipped_SSTM_shift_sum_3_variance_bound}}{\leq} \frac{M^2}{6\cdot 9^2\ln \frac{10nK}{\beta}},\quad \sum^{m-1}_{i=1}\bar{\sigma}_{i,T-1}^2 \overset{\eqref{eq:prox_clipped_SSTM_shift_sum_7_variance_bound}}{\leq} \frac{M^2}{6^5\ln\frac{10nK}{\beta}}
    \end{gather*}
    In addition, we also establish (see and our induction assumption): 
    \begin{gather*}
        \PP\{E_{T-1}\cap \widetilde{E}_{T-1,m-1}\} \geq 1 - \frac{(T-1)\beta}{K+1} - \frac{2(m-1)\beta}{5nK},\\
        \PP\{E_{\circledThree}\} \geq 1 - \frac{\beta}{5nK}, \quad \PP\{E_{\circledSeven}\} \geq 1 - \frac{\beta}{5nK},
    \end{gather*}
    where
    \begin{eqnarray}
        E_{\circledThree} &=& \left\{\text{either } \sum\limits_{k=0}^{T-1}\sum\limits_{i=1}^n\widetilde\sigma_{i,k}^2 > \frac{M^2}{6\cdot 9^2\ln\frac{10nK}{\beta}} \text{ or } |\circledFour| \leq \frac{M}{9}\right\},\notag\\
        E_{\circledSeven} &=& \left\{\text{either } \sum\limits_{i=1}^{m-1}\bar{\sigma}_{i,k}^2 > \frac{M^2}{6^5\ln\frac{10nK}{\beta}} \text{ or } |\circledSeven| \leq \frac{M}{36}\right\} \notag
    \end{eqnarray}
    Therefore, probability event $E_{T-1}\cap \widetilde E_{T-1, m-1} \cap E_{\circledThree} \cap E_{\circledSeven}$ implies
    \begin{equation*}
        \left\|\frac{\alpha_{T}}{n}\sum\limits_{i=1}^{m-1} \omega_{i,T}^u\right\| \leq \sqrt{\frac{M}{9} + \frac{M}{9} + \frac{M}{36}} \leq \frac{\sqrt{M}}{2}.
    \end{equation*}
    This implies $\widetilde E_{T-1,m}$ and
    \begin{eqnarray*}
        \PP\{E_{T-1} \cap \widetilde{E}_{T-1,m}\} &\geq& \PP\{E_{T-1}\cap \widetilde E_{T-1, m-1} \cap E_{\circledThree} \cap E_{\circledSeven}\} \\
        &=& 1 - \PP\left\{\overline{E_{T-1} \cap \widetilde{E}_{T-1,m-1}} \cup \overline{E}_{\circledThree} \cup \overline{E}_{\circledSeven}\right\}\\
        &\geq& 1 - \frac{(T-1)\beta}{K} - \frac{2m\beta}{5nK}.
    \end{eqnarray*}
    Therefore, for all $m = 2,\ldots,n$ the statement holds and, in particular, $\PP\{E_{T-1} \cap \widetilde{E}_{T-1, n}\} \geq 1 - \frac{(T-1)\beta}{K} - \frac{2\beta}{5K}$. Taking into account \eqref{eq:prox_clipped_SSTM_technical_8}, we conclude that $E_{T-1} \cap \widetilde{E}_{T-1, n} \cap E_{\circledOne} \cap E_{\circledThree} \cap E_{\circledSix'} \cap E_{\circledSeven}$ implies
    \begin{equation*}
        B_T \leq M
    \end{equation*}
    that is equivalent to \eqref{eq:SSTM_induction_inequality_1} for $t = T$. Moreover,
    \begin{eqnarray*}
        \PP\left\{E_T\right\} &\geq& \PP\left\{E_{T-1} \cap \widetilde{E}_{T-1, n} \cap E_{\circledOne} \cap E_{\circledThree} \cap E_{\circledSix'} \cap E_{\circledSeven} \right\} \\
        &=& 1 - \PP\left\{\overline{E_{T-1} \cap \widetilde{E}_{n}} \cup \overline{E}_{\circledOne} \cup \overline{E}_{\circledThree} \cup \overline{E}_{\circledSix'} \cup \overline{E}_{\circledSeven}\right\}\\
        &=& 1 - \frac{(T-1)\beta}{K} - \frac{2\beta}{5K} - 3\cdot \frac{\beta}{5nK} \geq 1 - \frac{T\beta}{K}.
    \end{eqnarray*}

    Finally, if 
    \begin{equation*}
        a = \max\left\{2, \frac{8\cdot 3^5\cdot 72^4}{n} \ln^4\frac{10nK}{\beta}, \frac{ 18\cdot 6^5\sigma K^{\frac{1}{\alpha}}(K+1)}{\sqrt{M}Ln^{\frac{\alpha-1}{\alpha}}} \ln^{\frac{\alpha-1}{\alpha}}\frac{10nK}{\beta} \right\},
    \end{equation*}

    then with probability at least $1-\beta$
    \begin{eqnarray*}
        \Phi(y^K) - \Phi(x^*) &\leq& \frac{6aLM}{K(K+3)} \\
        &=& \max \left\{\frac{12LM}{K(K+3)}, \frac{162\cdot 72^5LM}{nK(K+3)}\ln^4\frac{10nK}{\beta}, \frac{3\cdot 6^7 \sigma \frac{K+1}{K+3}}{\sqrt{M}(Kn)^{\frac{\alpha-1}{\alpha}}} \ln^{\frac{\alpha-1}{\alpha}}\frac{10nK}{\beta}\right\}\\
        &=& \cO\left(\max\left\{\frac{LM}{K^2}, \frac{LM\ln^4\frac{nK}{\beta}}{nK^2},\frac{\sigma \sqrt{M} \ln^{\frac{\alpha-1}{\alpha}}\frac{nK}{\beta}}{n^{\frac{\alpha-1}{\alpha}}K^{\frac{\alpha-1}{\alpha}}}\right\}\right).
    \end{eqnarray*}
    To get $\Phi(y^K) - \Phi(x^*) \leq \varepsilon$ wit probability $1-\beta$, $K$ should be 
    \begin{equation*}
        \cO\left(\max\left\{\sqrt{\frac{LM}{\varepsilon}}, \sqrt{\frac{LM}{\varepsilon n}}\ln^2\frac{nLM}{\varepsilon\beta}, \frac{1}{n}\left(\frac{\sigma \sqrt{M}}{\varepsilon}\right)^{\frac{\alpha}{\alpha-1}}\ln\frac{\sigma \sqrt{M}}{\varepsilon\beta}\right\}\right) 
    \end{equation*}
    that concludes the proof.
\end{proof}

\subsection{Strongly Convex case}\label{appendix:SSTM_str_cvx}

In this section, we provide the complete formulation of our result for \algname{R-DProx-clipped-SSTM-shift} (a restarted version for \algname{DProx-clipped-SSTM-shift}) and proofs. We should mention that the results for \algname{DProx-clipped-SSTM-shift}), Theorem~\ref{thm:D_prox_clipped_SSTM_shift_full} and Lemma~\ref{lem:SSTM_optimization_lemma_convex}, can be proven in the same way if we assume that $h_i^0 = \nabla f_i(x^0)$ and $M \geq \|x^0-x^*\|^2 + C^2\alpha^2_{K_0+1}\frac{1}{n}\sum^n_{i=1}\|h_i^0  - \nabla f_i(x^*)\|^2$. 

For the readers' convenience, the method’s update rule is repeated below:

\begin{algorithm}[h]
\caption{Restarted \algname{DProx-clipped-SSTM-shift} (\algname{R-DProx-clipped-SSTM-shift})}
\label{alg:R-DProx-clipped-SSTM-shift}   
\begin{algorithmic}[1]
\REQUIRE starting point $x^0$, number of restarts $\tau$, number of steps of \algname{DProx-clipped-SSTM-shift} between restarts $\{K_t\}_{t=1}^{\tau}$, stepsize parameters $\{a_t\}_{t=1}^\tau$, clipping levels $\{\lambda_{k}^1\}_{k=0}^{K_1-1}$, $\{\lambda_{k}^2\}_{k=0}^{K_2-1}$, \ldots, $\{\lambda_{k}^\tau\}_{k=0}^{K_\tau - 1}$, smoothness constant $L$, the constant $\{N_t\}^{\tau}_{t=1}$.
\STATE $\hat{x}^0 = x^0$
\FOR{$t=1,\ldots, \tau$}
\STATE Run \algname{DProx-clipped-SSTM-shift} for $K_t$ iterations with stepsize parameter $a_{t}$, clipping levels $\{\lambda_{k}^t\}_{k=0}^{K_t-1}$, and starting point $\hat{x}^{t-1}$. Define the output of \algname{DProx-clipped-SSTM-shift} by $\hat{x}^{t}$.
\ENDFOR
\ENSURE $\hat{x}^\tau$ 
\end{algorithmic}
\end{algorithm}

\begin{theorem}\label{thm:R_DProx_clipped_SSTM_shift_main_theorem_appendix}
    Let Assumptions~\ref{as:bounded_alpha_moment}, \ref{as:L_smoothness}, \ref{as:str_cvx} with $\mu > 0$ hold for $Q = B_{5n\sqrt{M}}(x^*)$, where $M \geq \|x^0-x^*\|^2 + C^2_t\alpha^2_{N_1+1}\frac{1}{n}\sum^n_{i=1}\|h_i^0  - \nabla f_i(x^*)\|^2$ and \algname{R-DProx-clipped-SSTM-shift} runs \algname{DProx-clipped-SSTM-shift} $\tau$ times. Let
    \begin{gather}
        K_t =  \Bigg\lceil \max \Bigg\{ \sqrt{\frac{24LM_{t-1}}{\varepsilon_t}}, 2\cdot10^{15}\sqrt{\frac{LM_{t-1}}{n\varepsilon_t}}\ln\frac{2\cdot10^{16}n\sqrt{LM_{t-1}}\tau}{\sqrt{\varepsilon_t}\beta},\notag\\ 
        \frac{1}{n}\left(\frac{6^8\sigma \sqrt{M_{t-1}}}{\varepsilon_t}\right)^{\frac{\alpha}{\alpha-1}} \ln\left(\frac{10\tau}{\beta}\left(\frac{6^8\sigma \sqrt{M_{t-1}}}{\varepsilon_t}\right)^{\frac{\alpha}{\alpha-1}}\right) \notag\\ 
        \frac{1}{n^{\frac{5\alpha-1}{\alpha-1}}}\left(\frac{16\cdot 10^{24}\sigma \sqrt{M_{t-1}}}{\varepsilon_t}\right)^{\frac{\alpha}{\alpha-1}} \ln^{\frac{7\alpha-1}{\alpha-1}}\left(\frac{10\tau}{\beta}\left(\frac{16\cdot 10^{24}\sigma \sqrt{M_{t-1}}}{\varepsilon_t}\right)^{\frac{\alpha}{\alpha-1}}\right)\Bigg\} \Bigg \rceil, \label{eq:R_clipped_SSTM_K_t} \\
        \varepsilon_t = \frac{\mu M_{t-1}}{4},\quad M_{t-1} = \frac{M}{2^{(t-1)}}, \quad \tau = \left\lceil \log_2 \frac{\mu M}{2\varepsilon} \right\rceil,\quad \ln\frac{10nK_t\tau}{\beta} \geq 1, \label{eq:R_clipped_SSTM_epsilon_R_t_tau} \\
        a \geq \max\left\{2, \frac{8\cdot 3^5\cdot 72^4}{n} \ln^4\frac{10nK_t}{\beta}, \frac{ 18\cdot 6^5\sigma K_t^{\frac{1}{\alpha}}(K_t+1)}{\sqrt{M_t}Ln^{\frac{\alpha-1}{\alpha}}} \ln^{\frac{\alpha-1}{\alpha}}\frac{10nK_t}{\beta} \right\},\label{eq:R_Dprox_clipped_SSTM_shift_parameter_a}\\
        \lambda^t_{k} = \frac{n\sqrt{M_t}}{72\widetilde\alpha^t_{k+1}\ln\frac{10nK_t}{\beta}}, \label{eq:R_Dprox_clipped_SSTM_shift_clipping_level}
    \end{gather}
    for $t = 1, \ldots, \tau$, where $C_t =\frac{864}{n}\ln \frac{10nK_t}{\beta}$, $ N_t =  \left\lceil\frac{3}{2}C^2_tn\right\rceil > 0$. Then to achieve $\Phi(\hat x^\tau) - \Phi(x^*) \leq \varepsilon$ with probability at least $ 1 - \beta$ \algname{R-DProx-clipped-SSTM-shift} requires
    \begin{gather}
        \cO\Bigg(\max\Bigg\{ \sqrt{\frac{L}{\mu}}\ln\left(\frac{\mu M}{\varepsilon}\right), \sqrt{\frac{L}{n\mu}}\ln\left(\frac{\mu M}{\varepsilon}\right)\ln^5\left(\frac{\sqrt{L}}{\sqrt{\mu}\beta}\ln\left(\frac{\mu M}{\varepsilon}\right)\right), \notag \\
        \frac{1}{n}\left(\frac{\sigma^2}{\mu \varepsilon}\right)^{\frac{\alpha}{2(\alpha-1)}} \ln\left(\frac{1}{\beta}\left(\frac{\sigma^2}{\mu \varepsilon}\right)^{\frac{\alpha}{2(\alpha-1)}}\ln\left(\frac{\mu M}{\varepsilon}\right)\right),\notag\\
        \frac{1}{n^{\frac{5\alpha-1}{\alpha-1}}}\left(\frac{\sigma^2}{\mu \varepsilon}\right)^{\frac{\alpha}{2(\alpha-1)}} \ln^{\frac{7\alpha-1}{\alpha-1}}\left(\frac{1}{\beta}\left(\frac{\sigma^2}{\mu \varepsilon}\right)^{\frac{\alpha}{2(\alpha-1)}}\ln\left(\frac{\mu M}{\varepsilon}\right)\right)\Bigg\} \Bigg) \label{eq:R_clipped_SSTM_main_result_appendix}
    \end{gather}
    iterations/oracle calls per worker. Moreover, with probability $\geq 1-\beta$ the iterates of \algname{R-DProx-clipped-SSTM-shift} at stage $t$ stay in the ball $B_{2\sqrt{M_{t-1}}}(x^*)$.
\end{theorem}
\begin{proof}
The key idea behind the proof is similar to the one used in \citep{gorbunov2021near,sadiev2023high}. We prove by induction that for any $t = 1,\ldots, \tau$ with probability at least $1 - \nicefrac{t\beta}{\tau}$ inequalities
    \begin{equation}
        \Phi(\hat{x}^l) - \Phi(x^*) \leq \varepsilon_l, \quad \hat{M}_l \leq  M_l = \frac{M}{2^l}
    \end{equation}
    hold for $l = 1,\ldots, t$ simultaneously. We recall the Lyapunov function is determined as
    \begin{equation*}
        \hat{M}_l = \|\hat{x}^l-x^*\|^2 +C^2_l(\alpha^l_{N_l+1})^2\frac{1}{n}\sum^n_{i=1}\|\nabla f_i(\hat{x}^l)-\nabla f_i(x^*)\|^2 \overset{\eqref{eq:L_smoothness}}{\leq} \underbrace{(1+C^2_l(\alpha^l_{N_l+1})^2L^2)}_{ \eqdef G_l}\|\hat{x}^l-x^*\|^2,
    \end{equation*}
    where by definition of $C_l$, $\alpha^l_{N_l+1}$, we can estimate $G_l$
    \begin{equation}
        \label{eq:G_t}
        G_l \leq 2 \max\left\{1, \frac{9\cdot 864^6}{4n^4}\ln^6\frac{10nK_{l+1}\tau}{\beta}\right\}
    \end{equation}
    Now, we prove the base of the induction. Theorem~\ref{thm:D_prox_clipped_SSTM_shift_full} implies that with probability at least $1 - \nicefrac{\beta}{\tau}$
    \begin{eqnarray}
        G_1(\Phi(\hat x^1) - \Phi(x^*)) &\leq& G_1\frac{6a_1LR^2}{K_1(K_1+3)} \overset{\eqref{eq:R_Dprox_clipped_SSTM_shift_parameter_a}}{=} 2 \max\left\{1, \frac{9\cdot 864^6}{4n^4}\ln^6\frac{10nK_{1}\tau}{\beta}\right\} \notag\\
        &&\times~ \max \left\{\frac{12LM}{K_1(K_1+3)}, \frac{162\cdot 72^5LM}{nK_1(K_1+3)}\ln^4\frac{10nK_1\tau}{\beta}, \frac{3\cdot 6^7 \sigma \frac{K_1+1}{K_1+3}}{\sqrt{M}(K_1n)^{\frac{\alpha-1}{\alpha}}} \ln^{\frac{\alpha-1}{\alpha}}\frac{10nK_1\tau}{\beta}\right\} \notag\\
        &\leq& \max\Bigg\{\frac{24LM}{K^2_1}, \frac{162\cdot 72^5\cdot9\cdot 864^6LM}{nK_1^2}\ln^{10}\frac{10nK_1\tau}{\beta},\notag\\
        &&\quad\frac{6^8\sigma \sqrt{M}\ln^{\frac{\alpha-1}{\alpha}}\frac{10nK_1\tau}{\beta}}{(nK_1)^{\frac{\alpha-1}{\alpha}}}, 
        \frac{81\cdot5184^6\cdot \sigma \sqrt{M}\ln^{\frac{7\alpha-1}{\alpha}}\frac{10nK_1\tau}{\beta}}{n^{\frac{5\alpha-1}{\alpha}}K_1^{\frac{\alpha-1}{\alpha}}}\Bigg\}\notag\\
        &\overset{\eqref{eq:R_clipped_SSTM_K_t}}{\leq}& \varepsilon_1 = \frac{\mu M}{4} \notag
    \end{eqnarray}
    and, due to the strong convexity,
    \begin{equation*}
        \hat{M}_1\leq G_1\|\hat x^1 - x^*\|^2 \leq \frac{2G_1(\Phi(\hat x^1) - \Phi(x^*))}{\mu} \leq \frac{M}{2} = M_1.
    \end{equation*}
    The base of the induction is proven. Now, assume that the statement holds for some $t = T < \tau$, i.e., with probability at least $1 - \nicefrac{T\beta}{\tau}$ inequalities
    \begin{equation}
        \Phi(\hat{x}^l) - \Phi(x^*) \leq \varepsilon_l, \quad \hat{M}_l \leq  M_l = \frac{M}{2^l}
    \end{equation}
    hold for $l = 1,\ldots, T$ simultaneously. In particular, with probability at least $1 - \nicefrac{T\beta}{\tau}$ we have $\hat{M}_T \leq M_T$. Applying Theorem~\ref{thm:D_prox_clipped_SSTM_shift_full} and using union bound for probability events, we get that with probability at least $1 - \nicefrac{(T+1)\beta}{\tau}$
    \begin{eqnarray}
        G_{T+1}(\Phi(\hat x^{T+1}) - \Phi(x^*)) &\leq& G_{T+1}\frac{6a_{T+1}LM^2_T}{K_{T+1}(K_{T+1}+3)} \overset{\eqref{eq:R_Dprox_clipped_SSTM_shift_parameter_a}}{=} 2 \max\left\{1, \frac{9\cdot 864^6}{4n^4}\ln^6\frac{10nK_{T+1}\tau}{\beta}\right\} \notag\\
        &&\times~ \max \Bigg\{\frac{12LM_{T}}{K_{T+1}(K_{T+1}+3)}, \frac{162\cdot 72^5LM_{T+1}}{nK_{T+1}(K_{T+1}+3)}\ln^4\frac{10nK_{T+1}\tau}{\beta}, \notag\\
        &&\quad\quad\frac{3\cdot 6^7 \sigma \frac{K_{T+1}+1}{K_{T+1}+3}}{\sqrt{M}(K_{T+1}n)^{\frac{\alpha-1}{\alpha}}} \ln^{\frac{\alpha-1}{\alpha}}\frac{10nK_{T+1}\tau}{\beta}\Bigg\} \notag\\
        &\leq& \max\Bigg\{\frac{24LM_T}{K^2_{T+1}}, \frac{162\cdot 72^5\cdot9\cdot 864^6LM_{T+1}}{nK_{T+1}^2}\ln^{10}\frac{10nK_{T+1}\tau}{\beta},\notag\\
        &&\quad\frac{6^8\sigma \sqrt{M_T}\ln^{\frac{\alpha-1}{\alpha}}\frac{10nK_{T+1}\tau}{\beta}}{(nK_{T+1})^{\frac{\alpha-1}{\alpha}}}, 
        \frac{81\cdot5184^6\cdot \sigma \sqrt{M_T}\ln^{\frac{7\alpha-1}{\alpha}}\frac{10nK_{T}\tau}{\beta}}{n^{\frac{5\alpha-1}{\alpha}}K_{T+1}^{\frac{\alpha-1}{\alpha}}}\Bigg\}\notag\\
        &\overset{\eqref{eq:R_clipped_SSTM_K_t}}{\leq}& \varepsilon_{T+1} = \frac{\mu M_T}{4} \notag
    \end{eqnarray}
    and, due to the strong convexity,
    \begin{equation*}
        \hat{M}_{T+1}\leq G_{T+1}\|\hat x^{T+1} - x^*\|^2 \leq \frac{2G_{T+1}(\Phi(\hat x^{T+1}) - \Phi(x^*))}{\mu} \leq \frac{M_T}{2} = M_{T+1}.
    \end{equation*}
    Thus, we finished the inductive part of the proof. In particular, with probability at least $1 - \beta$ inequalities
    \begin{equation}
        \Phi(\hat{x}^l) - \Phi(x^*) \leq \varepsilon_l, \quad \hat{M}_l \leq  M_l = \frac{M}{2^l}\notag
    \end{equation}
    hold for $l = 1,\ldots, \tau$ simultaneously, which gives for $l = \tau$ that with probability at least $1 - \beta$
    \begin{equation*}
        \Phi(\hat{x}^l) - \Phi(x^*) \leq \varepsilon_\tau =  \frac{\mu M_{\tau-1}}{4} = \frac{\mu M}{2^{\tau+1}}\overset{\eqref{eq:R_clipped_SSTM_epsilon_R_t_tau}}{\leq} \varepsilon.
    \end{equation*}
    It remains to calculate the overall number of oracle calls during all runs of \algname{clipped-SSTM}. We have
    \begin{eqnarray*}
        \sum\limits_{t=1}^\tau K_t &=& \cO\Bigg(\sum\limits_{t=1}^\tau \max\Bigg\{ 
        \sqrt{\frac{LM_{t-1}}{\varepsilon_t}},
        \sqrt{\frac{LM_{t-1}^2}{n\varepsilon_t}}\ln^5\left(\frac{n\sqrt{LM_{t-1}^2}\tau}{\sqrt{\varepsilon_t}\beta}\right), \\
        &&\frac{1}{n}\left(\frac{\sigma \sqrt{M_{t-1}}}{\varepsilon_t}\right)^{\frac{\alpha}{\alpha-1}} \ln\left(\frac{\tau}{\beta}\left(\frac{\sigma \sqrt{M_{t-1}}}{\varepsilon_t}\right)^{\frac{\alpha}{\alpha-1}}\right),\\
        &&\frac{1}{n^{\frac{7\alpha-1}{\alpha-1}}}\left(\frac{\sigma \sqrt{M_{t-1}}}{\varepsilon_t}\right)^{\frac{\alpha}{\alpha-1}} \ln^{\frac{5\alpha-1}{\alpha-1}}\left(\frac{\tau}{\beta}\left(\frac{\sigma \sqrt{M_{t-1}}}{\varepsilon_t}\right)^{\frac{\alpha}{\alpha-1}}\right) \Bigg\} \Bigg)\\
        &=& \cO\Bigg(\sum\limits_{t=1}^\tau \max\Bigg\{ \sqrt{\frac{L}{\mu}}, \sqrt{\frac{L}{n\mu}}\ln\left(\frac{n\sqrt{L}\tau}{\sqrt{\mu}\beta}\right),\frac{1}{n}\left(\frac{\sigma}{\mu \sqrt{M_{t-1}}}\right)^{\frac{\alpha}{\alpha-1}} \ln\left(\frac{\tau}{\beta}\left(\frac{\sigma }{\mu\sqrt{M_{t-1}}}\right)^{\frac{\alpha}{\alpha-1}}\right),\\
        && \frac{1}{n^{\frac{5\alpha-1}{\alpha-1}}}\left(\frac{\sigma}{\mu \sqrt{M_{t-1}}}\right)^{\frac{\alpha}{\alpha-1}} \ln^{\frac{7\alpha-1}{\alpha-1}}\left(\frac{\tau}{\beta}\left(\frac{\sigma }{\mu\sqrt{M_{t-1}}}\right)^{\frac{\alpha}{\alpha-1}}\right) \Bigg\} \Bigg)\\
        &=& \cO\Bigg(\max\Bigg\{ \{ \tau \sqrt{\frac{L}{\mu}}, \tau \sqrt{\frac{L}{n\mu}}\ln\left(\frac{n\sqrt{L}\tau}{\sqrt{\mu}\beta}\right), \frac{1}{n} \sum\limits_{t=1}^\tau\left(\frac{\sigma \cdot 2^{\nicefrac{t}{2}}}{\mu \sqrt{M}}\right)^{\frac{\alpha}{\alpha-1}} \ln\left(\frac{\tau}{\beta}\left(\frac{\sigma \cdot 2^{\nicefrac{t}{2}}}{\mu \sqrt{M}}\right)^{\frac{\alpha}{\alpha-1}}\right),\\
        &&\frac{1}{n^{\frac{5\alpha-1}{\alpha-1}}} \sum\limits_{t=1}^\tau\left(\frac{\sigma \cdot 2^{\nicefrac{t}{2}}}{\mu R}\right)^{\frac{\alpha}{\alpha-1}} \ln^{\frac{7\alpha-1}{\alpha-1}}\left(\frac{\tau}{\beta}\left(\frac{\sigma \cdot 2^{\nicefrac{t}{2}}}{\mu \sqrt{M}}\right)^{\frac{\alpha}{\alpha-1}}\right)\Bigg\} \Bigg)\\
        &=& \cO\Bigg(\max\Bigg\{ \sqrt{\frac{L}{\mu}}\ln\left(\frac{\mu M}{\varepsilon}\right),
        \sqrt{\frac{L}{n\mu}}\ln\left(\frac{\mu M}{\varepsilon}\right)\ln^5\left(\frac{n\sqrt{L}}{\sqrt{\mu}\beta}\ln\left(\frac{\mu M}{\varepsilon}\right)\right),\\
        &&
        \frac{1}{n}\left(\frac{\sigma}{\mu \sqrt{M}}\right)^{\frac{\alpha}{\alpha-1}} \ln\left(\frac{\tau}{\beta}\left(\frac{\sigma \cdot 2^{\nicefrac{\tau}{2}}}{\mu \sqrt{M}}\right)^{\frac{\alpha}{\alpha-1}}\right)\sum\limits_{t=1}^\tau 2^\frac{\alpha t}{2(\alpha-1)},\\
        && \frac{1}{n^{\frac{5\alpha-1}{\alpha-1}}}\left(\frac{\sigma}{\mu \sqrt{M}}\right)^{\frac{\alpha}{\alpha-1}} \ln^{\frac{7\alpha-1}{\alpha-1}}\left(\frac{\tau}{\beta}\left(\frac{\sigma \cdot 2^{\nicefrac{\tau}{2}}}{\mu \sqrt{M}}\right)^{\frac{\alpha}{\alpha-1}}\right)\sum\limits_{t=1}^\tau 2^\frac{\alpha t}{2(\alpha-1)}\Bigg\} \Bigg)\\
        &=& \cO\Bigg(\max\Bigg\{ \sqrt{\frac{L}{\mu}}\ln\left(\frac{\mu M} {\varepsilon}\right), 
        \sqrt{\frac{L}{n\mu}}\ln\left(\frac{\mu M} {\varepsilon}\right)\ln\left(\frac{\sqrt{L}}{\sqrt{\mu}\beta}\ln\left(\frac{\mu R^2}{\varepsilon}\right)\right),\\
        &&
        \frac{1}{n}\left(\frac{\sigma}{\mu \sqrt{M}}\right)^{\frac{\alpha}{\alpha-1}} \ln\left(\frac{\tau}{\beta}\left(\frac{\sigma}{\mu \sqrt{M}}\right)^{\frac{\alpha}{\alpha-1}}\cdot 2^{\frac{\alpha}{2(\alpha-1)}}\right)2^\frac{\alpha \tau}{2(\alpha-1)},\\
        &&
        \frac{1}{n^{\frac{5\alpha-1}{\alpha-1}}}\left(\frac{\sigma}{\mu \sqrt{M}}\right)^{\frac{\alpha}{\alpha-1}} \ln^{\frac{7\alpha-1}{\alpha-1}}\left(\frac{\tau}{\beta}\left(\frac{\sigma}{\mu \sqrt{M}}\right)^{\frac{\alpha}{\alpha-1}}\cdot 2^{\frac{\alpha}{2(\alpha-1)}}\right)2^\frac{\alpha \tau}{2(\alpha-1)}\Bigg\} \Bigg).
    \end{eqnarray*}
    Thus, we have 
    \begin{eqnarray*}
        \sum\limits_{t=1}^\tau K_t &=&\cO\Bigg(\max\Bigg\{ \sqrt{\frac{L}{\mu}}\ln\left(\frac{\mu M}{\varepsilon}\right), \sqrt{\frac{L}{n\mu}}\ln\left(\frac{\mu M}{\varepsilon}\right)\ln^5\left(\frac{\sqrt{L}}{\sqrt{\mu}\beta}\ln\left(\frac{\mu M}{\varepsilon}\right)\right), \notag \\
        &&
        \frac{1}{n}\left(\frac{\sigma^2}{\mu \varepsilon}\right)^{\frac{\alpha}{2(\alpha-1)}} \ln\left(\frac{1}{\beta}\left(\frac{\sigma^2}{\mu \varepsilon}\right)^{\frac{\alpha}{2(\alpha-1)}}\ln\left(\frac{\mu M}{\varepsilon}\right)\right),\\
        &&
        \frac{1}{n^{\frac{5\alpha-1}{\alpha-1}}}\left(\frac{\sigma^2}{\mu \varepsilon}\right)^{\frac{\alpha}{2(\alpha-1)}} \ln^{\frac{7\alpha-1}{\alpha-1}}\left(\frac{1}{\beta}\left(\frac{\sigma^2}{\mu \varepsilon}\right)^{\frac{\alpha}{2(\alpha-1)}}\ln\left(\frac{\mu M}{\varepsilon}\right)\right)\Bigg\} \Bigg),
    \end{eqnarray*}
    which concludes the proof.
\end{proof}

\clearpage

\section{Missing Proofs for \algname{DProx-clipped-SGDA-shift}}\label{appendix:dprox_clipped_SGDA_shift}

\subsection{Cocoercive case}
In this section, we give the complete formulations of our results for \algname{DProx-clipped-SGDA-shift} and rigorous proofs. For the readers' convenience, the method's update rule is repeated below:
\begin{gather*}
    x^{k+1} = \proxkPsi \left(x^k - \gamma \tg^k\right),\;\; \text{where}\;\; \tg^k = \frac{1}{n}\sum\limits_{i=1}^n \tg_i^k,\;\; \tg_i^k = h_i^k + \hat\Delta_i^k, \\
    h_i^{k+1} = h_i^k + \nu \hat \Delta_i^k,\quad \hat\Delta_i^k = \clip\left(F_{\xi_i^k}(x^k) - h_i^k, \lambda_k\right). 
\end{gather*}

\begin{lemma}\label{lem:optimization_lemma_gap_prox_clipped_SGDA}
    Let Assumptions~\ref{as:monotonicity}, \ref{as:star_cocoercivity} and \ref{as:cocoercivity} hold for $Q = B_{3\sqrt{V}}(x^*)$, where $V \geq \|x^0 - x^*\|^2 + \frac{25600\gamma^2\ln^2\frac{48n(K+1)}{\beta}}{n^2}\sum_{i=1}^n\|F_i(x^*)\|^2$ and $0 < \gamma \leq \nicefrac{1}{\ell}$. If $x^k$ lies in $B_{3\sqrt{V}}(x^*)$ for all $k = 0,1,\ldots, K-1$ for some $K\geq 0$, then for all $u \in B_{3\sqrt{V}}(x^*)$ the iterates produced by \algname{DProx-clipped-SGDA-shift} satisfy
    \begin{eqnarray}
        \langle F(u), x^K_{\avg} - u\rangle + \Psi(x^K_{\avg}) - \Psi(u) &\leq& \frac{\|x^0 - u\|^2 - \|x^{K} - u\|^2}{2\gamma K} + \frac{\gamma}{K}\sum\limits_{k=0}^{K-1} \|\omega_k\|^2\notag\\
        &&\quad + \frac{1}{K}\sum\limits_{k=0}^{K-1}\langle x^k - u , \omega_k\rangle, \label{eq:optimization_lemma_prox_clipped_SGDA}\\
        x^K_{\avg} &\eqdef& \frac{1}{K}\sum\limits_{k=0}^{K-1}x^{k+1}, \label{eq:x_avg_prox_clipped_SGDA}\\
        \omega_k &\eqdef& F(x^k) - \tg^k. \label{eq:omega_k_rpox_clipped_SGDA}
    \end{eqnarray}
\end{lemma}
\begin{proof}
    The proof of this lemma follows the proof of Theorem D.3 from \citep{beznosikov2022stochastic}. For completeness, we provide here the full proof. We start with the application of Lemma~\ref{lem:prox_lemma} with $x^+= x^{k+1}$, $x = x^k - \gamma g^k$, and $y = u$ for arbitrary $u \in B_{3\sqrt{V}}(x^*)$:
    \begin{equation*}
        \langle x^{k+1} - x^k + \gamma \tg^k, u - x^{k+1} \rangle \geq \gamma\left(\Psi(x^{k+1}) - \Psi(u)\right).
    \end{equation*}
    Rearranging the terms, we get
    \begin{eqnarray*}
        2\gamma\left(\Psi(x^{k+1}) - \Psi(u)\right) &\leq& 2\gamma\langle \tg^k, u - x^k \rangle + 2\langle x^{k+1} - x^k, u - x^k \rangle \\
        &&\quad + 2\langle x^{k+1} - x^k + \gamma \tg^k, x^k - x^{k+1} \rangle 
    \end{eqnarray*}
    implying
    \begin{eqnarray*}
        2\gamma\left(\langle F(x^k), x^k - u \rangle + \Psi(x^{k+1}) - \Psi(u)\right) &\leq& 2\langle x^{k+1} - x^k, u - x^k \rangle + 2\gamma\langle F(x^k) - \tg^k, x^k - u\rangle \\
        &&\quad + 2\langle x^{k+1} - x^k + \gamma \tg^k, x^k - x^{k+1} \rangle\\
        &=& \|x^{k+1} - x^k\|^2 + \|x^k - u\|^2 - \|x^{k+1} - u\|^2\\
        &&\quad + 2\gamma\langle F(x^k) - \tg^k, x^k - u\rangle \\
        &&\quad - 2\|x^{k+1} - x^k\|^2 + 2\gamma\langle \tg^k, x^k - x^{k+1} \rangle\\
        &=& \|x^k - u\|^2 - \|x^{k+1} - u\|^2 - \|x^{k+1} - x^k\|^2\\
        &&\quad + 2\gamma\langle F(x^k) - \tg^k, x^k - u\rangle\\
        &&\quad + 2\gamma\langle F(u), x^k - x^{k+1} \rangle\\
        &&\quad + 2\gamma\langle \tg^k - F(u), x^k - x^{k+1} \rangle\\
        &\leq& \|x^k - u\|^2 - \|x^{k+1} - u\|^2\\
        &&\quad + 2\gamma\langle F(x^k) - \tg^k, x^k - u\rangle\\
        &&\quad + 2\gamma\langle F(u), x^k - x^{k+1} \rangle + \gamma^2\|\tg^k - F(u)\|^2,
    \end{eqnarray*}
    where in the last step we apply $2\gamma\langle \tg^k - F(u), x^k - x^{k+1} \rangle \leq \gamma^2\|\tg^k - F(u)\|^2 + \|x^k - x^{k+1}\|^2$. Adding $2\gamma\langle F(u), x^{k+1} - u \rangle - 2\gamma\langle F(x^k), x^k - u \rangle$ to the both sides, we derive
    \begin{eqnarray*}
        2\gamma\left(\langle F(u), x^{k+1} - u \rangle + \Psi(x^{k+1}) - \Psi(u)\right) &\leq& \|x^k - u\|^2 - \|x^{k+1} - u\|^2\\
        &&\quad + 2\gamma\langle F(u) - \tg^k, x^k - u\rangle + \gamma^2\|\tg^k - F(u)\|^2\\
        &=& \|x^k - u\|^2 - \|x^{k+1} - u\|^2\\
        &&\quad - 2\gamma\langle F(x^k) - F(u), x^k - u\rangle + \gamma^2\|\tg^k - F(u)\|^2\\
        &&\quad + 2\gamma\langle F(x^k) - \tg^k, x^k - u\rangle\\
        &\overset{\eqref{eq:cocoercivity}}{\leq}& \|x^k - u\|^2 - \|x^{k+1} - u\|^2\\
        &&\quad - \frac{2\gamma}{\ell}\|F(x^k) - F(u)\|^2 + 2\gamma^2\|F(x^k) - F(u)\|^2\\
        &&\quad + 2\gamma\langle F(x^k) - \tg^k, x^k - u\rangle + 2\gamma^2\|F(x^k) - \tg^k\|^2\\
        &\leq& \|x^k - u\|^2 - \|x^{k+1} - u\|^2\\
        &&\quad+ 2\gamma\langle \omega_k, x^k - u\rangle + 2\gamma^2\|\omega_k\|^2.
    \end{eqnarray*}
    Next, we sum up the above inequality for $k=0,1,\ldots,K-1$ and divide both sides by $2\gamma K$:
    \begin{eqnarray*}
         \frac{1}{K}\sum\limits_{k=0}^{K-1}\left(\langle F(u), x^{k+1} - u\rangle + \Psi(x^{k+1}) - \Psi(u)\right) &\leq& \frac{\|x^0 - u\|^2 - \|x^{K} - u\|^2}{2\gamma K} + \frac{\gamma}{K}\sum\limits_{k=0}^{K-1} \|\omega_k\|^2\notag\\
        &&\quad + \frac{1}{K}\sum\limits_{k=0}^{K-1}\langle x^k - u , \omega_k\rangle,
    \end{eqnarray*}
    To finish the proof, we need to use Jensen's inequality $\Psi\left(\frac{1}{K}\sum\limits_{k=0}^{K-1}x^{k+1}\right) \leq \frac{1}{K}\sum\limits_{k=0}^{K-1}\Psi(x^{k+1})$:
    \begin{eqnarray*}
         \langle F(u), x^K_{\avg} - u\rangle + \Psi(x^K_{\avg}) - \Psi(u) &\leq& \frac{\|x^0 - u\|^2 - \|x^{K} - u\|^2}{2\gamma K} + \frac{\gamma}{K}\sum\limits_{k=0}^{K-1} \|\omega_k\|^2\notag\\
        &&\quad + \frac{1}{K}\sum\limits_{k=0}^{K-1}\langle x^k - u , \omega_k\rangle,
    \end{eqnarray*}
    where 
    $x^K_{\avg} = \frac{1}{K}\sum\limits_{k=0}^{K-1}x^{k+1}.$
\end{proof}

\begin{theorem}\label{thm:main_result_gap_prox_clipped_SGDA}
    Let Assumptions~\ref{as:monotonicity}, \ref{as:star_cocoercivity}, and \ref{as:cocoercivity} hold for $Q = B_{3\sqrt{V}}(x^*)$, where $V \geq \|x^0 - x^*\|^2 + \frac{25600\gamma^2\ln^2\frac{48n(K+1)}{\beta}}{n^2}\sum_{i=1}^n\|F_i(x^*)\|^2$, and    
    \begin{eqnarray}
        0< \gamma &\leq& \min\left\{\frac{1}{480\ell \ln \tfrac{48n(K+1)}{\beta}}, \frac{\sqrt{V} n^{\frac{\alpha-1}{\alpha}}}{(86400)^{\frac{1}{\alpha}}(K+1)^{\frac{1}{\alpha}}\sigma \ln^{\frac{\alpha-1}{\alpha}} \tfrac{48n(K+1)}{\beta}}\right\}, \label{eq:gamma_prox_clipped_SGDA}\\
        \lambda_{k} \equiv \lambda &=& \frac{n\sqrt{V}}{40\gamma \ln \tfrac{48n(K+1)}{\beta}}, \label{eq:lambda_prox_clipped_SGDA}
    \end{eqnarray}
    for some $K \geq 0$ and $\beta \in (0,1]$. Then, after $K$ iterations the iterates produced by \algname{DProx-clipped-SGDA-shift} with probability at least $1 - \beta$ satisfy 
    \begin{equation}
        \gap_{\sqrt{V}}(x_{\avg}^{K+1}) \leq \frac{4V}{\gamma(K+1)} \quad \text{and}\quad \{x^k\}_{k=0}^{K+1} \subseteq B_{3\sqrt{V}}(x^*), \label{eq:main_result_prox_clipped_SGDA}
    \end{equation}
    where $x_{\avg}^{K+1}$ is defined in \eqref{eq:x_avg_prox_clipped_SGDA}. In particular, when $\gamma$ equals the minimum from \eqref{eq:gamma_prox_clipped_SGDA}, then the iterates produced by \algname{DProx-clipped-SGDA-shift} after $K$ iterations with probability at least $1-\beta$ satisfy
    \begin{equation}
        \gap_{\sqrt{V}}(x_{\avg}^{K+1}) = \cO\left(\max\left\{\frac{\ell V\ln\frac{nK}{\beta}}{K}, \frac{\sigma \sqrt{V} \ln^{\frac{\alpha-1}{\alpha}}\frac{nK}{\beta}}{n^{\frac{\alpha-1}{\alpha}}K^{\frac{\alpha-1}{\alpha}}}\right\}\right), \label{eq:prox_clipped_SGDA_monotone_case_2}
    \end{equation}
    meaning that to achieve $\gap_{\sqrt{V}}(x_{\avg}^{K+1}) \leq \varepsilon$ with probability at least $1 - \beta$ \algname{DProx-clipped-SGDA-shift} requires
    \begin{equation}
        K = \cO\left(\frac{\ell V}{\varepsilon}\ln\frac{n\ell V}{\varepsilon\beta}, \frac{1}{n}\left(\frac{\sigma \sqrt{V}}{\varepsilon}\right)^{\frac{\alpha}{\alpha-1}}\ln\left(\frac{1}{\beta}\left(\frac{\sigma \sqrt{V}}{\varepsilon}\right)^{\frac{\alpha}{\alpha-1}}\right)\right)\quad \text{iterations/oracle calls.} \label{eq:prox_clipped_SGDA_monotone_case_complexity}
    \end{equation}
\end{theorem}

\begin{proof}
    The key idea behind the proof is similar to the one used in \citep{gorbunov2022clipped, sadiev2023high}: we prove by induction that the iterates do not leave some ball and the sums decrease as $\nicefrac{1}{K+1}$. To formulate the statement rigorously, we introduce probability event $E_k$ for each $k = 0,1,\ldots, K+1$ as follows: inequalities
    \begin{gather}
        \underbrace{\max\limits_{u \in B_{\sqrt{V}}(x^*)}\left\{ \|x^0 - u\|^2 + 2\gamma \sum\limits_{l = 0}^{t-1} \langle x^l - u, \omega_l \rangle + 2\gamma^2 \sum\limits_{l=0}^{t-1}\|\omega_l\|^2\right\}}_{A_t} \leq 8V, \label{eq:induction_inequality_1_prox_clipped_SGDA}\\
        \left\|\gamma\sum\limits_{l=0}^{t-1}\omega_l\right\| \leq \sqrt{V} \label{eq:induction_inequality_prox_clipped_SGDA_gap},\\
        \left\|\gamma\sum\limits_{i=1}^{r-1}\omega_{i,t-1}^u\right\| \leq \frac{\sqrt{V}}{2} \label{eq:induction_inequality_prox_clipped_SGDA_gap_2}
    \end{gather}
    hold for $t = 0,1,\ldots,k$ and $r = 1,2,\ldots, n$ simultaneously, where
    \begin{gather}
        \omega_l = \omega_l^u + \omega_l^b, \label{eq:gap_thm_prox_clipped_SGDA_technical_4_full_theta_omega}\\
         \omega_l^u \eqdef \frac{1}{n}\sum\limits_{i=1}^n \omega_{i,l}^u,\quad \omega_l^b \eqdef \frac{1}{n}\sum\limits_{i=1}^n \omega_{i,l}^b, \\
        \omega_{i,l}^u \eqdef \EE_{\xi_{i}^l}\left[\tg_{i}^l\right] - \tg_{i}^l,\quad \omega_{i,l}^b \eqdef F_i(x^l) - \EE_{\xi_{i}^l}\left[\tg_{i}^l\right] \quad \forall\; i\in [n]. \label{eq:thm_prox_clipped_SGDA_technical_gap_4}
    \end{gather}
    We will prove by induction that $\PP\{E_k\} \geq  1 - \nicefrac{k\beta}{(K+1)}$ for all $k = 0,1,\ldots,K+1$. The base of induction follows immediately: for all $u \in B_{\sqrt{V}}(x^*)$ we have $\|x^0 - u\|^2 \leq 2\|x^0 - x^*\|^2 + 2\|x^* - u\|^2 \leq 4V < 8V$ and for $k = 0$ we have $\|\gamma\sum_{l=0}^{k-1} \omega_l\| = 0$. Next, we assume that the statement holds for $k = T-1 \leq K$, i.e., $\PP\{E_{T-1}\} \geq  1 - \nicefrac{(T-1)\beta}{(K+1)}$. Let us show that it also holds for $k = T$, i.e., $\PP\{E_{T}\} \geq  1 - \nicefrac{T\beta}{(K+1)}$. 

    To proceed, we need to show that $E_{T-1}$ implies $\|x^t - x^*\| \leq 3\sqrt{V}$ for all $t = 0,1,\ldots,T-1$. We will use the induction argument as well. The base is already proven. Next, we assume that $\|x^t - x^*\| \leq 3\sqrt{V}$ for all $t = 0,1,\ldots,t'$ for some $t' < T-1$. This means that $x^t \in B_{3\sqrt{V}}(x^*)$ for $t = 0,1,\ldots, t'$ and we can apply Lemma~\ref{lem:optimization_lemma_gap_prox_clipped_SGDA}: $E_{T-1}$ implies
    \begin{align*}
        \max_{u \in B_{\sqrt{V}}(x^*)}\Big\{ 2\gamma(t'+1)\Big(\langle F(u), x^{t'+1}_{\avg} - u\rangle &+ \Psi(x^{t'+1}_{\avg}) - \Psi(u)\Big) + \|x^{t'+1} - u\|^2  \Big\}\\
        &\leq \max_{u \in B_{\sqrt{V}}(x^*)}\left\{\|x^0 - u\|^2 + 2\gamma \sum\limits_{l = 0}^{t} \langle x^l - u, \omega_l \rangle \right\}\\
        &\quad + 2\gamma^2 \sum\limits_{l=0}^{t}\|\omega_l\|^2\\
        &\overset{\eqref{eq:induction_inequality_1_prox_clipped_SGDA}}{\leq} 8V.
    \end{align*}
    that gives
    \begin{eqnarray*}
        \|x^{t'+1} - x^*\|^2 &\leq& \max_{u\in B_{\sqrt{V}}(x^*)}\left\{ 2\gamma(t'+1)\left(\langle F(u), x^{t'}_{\avg} - u\rangle + \Psi(x^{t'}_{\avg}) - \Psi(u)\right) + \|x^{t'+1} - u\|^2  \right\}\\
        &\leq& 8V.
    \end{eqnarray*}
    That is, we showed that $E_{T-1}$ implies $\|x^t - x^*\| \leq 3\sqrt{V}$ and
    \begin{equation}
        \max_{u \in B_{\sqrt{V}}(x^*)}\left\{ 2\gamma t\left(\langle F(u), \tx^{t}_{\avg} - u\rangle + \Psi(\tx^{t}_{\avg}) - \Psi(u)\right) + \|x^{t+1} - u\|^2  \right\} \leq 8V \label{eq:gap_thm_prox_clipped_SGDA_technical_1_5}
    \end{equation}
    for all $t = 0, 1, \ldots, T-1$. Before we proceed, we introduce a new notation:
    \begin{equation*}
        \eta_t = \begin{cases}x^t - x^* ,& \text{if } \|x^t - x^*\| \leq 3\sqrt{V},\\ 0,& \text{otherwise,} \end{cases}
    \end{equation*}
    for all $t = 0, 1, \ldots, T-1$. Random vectors $\{\eta_t\}_{t=0}^T$ are bounded almost surely:
     \begin{equation}
        \|\eta_t\| \leq 3\sqrt{V}  \label{eq:thm_prox_clipped_SGDA_technical_gap_3}
    \end{equation}
    for all $t = 0, 1, \ldots, T$. In addition, $\eta_t = x^t - x^*$ follows from $E_{T-1}$ for all $t = 0, 1, \ldots, T$ and, thus, $E_{T-1}$ implies
   \begin{eqnarray}
        A_T &\overset{\eqref{eq:induction_inequality_1_prox_clipped_SGDA}}{=}& \max\limits_{u \in B_{\sqrt{V}}(x^*)}\left\{ \|x^0 - u\|^2 + 2\gamma \sum\limits_{l = 0}^{T-1} \langle x^* - u, \omega_l \rangle \right\}  + 2\gamma \sum\limits_{l = 0}^{T-1} \langle x^l - x^*, \omega_l \rangle + 2\gamma^2 \sum\limits_{l=0}^{T-1}\|\omega_l\|^2\notag\\
        &\leq& 4V + 2\gamma\max\limits_{u \in B_{\sqrt{V}}(x^*)}\left\{\left\langle x^* - u, \sum\limits_{l = 0}^{T-1}\omega_l \right\rangle \right\} + 2\gamma \sum\limits_{l = 0}^{T-1} \langle \eta_l, \omega_l \rangle + 2\gamma^2 \sum\limits_{l=0}^{T-1}\|\omega_l\|^2\notag\\
        &=& 4V + 2\gamma \sqrt{V} \left\|\sum\limits_{l = 0}^{T-1}\omega_l\right\| + 2\gamma \sum\limits_{l = 0}^{T-1} \langle \eta_l, \omega_l \rangle + 2\gamma^2 \sum\limits_{l=0}^{T-1}\|\omega_l\|^2. \notag
    \end{eqnarray}
    Using the notation from \eqref{eq:gap_thm_prox_clipped_SGDA_technical_4_full_theta_omega}-\eqref{eq:thm_prox_clipped_SGDA_technical_gap_4}, we can rewrite $\|\omega_l\|^2$ as
    \begin{eqnarray}
        \|\omega_l\|^2 &\leq&  2\|\omega_l^u\|^2 + 2\|\omega_l^b\|^2 = \frac{2}{n}\left\|\sum\limits_{i=1}^n \omega_{i,l}^u\right\|^2 + 2\|\omega_l^b\|^2\notag \\
        &=& \frac{2}{n^2}\sum\limits_{i=1}^n\|\omega_{i,l}^u\|^2 + \frac{4}{n^2}\sum\limits_{j=2}^n\left\langle \sum\limits_{i=1}^{j-1} \omega_{i,l}^u, \omega_{j,l}^u \right\rangle + 2\|\omega_l^b\|^2. \label{eq:omega_bound_x_t_SGDA_3_prox_clipped}
    \end{eqnarray}
    Putting all together, we obtain that $E_{T-1}$ implies
    \begin{eqnarray}
        A_T &\leq&4V + 2\gamma \sqrt{V} \left\|\sum\limits_{l = 0}^{T-1}\omega_l\right\| + \underbrace{\frac{2\gamma}{n} \sum\limits_{l = 0}^{T-1}\sum\limits_{i=1}^n \langle \eta_l, \omega_{i,l}^u \rangle}_{\circledOne} + \underbrace{2\gamma \sum\limits_{l = 0}^{T-1} \langle \eta_l, \omega_l^b \rangle}_{\circledTwo}\notag\\
        &&\quad + \underbrace{\frac{4\gamma^2}{n^2} \sum\limits_{l=0}^{T-1} \sum\limits_{i=1}^n\EE_{\xi_{i}^l}\left[\|\omega_{i,l}^u\|^2\right]}_{\circledThree} + \underbrace{\frac{4\gamma^2}{n^2} \sum\limits_{l=0}^{T-1} \sum\limits_{i=1}^n\left(\|\omega_{i,l}^u\|^2 - \EE_{\xi_{i}^l}\left[\|\omega_{i,l}^u\|^2\right]\right)}_{\circledFour}\notag\\
        &&\quad + \underbrace{4\gamma^2 \sum\limits_{l=0}^{T-1}\|\omega_l^b\|^2}_{\circledFive} + \underbrace{\frac{8\gamma^2}{n^2} \sum\limits_{l=0}^{T-1} \sum\limits_{j=2}^n \left\langle \sum\limits_{i=1}^{j-1} \omega_{i,l}^u, \omega_{j,l}^u \right\rangle}_{\circledSix}.\label{eq:thm_prox_clipped_SGDA_technical_gap_5}
    \end{eqnarray}
    To finish the proof, it remains to estimate $2\gamma \sqrt{V} \left\|\sum_{l = 0}^{T-1}\omega_l\right\|, \circledOne, \circledTwo, \circledThree, \circledFour, \circledFive, \circledSix$ with high probability. More precisely, the goal is to prove that $2\gamma \sqrt{V} \left\|\sum_{l = 0}^{T-1}\omega_l\right\| + \circledOne + \circledTwo + \circledThree + \circledFour + \circledFive + \circledSix \leq 4V$ with high probability. Before we proceed, we need to derive several useful inequalities related to $ \omega_{i,l}^u, \omega_{l}^b$. First of all, we have
    \begin{equation}
        \|\omega_{i,l}^u\| \leq 2\lambda \label{eq:omega_magnitude_prox_clipped_SGDA_gap}
    \end{equation}
    by definition of the clipping operator. Next, probability event $E_{T-1}$ implies 
    \begin{eqnarray}
        \|F_i(x^l)\| &\leq& \|F_i(x^l) - F_i(x^*)\| + \|F_i(x^*)\| \leq \ell\|x^l - x^*\| + \sqrt{\sum\limits_{i=1}^n\|F_i(x^*)\|^2}\notag\\
        &\leq& 3\ell \sqrt{V} + \frac{n\sqrt{V}}{160\gamma \ln\tfrac{48n(K+1)}{\beta}} \overset{\eqref{eq:gamma_prox_clipped_SGDA}}{\leq} \frac{n\sqrt{V}}{80\gamma \ln\tfrac{48n(K+1)}{\beta}} \overset{\eqref{eq:lambda_prox_clipped_SGDA}}{=} \frac{\lambda}{2} \label{eq:operator_bound_x_t_prox_clipped_SGDA_gap}
    \end{eqnarray}
    for $l = 0,1,\ldots,T-1$ and $i\in [n]$.
    Therefore, Lemma~\ref{lem:bias_and_variance_clip} and $E_{T-1}$ imply
    \begin{gather}
        \left\|\omega_l^b\right\| \leq \frac{1}{n}\sum\limits_{i=1}^n \|\omega_{i,l}^b\| \leq \frac{2^\alpha\sigma^\alpha}{\lambda^{\alpha-1}}, \label{eq:bias_omega_prox_clipped_SGDA_gap}\\
        \EE_{\xi_{i}^l}\left[\left\|\omega_{i,l}^u\right\|^2\right] \leq 18 \lambda^{2-\alpha}\sigma^\alpha, \label{eq:variance_omega_prox_clipped_SGDA_gap}
    \end{gather}
    for all $l = 0,1, \ldots, T-1$ and $i\in [n]$.    
    
 \paragraph{Upper bound for $\circledOne$.} To estimate this sum, we will use Bernstein's inequality. The summands have conditional expectations equal to zero:
    \begin{equation*}
        \EE_{\xi^l_i}\left[\frac{2\gamma}{n}\langle \eta_l, \omega_{i,l}^u \rangle\right] = \frac{2\gamma}{n} \left\langle \eta_l, \EE_{\xi_{i}^l}[\omega_{i,l}^u] \right\rangle = 0.
    \end{equation*}
    Moreover, for all $l = 0,\ldots, T-1$ random vectors $\{\omega_{i,l}^u\}_{i=1}^n$ are independent. Thus, sequence $\left\{\frac{2\gamma}{n} \langle \eta_l, \omega_{i,l}^u \rangle\right\}_{l,i = 0,1}^{T-1,n}$ is a martingale difference sequence. Next, the summands are bounded:
    \begin{eqnarray}
        \left|\frac{2\gamma}{n}\langle \eta_l, \omega_{i,l}^u \rangle \right| \leq \frac{2\gamma}{n} \|\eta_l\|\cdot \|\omega_{i,l}^u\| \overset{\eqref{eq:thm_prox_clipped_SGDA_technical_gap_3},\eqref{eq:omega_magnitude_prox_clipped_SGDA_gap}}{\leq} \frac{12 \gamma \sqrt{V} \lambda}{n} \overset{\eqref{eq:lambda_prox_clipped_SGDA}}{\leq} \frac{3V}{10\ln\tfrac{48n(K+1)}{\beta}} \eqdef c. \label{eq:SGDA_neg_mon_technical_1_1}
    \end{eqnarray}
    Finally, conditional variances $\sigma_{i,l}^2 \eqdef \EE_{\xi^l_i}\left[\frac{4\gamma^2}{n^2} \langle \eta_l, \omega_{i,l}^u \rangle^2\right]$ of the summands are bounded:
    \begin{eqnarray}
        \sigma_{i,l}^2 \leq \EE_{\xi^l_i}\left[\frac{4\gamma^2}{n^2} \|\eta_l\|^2\cdot \|\omega_{i,l}^u\|^2\right] \overset{\eqref{eq:thm_prox_clipped_SGDA_technical_gap_3}}{\leq} \frac{36\gamma^2 V}{n^2} \EE_{\xi^l_i}\left[\|\omega_{i,l}^u\|^2\right]. \label{eq:prox_clipped_SGDA_neg_mon_technical_1_2}
    \end{eqnarray}
    Applying Bernstein's inequality (Lemma~\ref{lem:Bernstein_ineq}) with $X_{i,l} = \frac{2\gamma}{n} \langle \eta_l, \omega_{i,l}^u \rangle$, constant $c$ defined in \eqref{eq:SGDA_neg_mon_technical_1_1}, $b = \frac{3V}{10}$, $G = \tfrac{3V^2}{200\ln\frac{48n(K+1)}{\beta}}$, we get
    \begin{equation*}
        \PP\left\{|\circledOne| > \frac{3V}{10} \text{ and } \sum\limits_{l=0}^{T}\sum\limits_{i=1}^n\sigma_{i,l}^2 \leq \frac{3V^2}{200\ln\tfrac{48n(K+1)}{\beta}}\right\} \leq 2\exp\left(- \frac{b^2}{2G + \nicefrac{2cb}{3}}\right) = \frac{\beta}{24n(K+1)}.
    \end{equation*}
    The above is equivalent to 
    \begin{equation}
        \PP\{E_{\circledOne}\} \geq 1 - \tfrac{\beta}{24n(K+1)}, \text{ for }E_{\circledOne} = \left\{\text{either} \quad \sum\limits_{l=0}^{T}\sum\limits_{i=1}^n\sigma_{i,l}^2 > \frac{3V^2}{200\ln\tfrac{48n(K+1)}{\beta}}\quad \text{or}\quad |\circledOne| \leq \frac{3V}{10}\right\}. \label{eq:bound_1_prox_clipped_SGDA_neg_mon}
    \end{equation}
    Moreover, $E_{T-1}$ implies
    \begin{eqnarray}
        \sum\limits_{l=0}^{T}\sum\limits_{i=1}^n\sigma_{i,l}^2 &\overset{\eqref{eq:prox_clipped_SGDA_neg_mon_technical_1_2}}{\leq}& \frac{36\gamma^2V}{n^2}\sum\limits_{l=0}^{T} \sum\limits_{i=1}^n\EE_{\xi^l_i}\left[\|\omega_{i,l}^u\|^2\right] \notag\\ &\overset{\eqref{eq:variance_omega_prox_clipped_SGDA_gap}, T \leq K+1}{\leq}& \frac{648\gamma^2 V \sigma^{\alpha} (K+1) \lambda^{2-\alpha}}{n} \notag\\
        &\overset{\eqref{eq:lambda_prox_clipped_SGDA}}{\leq}& \frac{648 \gamma^{\alpha}\sqrt{V}^{4-\alpha}\sigma^{\alpha} (K+1)\ln^{\alpha-2}\frac{48n(K+1)}{\beta}}{40^{2-\alpha}n^{\alpha-1}} \notag\\
        &\overset{\eqref{eq:gamma_prox_clipped_SGDA}}{\leq}& \frac{3V^2}{200\ln\tfrac{48n(K+1)}{\beta}}. \label{eq:bound_1_variances_prox_clipped_SGDA_neg_mon}
    \end{eqnarray}
    
    \paragraph{Upper bound for $\circledTwo$.} Probability event $E_{T-1}$ implies
    \begin{eqnarray}
        \circledTwo &\leq& 2\gamma \sum\limits_{l=0}^{T}\|\eta_l\| \cdot \|\omega_l^b\| \overset{\eqref{eq:thm_prox_clipped_SGDA_technical_gap_3},\eqref{eq:bias_omega_prox_clipped_SGDA_gap}, T \leq K+1}{\leq} 6\cdot 2^{\alpha}\gamma \sqrt{V}(K+1) \frac{\sigma^{\alpha}}{\lambda^{\alpha-1}}\notag\\
        &\overset{\eqref{eq:lambda_prox_clipped_SGDA}}{=}& \frac{6\cdot 40^{\alpha-1}\cdot2^{\alpha}}{n^{\alpha-1}}\gamma^{\alpha}\sigma^{\alpha}\sqrt{V}^{2-\alpha} (K+1)\ln^{\alpha-1}\left(\frac{48n(K+1)}{\beta}\right) \overset{\eqref{eq:gamma_prox_clipped_SGDA}}{\leq} \frac{3V}{100}. \label{eq:bound_2_prox_clipped_SGDA_neg_mon}
    \end{eqnarray}
    
    \paragraph{Upper bound for $\circledThree$.}  Probability event $E_{T-1}$ implies
    \begin{eqnarray}
        \circledThree &=&  \frac{4\gamma^2}{n^2} \sum\limits_{l = 0}^{T}\sum\limits_{i=1}^n\EE_{\xi^l_i}\left[\|\omega_{i,l}^u\|^2\right] \overset{\eqref{eq:variance_omega_prox_clipped_SGDA_gap}, T \leq K+1}{\leq} \frac{72\gamma^2 \lambda^{2-\alpha} \sigma^{\alpha} (K+1)}{n} \notag\\
        &\overset{\eqref{eq:lambda_prox_clipped_SGDA}}{\leq}& \frac{72}{40^{2-\alpha} n^{\alpha-1}} \gamma^{\alpha} \sqrt{V}^{2-\alpha} \sigma^{\alpha} (K+1)  \ln^{\alpha-2}\left(\frac{48n(K+1)}{\beta}\right)
        \overset{\eqref{eq:gamma_prox_clipped_SGDA}}{\leq} \frac{3V}{100}. \label{eq:bound_3_prox_clipped_SGDA_neg_mon}
    \end{eqnarray}
    
    \paragraph{Upper bound for $\circledFour$.} To estimate this sum, we will use Bernstein's inequality. The summands have conditional expectations equal to zero:
    \begin{equation*}
        \frac{4\gamma^2}{n^2}\EE_{\xi^l_i}\left[\|\omega_{i,l}^u\|^2 - \EE_{\xi^l_i}\left[\|\omega_{i,l}^u\|^2\right]\right] = 0.
    \end{equation*}
    Moreover, for all $l = 0,\ldots, T-1$ random vectors  $\{\omega_{i,l}^u\}_{i=1}^n$ are independent. Thus, sequence $\left\{\frac{4\gamma^2}{n^2}\left(\|\omega_{i,l}^u\|^2 -\EE_{\xi_{i}^l}\left[\|\omega_{i,l}^u\|^2\right]\right)\right\}_{l,i = 0,1}^{T-1,n}$ is a martingale difference sequence. Next, the summands are bounded:
    \begin{eqnarray}
        \frac{4\gamma^2}{n^2} \left|\|\omega_{i,l}^u\|^2 - \EE_{\xi^l_i}\left[\|\omega_{i,l}^u\|^2\right] \right| &\leq& \frac{4\gamma^2}{n^2}\left( \|\omega_{i,l}^u\|^2 + \EE_{\xi^l_i}\left[\|\omega_{i,l}^u\|^2\right] \right) 
        \overset{\eqref{eq:omega_magnitude_prox_clipped_SGDA_gap}}{\leq} \frac{32\gamma^2 \lambda^2}{n^2} \notag\\
        &\overset{\eqref{eq:lambda_prox_clipped_SGDA}}{\leq}& \frac{V}{20\ln^2\tfrac{48n(K+1)}{\beta}} \leq \frac{V}{10\ln\tfrac{48n(K+1)}{\beta}} \eqdef c. \label{eq:prox_clipped_SGDA_neg_mon_technical_4_1}
    \end{eqnarray}
    Finally, conditional variances $$\widetilde\sigma_{i,l}^2 \eqdef \frac{16\gamma^4}{n^2}\EE_{\xi^l_i}\left[\left(\|\omega_{i,l}^u\|^2 - \EE_{\xi^l_i}\left[\|\omega_{i,l}^u\|^2\right]\right)^2\right]$$ of the summands are bounded:
    \begin{eqnarray}
        \widetilde\sigma_{i,l}^2 \overset{\eqref{eq:prox_clipped_SGDA_neg_mon_technical_4_1}}{\leq} \frac{\gamma^2 V}{5n^2\ln^2 \frac{48n(K+1)}{\beta}} \EE_{\xi^l_i}\left[\left| \|\omega_{i,l}^u\|^2 - \EE_{\xi^l_i}\left[\|\omega_{i,l}^u\|^2\right] \right|\right]  \leq \frac{2\gamma^2 V}{5n^2\ln^2 \frac{48n(K+1)}{\beta}} \EE_{\xi^l_i}\left[\|\omega_{i,l}^u\|^2\right]. \label{eq:prox_clipped_SGDA_neg_mon_technical_4_2}
    \end{eqnarray}
    Applying Bernstein's inequality (Lemma~\ref{lem:Bernstein_ineq}) with  $X_{i,l} = \frac{4\gamma^2}{n^2}\left(\|\omega_{i,l}^u\|^2 - \EE_{\xi^l_i}[\|\omega_{i,l}^u\|^2]\right)$, constant $c$ defined in \eqref{eq:prox_clipped_SGDA_neg_mon_technical_4_1}, $b = \frac{V}{10}$, $G = \tfrac{V^2}{600\ln\frac{48n(K+1)}{\beta}}$, we get
    \begin{equation*}
        \PP\left\{|\circledFour| > \frac{V}{10} \text{ and } \sum\limits_{t=0}^{T}\sum\limits_{i=1}^n\widetilde\sigma_{i,t}^2 \leq \frac{V^2}{600\ln\tfrac{48n(K+1)}{\beta}}\right\} \leq 2\exp\left(- \frac{b^2}{2G + \nicefrac{2cb}{3}}\right) = \frac{\beta}{24n(K+1)}.
    \end{equation*}
    The above is equivalent to 
    \begin{equation}
        \PP\{E_{\circledFour}\} \geq 1 - \tfrac{\beta}{24n(K+1)}, \text{ for }E_{\circledFour} = \left\{\text{either} \quad \sum\limits_{l=0}^{T}\sum\limits_{i=1}^n\widetilde\sigma_{i,l}^2 > \frac{V^2}{600\ln\tfrac{48n(K+1)}{\beta}}\quad \text{or}\quad |\circledFour| \leq \frac{V}{10}\right\}. \label{eq:bound_4_prox_clipped_SGDA_neg_mon}
    \end{equation}
    Moreover, $E_{T-1}$ implies
    \begin{eqnarray}
        \sum\limits_{l=0}^{T}\sum\limits_{i=1}^n\widetilde\sigma_{i,l}^2 &\overset{\eqref{eq:prox_clipped_SGDA_neg_mon_technical_4_2}}{\leq}& \frac{2\gamma^2 V}{5n^2\ln^2 \frac{48n(K+1)}{\beta}}\sum\limits_{l=0}^{T}\sum\limits_{i=1}^n \EE_{\xi^l_i}\left[\|\omega_{i,l}^u\|^2\right] \overset{\eqref{eq:variance_omega_prox_clipped_SGDA_gap}, T \leq K+1}{\leq} \frac{36\gamma^2 V (K+1)}{5 n \ln\tfrac{48n(K+1)}{\beta}} \lambda^{2-\alpha} \sigma^{\alpha} \notag\\
        &\overset{\eqref{eq:lambda_prox_clipped_SGDA}}{\leq}& \frac{9\cdot 40^{\alpha}\sqrt{2}^{\alpha}}{2000 n^{\alpha-1}}\gamma^{\alpha}
         \sqrt{V}^{4-\alpha} (K+1) \sigma^{\alpha}\ln^{\alpha-4}\frac{48n(K+1)}{\beta}\notag\\
        &\overset{\eqref{eq:gamma_prox_clipped_SGDA}}{\leq}&\frac{V^2}{600\ln\tfrac{48n(K+1)}{\beta}}. \label{eq:bound_4_variances_prox_clipped_SGDA_neg_mon}
    \end{eqnarray}
    
    \paragraph{Upper bound for $\circledFive$.} Probability event $E_{T-1}$ implies
    \begin{eqnarray}
        \circledFive &=&  4\gamma^2 \sum\limits_{l = 0}^{T}\|\omega_l^b\|^2  \overset{\eqref{eq:bias_omega_prox_clipped_SGDA_gap}, T \leq K+1}{\leq} 2^{2\alpha+2}\gamma^2 (K+1) \frac{\sigma^{2\alpha}}{\lambda^{2\alpha-2}}\notag\\
        &\overset{\eqref{eq:lambda_prox_clipped_SGDA}}{=}& \frac{12800^{\alpha}}{800}\gamma^{2\alpha} (K+1) \frac{\sigma^{2\alpha}}{n^{2\alpha-2}\sqrt{V}^{2\alpha-2}}\ln^{2\alpha-2}\frac{48n(K+1)}{\beta} \notag\\ &\overset{\eqref{eq:gamma_prox_clipped_SGDA}}{\leq}& \frac{V}{10}. \label{eq:bound_5_prox_clipped_SGDA_neg_mon}
    \end{eqnarray}

     \paragraph{Upper bound for $\circledSix$.} This sum requires more refined analysis. We introduce new vectors:
     \begin{equation}
        \delta_{j}^l = \begin{cases}
            \frac{\gamma}{n}\sum\limits_{i=1}^{j-1} \omega_{i,l}^u,& \text{if } \left\|\frac{\gamma}{n}\sum\limits_{i=1}^{j-1} \omega_{i,l}^u\right\| \leq \frac{\sqrt{V}}{2},\\
            0,& \text{otherwise,}
        \end{cases}\label{eq:prox_clipped_SGDA_delta_def}
    \end{equation}
    for all $j \in [n]$ and $l = 0,\ldots, T-1$. Then, by definition
    \begin{equation}
         \|\delta_j^l\| \leq \frac{\sqrt{V}}{2} \label{eq:prox_clipped_SGDA_delta_bound}
    \end{equation}
    and

\begin{eqnarray}
        \circledSix &=& \underbrace{\frac{8\gamma}{n} \sum\limits_{l=0}^{T-1} \sum\limits_{j=2}^n \left\langle \delta_j^l, \omega_{j,l}^u \right\rangle}_{\circledSix'} + \frac{8\gamma}{n} \sum\limits_{l=0}^{T-1} \sum\limits_{j=2}^n \left\langle \frac{\gamma}{n}\sum\limits_{i=1}^{j-1} \omega_{i,l}^u - \delta_j^l, \omega_{j,l}^u \right\rangle. \label{eq:prox_clipped_SGDA_extra_sums_six_distributed}
    \end{eqnarray}

    We also note here that $E_{T-1}$ implies
    \begin{eqnarray}
        \frac{8\gamma}{n} \sum\limits_{l=0}^{T-1} \sum\limits_{j=2}^n \left\langle \frac{\gamma}{n}\sum\limits_{i=1}^{j-1} \omega_{i,l}^u - \delta_j^l, \omega_{j,l}^u \right\rangle &=& \frac{8\gamma}{n}  \sum\limits_{j=2}^n \left\langle \frac{\gamma}{n}\sum\limits_{i=1}^{j-1} \omega_{i,l}^u - \delta_j^l, \omega_{j,l}^u \right\rangle. \label{eq:SGDA_extra_sums_six_distributed_1}
    \end{eqnarray}

    \paragraph{Upper bound for $\circledSix'$.} To estimate this sum, we will use Bernstein's inequality. The summands have conditional expectations equal to zero:

    \begin{equation*}
        \EE_{\xi_{j}^l}\left[\frac{8\gamma}{n}\left\langle \delta_j^l, \omega_{j,l}^u \right\rangle\right] = \frac{8\gamma}{n}\left\langle \delta_j^l, \EE_{\xi_{j}^l}[\omega_{j,l}^u] \right\rangle = 0.
    \end{equation*}

    Moreover, for all $l = 0,\ldots, T-1$ random vectors $\{\omega_{i,l}^u\}_{i=1}^n$ are independent. Thus, sequence $\left\{\frac{8\gamma}{n}\left\langle \delta_j^l, \omega_{j,l}^u \right\rangle\right\}_{l,j = 0,2}^{T-1,n}$ is a martingale difference sequence. Next, the summands are bounded:
    \begin{eqnarray}
        \left|\frac{8\gamma}{n}\left\langle \delta_j^l, \omega_{j,l}^u \right\rangle \right| \leq \frac{8\gamma}{n} \left\|\delta_j^l\right\| \cdot \|\omega_{j,l}^u\| \overset{\eqref{eq:prox_clipped_SGDA_delta_bound},\eqref{eq:omega_magnitude_prox_clipped_SGDA_gap}}{\leq} \frac{8\gamma}{n}\cdot \frac{\sqrt{V}}{2}\cdot 2\lambda \leq \frac{V}{5\ln\frac{48n(K+1)}{\beta}} \eqdef c. \label{eq:gap_thm_prox_clipped_SGDA_technical_6_5_1234}
    \end{eqnarray}
    Finally, conditional variances $(\sigma_{j,l}')^2 \eqdef \EE_{\xi_{j}^l}\left[\frac{64\gamma^2}{n^2} \langle \delta_j^l, \omega_{j,l}^u \rangle^2\right]$ of the summands are bounded:
    \begin{equation}
        (\sigma_{j,l}')^2 \leq \EE_{\xi_{j}^l}\left[\frac{64\gamma^2 }{n^2}\|\delta_j^l\|^2\cdot \|\omega_{j,l}^u\|^2\right] \overset{\eqref{eq:prox_clipped_SGDA_delta_bound}}{\leq} \frac{16\gamma^2 V}{n^2} \EE_{\xi_{j}^l}\left[\|\omega_{j,l}^u\|^2\right]. \label{eq:gap_thm_prox_clipped_SGDA_technical_7_1234}
    \end{equation}
    Applying Bernstein's inequality (Lemma~\ref{lem:Bernstein_ineq}) with $X_{i,l} = \frac{8\gamma}{n}\left\langle \delta_j^l, \omega_{j,l}^u \right\rangle$, constant $c$ defined in \eqref{eq:gap_thm_prox_clipped_SGDA_technical_6_5_1234}, $b = \frac{V}{5}$, $G = \frac{V^2}{150\ln\frac{48n(K+1)}{\beta}}$, we get
    \begin{equation*}
        \PP\left\{|\circledSix'| > \frac{V}{5} \text{ and } \sum\limits_{l=0}^{T-1}\sum\limits_{j=2}^n(\sigma_{i,l}')^2 \leq \frac{V^2}{150\ln\frac{48n(K+1)}{\beta}}\right\} \leq 2\exp\left(- \frac{b^2}{2G + \nicefrac{2cb}{3}}\right) = \frac{\beta}{24n(K+1)}.
    \end{equation*}
    The above is equivalent to 
    \begin{equation}
        \PP\{E_{\circledSix'}\} \geq 1 - \frac{\beta}{24n(K+1)},\; \text{ for }\; E_{\circledSix'} = \left\{\text{either } \sum\limits_{l=0}^{T-1}\sum\limits_{j=2}^n(\sigma_{i,l}')^2 > \frac{V^2}{150\ln\frac{48n(K+1)}{\beta}} \text{ or } |\circledSix'| \leq \frac{V}{5}\right\}. \label{eq:bound_7_gap_prox_clipped_SGDA}
    \end{equation}
    Moreover, $E_{T-1}$ implies
    \begin{eqnarray}
        \sum\limits_{l=0}^{T-1}\sum\limits_{j=2}^n(\sigma_{j,l}')^2 &\overset{\eqref{eq:gap_thm_prox_clipped_SGDA_technical_7_1234}}{\leq}& \frac{16\gamma^2 V}{n^2} \sum\limits_{l=0}^{T-1} \sum\limits_{j=2}^n \EE_{\xi_{j}^l}\left[\|\omega_{j,l}^u\|^2\right] \overset{\eqref{eq:variance_omega_prox_clipped_SGDA_gap}, T \leq K+1}{\leq} \frac{288(K+1)\gamma^2 V \lambda^{2-\alpha} \sigma^\alpha}{n} \notag \\
        &\overset{\eqref{eq:lambda_prox_clipped_SGDA}}{\leq}& \frac{288 (K+1)\gamma^{\alpha}\sigma^\alpha V^{2-\frac{\alpha}{2}}}{40^{2-\alpha} \sqrt{2}^{2-\alpha} n^{\alpha-1} \ln^{2-\alpha}\frac{48n(K+1)}{\beta}} \overset{\eqref{eq:gamma_prox_clipped_SGDA}}{\leq} \frac{V^2}{150\ln\frac{48n(K+1)}{\beta}}. \label{eq:bound_6_variances_gap_prox_clipped_SGDA}
    \end{eqnarray}

    \paragraph{Upper bound for $2\gamma \sqrt{V} \left\|\sum_{l=0}^{T-1}\omega_t\right\|$.} We introduce new random vectors:
    \begin{equation*}
        \zeta_l = \begin{cases} \gamma \sum\limits_{r=0}^{l-1}\omega_r,& \text{if } \left\|\gamma \sum\limits_{r=0}^{l-1}\omega_r\right\| \leq \sqrt{V},\\ 0, & \text{otherwise} \end{cases}
    \end{equation*}
    for $l = 1, 2, \ldots, T-1$. With probability $1$ we have
    \begin{equation}
        \|\zeta_l\| \leq \sqrt{V}.  \label{eq:gap_thm_prox_clipped_SGDA_technical_8}
    \end{equation}
    Using this and \eqref{eq:induction_inequality_prox_clipped_SGDA_gap}, we obtain that $E_{T-1}$ implies
    \begin{eqnarray}
        \gamma\left\|\sum\limits_{l = 0}^{T-1}\omega_l\right\| &=& \sqrt{\gamma^2\left\|\sum\limits_{l = 0}^{T-1}\omega_l\right\|^2}\notag\\
        &=& \sqrt{\gamma^2\sum\limits_{l=0}^{T-1}\|\omega_l\|^2 + 2\gamma\sum\limits_{l=0}^{T-1}\left\langle \gamma\sum\limits_{r=0}^{l-1} \omega_r, \omega_l \right\rangle} \notag\\
        &=& \sqrt{\gamma^2\sum\limits_{l=0}^{T-1}\|\omega_l\|^2 + 2\gamma \sum\limits_{l=0}^{T-1} \langle \zeta_l, \omega_l\rangle} \notag\\
        &\overset{\eqref{eq:thm_prox_clipped_SGDA_technical_gap_5}}{\leq}& \sqrt{\frac{1}{4}\left(\circledThree + \circledFour + \circledFive + \circledSix\right) + \underbrace{\frac{2\gamma}{n} \sum\limits_{l=0}^{T-1}\sum\limits_{i=1}^n \langle \zeta_l, \omega_{i,l}^u\rangle}_{\circledSeven} + \underbrace{2\gamma \sum\limits_{l=0}^{T-1} \langle \zeta_l, \omega_l^b}_{\circledEight}\rangle}. \label{eq:norm_sum_omega_bound_gap_prox_clipped_SGDA}
    \end{eqnarray}
    
    \paragraph{Upper bound for $\circledSeven$.} 
    To estimate this sum, we will use Bernstein's inequality. The summands have conditional expectations equal to zero:
    
    \begin{equation*}
        \EE_{\xi_i^l}\left[\frac{2\gamma}{n}\langle \zeta_l, \omega_{i,l}^u \rangle\right] = \frac{2\gamma}{n} \left\langle \zeta_l, \EE_{\xi_{i}^l}[\omega_{i,l}^u] \right\rangle = 0.
    \end{equation*}

    Moreover, for all $l = 0,\ldots, T-1$ random vectors $\{\omega_{i,l}^u\}_{i=1}^n$ are independent. Thus, sequence $\left\{\frac{2\gamma}{n} \langle \zeta_l, \omega_{i,l}^u \rangle\right\}_{l,i = 0,1}^{T-1,n}$ is a martingale difference sequence. Next, the summands are bounded:
    \begin{eqnarray}
        \left|\frac{2\gamma}{n}\langle \zeta_l, \omega_{i,l}^u \rangle \right| \leq \frac{2\gamma}{n}  \|\zeta_l\|\cdot \|\omega_{i,l}^u\| 
        \overset{\eqref{eq:gap_thm_prox_clipped_SGDA_technical_8},\eqref{eq:omega_magnitude_prox_clipped_SGDA_gap}}{\leq} \frac{4 \gamma}{n} R \lambda \overset{\eqref{eq:lambda_prox_clipped_SGDA}}{\leq} \frac{V}{5\ln\tfrac{48n(K+1)}{\beta}} \eqdef c. \label{eq:prox_clipped_SGDA_neg_mon_technical_6_1}
    \end{eqnarray}
    Finally, conditional variances $\widehat\sigma_{i,l}^2 \eqdef \EE_{\xi_i^l}\left[\frac{4\gamma^2}{n^2} \langle \zeta_l, \omega_{i,l}^u \rangle^2\right]$ of the summands are bounded:
    
    \begin{eqnarray}
        \widehat\sigma_{i,l}^2 \leq \EE_{\xi^l_i}\left[\frac{4\gamma^2}{n^2} \|\zeta_l\|^2\cdot \|\omega_{i,l}^u\|^2\right] \overset{\eqref{eq:gap_thm_prox_clipped_SGDA_technical_8}}{\leq} \frac{4\gamma^2}{n^2} V \EE_{\xi^l_i}\left[\|\omega_{i,l}^u\|^2\right]. \label{eq:prox_clipped_SGDA_neg_mon_technical_6_2}
    \end{eqnarray}
    Applying Bernstein's inequality (Lemma~\ref{lem:Bernstein_ineq}) with $X_{i,l} = \frac{2\gamma}{n} \langle \zeta_l, \omega_{i,l}^u \rangle$, constant $c$ defined in \eqref{eq:prox_clipped_SGDA_neg_mon_technical_6_1}, $b = \frac{V}{5}$, $G = \tfrac{V^2}{150\ln\frac{48n(K+1)}{\beta}}$, we get
    \begin{equation*}
        \PP\left\{|\circledSeven| > \frac{V}{5} \text{ and } \sum\limits_{l=0}^{T}\sum\limits_{i=1}^n\widehat\sigma_{i,l}^2 \leq \frac{V^2}{150\ln\tfrac{48n(K+1)}{\beta}}\right\} \leq 2\exp\left(- \frac{b^2}{2G + \nicefrac{2cb}{3}}\right) = \frac{\beta}{24n(K+1)}.
    \end{equation*}
    The above is equivalent to 
    \begin{equation}
        \PP\{E_{\circledSeven}\} \geq 1 - \tfrac{\beta}{24n(K+1)} \text{ for }E_{\circledSeven} = \left\{\text{either} \quad \sum\limits_{l=0}^{T}\sum\limits_{i=1}^n\widehat\sigma_{i,l}^2 > \frac{V^2}{150\ln\tfrac{48n(K+1)}{\beta}}\quad \text{or}\quad |\circledSeven| \leq \frac{V}{5}\right\}. \label{eq:bound_6_prox_clipped_SGDA_neg_mon}
    \end{equation}
    Moreover, $E_{T-1}$ implies
    \begin{eqnarray}
        \sum\limits_{l=0}^{T}\sum\limits_{i=1}^n\widehat\sigma_{i,l}^2 &\overset{\eqref{eq:prox_clipped_SGDA_neg_mon_technical_6_2}}{\leq}& \frac{4\gamma^2}{n^2}V\sum\limits_{l=0}^{T} \EE_{\xi^l}\left[\|\omega_l^u\|^2\right] \notag\\ &\overset{\eqref{eq:variance_omega_prox_clipped_SGDA_gap}, T \leq K+1}{\leq}& \frac{72\gamma^2 V \sigma^{\alpha} (K+1) \lambda^{2-\alpha}}{n} \notag\\
        &\overset{\eqref{eq:lambda_prox_clipped_SGDA}}{\leq}& \frac{9 \cdot 20^{\alpha} \sqrt{2}^{\alpha}}{100 \cdot n^{\alpha-1}} \gamma^{\alpha}R^{4-\alpha}\sigma^{\alpha} (K+1)\ln^{\alpha-2}\frac{48n(K+1)}{\beta} \notag\\
        &\overset{\eqref{eq:gamma_prox_clipped_SGDA}}{\leq}& \frac{V^2}{150\ln\tfrac{48n(K+1)}{\beta}} . \label{eq:bound_6_variances_prox_clipped_SGDA_neg_mon}
    \end{eqnarray}
    
    \paragraph{Upper bound for $\circledEight$.} Probability event $E_{T-1}$ implies
    \begin{eqnarray}
        \circledEight &\leq& 2\gamma \sum\limits_{l=0}^{T}\|\zeta_l\| \cdot \|\omega_l^b\| \overset{\eqref{eq:thm_prox_clipped_SGDA_technical_gap_3},\eqref{eq:bias_omega_prox_clipped_SGDA_gap}, T \leq K+1}{\leq} 2\cdot 2^{\alpha}\gamma R(K+1) \frac{\sigma^{\alpha}}{\lambda^{\alpha-1}}\notag\\
        &\overset{\eqref{eq:lambda_prox_clipped_SGDA}}{=}& \frac{40^{\alpha}\sqrt{2}^{\alpha}}{n^{\alpha-1}10\sqrt{2}}\gamma^{\alpha}\sigma^{\alpha}R^{2-\alpha}(K+1)\ln^{\alpha-1}\frac{48n(K+1)}{\beta} \overset{\eqref{eq:gamma_prox_clipped_SGDA}}{\leq} \frac{V}{5}. \label{eq:bound_7_prox_clipped_SGDA_neg_mon}
    \end{eqnarray}

That is, we derive the upper bounds for  $2\gamma \sqrt{V}\left\|\sum_{l=0}^{T-1}\omega_l\right\|, \circledOne, \circledTwo, \circledThree, \circledFour, \circledFive, \circledSix$. More precisely, $E_{T-1}$ implies

    \begin{gather*}
        A_T \overset{\eqref{eq:thm_prox_clipped_SGDA_technical_gap_5}}{\leq} 4V + 2\gamma \sqrt{V}\left\|\sum\limits_{l=0}^{T-1}\omega_l\right\| + \circledOne + \circledTwo + \circledThree + \circledFour + \circledFive + \circledSix, \\
        \circledSix \overset{\eqref{eq:prox_clipped_SGDA_extra_sums_six_distributed}}{=} \circledSix' + \frac{8\gamma}{n} \sum\limits_{j=2}^n \left\langle \frac{\gamma}{n}\sum\limits_{i=1}^{j-1} \omega_{i,T-1}^u - \delta_j^{T-1}, \omega_{j,T-1}^u \right\rangle, \\
        2\gamma\sqrt{V}\left\|\sum\limits_{l = 0}^{T-1}\omega_l\right\| \overset{\eqref{eq:norm_sum_omega_bound_gap_prox_clipped_SGDA}}{\leq} 2\sqrt{V}\sqrt{\frac{1}{4}\left(\circledThree + \circledFour + \circledFive + \circledSix\right) + \circledSeven + \circledEight},\\
        \circledTwo \overset{\eqref{eq:bound_2_prox_clipped_SGDA_neg_mon}}{\leq} \frac{3V}{100},\quad \circledThree \overset{\eqref{eq:bound_3_prox_clipped_SGDA_neg_mon}}{\leq} \frac{3V}{100},\quad \circledFive \overset{\eqref{eq:bound_5_prox_clipped_SGDA_neg_mon}}{\leq} \frac{V}{10},\quad \circledEight \overset{\eqref{eq:bound_7_prox_clipped_SGDA_neg_mon}}{\leq} \frac{V}{5},\\
        \sum\limits_{l=0}^{T-1}\sum\limits_{i=1}^n\sigma_{i,l}^2 \overset{\eqref{eq:bound_1_variances_prox_clipped_SGDA_neg_mon}}{\leq}  \frac{3V^2}{200\ln\tfrac{48n(K+1)}{\beta}},\quad \sum\limits_{l=0}^{T-1}\sum\limits_{i=1}^n\widetilde\sigma_{i,l}^2 \overset{\eqref{eq:bound_4_variances_prox_clipped_SGDA_neg_mon}}{\leq} \frac{V^2}{600\ln\tfrac{48n(K+1)}{\beta}}, \\
        \sum\limits_{l=0}^{T-1}\sum\limits_{i=1}^n\widehat\sigma_{i,l}^2 \overset{\eqref{eq:bound_6_variances_prox_clipped_SGDA_neg_mon}}{\leq} \frac{V^2}{150\ln\tfrac{48n(K+1)}{\beta}}, \quad \sum\limits_{l=0}^{T-1}\sum\limits_{i=1}^n(\sigma_{j,l}')^2  \overset{\eqref{eq:bound_6_variances_gap_prox_clipped_SGDA}}{\leq} \frac{V^2}{150\ln\tfrac{48n(K+1)}{\beta}}.
    \end{gather*}
    In addition, we also establish (see \eqref{eq:bound_1_prox_clipped_SGDA_neg_mon}, \eqref{eq:bound_4_prox_clipped_SGDA_neg_mon}, \eqref{eq:bound_7_prox_clipped_SGDA_neg_mon}, \eqref{eq:bound_7_gap_prox_clipped_SGDA} and our induction assumption)
    \begin{gather*}
        \PP\{E_{T-1}\} \geq 1 - \frac{(T-1)\beta}{K+1},\\
        \PP\{E_{\circledOne}\} \geq 1 - \frac{\beta}{24n(K+1)}, \quad \PP\{E_{\circledFour}\} \geq 1 - \frac{\beta}{24n(K+1)}, \\ \PP\{E_{\circledSix'}\} \geq 1 - \frac{\beta}{24n(K+1)}, \quad \PP\{E_{\circledSeven}\} \geq 1 - \frac{\beta}{24n(K+1)},
    \end{gather*}
    where
    \begin{eqnarray}
        E_{\circledOne}&=& \left\{\text{either} \quad \sum\limits_{l=0}^{T-1}\sum\limits_{i=1}^n\sigma_{i,l}^2 > \frac{3V^2}{200\ln\tfrac{48n(K+1)}{\beta}}\quad \text{or}\quad |\circledOne| \leq \frac{3V}{10}\right\},\notag\\
        E_{\circledFour}&=& \left\{\text{either} \quad \sum\limits_{l=0}^{T-1}\sum\limits_{i=1}^n\widetilde\sigma_{i,l}^2 > \frac{V^2}{600\ln\tfrac{48n(K+1)}{\beta}}\quad \text{or}\quad |\circledFour| \leq \frac{V}{10}\right\},\notag\\
        E_{\circledSix'} &=& \left\{\text{either } \sum\limits_{l=0}^{T-1}\sum\limits_{j=2}^n(\sigma_{i,l}')^2 > \frac{V^2}{150\ln\frac{48n(K+1)}{\beta}} \text{ or } |\circledSix'| \leq \frac{V}{5}\right\},\notag\\
        E_{\circledSeven}&=& \left\{\text{either} \quad \sum\limits_{l=0}^{T-1}\sum\limits_{i=1}^n\widehat\sigma_{i,l}^2 > \frac{V^2}{150\ln\tfrac{48n(K+1)}{\beta}}\quad \text{or}\quad |\circledSeven| \leq \frac{V}{5}\right\}.\notag
    \end{eqnarray}
    Therefore, probability event $E_{T-1} \cap E_{\circledOne} \cap E_{\circledFour} \cap E_{\circledSix'} \cap E_{\circledSeven}$ implies

    \begin{eqnarray}
        \left\|\gamma\sum\limits_{l = 0}^{T-1}\omega_l\right\| &\leq& \sqrt{\frac{1}{4}\left(\frac{3V}{10} + \frac{V}{10} + \frac{V}{10} + \frac{V}{5}\right) + \frac{V}{5} + \frac{V}{5} + \frac{2\gamma}{n} \sum\limits_{j=2}^n \left\langle \frac{\gamma}{n}\sum\limits_{i=1}^{j-1} \omega_{i,T-1}^u - \delta_j^{T-1}, \omega_{j,T-1}^u \right\rangle}\notag\\
        &\leq& \sqrt{V} + \sqrt{\frac{2\gamma}{n} \sum\limits_{j=2}^n \left\langle \frac{\gamma}{n}\sum\limits_{i=1}^{j-1} \omega_{i,T-1}^u - \delta_j^{T-1}, \omega_{j,T-1}^u \right\rangle},\label{eq:gap_thm_prox_clipped_SGDA_technical_10}\\
        A_T &\leq& 4V + 2V + 2\sqrt{V}\sqrt{\frac{2\gamma}{n} \sum\limits_{j=2}^n \left\langle \frac{\gamma}{n}\sum\limits_{i=1}^{j-1} \omega_{i,T-1}^u - \delta_j^{T-1}, \omega_{j,T-1}^u \right\rangle} \notag\\
        &&\quad + \frac{3V}{10} + \frac{3V}{100} + \frac{3V}{100} + \frac{V}{10} + \frac{V}{5} + \frac{V}{5}\notag\\
        &&\quad + \frac{8\gamma}{n} \sum\limits_{j=2}^n \left\langle \frac{\gamma}{n}\sum\limits_{i=1}^{j-1} \omega_{i,T-1}^u - \delta_j^{T-1}, \omega_{j,T-1}^u \right\rangle \notag \\
        &\leq& 8V + 2\sqrt{V}\sqrt{\frac{2\gamma}{n} \sum\limits_{j=2}^n \left\langle \frac{\gamma}{n}\sum\limits_{i=1}^{j-1} \omega_{i,T-1}^u - \delta_j^{T-1}, \omega_{j,T-1}^u \right\rangle}\notag\\
        &&\quad + \frac{8\gamma}{n} \sum\limits_{j=2}^n \left\langle \frac{\gamma}{n}\sum\limits_{i=1}^{j-1} \omega_{i,T-1}^u - \delta_j^{T-1}, \omega_{j,T-1}^u \right\rangle.\label{eq:gap_thm_prox_clipped_SGDA_technical_11}
    \end{eqnarray}
    
    In the final part of the proof, we will show that $\frac{\gamma}{n}\sum\limits_{i=1}^{j-1} \omega_{i,T-1}^u = \delta_j^{T-1}$ with high probability. In particular, we consider probability event $\widetilde{E}_{T-1,j}$ defined as follows: inequalities
    \begin{equation}
       \left\|\frac{\gamma}{n}\sum\limits_{i=1}^{r-1} \omega_{i,T-1}^u\right\| \leq \frac{\sqrt{V}}{2}
    \end{equation}
    hold for $r = 2, \ldots, j$ simultaneously. We want to show that $\PP\{E_{T-1} \cap \widetilde{E}_{T-1,j}\} \geq 1 - \frac{(T-1)\beta}{K+1} - \frac{j\beta}{8n(K+1)}$ for all $j = 2, \ldots, n$. For $j = 2$ the statement is trivial since
    \begin{eqnarray*}
         \left\|\frac{\gamma}{n} \omega_{1,T-1}^u\right\| \overset{\eqref{eq:omega_magnitude_prox_clipped_SGDA_gap}}{\leq} \frac{2\gamma\lambda}{n} \leq \frac{\sqrt{V}}{2}.
    \end{eqnarray*}
    Next, we assume that the statement holds for some $j = m-1 < n$, i.e.,  $\PP\{E_{T-1}\cap\widetilde{E}_{T-1,m-1}\} \geq 1 - \frac{(T-1)\beta}{K+1} - \frac{(m-1)\beta}{8n(K+1)}$. Our goal is to prove that $\PP\{E_{T-1}\cap\widetilde{E}_{T-1,m}\} \geq 1 - \frac{(T-1)\beta}{K+1} - \frac{m\beta}{8n(K+1)}$. First, we consider $\left\|\frac{\gamma}{n}\sum\limits_{i=1}^{m-1} \omega_{i,T-1}^u\right\|$:
    \begin{eqnarray}
        \left\|\frac{\gamma}{n}\sum\limits_{i=1}^{m-1} \omega_{i,T-1}^u\right\| &=& \sqrt{\frac{\gamma^2}{n^2}\left\|\sum\limits_{i=1}^{m-1} \omega_{i,T-1}^u\right\|^2} \notag\\
        &=& \sqrt{\frac{\gamma^2}{n^2}\sum\limits_{i=1}^{m-1}\|\omega_{i,T-1}^u\|^2 + \frac{2\gamma}{n}\sum\limits_{i=1}^{m-1}\left\langle \frac{\gamma}{n}\sum\limits_{r=1}^{i-1}\omega_{r,T-1}^u , \omega_{i,T-1}^u \right\rangle} \notag\\
        &\leq& \sqrt{\frac{\gamma^2}{n^2}\sum\limits_{l=0}^{T-1}\sum\limits_{i=1}^{m-1}\|\omega_{i,l}^u\|^2 + \frac{2\gamma}{n}\sum\limits_{i=1}^{m-1}\left\langle \frac{\gamma}{n}\sum\limits_{r=1}^{i-1}\omega_{r,T-1}^u , \omega_{i,T-1}^u \right\rangle}. \notag
    \end{eqnarray}
    Next, we introduce a new notation:
    \begin{equation*}
        \rho_{i,T-1} = \begin{cases}
            \frac{\gamma}{n}\sum\limits_{r=1}^{i-1}\omega_{r,T-1}^u,& \text{if } \left\|\frac{\gamma}{n}\sum\limits_{r=1}^{i-1}\omega_{r,T-1}^u\right\| \leq \frac{\sqrt{V}}{2},\\
            0,& \text{otherwise}
        \end{cases}
    \end{equation*}
    for $i = 1,\ldots,m-1$. By definition, we have
    \begin{equation}
        \|\rho_{i,T-1}\| \leq \frac{\sqrt{V}}{2} \label{eq:prox_clipped_SGDA_bound_rho}
    \end{equation}
    for $i = 1,\ldots,m-1$. Moreover, $\widetilde{E}_{T-1,m-1}$ implies $\rho_{i,T-1} = \frac{\gamma}{n}\sum\limits_{r=1}^{i-1}\omega_{r,T-1}^u$ for $i = 1,\ldots,m-1$ and
    \begin{eqnarray}
        \left\|\frac{\gamma}{n}\sum\limits_{i=1}^{m-1} \omega_{i,l}^u\right\| &\leq& \sqrt{\circledThree + \circledFour + \circledNine}, \notag
    \end{eqnarray}
    where
    \begin{gather*}
        \circledNine =  \frac{2\gamma}{n}\sum\limits_{i=1}^{m-1}\left\langle \rho_{i,T-1} , \omega_{i,T-1}^u \right\rangle.
    \end{gather*}
    It remains to estimate $\circledNine$.
    
    \textbf{Upper bound for $\circledNine$.} To estimate this sum, we will use Bernstein's inequality. The summands have conditional expectations equal to zero:
    \begin{equation*}
        \EE_{\xi_{i}^{T-1}}\left[\frac{2\gamma}{n} \langle \rho_{i,T-1}, \omega_{i,T-1}^u \rangle\right] = \frac{2\gamma}{n} \left\langle \rho_{i,T-1}, \EE_{\xi_{i}^T-1}[\omega_{i,T-1}^u] \right\rangle = 0
    \end{equation*}
    since random vectors $\{\omega_{i,T-1}^u\}_{i=1}^n$ are independent. Thus, sequence $\left\{\frac{2\gamma}{n} \langle \rho_{i,T-1}, \omega_{i,T-1}^u \rangle\right\}_{i = 1}^{m-1}$ is a martingale difference sequence. Next, the summands are bounded:
    \begin{eqnarray}
        \left|\frac{2\gamma}{n}\langle \rho_{i,T-1}, \omega_{i,T-1}^u \rangle \right| \leq \frac{2\gamma}{n} \|\rho_{i,T-1}\|\cdot \|\omega_{i,T-1}^u\| \overset{\eqref{eq:prox_clipped_SGDA_bound_rho},\eqref{eq:omega_magnitude_prox_clipped_SGDA_gap}}{\leq} \frac{\gamma}{n} \sqrt{V}\lambda \overset{\eqref{eq:lambda_prox_clipped_SGDA}}{\leq} \frac{V}{20\ln\tfrac{48n(K+1)}{\beta}} \eqdef c. \label{eq:gap_thm_prox_clipped_SGDA_technical_6_5_10}
    \end{eqnarray}
    Finally, conditional variances $(\widehat\sigma_{i,T-1}')^2 \eqdef \EE_{\xi_{i}^{T-1}}\left[\frac{4\gamma^2}{n^2} \langle \rho_{i,T-1}, \omega_{i,T-1}^u \rangle^2\right]$ of the summands are bounded:
    \begin{equation}
        (\widehat\sigma_{i,T-1}')^2 \leq \EE_{\xi_{i}^{T-1}}\left[\frac{4\gamma^2 }{n^2}\|\rho_{i,T-1}\|^2\cdot \|\omega_{i,T-1}^u\|^2\right] \overset{\eqref{eq:prox_clipped_SGDA_bound_rho}}{\leq} \frac{\gamma^2 V}{n^2} \EE_{\xi_{i}^{T-1}}\left[\|\omega_{i,T-1}^u\|^2\right]. \label{eq:gap_thm_prox_clipped_SGDA_technical_7_10}
    \end{equation}
    Applying Bernstein's inequality (Lemma~\ref{lem:Bernstein_ineq}) with $X_{T-1,i} = \frac{2\gamma}{n} \langle \rho_{i,T-1}, \omega_{i,T-1}^u \rangle$, constant $c$ defined in \eqref{eq:gap_thm_prox_clipped_SGDA_technical_6_5_10}, $b = \frac{V}{20}$, $G = \frac{V^2}{2400\ln\frac{48n(K+1)}{\beta}}$, we get
    \begin{equation*}
        \PP\left\{|\circledNine| > \frac{V}{20} \text{ and } \sum\limits_{i=1}^{m-1}(\widehat\sigma_{i,T-1}')^2 \leq \frac{V^2}{2400\ln\frac{48n(K+1)}{\beta}}\right\} \leq 2\exp\left(- \frac{b^2}{2G + \nicefrac{2cb}{3}}\right) = \frac{\beta}{24n(K+1)}.
    \end{equation*}
    The above is equivalent to 
    \begin{equation}
        \PP\{E_{\circledNine}\} \geq 1 - \frac{\beta}{24n(K+1)},\; \text{ for }\; E_{\circledNine} = \left\{\text{either } \sum\limits_{i=1}^{m-1}(\widehat\sigma_{i,T-1}')^2 > \frac{V^2}{2400\ln\frac{48n(K+1)}{\beta}} \text{ or } |\circledNine| \leq \frac{V}{20}\right\}. \label{eq:bound_10_gap_prox_clipped_SGDA}
    \end{equation}
    Moreover, $E_{T-1}$ implies
    \begin{eqnarray}
        \sum\limits_{i=1}^{m-1}(\widehat\sigma_{i,T-1}')^2 &\overset{\eqref{eq:gap_thm_prox_clipped_SGDA_technical_7_10}}{\leq}& \frac{\gamma^2 V}{n^2} \sum\limits_{i=1}^n \EE_{\xi_{i}^l}\left[\|\omega_{i,l}^u\|^2\right] \overset{\eqref{eq:variance_omega_prox_clipped_SGDA_gap}}{\leq} \frac{18\gamma^2 V \lambda^{2-\alpha} \sigma^\alpha}{n} \notag \\
        &\overset{\eqref{eq:lambda_prox_clipped_SGDA}}{\leq}& \frac{18 \gamma^{\alpha}\sigma^\alpha V^{2-\frac{\alpha}{2}}}{40^{2-\alpha} n^{\alpha-1} \ln^{2-\alpha}\frac{48n(K+1)}{\beta}} \overset{\eqref{eq:gamma_prox_clipped_SGDA}}{\leq} \frac{V^2}{2400\ln\tfrac{48n(K+1)}{\beta}}. \label{eq:bound_10_variances_gap_prox_clipped_SGDA}
    \end{eqnarray}
    Putting all together we get that $E_{T-1}\cap \widetilde{E}_{T-1, m-1}$ implies
    \begin{gather*}
        \left\|\frac{\gamma}{n}\sum\limits_{i=1}^{m-1} \omega_{i,T-1}^u\right\| \leq \sqrt{\circledThree + \circledFour + \circledNine},\quad \circledThree \overset{\eqref{eq:bound_3_prox_clipped_SGDA_neg_mon}}{\leq} \frac{V}{10},\\
        \sum\limits_{l=0}^{T-1}\sum\limits_{i=1}^n\widetilde\sigma_{i,l}^2 \overset{\eqref{eq:bound_4_variances_prox_clipped_SGDA_neg_mon}}{\leq} \frac{V^2}{600\ln\frac{48n(K+1)}{\beta}},\quad \sum\limits_{i=1}^{m-1}(\widehat\sigma_{i,T-1}')^2 \leq \frac{V^2}{2400\ln\tfrac{48n(K+1)}{\beta}}
    \end{gather*}
    In addition, we also establish (see \eqref{eq:bound_4_prox_clipped_SGDA_neg_mon}, \eqref{eq:bound_10_gap_prox_clipped_SGDA} and our induction assumption)
    \begin{gather*}
        \PP\{E_{T-1}\cap \widetilde{E}_{T-1,m-1}\} \geq 1 - \frac{(T-1)\beta}{K+1} - \frac{(m-1)\beta}{8n(K+1)},\\
        \PP\{E_{\circledFour}\} \geq 1 - \frac{\beta}{24n(K+1)}, \quad \PP\{E_{\circledNine}\} \geq 1 - \frac{\beta}{24n(K+1)},
    \end{gather*}
    where
    \begin{eqnarray}
        E_{\circledFour} &=& \left\{\text{either } \sum\limits_{l=0}^{T-1}\sum\limits_{i=1}^n\widetilde\sigma_{i,l}^2 > \frac{V^2}{600\ln\frac{48n(K+1)}{\beta}} \text{ or } |\circledFour| \leq \frac{V}{10}\right\},\notag\\
        E_{\circledNine} &=& \left\{\text{either } \sum\limits_{i=1}^{m-1}(\widehat\sigma_{i,T-1}')^2 > \frac{V^2}{2400\ln\frac{48n(K+1)}{\beta}} \text{ or } |\circledTen| \leq \frac{V}{20}\right\}. \notag
    \end{eqnarray}
    Therefore, probability event $E_{T-1} \cap \widetilde{E}_{T-1, m-1} \cap E_{\circledFour} \cap E_{\circledNine}$ implies
    \begin{equation*}
        \left\|\frac{\gamma}{n}\sum\limits_{i=1}^{m-1} \omega_{i,l}^u\right\| \leq \sqrt{\frac{V}{10} + \frac{V}{10} + \frac{V}{20}} = \frac{\sqrt{V}}{2}.
    \end{equation*} 
    This implies $\widetilde E_{T-1,m}$ and
    \begin{eqnarray*}
        \PP\{E_{T-1} \cap \widetilde{E}_{T-1,m}\} &\geq& \PP\{E_{T-1} \cap \widetilde{E}_{T-1,m-1} \cap E_{\circledFour} \cap E_{\circledNine} \} \\
        &=& 1 - \PP\left\{\overline{E_{T-1} \cap \widetilde{E}_{T-1,m-1}} \cup \overline{E}_{\circledFour} \cup \overline{E}_{\circledNine} \right\}\\
        &\geq& 1 - \frac{(T-1)\beta}{K+1} - \frac{m\beta}{8n(K+1)}.
    \end{eqnarray*}
    Therefore, for all $m = 2,\ldots,n$ the statement holds and, in particular, $\PP\{E_{T-1} \cap \widetilde{E}_{T-1,n}\} \geq 1 - \frac{(T-1)\beta}{K+1} - \frac{\beta}{8(K+1)}$. Taking into account  \eqref{eq:gap_thm_prox_clipped_SGDA_technical_10} and \eqref{eq:gap_thm_prox_clipped_SGDA_technical_11}, we conclude that $E_{T-1} \cap \widetilde{E}_{T-1,n}$ implies  
    \begin{equation*}
       \left\|\gamma\sum\limits_{l = 0}^{T-1}\omega_l\right\| \leq \sqrt{V},\quad  A_T \leq 8V,
    \end{equation*}
    which is equivalent to \eqref{eq:induction_inequality_1_prox_clipped_SGDA} and \eqref{eq:induction_inequality_prox_clipped_SGDA_gap} for $t = T$. Moreover,
    \begin{eqnarray*}
        \PP\left\{E_T\right\} &\geq& \PP\left\{E_{T-1} \cap \widetilde{E}_{T-1,n} \cap E_{\circledOne} \cap E_{\circledFour} \cap E_{\circledSix'} \cap E_{\circledSeven} \cap E_{\circledNine}\right\} \\
        &=& 1 - \PP\left\{\overline{E_{T-1} \cap \widetilde{E}_{T-1,n}} \cup \overline{E}_{\circledOne} \cup \overline{E}_{\circledFour} \cup \overline{E}_{\circledSix'} \cup \overline{E}_{\circledSeven} \cup \overline{E}_{\circledNine} \right\}\\
        &=& 1 - \frac{(T-1)\beta}{K+1} - \frac{\beta}{8(K+1)} - 5\cdot \frac{\beta}{8(K+1)} \geq 1 - \frac{T\beta}{K+1}.
    \end{eqnarray*}

    In other words, we showed that $\PP\{E_k\} \geq 1 - \nicefrac{k\beta}{(K+1)}$ for all $k = 0,1,\ldots,K+1$. For $k = K+1$ we have that with probability at least $1 - \beta$

    \begin{eqnarray*}
        \gap_{\sqrt{V}}(x_{\avg}^{K+1})
        &\overset{\eqref{eq:gap_thm_prox_clipped_SGDA_technical_1_5}}{\leq}& \frac{4V}{\gamma(K+1)}.
    \end{eqnarray*}

    Finally, if 
    \begin{equation*}
        \gamma = \min\left\{\frac{1}{480\ell \ln \tfrac{48n(K+1)}{\beta}}, \left(\frac{1}{86400}\right)^{\frac{1}{\alpha}}\cdot\frac{\sqrt{V}n^{\frac{\alpha-1}{\alpha}}}{(K+1)^{\frac{1}{\alpha}}\sigma \ln^{\frac{\alpha-1}{\alpha}} \tfrac{48n(K+1)}{\beta}}\right\}
    \end{equation*}
    then with probability at least $1-\beta$
    \begin{eqnarray*}
        \gap_{\sqrt{V}}(x^{K+1}_{\avg}) &\leq& \frac{4V}{\gamma(K+1)} = \max\left\{\frac{480 \ell V \ln \tfrac{48n(K+1)}{\beta}}{K+1}, \left(\frac{86400}{1}\right)^{\frac{1}{\alpha}}\cdot\frac{4 \sigma \sqrt{V} \ln^{\frac{\alpha-1}{\alpha}} \tfrac{48n(K+1)}{\beta}}{n^{\frac{\alpha-1}{\alpha}}(K+1)^{\frac{\alpha-1}{\alpha}}}\right\}\\
        &=& \cO\left(\max\left\{\frac{\ell \sqrt{V}\ln\frac{nK}{\beta}}{K}, \frac{\sigma \sqrt{V} \ln^{\frac{\alpha-1}{\alpha}}\frac{K}{\beta}}{n^{\frac{\alpha-1}{\alpha}}K^{\frac{\alpha-1}{\alpha}}}\right\}\right).
    \end{eqnarray*}
    To get $\gap_{R}(x^{K+1}_{\avg}) \leq \varepsilon$ with probability at least $1-\beta$ it is sufficient to choose $K$ such that both terms in the maximum above are $\cO(\varepsilon)$. This leads to
    \begin{equation*}
         K = \cO\left(\frac{\ell V}{\varepsilon}\ln\frac{n\ell V}{\varepsilon\beta}, \frac{1}{n}\left(\frac{\sigma \sqrt{V}}{\varepsilon}\right)^{\frac{\alpha}{\alpha-1}}\ln\left(\frac{1}{\beta}\left(\frac{\sigma \sqrt{V}}{\varepsilon}\right)^{\frac{\alpha}{\alpha-1}}\right)\right)
    \end{equation*}
    that concludes the proof.
\end{proof}

\subsection{Quasi-Strongly Monotone case}

\begin{lemma}\label{lem:optimization_lemma_str_mon_SGDA_2}
    Let Assumptions~\ref{as:str_monotonicity}, \ref{as:star_cocoercivity} hold for $Q = B_{\sqrt{2V}}(x^*)$, where $V \geq \|x^0 - x^*\| + \frac{9000000\gamma^2\ln^2\left(\frac{48n(K+1)}{\beta}\right)}{n^2}\revision{\sum\limits_{i=1}^n\|F_i(x^*)\|^2}$, and $0 < \gamma \leq \frac{1}{\ell + \nicefrac{18000000 \nu \ln^2\left(\frac{48n(K+1)}{\beta}\right)\ell}{n}}, \; \nu \leq \frac{1}{18000000 \ln^2 \left(\frac{48n(K+1)}{\beta}\right)}$. If $x^k$ lies in $B_{\sqrt{2V}}(x^*)$ for all $k = 0,1,\ldots, K$ for some $K\geq 0$, then the iterates produced by \algname{DProx-clipped-SGDA-shift} satisfy
    \begin{eqnarray}
        V_{K+1} &\leq& (1 - \gamma \mu)^{K+1}V_0 + \frac{2\gamma}{n} \sum\limits_{k=0}^K\sum\limits_{i=1}^n (1-\gamma\mu)^{K-k}\langle x^k - x^* - \gamma (F(x^k)-h^*), \omega_{i,k} \rangle\notag\\
        &&\quad  + \frac{ \gamma^2  }{n^2}\sum\limits_{k=0}^K\sum\limits_{i=1}^n (1-\gamma \mu)^{K-k} \|\omega_{i,k}\|^2  + \gamma^2 \sum\limits_{k=0}^K (1-\gamma\mu)^{K-k} \|\omega_k\|^2, \label{eq:optimization_lemma_prox_SGDA_str_mon_2}
    \end{eqnarray}
    where $V_k = \|x^k - x^*\|^2 + \frac{9000000\gamma^2\ln^2\left(\frac{48n(K+1)}{\beta}\right)}{n^2}\sum\limits_{i=1}^n\|h^k_i - h^*_i\|^2$, $h_i^* = F_i(x^*)$, and $\omega_k, \omega_k^u, \omega_k^b, \omega_{k,i}^u, \omega_{k,i}^b$ are defined in \eqref{eq:gap_thm_prox_clipped_SGDA_technical_4_full_theta_omega}-\eqref{eq:thm_prox_clipped_SGDA_technical_gap_4}.
\end{lemma}

\begin{proof}
     Using the update rule of \algname{DProx-clipped-SGDA-shift} and
    $\omega_k = F(x^k) - \tg^k$ we obtain
     \begin{eqnarray*}
        \|x^{k+1} - x^*\|^2 &=& \|\proxkPsi\left(x^k - \gamma\tg^k\right) - \proxkPsi\left(x^* - \gamma h^*\right)\|^2\\
        &\le& \|x^k - x^* - \gamma(\tg^k - h^*)\|^2\\
        &=& \|x^k - x^*\|^2 -2\gamma \langle x^k - x^*, \tg^k - h^* \rangle + \gamma^2\|\tg^k - h^*\|^2\\
        &=& \|x^k - x^*\|^2 - 2\gamma \langle x^k - x^*, F(x^k)-h^* \rangle - 2\gamma^2\langle F(x^k) - h^*, \omega_k\rangle\\
        &&\quad  + 2\gamma \langle x^k - x^*, \omega_k \rangle  + \gamma^2\|F(x^k) - h^*\|^2 + \gamma^2\|\omega_k\|^2,
    \end{eqnarray*}
    Next, let us recall that
    \begin{gather*}
    h_i^{k+1} = h_i^k + \nu \hat \Delta_i^k,\quad \hat\Delta_i^k = \clip\left(F_{\xi_i^k}(x^k) - h_i^k, \lambda_k\right),\quad \tg_i^k = h_i^k + \hat\Delta_i^k,\quad \omega_{i,k} = F_i(x^k) - \tg_i^k.
\end{gather*}
    Then, $\forall i \in [n]$ we have
    \begin{eqnarray*}
        \|h^{k+1}_i - h^*_i\|^2 &=& \|h^k_i - h^*_i + \nu \hat\Delta^k_i\|^2 = \|h^k_i - h^*_i\|^2 + 2\nu \langle h^k_i - h^*_i, \hat\Delta^k_i\rangle + \nu^2\|\hat\Delta^k_i\|^2\\
        &=& \|h^k_i - h^*_i\|^2 + 2\nu \langle h^k_i - h^*_i, \tg^k_i - h_i^k\rangle + \nu^2\|\tg^k_i - h_i^k\|^2\\
        &\overset{\nu \leq \frac{1}{2}}{\leq}& \|h^k_i - h^*_i\|^2 + 2\nu \langle h^k_i - h^*_i, \tg^k_i - h_i^k\rangle + \nu\|\tg^k_i - h_i^k\|^2\\
        &=& \|h^k_i - h^*_i\|^2 + \nu \langle \tg^k_i - h_i^k, \tg^k_i + h^k_i - 2h^*_i \rangle \\
        &=& (1-\nu)\|h^k_i - h^*_i\|^2 + \nu\|\tg^k_i - h_i^*\|^2\\
        &\le&  (1-\nu)\|h^k_i - h^*_i\|^2 + 2\nu\|\tg^k_i - F_i(x^k)\|^2 + 2\nu\|F_i(x^k) - h_i^*\|^2\\
        &=& (1-\nu)\|h^k_i - h^*_i\|^2 + 2\nu\|\omega_{i,k}\|^2 + 2\nu \|F_i(x^k) - h^*_i\|^2.
    \end{eqnarray*}
    Let us consider the following stepsize condition
    \begin{eqnarray}
        0 < \gamma \leq \frac{1}{\ell + \frac{18000000 \nu \ln^2\left(\frac{48n(K+1)}{\beta}\right)\ell}{n}}.\label{eq:gamma_lemma_prox_SGDA_str_mon_2}
    \end{eqnarray}
    Lyapunov function
    \begin{eqnarray*}
        V_k = \|x^k - x^*\|^2 + \frac{9000000\gamma^2\ln^2\left(\frac{48n(K+1)}{\beta}\right)}{n^2}\sum\limits_{i=1}^n\|h^k_i - h^*_i\|^2.
    \end{eqnarray*}
    \begin{eqnarray*}
        V_{k+1} &\le& \|x^k - x^*\|^2 - 2\gamma \langle x^k - x^*, F(x^k)-h^* \rangle - 2\gamma^2\langle F(x^k) - h^*, \omega_k\rangle\\
        &&\quad  + 2\gamma \langle x^k - x^*, \omega_k \rangle  + \gamma^2\|F(x^k) - h^*\|^2 + \gamma^2\|\omega_k\|^2\\ 
        &&\quad + \frac{9\cdot 10^6\gamma^2\ln^2\left(\frac{48n(K+1)}{\beta}\right)}{n^2} \sum\limits_{i=1}^n \Bigr[ (1-\nu)\|h^k_i - h^*_i\|^2 + 2\nu \|\omega_{i,k}\|^2 + 2\nu \|F_i(x^k) - h^*_i\|^2\Bigl]\\
         &\overset{\eqref{eq:star_cocoercivity}}{\le}& \|x^k-x^*\|^2 + (1-\nu)\frac{9\cdot 10^6\gamma^2\ln^2\left(\frac{48n(K+1)}{\beta}\right)}{n^2}\sum\limits_{i=1}^n\|h^k_i - h^*_i\|^2 \\
         &&\quad -2 \gamma \left(1 - \frac{\gamma \ell}{2} - \frac{\nu n}{\gamma}\cdot \frac{9\cdot 10^6\gamma^2\ln^2\left(\frac{48n(K+1)}{\beta}\right)}{n^2} \ell_{\max}\right)\langle x^k - x^*, F(x^k)-h^* \rangle \\
         &&\quad + \frac{2\gamma}{n}\sum\limits_{i=1}^n\langle x^k - x^* - \gamma(F(x^k) - h^*), \omega_{i,k} \rangle + \gamma^2\|\omega_k\|^2 + \frac{\gamma^2}{n^2}\sum\limits_{i=1}^n\|\omega_{i,k}\|^2\\
          &\overset{ \eqref{eq:gamma_lemma_prox_SGDA_str_mon_2}}{\le}& \|x^k-x^*\|^2 + (1-\nu)\frac{9\cdot 10^6\gamma^2\ln^2\left(\frac{48n(K+1)}{\beta}\right)}{n^2}\sum\limits_{i=1}^n\|h^k_i - h^*_i\|^2\\
         &&\quad - \gamma \langle x^k - x^*, F(x^k)-h^* \rangle \\
         &&\quad + \frac{2\gamma}{n}\sum\limits_{i=1}^n\langle x^k - x^* - \gamma(F(x^k) - h^*), \omega_{i,k} \rangle + \gamma^2\|\omega_k\|^2 + \frac{\gamma^2}{n^2}\sum\limits_{i=1}^n\|\omega_{i,k}\|^2\\
         &\overset{\eqref{eq:str_monotonicity}}{\le}& (1-\gamma \mu) \|x^k - x^*\|^2 + (1-\nu)\frac{9\cdot 10^6\gamma^2\ln^2\left(\frac{48n(K+1)}{\beta}\right)}{n^2}\sum\limits_{i=1}^n\|h^k_i - h^*_i\|^2\\
         &&\quad + \frac{2\gamma}{n}\sum\limits_{i=1}^n\langle x^k - x^* - \gamma(F(x^k) - h^*), \omega_{i,k} \rangle + \gamma^2\|\omega_k\|^2 + \frac{\gamma^2}{n^2}\sum\limits_{i=1}^n\|\omega_{i,k}\|^2\\
         &\overset{ \gamma \leq \frac{\nu}{\mu}}{\le}& (1 - \gamma \mu) V_k + \frac{2\gamma}{n}\sum\limits_{i=1}^n\langle x^k - x^* - \gamma(F(x^k) - h^*), \omega_{i,k} \rangle\\
         &&\quad + \gamma^2\|\omega_k\|^2 + \frac{\gamma^2}{n^2}\sum\limits_{i=1}^n\|\omega_{i,k}\|^2.
    \end{eqnarray*}
    Unrolling the recurrence, we obtain \eqref{eq:optimization_lemma_prox_SGDA_str_mon}.
\end{proof}

\begin{theorem}\label{thm:main_result_prox_clipped_SGDA_2}
Let Assumptions~\ref{as:str_monotonicity}, \ref{as:star_cocoercivity}, hold for $Q = B_{\sqrt{2V}}(x^*)$, 
where $V \geq \|x^0 - x^*\| + \frac{9000000\gamma^2\ln^2\left(\frac{48n(K+1)}{\beta}\right)}{n^2}\revision{\sum\limits_{i=1}^n\|F_i(x^*)\|^2}$, and \revision{$R \geq \|x^0 - x^*\|$,}
    \begin{eqnarray}
        0< \gamma &\leq& \min\left\{\frac{1}{4096 \ell \ln \tfrac{48n(K+1)}{\beta}}, \revision{\frac{\sqrt{n}R}{3000 \zeta_* \ln \tfrac{48n(K+1)}{\beta}},} \frac{\ln(B_K)}{\mu(K+1)}\right\}, \label{eq:gamma_SGDA_2_prox_clipped}\\
        B_K &=& \max\left\{2, \left(\frac{\sqrt{2}}{3456}\right)^{\frac{2}{\alpha}}\cdot\frac{(K+1)^{\frac{2(\alpha-1)}{\alpha}}\mu^2V n^{\frac{2(\alpha-1)}{\alpha}}}{\sigma^2\ln^{\frac{2(\alpha-1)}{\alpha}}\left(\frac{48n(K+1)}{\beta}\right)\ln^2(B_K)} \right\} \label{eq:B_K_SGDA_2_prox_clipped_1} \\
       &=& \cO\left(\max\left\{2, \frac{K^{\frac{2(\alpha-1)}{\alpha}}\mu^2V n^{\frac{2(\alpha-1)}{\alpha}}}{\sigma^2\ln^{\frac{2(\alpha-1)}{\alpha}}\left(\frac{nK}{\beta}\right)\ln^2\left(\max\left\{2, \frac{K^{\frac{2(\alpha-1)}{\alpha}}\mu^2V n^{\frac{2(\alpha-1)}{\alpha}}}{\sigma^2\ln^{\frac{2(\alpha-1)}{\alpha}}\left(\frac{nK}{\beta}\right)} \right\}\right)} \right\}\right), \label{eq:B_K_SGDA_2_prox_clipped_2} \\
        \lambda_k &=& \frac{n\cdot\exp(-\gamma\mu(1 + \nicefrac{k}{2}))\sqrt{V}}{256\sqrt{2}\gamma \ln \tfrac{48n(K+1)}{\beta}}, \label{eq:lambda_SGDA_2_prox_clipped}
    \end{eqnarray}
    for some $K \geq 0$ and $\beta \in (0,1]$. Then, after $K$ iterations the iterates produced by \algname{DProx-clipped-SGDA-shift} with probability at least $1 - \beta$ satisfy 
    \begin{equation}
        V_{K+1} \leq 2\exp(-\gamma\mu(K+1))V, \label{eq:main_result_str_mon_SGDA_2}
    \end{equation}
    \revision{where $V_k = \|x^k - x^*\|^2 + \frac{9000000\gamma^2\ln^2\left(\frac{48n(K+1)}{\beta}\right)}{n^2}\sum\limits_{i=1}^n\|h^k_i - h^*_i\|^2$, $h_i^* = F_i(x^*)$.} 
    In particular, \revision{$V \leq 2R^2$, and} when $\gamma$ equals the minimum from \eqref{eq:gamma_SGDA_2_prox_clipped}, then the iterates produced by \algname{Dprox-clipped-SGDA-shift} after $K$ iterations with probability at least $1-\beta$ satisfy
    \begin{equation}
       V_K = \cO\left(\max\left\{\revision{R^2}\exp\left(- \frac{\mu K}{\ell \ln \tfrac{nK}{\beta}}\right), \revision{\revision{R^2}\exp\left(- \frac{\mu \sqrt{n} RK}{\zeta_* \ln \tfrac{nK}{\beta}}\right),} \frac{\sigma^2\ln^{\frac{2(\alpha-1)}{\alpha}}\left(\frac{nK}{\beta}\right)\ln^2\revision{B_K}}{K^{\frac{2(\alpha-1)}{\alpha}}\mu^2 n^{\frac{2(\alpha-1)}{\alpha}}}\right\}\right), \label{eq:clipped_SGDA_2_prox_clipped_case_1}
    \end{equation}
    meaning that to achieve $V_K \leq \varepsilon$ with probability at least $1 - \beta$ \algname{DProx-clipped-SGDA-shift} requires
    \begin{align}
        K = \cO\Bigg(\max\Bigg\{\frac{\ell}{\mu}\ln\left(\frac{\revision{R^2}}{\varepsilon}\right)&\ln\left(\frac{n\ell}{\mu \beta}\ln\frac{\revision{R^2}}{\varepsilon}\right), \revision{\frac{\zeta_*}{\sqrt{n}R\mu}\ln\left(\frac{\revision{R^2}}{\varepsilon}\right)\ln\left(\frac{\sqrt{n}\zeta_*}{R\mu \beta}\ln\frac{\revision{R^2}}{\varepsilon}\right),}\notag\\
        &\frac{1}{n}\left(\frac{\sigma^2}{\mu^2\varepsilon}\right)^{\frac{\alpha}{2(\alpha-1)}}\ln \left(\frac{1}{\beta} \left(\frac{\sigma^2}{\mu^2\varepsilon }\right)^{\frac{\alpha}{2(\alpha-1)}}\right)\ln^{\frac{\alpha}{\alpha-1}}\left(B_\varepsilon\right)\Bigg\}\Bigg) \label{eq:clipped_SGDA_2_prox_clipped_case_complexity_appendix}
    \end{align}
    iterations/oracle calls, where
    \begin{equation*}
         B_\varepsilon = \max\left\{2, \frac{\revision{2R^2}}{\varepsilon \ln \left(\frac{1}{\beta} \left(\frac{\sigma^2}{\mu^2\varepsilon }\right)^{\frac{\alpha}{2(\alpha-1)}}\right)}\right\}.
    \end{equation*}
\end{theorem}
\begin{proof}
The Lyapunov function has the following form
    \begin{eqnarray*}
        V_k = \|x^k - x^*\|^2 + \frac{9000000\gamma^2\ln^2\left(\frac{48n(K+1)}{\beta}\right)}{n^2}\sum\limits_{i=1}^n\|h^k_i - h^*_i\|^2.
    \end{eqnarray*} Similar to previous results, our proof is induction-based.  To formulate the statement rigorously, we introduce probability event $E_k$ for each $k = 0,1,\ldots, K+1$ as follows: inequalities
    \begin{equation}
        V_t \leq 2 \exp(-\gamma\mu t) V \label{eq:induction_inequality_prox_clipped_SGDA_2}
    \end{equation}
    \begin{gather}
         \left\|\frac{\gamma}{n}\sum\limits_{i=1}^{r-1}\omega_{i,t-1}^u\right\| \leq \exp\left(-\frac{\gamma\mu (t-1)}{2}\right)\frac{\sqrt{V}}{2} \label{eq:induction_inequality_str_mon_SGDA_2}
    \end{gather}
    hold for $t = 0,1,\ldots,k$ and $r = 1, 2, \ldots, n$ simultaneously. We will prove by induction that $\PP\{E_k\} \geq  1 - \nicefrac{k\beta}{(K+1)}$ for all $k = 0,1,\ldots,K+1$. The base of induction follows immediately by the definition of $V$. Next, we assume that the statement holds for $k = T-1 \leq K$, i.e., $\PP\{E_{T-1}\} \geq  1 - \nicefrac{(T-1)\beta}{(K+1)}$. Let us show that it also holds for $k = T$, i.e., $\PP\{E_{T}\} \geq  1 - \nicefrac{T\beta}{(K+1)}$.

    Similarly to the monotone case, one can show that due to our choice of the clipping level, we have that $E_{T-1}$ implies $x^t \in B_{\sqrt{2}n\sqrt{V}}(x^*)$ for $t=0,\ldots,T-1$. Indeed, for $t = 0,1,\ldots,T-1$ inequality \eqref{eq:induction_inequality_prox_clipped_SGDA_2} gives $x^t \in B_{\sqrt{2V}}(x^*)$. This means that we can apply Lemma~\ref{lem:optimization_lemma_str_mon_SGDA_2}: $E_{T-1}$ implies

    \begin{eqnarray*}
        V_{T} &\leq& (1 - \gamma \mu)^{T}V + \frac{2\gamma}{n} \sum\limits_{t=0}^{T-1}\sum\limits_{i=1}^n (1-\gamma\mu)^{T-1-t}\langle x^t - x^* - \gamma (F(x^t)-h^*), \omega_{i,t} \rangle\notag\\
        &&\quad  + \frac{\gamma^2}{n^2}\sum\limits_{t=0}^{T-1}\sum\limits_{i=1}^n (1-\gamma \mu)^{T-1-t} \|\omega_{i,t}\|^2 + \gamma^2 \sum\limits_{t=0}^{T-1} (1-\gamma\mu)^{T-1-t} \|\omega_t\|^2.
    \end{eqnarray*}
Before we proceed, we introduce a new notation:
\begin{gather}
        \xi_t = \begin{cases} x^t - x^* - \gamma (F(x^t)-h^*),& \text{if } \|x^t - x^* - \gamma (F(x^t)-h^*)\| \leq 2\sqrt{2} \exp(- \nicefrac{\gamma\mu t}{2})\sqrt{V},\\ 0,& \text{otherwise}, \end{cases} \label{eq:xi_t_SGDA_2_str_mon}
\end{gather}
for $t = 0, 1, \ldots, T$. Random vectors $\{\xi_t\}_{t=0}^T$ are bounded almost surely:
\begin{equation}
         \|\xi_t\| \leq 2\sqrt{2}\exp(-\nicefrac{\gamma\mu t}{2})\sqrt{V} \label{eq:xi_t_bound_SGDA_2_prox_clipped} 
\end{equation}
for all $t = 0, 1, \ldots, T$. 
    In addition, $\xi_t = x^t - x^* - \gamma (F(x^t)-h^*)$ follows from $E_{T-1}$ for all $t = 0, 1, \ldots, T$ and, thus, $E_{T-1}$ implies
\begin{eqnarray}
        V_{T} &\leq& \exp(-\gamma\mu T)V + \underbrace{\frac{2\gamma}{n} \sum\limits_{t=0}^{T-1}\sum\limits_{i=1}^n (1-\gamma\mu)^{T-1-t}\langle \xi_t, \omega^u_{i,t} \rangle}_{\circledOne} + \underbrace{\frac{2\gamma}{n} \sum\limits_{t=0}^{T-1}\sum\limits_{i=1}^n (1-\gamma\mu)^{T-1-t}\langle \xi_t, \omega^b_{i,t} \rangle}_{\circledTwo}\notag\\
        &&\quad  + \underbrace{\frac{4\gamma^2}{n^2}  \sum\limits_{t=0}^{T-1}\sum\limits_{i=1}^n (1-\gamma \mu)^{T-1-t} \left[\|\omega^u_{i,t}\|^2 - \EE_{\xi^t}[\|\omega^u_{i,t}\|^2] \right]}_{\circledThree}\notag\\
        &&\quad  + \underbrace{\frac{4\gamma^2}{n^2} \sum\limits_{t=0}^{T-1}\sum\limits_{i=1}^n (1-\gamma \mu)^{T-1-t} \EE_{\xi^t}[\|\omega^u_{i,t}\|^2]}_{\circledFour} \notag\\
        &&\quad  + \underbrace{\frac{4\gamma^2}{n} \sum\limits_{t=0}^{T-1}\sum\limits_{i=1}^n (1-\gamma \mu)^{T-1-t} \|\omega^b_{i,t}\|^2}_{\circledFive}\notag\\
        &&\quad + \underbrace{\frac{4\gamma^2}{n^2} \sum\limits_{t=0}^{T-1}\sum\limits_{j=1}^n (1-\gamma\mu)^{T-1-t} \left\langle \sum\limits_{i=1}^{j-1} \omega_{i,t}^u, \omega_{j,t}^u\right\rangle}_{\circledSix}. \label{eq:SGDA_2_prox_clipped_12345_bound}
\end{eqnarray}
To derive high-probability bounds for $\circledOne, \circledTwo, \circledThree, \circledFour, \circledFive, \circledSix$ we need to establish several useful inequalities related to $\omega_{i,t}^u, \omega_{i,t}^b$. First, by definition of clipping
    \begin{equation}
         \|\omega_{i,t}^u\| \leq 2\lambda_t. \label{eq:omega_magnitude_prox_clipped_2}
    \end{equation}
Next, we notice that $E_{T-1}$ implies
\begin{eqnarray}
        \|F_i(x^t) - h_i^t\| &\le& \|F_i(x^t) - h^*_i\| + \|h^t_i- h^*_i\| \overset{\eqref{eq:star_cocoercivity}}{\le} \ell\|x^t - x^*\| + \sqrt{\sum\limits_{i=1}^n\|h^t_i -  h^*_i\|^2}\notag\\
        &\le& \left( \ell + \frac{n}{3000\gamma\ln\left(\frac{48n(K+1)}{\beta}\right)}\right)\sqrt{V_t}\notag\\ &\overset{\eqref{eq:induction_inequality_prox_clipped_SGDA_2}}{\leq}& \sqrt{2}\left( \ell + \frac{n}{3000\gamma\ln\left(\frac{48n(K+1)}{\beta}\right)}\right)\exp(- \nicefrac{\gamma\mu t}{2})\sqrt{V}
        \overset{\eqref{eq:gamma_SGDA_2_prox_clipped},\eqref{eq:lambda_SGDA_2_prox_clipped}}{\leq} \frac{\lambda_t}{2}. \label{eq:operator_bound_x_t_SGDA_2_prox_clipped}
    \end{eqnarray}
for $t = 0,1,\ldots,T-1$ and $i\in [n]$. Therefore, one can apply Lemma~\ref{lem:bias_and_variance_clip} and get \begin{gather}
         \left\|\omega_t^b\right\| \leq \frac{1}{n}\sum\limits_{i=1}^n \|\omega_{i,t}^b\| \leq \frac{2^\alpha\sigma^\alpha}{\lambda_t^{\alpha-1}}, \label{eq:bias_omega_prox_clipped_2}\\
         \EE_{\xi_{i}^t}\left[\left\|\omega_{i,t}^u\right\|^2\right] \leq 18 \lambda_t^{2-\alpha}\sigma^\alpha, \label{eq:variance_omega_prox_clipped_2}
    \end{gather}
    for all $t = 0,1, \ldots, T-1$ and $i\in [n]$. In addition, we require the following condition
    \begin{eqnarray}
        \nu \leq \frac{1}{18000000 \ln^2 \left(\frac{48n(K+1)}{\beta}\right)}.\label{eq:nu_theorem_prox_SGDA_str_mon_2}
    \end{eqnarray}

\paragraph{Upper bound for $\circledOne$.}
To estimate this sum, we will use Bernstein's inequality. The summands have conditional expectations equal to zero:
\begin{equation*}
        \EE_{\xi_{i}^l}\left[\frac{2\gamma}{n} (1-\gamma\mu)^{T-1-l} \langle \xi_t, \omega_{i,l}^u \rangle\right] = \frac{2\gamma}{n} \exp\left(-\gamma\mu(T-1-l)\right)\left\langle \xi_l, \EE_{\xi_{i}^l}[\omega_{i,l}^u] \right\rangle = 0.
\end{equation*}
Moreover, for all $l = 0,\ldots, T-1$ random vectors $\{\omega_{i,l}^u\}_{i=1}^n$ are independent. Thus, sequence $\left\{\frac{2\gamma}{n}\exp\left(-\gamma\mu(T-1-l)\right) \langle \xi_l, \omega_{i,l}^u \rangle\right\}_{l,i = 0,1}^{T-1,n}$ is a martingale difference sequence. Next, the summands are bounded:
    \begin{eqnarray}
        \left|\frac{2\gamma}{n} \exp(-\gamma\mu (T - 1 - l)) \langle \xi_l, \omega_{i,l}^u \rangle \right| &\leq& \frac{2\gamma}{n}\exp(-\gamma\mu (T - 1 - l)) \|\xi_l\|\cdot \|\omega_{i,l}^u\|\notag\\
        &\overset{\eqref{eq:xi_t_bound_SGDA_2_prox_clipped},\eqref{eq:omega_magnitude_prox_clipped_2}}{\leq}& \frac{8\sqrt{2}}{n}\gamma  \exp(-\gamma\mu (T - 1 - \nicefrac{l}{2})) \sqrt{V} \lambda_l\notag\\
        &\overset{\eqref{eq:lambda_SGDA_2_prox_clipped}}{\leq}& \frac{\exp(-\gamma\mu T)V}{8\ln\tfrac{48n(K+1)}{\beta}} \eqdef c. \label{eq:SGDA_2_prox_clipped_technical_1_1_xi}
    \end{eqnarray}
    Finally, conditional variances $\sigma_{i,l}^2 \eqdef \EE_{\xi_{i}^l}\left[\frac{4\gamma^2}{n^2}\exp\left(-2\gamma\mu(T-1-l)\right) \langle \xi_l, \omega_{i,l}^u \rangle^2\right]$ of the summands are bounded:
    \begin{eqnarray}
        \sigma_{i,l}^2 &\leq& \EE_{\xi^l_i}\left[\frac{4\gamma^2}{n^2}\exp(-2\gamma\mu (T - 1 - l)) \|\xi_l\|^2\cdot \|\omega_{i,l}^u\|^2\right]\notag\\
        &\overset{\eqref{eq:xi_t_bound_SGDA_2_prox_clipped}}{\leq}& \frac{32\gamma^2}{n^2}  \exp(-\gamma\mu (2T - 2 - l)) V \EE_{\xi^l_i}\left[\|\omega_{i,l}^u\|^2\right]. \label{eq:SGDA_2_prox_clipped_technical_1_2_xi}
    \end{eqnarray}
Applying Bernstein's inequality (Lemma~\ref{lem:Bernstein_ineq}) with $X_{i,l} = \frac{2\gamma}{n}\exp\left(-\gamma\mu(T-1-l)\right) \langle \xi_l, \omega_{i,l}^u \rangle$, constant $c$ defined in \eqref{eq:SGDA_2_prox_clipped_technical_1_1_xi}, $b = \frac{\exp\left(-\gamma\mu T\right)V}{8}$, $G = \frac{\exp\left(-2\gamma\mu T\right)V^2}{384\ln\frac{48n(K+1)}{\beta}}$, we get

    \begin{equation*}
        \PP\left\{|\circledOne| > \frac{\exp(-\gamma\mu T) V}{8} \text{ and } \sum\limits_{l=0}^{T-1}\sum\limits_{i=1}^{n}\sigma_{i,l}^2 \leq \frac{\exp(- 2\gamma\mu T) V^2}{384\ln\tfrac{48n(K+1)}{\beta}}\right\} \leq 2\exp\left(- \frac{b^2}{2G + \nicefrac{2cb}{3}}\right) = \frac{\beta}{24n(K+1)}.
    \end{equation*}
    The above is equivalent to $\PP\{E_{\circledOne}\} \geq 1 - \frac{\beta}{24n(K+1)}$ for
    \begin{equation}
         E_{\circledOne} = \left\{\text{either} \quad \sum\limits_{l=0}^{T-1}\sum\limits_{i=1}^{n}\sigma_{i,l}^2 > \frac{\exp(- 2\gamma\mu T) V^2}{384\ln\tfrac{48n(K+1)}{\beta}}\quad \text{or}\quad |\circledOne| \leq \frac{\exp(-\gamma\mu T) V}{8}\right\}. \label{eq:bound_1_SGDA_2_prox_clipped}
    \end{equation}
    Moreover, $E_{T-1}$ implies
    \begin{eqnarray}
        \sum\limits_{l=0}^{T-1}\sum\limits_{i=1}^n\sigma_{i,l}^2 &\overset{\eqref{eq:SGDA_2_prox_clipped_technical_1_2_xi}}{\leq}& \frac{32\gamma^2}{n}\exp(- 2\gamma\mu (T-1))V\sum\limits_{l=0}^{T-1} \frac{\EE_{\xi^l_i}\left[\|\omega_{i,l}^u\|^2\right]}{\exp(-\gamma\mu l)}\notag\\ 
        &\overset{\eqref{eq:variance_omega_prox_clipped_2}, T \leq K+1}{\leq}& \frac{576\gamma^2}{n}\exp(-2\gamma\mu (T-1) V \sigma^\alpha \sum\limits_{l=0}^{K} \frac{\lambda_l^{2-\alpha}}{\exp(-\gamma\mu l)}\notag\\
         &\overset{\eqref{eq:lambda_SGDA_2_prox_clipped}}{\leq}& \frac{9 (64\sqrt{2})^{\alpha}}{\sqrt{2}}\frac{\gamma^\alpha\exp(-2\gamma\mu (T-1) \sqrt{V}^{4-\alpha} \sigma^\alpha (K+1)\exp(\frac{\gamma\mu\alpha K}{2})}{ n^{\alpha-1}\ln^{2-\alpha}\tfrac{48n(K+1)}{\beta}}\notag\\
        &\overset{\eqref{eq:gamma_SGDA_2_prox_clipped}}{\leq}& \frac{\exp(-2\gamma\mu T)V^2}{384\ln\tfrac{48n(K+1)}{\beta}}. \label{eq:bound_1_variances_SGDA_2_prox_clipped}
    \end{eqnarray}
    
\paragraph{Upper bound for $\circledTwo$.} Probability event $E_{T-1}$ implies
    \begin{eqnarray}
        \circledTwo &\leq& \frac{2\gamma}{n} \exp(-\gamma\mu (T-1)) \sum\limits_{l=0}^{T-1}\sum\limits_{i=1}^n \frac{\|\xi_l\|\cdot \|\omega_{i,l}^b\|}{\exp(-\gamma\mu l)}\notag\\
        &\overset{\eqref{eq:xi_t_bound_SGDA_2_prox_clipped}, \eqref{eq:bias_omega_prox_clipped_2}}{\leq}& 2^{2+\alpha}\sqrt{2} \gamma \exp(-\gamma\mu (T-1)) \sqrt{V} \sigma^\alpha \sum\limits_{l=0}^{T-1} \frac{1}{\lambda_l^{\alpha-1} \exp(-\nicefrac{\gamma\mu l}{2})}\notag\\
        &\overset{\eqref{eq:lambda_SGDA_2_prox_clipped}, T \leq K+1}{\leq}& \frac{(128\sqrt{2})^{\alpha}}{16}\cdot\frac{\gamma^\alpha \sigma^{\alpha} \exp(-\gamma\mu (T-1)) (K+1) \exp\left(\frac{\gamma\mu\alpha K}{2}\right)\exp(\gamma\mu\alpha) \ln^{\alpha-1}\tfrac{48n(K+1)}{\beta}}{n^{\alpha-1}\sqrt{V}^{\alpha-2}} \notag \\
        &\overset{\eqref{eq:gamma_SGDA_2_prox_clipped}}{\leq}& \frac{\exp(-\gamma\mu T) V}{8}. \label{eq:bound_2_SGDA_2_prox_clipped}
    \end{eqnarray}

\paragraph{Upper bound for $\circledThree$.} To estimate this sum, we will use Bernstein's inequality. The summands have conditional expectations equal to zero:
    \begin{equation*}
        \EE_{\xi^l_i}\left[\frac{4\gamma^2}{n^2}(1-\gamma\mu)^{T-1-l}\left[\|\omega_{i,l}^u\|^2  -\EE_{\xi^l_i}\left[\|\omega_{i,l}^u\|^2\right]\right] \right] = 0.
    \end{equation*}
    Moreover, for all $l = 0,\ldots, T-1$ random vectors $\{\omega_{i,l}^u\}_{i=1}^n$ are independent. Thus, sequence $\left\{ \frac{4\gamma^2}{n^2}  \exp\left(-\gamma\mu(T-1-l)\right)\left(\|\omega_{i,l}^u\|^2 - \EE_{\xi_{i}^l}\left[\|\omega_{i,l}^u\|^2\right]\right)\right\}_{l,i = 0,1}^{T-1,n}$ is a martingale difference sequence. Next, the summands are bounded:
    \begin{eqnarray}
        \frac{4\gamma^2}{n^2} (1-\gamma\mu)^{T-1-l}\left| \|\omega_{i,l}^u\|^2  -\EE_{\xi^l_i}\left[\|\omega_{i,l}^u\|^2\right] \right| &\overset{\eqref{eq:omega_magnitude_prox_clipped_2}}{\leq}&
        \frac{32\gamma^2\lambda_l^2}{n^2} \frac{\exp(-\gamma\mu T) }{\exp(-\gamma\mu (1+l))}\notag\\
        &\overset{\eqref{eq:lambda_SGDA_2_prox_clipped}}{\leq}& \frac{\exp(-\gamma\mu (T+1))V}{256\ln^2\tfrac{48n(K+1)}{\beta}}\notag\\
        &\leq& \frac{\exp(-\gamma\mu T)V}{8\ln\tfrac{48n(K+1)}{\beta}}
        \eqdef c. \label{eq:SGDA_2_prox_clipped_technical_4_1}
    \end{eqnarray}

    Finally, conditional variances 
    \begin{equation*}
        \widetilde\sigma_{i,l}^2 \eqdef \EE_{\xi^l_i}\left[\frac{16\gamma^4}{n^4}  (1-\gamma\mu)^{2T-2-2l} \left|\|\omega_{i,l}^u\|^2  -\EE_{\xi^l_i}\left[\|\omega_{i,l}^u\|^2\right] \right|^2 \right]
    \end{equation*}
    of the summands are bounded:
    \begin{eqnarray}
        \widetilde\sigma_{i,l}^2 &\overset{\eqref{eq:SGDA_2_prox_clipped_technical_4_1}}{\leq}& \frac{4\gamma^2\exp(-2\gamma\mu T)V}{8n^2\exp(-\gamma\mu (1+l))\ln\tfrac{48n(K+1)}{\beta}} \EE_{\xi^l_i}\left[ \left|\|\omega_{i,l}^u\|^2  -\EE_{\xi^l_i}\left[\|\omega_{i,l}^u\|^2\right] \right|\right]\notag\\
        &\leq& \frac{\gamma^2\exp(-2\gamma\mu T)V}{n^2\exp(-\gamma\mu (1+l))\ln\tfrac{48n(K+1)}{\beta}} \EE_{\xi^l_i}\left[\|\omega_{i,l}^u\|^2\right]. \label{eq:SGDA_2_prox_clipped_technical_4_2}
    \end{eqnarray}
    Applying Bernstein's inequality (Lemma~\ref{lem:Bernstein_ineq}) with $X_{i,l} = \frac{4\gamma^2}{n^2} (1-\gamma\mu)^{T-1-l}\left[\|\omega_{i,l}^u\|^2  -\EE_{\xi^l_i}\left[\|\omega_{i,l}^u\|^2\right]\right]$, constant $c$ defined in \eqref{eq:SGDA_2_prox_clipped_technical_4_1}, $b = \frac{\exp\left(-\gamma\mu T\right)V}{8}$, $G = \frac{\exp\left(-2\gamma\mu T\right)V^2}{384\ln\frac{48n(K+1)}{\beta}}$, we get
    \begin{equation*}
        \PP\left\{|\circledThree| > \frac{\exp(-\gamma\mu T) V}{8} \text{ and } \sum\limits_{l=0}^{T-1}\sum\limits_{i=1}^n\widetilde\sigma_{i,l}^2 \leq \frac{\exp(-2\gamma\mu T) V^2}{384\ln\frac{48n(K+1)}{\beta}}\right\} \leq 2\exp\left(- \frac{b^2}{2G + \nicefrac{2cb}{3}}\right) = \frac{\beta}{24n(K+1)}.
    \end{equation*}
    The above is equivalent to $\PP\{E_{\circledThree}\} \geq 1 - \frac{\beta}{24n(K+1)}$ for
    \begin{equation}
        E_{\circledThree} = \left\{\text{either} \quad \sum\limits_{l=0}^{T-1}\sum\limits_{i=1}^n\widetilde\sigma_{i,l}^2 > \frac{\exp(-2\gamma\mu T) V^2}{384\ln\tfrac{48n(K+1)}{\beta}}\quad \text{or}\quad |\circledThree| \leq \frac{\exp(-\gamma\mu T) V}{8}\right\}. \label{eq:bound_5_SGDA_2_prox_clipped}
    \end{equation}
    Moreover, $E_{T-1}$ implies
    \begin{eqnarray}
        \sum\limits_{l=0}^{T-1}\sum\limits_{i=1}^n\widetilde\sigma_{i,l}^2 &\overset{\eqref{eq:SGDA_2_prox_clipped_technical_4_2}}{\leq}& \frac{\gamma^2\exp(-\gamma\mu (2T-1))V}{n^2\ln\tfrac{48n(K+1)}{\beta}} \sum\limits_{l=0}^{T-1}\sum\limits_{i=1}^n \frac{\EE_{\xi^l_i}\left[\|\omega_{i,l}^u\|^2\right]}{\exp(-\gamma\mu l)}\notag\\
        &\overset{\eqref{eq:variance_omega_prox_clipped_2}, T \leq K+1}{\leq}& \frac{18\gamma^2\exp(-\gamma\mu (2T-1)) V \sigma^\alpha}{n\ln\tfrac{48n(K+1)}{\beta}} \sum\limits_{l=0}^{K} \frac{\lambda_l^{2-\alpha}}{\exp(-\gamma\mu l)}\notag\\
        &\overset{\eqref{eq:lambda_SGDA_2_prox_clipped}}{\leq}& 
        \frac{9(64\sqrt{2})^{\alpha}}{4096}\cdot\frac{\gamma^\alpha\exp(-\gamma\mu (2T-1)) \sqrt{V}^{4-\alpha} \sigma^\alpha (K+1)\exp(\frac{\gamma\mu\alpha K}{2})}{n^{\alpha-1}\ln^{3-\alpha}\tfrac{48n(K+1)}{\beta}} \notag\\
        &\overset{\eqref{eq:gamma_SGDA_2_prox_clipped}}{\leq}& \frac{\exp(-2\gamma\mu T)V^2}{384\ln\tfrac{48n(K+1)}{\beta}}. \label{eq:bound_5_variances_SGDA_2_prox_clipped}
    \end{eqnarray}

\paragraph{Upper bound for $\circledFour$.} Probability event $E_{T-1}$ implies
    \begin{eqnarray}
        \circledFour &=& \frac{4\gamma^2}{n^2} \sum\limits_{l=0}^{T-1}\sum\limits_{i=1}^n (1-\gamma \mu)^{T-1-l} \EE_{\xi^l_i}[\|\omega^u_{i,l}\|^2]\notag\\
        &\overset{\eqref{eq:variance_omega_prox_clipped_2}}{\leq}& \frac{72\gamma^2\exp(-\gamma\mu (T-1)) \sigma^\alpha}{n}\sum\limits_{l=0}^{T-1} \frac{\lambda_l^{2-\alpha}}{\exp(-\gamma\mu l)} \notag\\
        &\overset{\eqref{eq:lambda_SGDA_2_prox_clipped}, T \leq K+1}{\leq}& \frac{9(64\sqrt{2})^{\alpha}}{1024} \cdot \frac{\gamma^\alpha \sqrt{V}^{2-\alpha}\exp(-\gamma\mu (T-1)) \sigma^\alpha (K+1)\exp(\frac{\gamma\mu\alpha K}{2})}{n^{\alpha-1} \ln^{2-\alpha}\tfrac{48n(K+1)}{\beta}} \notag\\
        &\overset{\eqref{eq:gamma_SGDA_2_prox_clipped}}{\leq}& \frac{\exp(-\gamma\mu T) V}{8} . \label{eq:bound_6_SGDA_2_prox_clipped}
    \end{eqnarray}

\paragraph{Upper bound for $\circledFive$.} Probability event $E_{T-1}$ implies
    \begin{eqnarray}
        \circledFive &=&  \frac{4\gamma^2}{n} \sum\limits_{l=0}^{T-1}\sum\limits_{i=1}^n (1-\gamma \mu)^{T-1-l} \|\omega^b_{i,l}\|^2\notag\\
        &\overset{\eqref{eq:bias_omega_prox_clipped_2}}{\leq}& 4\cdot 2^{2\alpha}\gamma^2 \exp(-\gamma\mu (T-1)) \sigma^{2\alpha} \sum\limits_{l=0}^{T-1} \frac{1}{\lambda_l^{2\alpha-2} \exp(-\gamma\mu l)} \notag\\
        &\overset{\eqref{eq:lambda_SGDA_2_prox_clipped}, T \leq K+1}{\leq}&  \frac{(128\sqrt{2})^{\alpha}}{2048}\cdot\frac{\gamma^{2\alpha} \exp(-\gamma\mu (T-3)) \sigma^{2\alpha} \ln^{2(\alpha-1)}\tfrac{48n(K+1)}{\beta} (K+1) \exp(\gamma\mu \alpha K)}{n^{2(\alpha-1)}V^{\alpha-1}}\notag\\
        &\overset{\eqref{eq:gamma_SGDA_2_prox_clipped}}{\leq}& \frac{\exp(-\gamma\mu T) V}{8}. \label{eq:bound_7_SGDA_2_prox_clipped}
    \end{eqnarray}

\paragraph{Upper bounds for $\circledSix$.} This sum requires more refined analysis. We introduce new vectors:
\begin{eqnarray}
        \delta_{j}^l = \begin{cases}
            \frac{\gamma}{n}\sum\limits_{i=1}^{j-1} \omega_{i,l}^u,& \text{if } \left\|\frac{\gamma}{n}\sum\limits_{i=1}^{j-1} \omega_{i,l}^u\right\| \leq \exp\left(-\frac{\gamma\mu l}{2}\right)\frac{\sqrt{V}}{2},\\
            0,& \text{otherwise,}
        \end{cases}\label{eq:SGDA_zeta_delta_def_str_mon_2}
\end{eqnarray}
    for all $j \in [n]$ and $l = 0,\ldots, T-1$. Then, by definition
    \begin{equation}
        \|\delta_j^l\| \leq \exp\left(-\frac{\gamma\mu l}{2}\right)\frac{\sqrt{V}}{2} \label{eq:SGDA_zeta_delta_bound_str_mon}
    \end{equation}
    and
    \begin{eqnarray}
        \circledSix &=& \underbrace{\frac{4\gamma}{n} \sum\limits_{l=0}^{T-1} \sum\limits_{j=2}^n \exp\left(-\gamma\mu(T-1-t)\right)\left\langle \delta_j^l, \omega_{j,l}^u \right\rangle}_{\circledSix'}\notag\\
        &&\quad + \frac{4\gamma}{n} \sum\limits_{l=0}^{T-1} \sum\limits_{j=2}^n \exp\left(-\gamma\mu(T-1-t)\right) \left\langle \frac{\gamma}{n}\sum\limits_{i=1}^{j-1} \omega_{i,l}^u - \delta_j^l, \omega_{j,l}^u \right\rangle. \label{eq:SGDA_extra_sums_seven_distributed_str_mon}
    \end{eqnarray}
We also note here that $E_{T-1}$ implies
    \begin{align}
        \frac{4\gamma}{n} \sum\limits_{l=0}^{T-1} \sum\limits_{j=2}^n \exp&\left(-\gamma\mu(T-1-t)\right)\left\langle \frac{\gamma}{n}\sum\limits_{i=1}^{j-1} \omega_{i,l}^u - \delta_j^l, \omega_{j,l}^u \right\rangle \notag\\
        &= \frac{4\gamma}{n}  \sum\limits_{j=2}^n \exp\left(-\gamma\mu(T-1-t)\right)\left\langle \frac{\gamma}{n}\sum\limits_{i=1}^{j-1} \omega_{i,T-1}^u - \delta_j^{T-1}, \omega_{j,T-1}^u \right\rangle. \label{eq:SGDA_extra_sums_seven_distributed_1_str_mon}
    \end{align}
\paragraph{Upper bound for $\circledSix'$.} To estimate this sum, we will use Bernstein's inequality. The summands have conditional expectations equal to zero:
    \begin{equation*}
        \EE_{\xi_{j}^l}\left[\frac{4\gamma}{n}\exp\left(-\gamma\mu(T-1-l)\right) \langle \delta_j^l, \omega_{j,l}^u \rangle\right] = \frac{4\gamma}{n} \exp\left(-\gamma\mu(T-1-l)\right)\left\langle \delta_j^l, \EE_{\xi_{j}^l}[\omega_{j,l}^u] \right\rangle = 0.
    \end{equation*}
    Moreover, for all $l = 0,\ldots, T-1$ random vectors $\{\omega_{j,l}^u\}_{j=1}^n$ are independent. Thus, sequence $\left\{\frac{4\gamma}{n}\exp\left(-\gamma\mu(T-1-l)\right) \langle \delta_j^l, \omega_{j,l}^u \rangle\right\}_{l,j = 0,1}^{T-1,n}$ is a martingale difference sequence. Next, the summands are bounded:
    \begin{eqnarray}
        \left|\frac{4\gamma}{n} \exp\left(-\gamma\mu(T-1-l)\right)\langle \delta_j^l, \omega_{j,l}^u \rangle \right| &\leq& \frac{4\gamma}{n}\exp\left(-\gamma\mu(T-1-l)\right) \|\delta_j^l\|\cdot \|\omega_{j,l}^u\|\notag\\
        &\overset{\eqref{eq:SGDA_zeta_delta_bound_str_mon},\eqref{eq:omega_magnitude_prox_clipped_2}}{\leq}& \frac{4\sqrt{V}\gamma \exp\left(-\gamma\mu(T-1)\right)}{n} \exp\left(\frac{\gamma\mu l}{2}\right)\lambda_l\notag\\
        &\overset{\eqref{eq:lambda_SGDA_2_prox_clipped}}{=}& \frac{\exp\left(-\gamma\mu T\right)V}{16\sqrt{2}\ln\frac{48n(K+1)}{\beta}}\notag\\
        &\leq& \frac{\exp\left(-\gamma\mu T\right)V}{8\ln\frac{48n(K+1)}{\beta}} \eqdef c. \label{eq:str_mon_thm_SGDA_technical_6_5_seven}
    \end{eqnarray}
    Finally, conditional variances $(\sigma_{j,l}')^2 \eqdef \EE_{\xi_{j}^l}\left[\frac{16\gamma^2}{n^2}\exp\left(-\gamma\mu(2T-2-2l)\right) \langle \delta_j^l, \omega_{j,l}^u \rangle^2\right]$ of the summands are bounded:
    \begin{eqnarray}
        (\sigma_{j,l}')^2 &\leq& \EE_{\xi_{j}^l}\left[\frac{16\gamma^2 }{n^2}\exp\left(-\gamma\mu(2T-2-2l)\right)\|\delta_j^l\|^2\cdot \|\omega_{j,l}^u\|^2\right] \notag\\
        &\overset{\eqref{eq:SGDA_zeta_delta_bound_str_mon}}{\leq}& \frac{4\gamma^2 V\exp\left(-\gamma\mu\left(2T-2-l\right)\right)}{n^2} \EE_{\xi_{j}^l}\left[\|\omega_{j,l}^u\|^2\right]. \label{eq:str_mon_thm_SGDA_technical_7_seven}
    \end{eqnarray}
    Applying Bernstein's inequality (Lemma~\ref{lem:Bernstein_ineq}) with $X_{j,l} = \frac{4\gamma}{n}\exp\left(-\gamma\mu(T-1-l)\right) \langle \delta_j^l, \omega_{j,l}^u \rangle$, constant $c$ defined in \eqref{eq:str_mon_thm_SGDA_technical_7_seven}, $b = \frac{\exp\left(-\gamma\mu T\right)V}{8}$, $G = \frac{\exp\left(-2\gamma\mu T\right)V^2}{384\ln\frac{48n(K+1)}{\beta}}$, we get
    \begin{eqnarray*}
        \PP\left\{|\circledSix'| > \frac{\exp\left(-\gamma\mu T\right)V}{8} \text{ and } \sum\limits_{l=0}^{T-1}\sum\limits_{j=2}^n (\sigma_{j,l}')^2 \leq \frac{\exp\left(-2\gamma\mu T\right)V^2}{384\ln\frac{48n(K+1)}{\beta}}\right\} &\leq& 2\exp\left(- \frac{b^2}{2G + \nicefrac{2cb}{3}}\right)\\
        &=& \frac{\beta}{24n(K+1)}.
    \end{eqnarray*}
    The above is equivalent to $\PP\{E_{\circledSix'}\} \geq 1 - \frac{\beta}{24n(K+1)}$ for
    \begin{equation}
        E_{\circledSix'} = \left\{\text{either } \sum\limits_{l=0}^{T-1}\sum\limits_{j=2}^n (\sigma_{j,l}')^2 > \frac{\exp\left(-2\gamma\mu T\right)V^2}{384\ln\frac{48n(K+1)}{\beta}} \text{ or } |\circledSix'| \leq \frac{\exp\left(-\gamma\mu T\right)V}{8}\right\}. \label{eq:bound_6_str_mon_SGDA_seven}
    \end{equation}
    Moreover, $E_{T-1}$ implies
    \begin{eqnarray}
        \sum\limits_{l=0}^{T-1}\sum\limits_{j=2}^n (\sigma_{j,l}')^2 &\overset{\eqref{eq:str_mon_thm_SGDA_technical_7_seven}}{\leq}& \frac{4\gamma^2 V\exp\left(-\gamma\mu\left(2T-2\right)\right)}{n^2} \sum\limits_{l=0}^{T-1}\exp\left(\gamma\mu l\right) \sum\limits_{i=1}^n \EE_{\xi_{i}^l}\left[\|\omega_{i,l}^u\|^2\right]\notag\\
        &\overset{\eqref{eq:variance_omega_prox_clipped_2}, T \leq K+1}{\leq}&  \frac{72\gamma^2 V\exp\left(-\gamma\mu\left(2T-2\right)\right)\sigma^\alpha}{n} \sum\limits_{l=0}^{T-1}\exp\left(\gamma\mu l\right) \lambda_l^{2-\alpha}\notag\\
        &\overset{\eqref{eq:lambda_SGDA_2_prox_clipped}}{\leq}& \frac{72\gamma^\alpha V^{2 - \frac{\alpha}{2}}\exp\left(-2\gamma\mu T\right)\sigma^\alpha}{(64\sqrt{2})^{2-\alpha} n^{\alpha-1} \ln^{2-\alpha}\frac{48n(K+1)}{\beta}} \sum\limits_{l=0}^{T-1}\exp\left(\frac{\gamma\mu l \alpha}{2}\right)\notag\\
        &\leq& \frac{72\gamma^\alpha V^{2 - \frac{\alpha}{2}}\exp\left(-2\gamma\mu T\right)\sigma^\alpha (K+1) \exp\left(\frac{\gamma\mu K \alpha}{2}\right)}{(64\sqrt{2})^{2-\alpha} n^{\alpha-1} \ln^{2-\alpha}\frac{48n(K+1)}{\beta}}\notag\\
        &\overset{\eqref{eq:gamma_SGDA_2_prox_clipped}}{\leq}& \frac{\exp\left(-2\gamma\mu T\right)V^2}{384\ln\frac{48n(K+1)}{\beta}}.\label{eq:bound_6_variances_str_mon_SGDA_seven}
    \end{eqnarray}

That is, we derive the upper bounds for  $\circledOne, \circledTwo, \circledThree, \circledFour, \circledFive, \circledSix$. More precisely, $E_{T-1}$ implies
    \begin{gather*}
        V_T \overset{\eqref{eq:SGDA_2_prox_clipped_12345_bound}}{\leq} \exp\left(\gamma\mu T\right)V + \circledOne + \circledTwo + \circledThree + \circledFour + \circledFive + \circledSix,\\
         \circledSix \overset{\eqref{eq:SGDA_extra_sums_seven_distributed_str_mon}}{=} \circledSix' + \frac{4\gamma}{n}  \sum\limits_{j=2}^n \exp\left(-\gamma\mu(T-1-t)\right)\left\langle \frac{\gamma}{n}\sum\limits_{i=1}^{j-1} \omega_{i,T-1}^u - \delta_j^{T-1}, \omega_{j,T-1}^u \right\rangle,\\
         \circledTwo \overset{\eqref{eq:bound_2_SGDA_2_prox_clipped}}{\leq} \frac{\exp\left(-\gamma\mu T\right) V}{8},\quad
         \circledFour \overset{\eqref{eq:bound_6_SGDA_2_prox_clipped}}{\leq} \frac{\exp\left(-\gamma\mu T\right) V}{8},\\ \circledFive \overset{\eqref{eq:bound_7_SGDA_2_prox_clipped}}{\leq} \frac{\exp\left(-\gamma\mu T\right) V}{8}, \\
         \sum\limits_{l=0}^{T-1}\sum\limits_{i=1}^n\sigma_{i,l}^2 \overset{\eqref{eq:bound_1_variances_SGDA_2_prox_clipped}}{\leq} \frac{\exp\left(-2\gamma\mu T\right)V^2}{384\ln\frac{48n(K+1)}{\beta}},\quad 
         \sum\limits_{l=0}^{T-1}\sum\limits_{j=2}^n \widetilde \sigma_{j,l}^2 \overset{\eqref{eq:bound_5_variances_SGDA_2_prox_clipped}}{\leq} \frac{\exp\left(-2\gamma\mu T\right)V^2}{384\ln\frac{48n(K+1)}{\beta}},\\ \sum\limits_{l=0}^{T-1}\sum\limits_{j=2}^n (\sigma_{j,l}')^2 \overset{\eqref{eq:bound_6_variances_str_mon_SGDA_seven}}{\leq} \frac{\exp\left(-2\gamma\mu T\right)V^2}{384\ln\frac{48n(K+1)}{\beta}}.
    \end{gather*}
    In addition, we also establish (see \eqref{eq:bound_1_SGDA_2_prox_clipped}, \eqref{eq:bound_5_SGDA_2_prox_clipped}, \eqref{eq:bound_6_str_mon_SGDA_seven}, and our induction assumption)
    \begin{gather*}
        \PP\{E_{T-1}\} \geq 1 - \frac{(T-1)\beta}{K+1},\\ \PP\{E_{\circledOne}\} \geq 1 - \frac{\beta}{24n(K+1)}, \quad \PP\{E_{\circledThree}\} \geq 1 - \frac{\beta}{24n(K+1)},\quad \PP\{E_{\circledSix'}\} \geq 1 - \frac{\beta}{24n(K+1)},
    \end{gather*}
    where
    \begin{eqnarray*}
        E_{\circledOne} &=& \left\{\text{either } \sum\limits_{l=0}^{T-1}\sum\limits_{i=1}^n\sigma_{i,l}^2 > \frac{\exp\left(-2\gamma\mu T\right)V^2}{384\ln\frac{48n(K+1)}{\beta}} \text{ or } |\circledOne| \leq \frac{\exp\left(-\gamma\mu T\right)V}{8}\right\},\\
        E_{\circledThree} &=& \left\{\text{either } \sum\limits_{l=0}^{T-1}\sum\limits_{j=2}^n\widetilde \sigma_{j,l}^2 > \frac{\exp\left(-2\gamma\mu T\right)V^2}{384\ln\frac{48n(K+1)}{\beta}} \text{ or } |\circledThree| \leq \frac{\exp\left(-\gamma\mu T\right)V}{8}\right\},\\
        E_{\circledSix'} &=& \left\{\text{either } \sum\limits_{l=0}^{T-1}\sum\limits_{j=2}^n (\sigma_{j,l}')^2 > \frac{\exp\left(-2\gamma\mu T\right)V^2}{384\ln\frac{48n(K+1)}{\beta}} \text{ or } |\circledSix'| \leq \frac{\exp\left(-\gamma\mu T\right)V}{8}\right\}.
    \end{eqnarray*}
    Therefore, probability event $E_{T-1} \cap E_{\circledOne}  \cap E_{\circledThree} \cap E_{\circledSix'}$ implies
    \begin{eqnarray}
        V_T &\leq& \exp\left(-\gamma\mu T\right)V \underbrace{\left(1 + \frac{1}{8} + \frac{1}{8} + \frac{1}{8} + \frac{1}{8} + \frac{1}{8} + \frac{1}{8}\right)}_{< 2} \notag\\
        &&\quad + \frac{4\gamma}{n}  \sum\limits_{j=2}^n \exp\left(-\gamma\mu (T-1-t)\right)\left\langle \frac{\gamma}{n}\sum\limits_{i=1}^{j-1} \omega_{i,T-1}^u - \delta_j^{T-1}, \omega_{j,T-1}^u \right\rangle. \label{eq:SGDA_new}
    \end{eqnarray}

To finish the proof, we need to show that  $\frac{\gamma}{n}\sum\limits_{i=1}^{j-1} \omega_{i,T-1}^u = \delta_j^{T-1}$ with high probability. In particular, we consider probability event $\widetilde{E}_{T-1,j}$ defined as follows: inequalities
    \begin{equation}
         \left\|\frac{\gamma}{n}\sum\limits_{i=1}^{r-1} \omega_{i,T-1}^u\right\| \leq \exp\left(-\frac{\gamma\mu (T-1)}{2}\right)\frac{\sqrt{V}}{2} \notag
    \end{equation}
    hold for $r = 2, \ldots, j$ simultaneously. We want to show that $\PP\{E_{T-1} \cap \widetilde{E}_{T-1,j}\} \geq 1 - \frac{(T-1)\beta}{K+1} - \frac{j\beta}{8n(K+1)}$ for all $j = 2, \ldots, n$. For $j = 2$ the statement is trivial since
    \begin{eqnarray*}
        \left\|\frac{\gamma}{n} \omega_{1,T-1}^u\right\| &\overset{\eqref{eq:omega_magnitude_prox_clipped_2}}{\leq}& \frac{2\gamma\lambda_{T-1}}{n} \leq \exp\left(-\frac{\gamma\mu (T-1)}{2}\right)\frac{\sqrt{V}}{2}.
    \end{eqnarray*}
    Next, we assume that the statement holds for some $j = m-1 < n$, i.e.,  $\PP\{E_{T-1}\cap\widetilde{E}_{T-1,m-1}\} \geq 1 - \frac{(T-1)\beta}{K+1} - \frac{(m-1)\beta}{8n(K+1)}$. Our goal is to prove that $\PP\{E_{T-1}\cap\widetilde{E}_{T-1,m}\} \geq 1 - \frac{(T-1)\beta}{K+1} - \frac{m\beta}{8n(K+1)}$. First, we consider $\left\|\frac{\gamma}{n}\sum\limits_{i=1}^{m-1} \omega_{i,T-1}^u\right\|$:
    \begin{eqnarray}
        \left\|\frac{\gamma}{n}\sum\limits_{i=1}^{m-1} \omega_{i,T-1}^u\right\| &=& \sqrt{\frac{\gamma^2}{n^2}\left\|\sum\limits_{i=1}^{m-1} \omega_{i,T-1}^u\right\|^2} \notag\\
        &=& \sqrt{\frac{\gamma^2}{n^2}\sum\limits_{i=1}^{m-1}\|\omega_{i,T-1}^u\|^2 + \frac{2\gamma}{n}\sum\limits_{i=1}^{m-1}\left\langle \frac{\gamma}{n}\sum\limits_{r=1}^{i-1}\omega_{r,T-1}^u , \omega_{i,T-1}^u \right\rangle} \notag\\
        &\leq& \sqrt{\frac{\gamma^2}{n^2}\sum\limits_{t=0}^{T-1}\exp\left(-\gamma\mu(T-1-t)\right)\sum\limits_{i=1}^{m-1}\|\omega_{i,t}^u\|^2 + \frac{2\gamma}{n}\sum\limits_{i=1}^{m-1}\left\langle \frac{\gamma}{n}\sum\limits_{r=1}^{i-1}\omega_{r,T-1}^u , \omega_{i,T-1}^u \right\rangle}. \notag
    \end{eqnarray}

    Next, we introduce a new notation:
    \begin{gather*}
        \rho_{i,T-1}' = \begin{cases}
            \frac{\gamma}{n}\sum\limits_{r=1}^{i-1}\omega_{r,T-1}^u,& \text{if } \left\|\frac{\gamma}{n}\sum\limits_{r=1}^{i-1}\omega_{r,T-1}^u\right\| \leq \exp\left(-\frac{\gamma\mu (T-1)}{2}\right)\frac{\sqrt{V}}{2},\\
            0,& \text{otherwise}
        \end{cases}
    \end{gather*}
    for $i = 1,\ldots,m-1$. By definition, we have
    \begin{equation}
         \|\rho_{i,T-1}'\| \leq \exp\left(-\frac{\gamma\mu (T-1)}{2}\right)\frac{\sqrt{V}}{2} \label{eq:SGDA_bound_rho_str_mon}
    \end{equation}
    for $i = 1,\ldots,m-1$. Moreover, $\widetilde{E}_{T-1,m-1}$ implies $\rho_{i,T-1}' = \frac{\gamma}{n}\sum\limits_{r=1}^{i-1}\omega_{r,T-1}^u$ for $i = 1,\ldots,m-1$ and
    \begin{eqnarray}
        \left\|\frac{\gamma}{n}\sum\limits_{i=1}^{m-1} \omega_{i,l}^u\right\| &\leq& \sqrt{\circledThree + \circledFour + \circledSeven}, \notag
    \end{eqnarray}
    where
    \begin{gather*}
        \circledSeven =  \frac{2\gamma}{n}\sum\limits_{i=1}^{m-1}\left\langle \rho_{i,T-1}' , \omega_{i,T-1}^u \right\rangle.
    \end{gather*}
    It remains to estimate $\circledNine$.

\paragraph{Upper bound for $\circledSeven$.} To estimate this sum, we will use Bernstein's inequality. The summands have conditional expectations equal to zero:
    \begin{equation*}
        \EE_{\xi_{i}^{T-1}}\left[\frac{2\gamma}{n} \langle \rho_{i,T-1}', \omega_{i,T-1}^u \rangle\right] = \frac{2\gamma}{n} \left\langle \rho_{i,T-1}', \EE_{\xi_{i}^{T-1}}[\omega_{i,T-1}^u] \right\rangle = 0.
    \end{equation*}
    Thus, sequence $\left\{\frac{2\gamma}{n}\langle \rho_{i,T-1}', \omega_{i,T-1}^u \rangle\right\}_{i = 1}^{n}$ is a martingale difference sequence. Next, the summands are bounded:
    \begin{eqnarray}
        \left|\frac{2\gamma}{n} \langle \rho_{i,T-1}', \omega_{i,T-1}^u \rangle \right| &\leq& \frac{2\gamma}{n}\|\rho_{i,T-1}'\|\cdot \|\omega_{i,T-1}^u\|\notag\\
        &\overset{\eqref{eq:SGDA_bound_rho_str_mon},\eqref{eq:omega_magnitude_prox_clipped_2}}{\leq}& \frac{2\sqrt{V}\gamma \exp\left(-\frac{\gamma\mu(T-1)}{2}\right)}{n} \lambda_{T-1}\notag\\
        &\overset{\eqref{eq:lambda_SGDA_2_prox_clipped}}{=}& \frac{\exp\left(-\gamma\mu T\right)V}{32\sqrt{2}\ln\frac{48n(K+1)}{\beta}} \notag\\
        &\leq& \frac{\exp\left(-\gamma\mu T\right)V}{8\ln\frac{48n(K+1)}{\beta}} \eqdef c. \label{eq:str_mon_thm_SGDA_technical_6_5_eight'}
    \end{eqnarray}
    Finally, conditional variances $(\widetilde\sigma_{i,T-1}')^2 \eqdef \EE_{\xi_{i}^{T-1}}\left[\frac{4\gamma^2}{n^2} \langle \rho_{i,T-1}', \omega_{i,T-1}^u \rangle^2\right]$ of the summands are bounded:
    \begin{eqnarray}
        (\widetilde\sigma_{i,T-1}')^2 &\leq& \EE_{\xi_{i}^{T-1}}\left[\frac{4\gamma^2 }{n^2}\|\rho_{i,T-1}'\|^2\cdot \|\omega_{i,T-1}^u\|^2\right] \notag\\
        &\overset{\eqref{eq:SGDA_bound_rho_str_mon}}{\leq}& \frac{\gamma^2 V\exp\left(-\gamma\mu(T-1)\right)}{n^2} \EE_{\xi_{i}^{T-1}}\left[\|\omega_{i,T-1}^u\|^2\right]. \label{eq:str_mon_thm_SGDA_technical_7_eight'}
    \end{eqnarray}
    Applying Bernstein's inequality (Lemma~\ref{lem:Bernstein_ineq}) with $X_{i} = \frac{2\gamma}{n}\langle \rho_{i,T-1}', \omega_{i,T-1}^u \rangle$, constant $c$ defined in \eqref{eq:str_mon_thm_SGDA_technical_6_5_eight'}, $b = \frac{\exp\left(-\gamma\mu T\right)V}{8}$, $G = \frac{\exp\left(-2\gamma\mu T\right)V^2}{384\ln\frac{48n(K+1)}{\beta}}$, we get
    \begin{eqnarray*}
        \PP\left\{|\circledSeven| > \frac{\exp\left(-\gamma\mu T\right)V}{8} \text{ and } \sum\limits_{i=1}^n(\widetilde\sigma_{i,T-1}')^2 \leq \frac{\exp\left(-2\gamma\mu T\right)V^2}{384\ln\frac{48n(K+1)}{\beta}}\right\} &\leq& 2\exp\left(- \frac{b^2}{2G + \nicefrac{2cb}{3}}\right)\\
        &=& \frac{\beta}{24n(K+1)}.
    \end{eqnarray*}
    The above is equivalent to $\PP\{E_{\circledSeven}\} \geq 1 - \frac{\beta}{24n(K+1)}$ for
    \begin{equation}
        E_{\circledSeven} = \left\{\text{either } \sum\limits_{i=1}^n(\widetilde\sigma_{i,T-1}')^2 > \frac{\exp\left(-2\gamma\mu T\right)V^2}{384\ln\frac{48n(K+1)}{\beta}} \text{ or } |\circledSeven| > \frac{\exp\left(-\gamma\mu T\right)V}{8}\right\}. \label{eq:bound_1_str_mon_SGDA_eight'}
    \end{equation}
    Moreover, $E_{T-1}$ implies
    \begin{eqnarray}
        \sum\limits_{i=1}^n(\widetilde\sigma_{i,T-1}')^2 &\overset{\eqref{eq:str_mon_thm_SGDA_technical_7_eight'}}{\leq}& \frac{\gamma^2 V\exp\left(-\gamma\mu\left(T-1\right)\right)}{n^2} \sum\limits_{i=1}^n \EE_{\xi_{i}^T-1}\left[\|\omega_{i,T-1}^u\|^2\right]\notag\\
        &\overset{\eqref{eq:variance_omega_prox_clipped_2}}{\leq}&  \frac{18\gamma^2 V\exp\left(-\gamma\mu\left(T-1\right)\right)\sigma^\alpha}{n} \lambda_{T-1}^{2-\alpha}\notag\\
        &\overset{\eqref{eq:lambda_SGDA_2_prox_clipped}}{\leq}& \frac{18\gamma^\alpha V^{2 - \frac{\alpha}{2}}\exp\left(-\gamma\mu (T-1)\right)\sigma^\alpha}{(64\sqrt{2})^{2-\alpha} n^{\alpha-1} \ln^{2-\alpha}\frac{48n(K+1)}{\beta}} \exp\left(\frac{\gamma\mu (T-1) \alpha}{4}\right)\notag\\
        &\overset{T-1 \leq K}{\leq}& \frac{18\gamma^\alpha V^{2 - \frac{\alpha}{2}}\exp\left(-\gamma\mu (T-1)\right)\sigma^\alpha \exp\left(\frac{\gamma\mu K \alpha}{2}\right)}{(64\sqrt{2})^{2-\alpha} n^{\alpha-1} \ln^{2-\alpha}\frac{48n(K+1)}{\beta}}\notag\\
        &\overset{\eqref{eq:gamma_SGDA_2_prox_clipped}}{\leq}& \frac{\exp\left(-2\gamma\mu T\right)V^2}{384\ln\frac{48n(K+1)}{\beta}}.\label{eq:bound_1_variances_str_mon_SGDA_eight'}
    \end{eqnarray}
Putting all together we get that $E_{T-1}\cap \widetilde{E}_{T-1,m-1}$ implies
    \begin{gather*}
        \left\|\frac{\gamma}{n}\sum\limits_{i=1}^{m-1} \omega_{i,T-1}^u\right\| \leq \sqrt{\circledThree + \circledFour + \circledSeven},\\ \circledFour \overset{\eqref{eq:bound_6_SGDA_2_prox_clipped}}{\leq} \frac{\exp\left(-\gamma\mu T\right)V}{8},\quad
        \sum\limits_{l=0}^{T-1}\sum\limits_{i=1}^n\widetilde\sigma_{i,l}^2 \overset{\eqref{eq:bound_5_variances_SGDA_2_prox_clipped}}{\leq} \frac{\exp\left(-2\gamma\mu T\right)V}{384\ln\frac{48n(K+1)}{\beta}},\\
        \sum\limits_{i=1}^{m-1}(\widetilde\sigma_{i,T-1}')^2 \leq \frac{\exp\left(-2\gamma\mu T\right)V^2}{384\ln\frac{48n(K+1)}{\beta}}.
    \end{gather*}
In addition, we also establish (see \eqref{eq:bound_5_SGDA_2_prox_clipped}, \eqref{eq:bound_1_str_mon_SGDA_eight'} and our induction assumption)
    \begin{gather*}
        \PP\{E_{T-1}\cap \widetilde{E}_{T-1,m-1}\} \geq 1 - \frac{(T-1)\beta}{K+1} - \frac{(m-1)\beta}{8n(K+1)},\\
        \PP\{E_{\circledThree}\} \geq 1 - \frac{\beta}{24n(K+1)}, \quad \PP\{E_{\circledSeven}\} \geq 1 - \frac{\beta}{24n(K+1)}
    \end{gather*}
    where
    \begin{eqnarray}
        E_{\circledThree} &=& \left\{\text{either } \sum\limits_{l=0}^{T-1}\sum\limits_{i=1}^n\widetilde\sigma_{i,l}^2 > \frac{\exp\left(-2\gamma\mu T\right)V^2}{384\ln\frac{48n(K+1)}{\beta}} \text{ or } |\circledThree| \leq \frac{\exp\left(-\gamma\mu T\right)V}{8}\right\}, \notag\\
        E_{\circledSeven} &=& \left\{\text{either } \sum\limits_{i=1}^n(\widetilde\sigma_{i,T-1}')^2 > \frac{\exp\left(-2\gamma\mu T\right)V^2}{384\ln\frac{48n(K+1)}{\beta}} \text{ or } |\circledSeven| > \frac{\exp\left(-\gamma\mu T\right)V}{8}\right\}. \notag
    \end{eqnarray}
    Therefore, probability event $E_{T-1} \cap \widetilde{E}_{m-1} \cap E_{\circledThree} \cap E_{\circledSeven}$ implies
    \begin{gather*}
        \left\|\frac{\gamma}{n}\sum\limits_{i=1}^{m-1} \omega_{i,T-1}^u\right\| \leq \exp\left(-\frac{\gamma\mu(T-1)}{2}\right)\sqrt{V}\sqrt{\frac{1}{8} + \frac{1}{8} } \leq \frac{\exp\left(-\frac{\gamma\mu(T-1)}{2}\right)\sqrt{V}}{2}.
    \end{gather*}
This implies $\widetilde E_{T-1,m}$ and
    \begin{eqnarray*}
        \PP\{E_{T-1} \cap \widetilde{E}_{T-1,m}\} &\geq& \PP\{E_{T-1} \cap \widetilde{E}_{T-1,m-1} \cap E_{\circledSix} \cap E_{\circledSix'}\} \\
        &=& 1 - \PP\left\{\overline{E_{T-1} \cap \widetilde{E}_{T-1,m-1}}  \cup \overline{E}_{\circledSix} \cup \overline{E}_{\circledSix'}\right\}\\
        &\geq& 1 - \frac{(T-1)\beta}{K+1} - \frac{m\beta}{8n(K+1)}.
    \end{eqnarray*}
    Therefore, for all $m = 2,\ldots,n$ the statement holds and, in particular, $\PP\{E_{T-1} \cap \widetilde{E}_{T-1, n}\} \geq 1 - \frac{(T-1)\beta}{K+1} - \frac{\beta}{8(K+1)}$, i.e., \eqref{eq:induction_inequality_str_mon_SGDA_2} holds. Taking into account \eqref{eq:SGDA_new}, we conclude that $E_{T-1} \cap \widetilde{E}_{T-1, n} \cap E_{\circledOne} \cap E_{\circledThree} \cap E_{\circledFive} \cap E_{\circledEight'}$ implies
    \begin{equation*}
        V_T \leq 2\exp\left(-\gamma\mu T\right)V
    \end{equation*}
    that is equivalent to \eqref{eq:induction_inequality_prox_clipped_SGDA_2} for $t = T$. Moreover,
    \begin{eqnarray*}
        \PP\left\{E_T\right\} &\geq& \PP\left\{E_{T-1} \cap \widetilde{E}_{T-1,n} \cap E_{\circledOne} \cap E_{\circledThree} \cap E_{\circledSix'} \right\} \\
        &=& 1 - \PP\left\{\overline{E_{T-1} \cap \widetilde{E}_{n}} \cup \overline{E}_{\circledOne} \cup \overline{E}_{\circledThree} \cup \overline{E}_{\circledSix'}\right\}\\
        &=& 1 - \frac{(T-1)\beta}{K+1} - \frac{\beta}{8(K+1)} - 3\cdot \frac{\beta}{24n(K+1)} \geq 1 - \frac{T\beta}{K+1}.
    \end{eqnarray*}

    In other words, we showed that $\PP\{E_k\} \geq 1 - \nicefrac{k\beta}{(K+1)}$ for all $k = 0,1,\ldots,K+1$. For $k = K+1$ we have that with probability at least $1 - \beta$
    \begin{equation}
        \|x^{K+1} - x^*\|^2 \leq 2\exp(-\gamma\mu (K+1))V. \notag
    \end{equation}
    Finally, if 
    \begin{eqnarray*}
        \gamma &=& \min\left\{\frac{1}{4096 \ell \ln \tfrac{48n(K+1)}{\beta}}, \revision{\frac{\sqrt{n}R}{3000 \zeta_* \ln \tfrac{48n(K+1)}{\beta}},} \frac{\ln(B_K)}{\mu(K+1)}\right\}, \notag\\
        B_K &=& \max\left\{2, \left(\frac{\sqrt{2}}{3456}\right)^{\frac{2}{\alpha}}\cdot\frac{(K+1)^{\frac{2(\alpha-1)}{\alpha}}\mu^2V n^{\frac{2(\alpha-1)}{\alpha}}}{\sigma^2\ln^{\frac{2(\alpha-1)}{\alpha}}\left(\frac{48n(K+1)}{\beta}\right)\ln^2(B_K)} \right\}  \\
        &=& \cO\left(\max\left\{2, \frac{K^{\frac{2(\alpha-1)}{\alpha}}\mu^2V n^{\frac{2(\alpha-1)}{\alpha}}}{\sigma^2\ln^{\frac{2(\alpha-1)}{\alpha}}\left(\frac{nK}{\beta}\right)\ln^2\left(\max\left\{2, \frac{K^{\frac{2(\alpha-1)}{\alpha}}\mu^2V n^{\frac{2(\alpha-1)}{\alpha}}}{\sigma^2\ln^{\frac{2(\alpha-1)}{\alpha}}\left(\frac{nK}{\beta}\right)} \right\}\right)} \right\}\right)
    \end{eqnarray*}
    then with probability at least $1-\beta$
    \begin{eqnarray*}
        \|x^{K+1} - x^*\|^2 &\leq& 2\exp(-\gamma\mu (K+1))V\\
        &=& 2V\max\left\{\exp\left(-\frac{\mu(K+1)}{4096 \ell \ln \tfrac{48n(K+1)}{\beta}}\right), \revision{\exp\left(- \frac{\mu \sqrt{n} RK}{3000\zeta_* \ln \tfrac{48nK}{\beta}}\right),} \frac{1}{B_K} \right\}\\
        &\hspace{-5cm}=&\hspace{-2.5cm} \cO\left(\max\left\{R^2\exp\left(- \frac{\mu K}{\ell \ln \tfrac{nK}{\beta}}\right), \revision{R^2\exp\left(- \frac{\mu \sqrt{n} RK}{\zeta_* \ln \tfrac{nK}{\beta}}\right),} \frac{\sigma^2\ln^2B_K}{\ln^{\frac{2(1-\alpha)}{\alpha}}\left(\frac{nK}{\beta}\right)K^{\frac{2(\alpha-1)}{\alpha}}\mu^2 n^{\frac{2(\alpha-1)}{\alpha}}}\right\}\right).
    \end{eqnarray*}
    To get $\|x^{K+1} - x^*\|^2 \leq \varepsilon$ with probability $\geq 1 - \beta$, $K$ should be
    \begin{align*}
         K = \cO\Bigg(\max\Bigg\{\frac{\ell}{\mu}\ln\left(\frac{\revision{R^2}}{\varepsilon}\right)&\ln\left(\frac{n\ell}{\mu \beta}\ln\frac{\revision{R^2}}{\varepsilon}\right), \revision{\frac{\zeta_*}{\sqrt{n}R\mu}\ln\left(\frac{\revision{R^2}}{\varepsilon}\right)\ln\left(\frac{\sqrt{n}\zeta_*}{R\mu \beta}\ln\frac{\revision{R^2}}{\varepsilon}\right),}\notag\\
        &\frac{1}{n}\left(\frac{\sigma^2}{\mu^2\varepsilon}\right)^{\frac{\alpha}{2(\alpha-1)}}\ln \left(\frac{1}{\beta} \left(\frac{\sigma^2}{\mu^2\varepsilon }\right)^{\frac{\alpha}{2(\alpha-1)}}\right)\ln^{\frac{\alpha}{\alpha-1}}\left(B_\varepsilon\right)\Bigg\}\Bigg),
    \end{align*}
    where
    \begin{equation*}
        B_\varepsilon = \max\left\{2, \frac{\revision{2R^2}}{\varepsilon \ln \left(\frac{1}{\beta} \left(\frac{\sigma^2}{\mu^2\varepsilon}\right)^{\frac{\alpha}{2(\alpha-1)}}\right)}\right\}.
    \end{equation*}
    This concludes the proof.
\end{proof}

\clearpage

\section{Missing Proofs for \algname{DProx-clipped-SEG-shift}}\label{appendix:dprox_clipped_SEG_shift}

In this section, we give the complete formulations of our results for \algname{DProx-clipped-SEG-shift} and rigorous proofs. For the readers' convenience, the method's update rule is repeated below:
\begin{gather*}
    \tx^k = \prox_{\gamma\Psi}\left(x^k - \gamma \tg^k\right),\;\; \tg^k = \frac{1}{n}\sum\limits_{i=1}^n \tg_i^k,\;\; \tg_i^k = \thh_{i}^k + \tilde\Delta_i^k,\;\; \thh_i^{k+1} = \thh_i^k + \nu \tilde \Delta_i^k,\\
    x^{k+1} = \prox_{\gamma\Psi}\left(x^k - \gamma \hg^k\right),\;\; \hg^k = \frac{1}{n}\sum\limits_{i=1}^n \hg_i^k,\;\; \hg_i^k = \hh_{i}^k + \hat\Delta_i^k,\;\; \hh_i^{k+1} = \hh_i^k + \nu \hat \Delta_i^k, \\
    \tilde\Delta_i^k = \clip(F_{\xi_{1,i}^k}(x^k) - \thh_i^k, \lambda_k),\;\; \hat\Delta_i^k = \clip(F_{\xi_{2,i}^k}(\tx^k) - \hh_i^k, \lambda_k).
\end{gather*}

\subsection{Monotone Case}

The following lemma is the main ``optimization'' part of the analysis of \algname{DProx-clipped-SEG-shift}.

\begin{lemma}\label{lem:SEG_optimization_lemma_monotone}
    Let Assumptions \ref{as:L_Lip} and \ref{as:monotonicity} hold for $Q = B_{4n\sqrt{V}}(x^*)$, where $V \geq \|x^0 - x^*\|^2 + \frac{409600\gamma^2\ln^2\frac{48n(K+1)}{\beta}}{n^2}\sum_{i=1}^n\|F_i(x^*)\|^2$, and $0 < \gamma \leq \nicefrac{1}{\sqrt{12}L}$. If $x^k$ and $\tx^k$ lie in $B_{4n\sqrt{V}}(x^*)$ for all $k = 0,1,\ldots,K$ for some $K \geq 0$, then for all $u \in B_{4n\sqrt{V}}(x^*)$ the iterates produced by \algname{DProx-clipped-SEG-shift} satisfy
    \begin{eqnarray}
        \langle F(u), \tx^K_{\avg} - u\rangle + \Psi(\tx^K_{\avg}) - \Psi(u) &\leq& \frac{\|x^0 - u\|^2 - \|x^{K+1}-u\|^2}{2\gamma(K+1)} \notag\\
        &&\quad + \frac{\gamma}{K+1}\sum\limits_{k=0}^K (3\|\omega_k\|^2 + 4\|\theta_k\|^2) \notag\\
        &&\quad + \frac{1}{K+1}\sum\limits_{k=0}^K \langle \theta_k, x^k - u \rangle, \label{eq:optimization_lemma_SEG}\\
        \tx^K_{\avg} &\eqdef& \frac{1}{K+1}\sum\limits_{k=0}^{K}\tx^k, \label{eq:tx_avg_SEG}\\
        \theta_k &\eqdef& F(\tx^k) - \hg^k, \label{eq:theta_k_SEG}\\
        \omega_k &\eqdef& F(x^k) - \tg^k. \label{eq:omega_k_SEG}
    \end{eqnarray}
\end{lemma}
\begin{proof}
    Since $\tx^k = \prox_{\gamma\Psi}\left(x^k - \gamma \tg^k\right)$ and $x^{k+1} = \prox_{\gamma\Psi}\left(x^k - \gamma \hg^k\right)$, we have $x^k - \gamma \tg^k - \tx^k \in \gamma \partial \Psi(\tx^k)$ and $x^k - \gamma \hg^k - x^{k+1} \in \gamma \partial\Psi(x^{k+1})$. By definition of the subgradient, we have $\forall u \in \R^d$
    \begin{eqnarray}
        \gamma\left(\Psi(\tx^{k}) - \Psi(x^{k+1})\right) &\leq& \langle \tx^k - x^k + \gamma \tg^k, x^{k+1} - \tx^k \rangle, \notag\\
        \gamma\left(\Psi(x^{k+1}) - \Psi(u)\right) &\leq& \langle x^{k+1} - x^k + \gamma \hg^k, u - x^{k+1} \rangle. \notag
    \end{eqnarray}
    Summing up the above inequalities, we get
    \begin{eqnarray}
        \gamma\left(\Psi(\tx^k) - \Psi(u)\right) &\leq& \langle \tx^k - x^k, x^{k+1} - \tx^k \rangle + \langle x^{k+1} - x^k, u - x^{k+1} \rangle \notag\\
        &&\quad + \gamma \langle \tg^k - \hg^k, x^{k+1} - \tx^k \rangle + \gamma \langle \hg^k, u - \tx^k \rangle. \label{eq:ddmvklfkdvnfkvdn}
    \end{eqnarray}
    Since
    \begin{eqnarray}
        \langle \tx^k - x^k, x^{k+1} - \tx^k \rangle &=& \frac{1}{2}\|x^{k+1} - x^k\|^2 - \frac{1}{2}\|\tx^k - x^k\|^2 - \frac{1}{2}\|x^{k+1} - \tx^k\|^2, \notag\\
        \langle x^{k+1} - x^k, u - x^{k+1} \rangle &=& \frac{1}{2}\|x^k - u\|^2 - \frac{1}{2}\|x^{k+1} - x^k\|^2 - \frac{1}{2}\|x^{k+1} - u\|^2, \notag
    \end{eqnarray}
    we can rewrite \eqref{eq:ddmvklfkdvnfkvdn} as follows
    \begin{eqnarray}
        \gamma\left(\langle F(\tx^k), \tx^k - u \rangle +  \Psi(\tx^k) - \Psi(u)\right) &\leq& \frac{1}{2}\|x^k - u\|^2 - \frac{1}{2}\|x^{k+1} - u\|^2  - \frac{1}{2}\|\tx^k - x^k\|^2 \notag\\
        &&\quad - \frac{1}{2}\|x^{k+1} - \tx^k\|^2 + \gamma \langle \tg^k - \hg^k, x^{k+1} - \tx^k \rangle\notag\\
        &&\quad + \gamma \langle \theta_k, \tx^k - u \rangle. \label{eq:shvghaxsacshdbjchjd}
    \end{eqnarray}
    Next, we upper-bound $\gamma \langle \tg^k - \hg^k, x^{k+1} - \tx^k \rangle$ using Young's inequality, stating that $\langle a, b \rangle \leq \frac{1}{2\eta}\|a\|^2 + \frac{\eta}{2}\|b\|^2$ for all $a,b\in\R^d$ and $\eta > 0$, and Jensen's inequality for the squared norm:
    \begin{eqnarray}
        \gamma \langle \tg^k - \hg^k, x^{k+1} - \tx^k \rangle &\leq& \gamma^2 \|\tg^k - \hg^k\|^2 + \frac{1}{4}\|x^{k+1} - \tx^k\|^2 \notag\\
        &=& \gamma^2\|F(x^k) - F(\tx^k) - \omega_k + \theta_k\|^2 + \frac{1}{4}\|x^{k+1} - \tx^k\|^2\notag\\
        &\leq& 3\gamma^2 \|F(x^k) - F(\tx^k)\|^2 + 3\gamma^2 \|\omega_k\|^2 + 3\gamma^2\|\theta_k\|^2 + \frac{1}{4}\|x^{k+1} - \tx^k\|^2\notag\\
        &\overset{\eqref{eq:L_Lip}}{\leq}& 3\gamma^2 L^2 \|x^k - \tx^k\|^2 + 3\gamma^2 \|\omega_k\|^2 + 3\gamma^2\|\theta_k\|^2 + \frac{1}{4}\|x^{k+1} - \tx^k\|^2 .\label{eq:njksdnjnsdkjncikds}
    \end{eqnarray}
    Plugging \eqref{eq:njksdnjnsdkjncikds} in \eqref{eq:shvghaxsacshdbjchjd}, we derive for all $u \in \R^d$
    \begin{eqnarray}
        \gamma\left(\langle F(\tx^k), \tx^k - u \rangle +  \Psi(\tx^k) - \Psi(u)\right) &\leq& \frac{1}{2}\|x^k - u\|^2 - \frac{1}{2}\|x^{k+1} - u\|^2  \notag\\
        &&\quad - \frac{1}{2}\left(1 - 6\gamma^2 L^2\right)\|\tx^k - x^k\|^2 - \frac{1}{4}\|x^{k+1} - \tx^k\|^2 \notag\\
        &&\quad + 3\gamma^2 \|\omega_k\|^2 + 3\gamma^2\|\theta_k\|^2  + \gamma \langle \theta_k, \tx^k - u \rangle. \label{eq:important_optimization_ineq_SEG}
    \end{eqnarray}
    We notice that the above inequality does not rely on monotonicity. Next, we apply monotonicity and get that for all $u \in B_{4n\sqrt{V}}(x^*)$:
    \begin{eqnarray}
        \gamma \left(\langle F(u), \tx^k - u \rangle +  \Psi(\tx^k) - \Psi(u)\right) &\leq& \frac{1}{2}\|x^k - u\|^2 - \frac{1}{2}\|x^{k+1} - u\|^2  \notag\\
        &&\quad - \frac{1}{2}\left(1 - 6\gamma^2 L^2\right)\|\tx^k - x^k\|^2 + \gamma\langle \theta^k, \tx^k - x^k \rangle\notag\\
        &&\quad + 3\gamma^2 \|\omega_k\|^2 + 3\gamma^2\|\theta_k\|^2  + \gamma \langle \theta_k, x^k - u \rangle \notag\\
        &\leq& \frac{1}{2}\|x^k - u\|^2 - \frac{1}{2}\|x^{k+1} - u\|^2  \notag\\
        &&\quad - \frac{1}{2}\left(\frac{1}{2} - 6\gamma^2 L^2\right)\|\tx^k - x^k\|^2 \notag\\
        &&\quad + 3\gamma^2 \|\omega_k\|^2 + 4\gamma^2\|\theta_k\|^2  + \gamma \langle \theta_k, x^k - u \rangle, \notag
    \end{eqnarray}
    where in the last step we apply $\gamma\langle \theta^k, \tx^k - x^k \rangle \leq \gamma^2\|\theta^k\|^2 + \frac{1}{4}\|\tx^k - x^k\|^2$. Since $\gamma \leq \nicefrac{1}{\sqrt{12}L}$, we have
    \begin{eqnarray}
        \gamma \left(\langle F(u), \tx^k - u \rangle +  \Psi(\tx^k) - \Psi(u)\right) &\leq& \frac{1}{2}\|x^k - u\|^2 - \frac{1}{2}\|x^{k+1} - u\|^2  \notag\\
        &&\quad + 3\gamma^2 \|\omega_k\|^2 + 4\gamma^2\|\theta_k\|^2  + \gamma \langle \theta_k, x^k - u \rangle, \notag
    \end{eqnarray}
    Summing up the above inequalities for $k = 0,1,\ldots, K$ and dividing both sides by $\gamma(K+1)$, we obtain
    \begin{eqnarray}
        \frac{1}{K+1}\sum\limits_{k=0}^{K+1}\left(\langle F(u), \tx^k - u \rangle +  \Psi(\tx^k) - \Psi(u)\right) &\leq& \frac{\|x^0 - u\|^2 - \|x^{K+1}-u\|^2}{2\gamma(K+1)} \notag\\
        &&\quad + \frac{\gamma}{K+1}\sum\limits_{k=0}^{K+1} (3\|\omega_k\|^2 + 4\|\theta_k\|^2) \notag\\
        &&\quad + \frac{1}{K+1}\sum\limits_{k=0}^{K+1} \langle \theta_k, x^k - u \rangle. \notag
    \end{eqnarray}
    Applying $\frac{1}{K+1}\sum\limits_{i=1}^n \langle F(u), \tx^k - u \rangle = \langle F(u), \tx^K_{\avg} - u \rangle$ and $\Psi(\tx^K_{\avg}) - \Psi(x^*) \leq \frac{1}{K+1}\sum\limits_{i=1}^n\Psi(\tx^k)$, we get the result.
\end{proof}

Next, we proceed with the full statement of our main result for \algname{DProx-clipped-SEG-shift} in the monotone case.

\begin{theorem}[Case 2 from Theorem~\ref{thm:D_prox_clipped_SEG_main}]\label{thm:D_prox_clipped_SEG_monotone_appendix}
    Let Assumptions \ref{as:bounded_alpha_moment}, \ref{as:L_Lip} and \ref{as:monotonicity} hold for $Q = B_{4n\sqrt{V}}(x^*)$, where $V \geq \|x^0 - x^*\|^2 + \frac{409600\gamma^2\ln^2 \frac{48n(K+1)}{\beta}}{n^2}\sum_{i=1}^n\|F_i(x^*)\|^2$ and
    \begin{eqnarray}
        \gamma &\leq& \min\left\{\frac{1}{1920 L \ln \frac{48n(K+1)}{\beta}}, \frac{60^{\frac{2-\alpha}{\alpha}}\sqrt{V}n^{\frac{\alpha-1}{\alpha}}}{97200^{\frac{1}{\alpha}}(K+1)^{\frac{1}{\alpha}}\sigma \ln^{\frac{\alpha-1}{\alpha}} \frac{48n(K+1)}{\beta}}\right\}, \label{eq:gamma_SEG_monotone}\\
        \lambda_k \equiv \lambda &=& \frac{n\sqrt{V}}{60\gamma \ln\frac{48n(K+1)}{\beta}}, \label{eq:lambda_SEG_monotone}\\
        \nu &=& 0 \label{eq:nu_SEG_monotone}
    \end{eqnarray}
    for some $K \geq 1$ and $\beta \in (0,1]$. Then, after $K$ iterations of \algname{DProx-clipped-SEG-shift}, the following inequality holds with probability at least $1 - \beta$:
    \begin{equation}
        \gap_{\sqrt{V}}(\tx_{\avg}^K) \leq \frac{9V}{2\gamma(K+1)} \quad \text{and}\quad \{x^k\}_{k=0}^{K+1} \subseteq B_{3\sqrt{V}}(x^*), \{\tx^k\}_{k=0}^{K+1} \subseteq B_{4n\sqrt{V}}(x^*), \label{eq:main_result_SEG_monotone}
    \end{equation}
    where $\tx_{\avg}^K$ is defined in \eqref{eq:tx_avg_SEG}. In particular, when $\gamma$ equals the minimum from \eqref{eq:gamma_SEG_monotone}, then after $K$ iterations of \algname{DProx-clipped-SEG-shift}, we have with probability at least $1-\beta$
    \begin{equation}
        \gap_{\sqrt{V}}(\tx_{\avg}^K) = \cO\left(\max\left\{\frac{LV\ln\frac{nK}{\beta}}{K}, \frac{\sigma \sqrt{V} \ln^{\frac{\alpha-1}{\alpha}}\frac{nK}{\beta}}{n^{\frac{\alpha-1}{\alpha}}K^{\frac{\alpha-1}{\alpha}}}\right\}\right), \label{eq:clipped_SEG_monotone_rate_appendix}
    \end{equation}
    i.e., to achieve $\gap_{\sqrt{V}}(\tx_{\avg}^K) \leq \varepsilon$ with probability at least $1 - \beta$ \algname{DProx-clipped-SEG-shift} needs
    \begin{equation}
        K = \cO\left(\max\left\{\frac{LV}{\varepsilon}\ln\frac{nLV}{\varepsilon\beta}, \frac{1}{n}\left(\frac{\sigma \sqrt{V}}{\varepsilon}\right)^{\frac{\alpha}{\alpha-1}}\ln \frac{\sigma \sqrt{V}}{\varepsilon\beta}\right\}\right) \label{eq:clipped_SEG_monotone_complexity_appendix}
    \end{equation}
    iterations/oracle calls per worker.
\end{theorem}
\begin{proof}
    The key idea behind the proof is similar to the one used in \citep{gorbunov2022clipped, sadiev2023high}: we prove by induction that the iterates do not leave some ball and the sums decrease as $\nicefrac{1}{K+1}$. To formulate the statement rigorously, we introduce probability event $E_k$ for each $k = 0,1,\ldots, K+1$ as follows: inequalities
    \begin{gather}
        \underbrace{\max\limits_{u \in B_{\sqrt{V}}(x^*)}\left\{ \|x^0 - u\|^2 + 2\gamma \sum\limits_{l = 0}^{t-1} \langle x^l - u, \theta_l \rangle + \gamma^2 \sum\limits_{l=0}^{t-1}\left(8\|\theta_l\|^2 + 6\|\omega_l\|^2\right)\right\}}_{A_t} \leq 9V, \label{eq:induction_inequality_1_SEG}\\
        \left\|\gamma\sum\limits_{l=0}^{t-1}\theta_l\right\| \leq \sqrt{V}, \label{eq:induction_inequality_2_SEG}\\
        \left\|\frac{\gamma}{n}\sum\limits_{i=1}^{r-1}\theta_{i,t-1}^u\right\| \leq \frac{\sqrt{V}}{2}, \quad \left\|\frac{\gamma}{n}\sum\limits_{i=1}^{r-1}\omega_{i,t-1}^u\right\| \leq \frac{\sqrt{V}}{2}
    \end{gather}
    hold for $t = 0,1,\ldots,k$ and $r = 1, 2, \ldots, n$ simultaneously, where
    \begin{gather}
        \theta_l = \theta_l^u + \theta_l^b,\quad  \omega_l = \omega_l^u + \omega_l^b, \label{eq:gap_thm_SEG_technical_4_full_theta_omega}\\
        \theta_l^u \eqdef \frac{1}{n}\sum\limits_{i=1}^n \theta_{i,l}^u,\quad \theta_l^b \eqdef \frac{1}{n}\sum\limits_{i=1}^n \theta_{i,l}^b,\quad \omega_l^u \eqdef \frac{1}{n}\sum\limits_{i=1}^n \omega_{i,l}^u,\quad \omega_l^b \eqdef \frac{1}{n}\sum\limits_{i=1}^n \omega_{i,l}^b, \\
        \theta_{i,l}^u \eqdef \EE_{\xi_{2,i}^l}\left[\hg_{i}^l\right] - \hg_{i}^l,\quad \theta_{i,l}^b \eqdef F_i(\tx^l) - \EE_{\xi_{2,i}^l}\left[\hg_{i}^l\right] \quad \forall\; i\in [n], \label{eq:gap_thm_SEG_technical_4}\\
        \omega_{i,l}^u \eqdef \EE_{\xi_{1,i}^l}\left[\tg_{i}^l\right] - \tg_{i}^l,\quad \omega_{i,l}^b \eqdef F_i(x^l) - \EE_{\xi_{1,i}^l}\left[\hg_{i}^l\right] \quad \forall\; i\in [n]. \label{eq:gap_thm_SEG_technical_5}
    \end{gather}

    We will prove by induction that $\PP\{E_k\} \geq  1 - \nicefrac{k\beta}{(K+1)}$ for all $k = 0,1,\ldots,K+1$. The base of induction follows immediately: for all $u \in B_{\sqrt{V}}(x^*)$ we have $\|x^0 - u\|^2 \leq 2\|x^0 - x^*\|^2 + 2\|x^* - u\|^2 \leq 4V < 9V$ and for $k = 0$ we have $\|\gamma\sum_{l=0}^{k-1} \theta_l\| = 0$, $\|\frac{\gamma}{n}\sum_{i=1}^{r-1} \theta_{i,k-1}^u\| = \|\frac{\gamma}{n}\sum_{i=1}^{r-1} \omega_{i,k-1}^u\| = 0$ since $\theta_{i,-1}^u = \omega_{i,-1}^u = 0$. Next, we assume that the statement holds for $k = T-1 \leq K$, i.e., $\PP\{E_{T-1}\} \geq  1 - \nicefrac{(T-1)\beta}{(K+1)}$. Let us show that it also holds for $k = T$, i.e., $\PP\{E_{T}\} \geq  1 - \nicefrac{T\beta}{(K+1)}$. 
    
    To proceed, we need to show that $E_{T-1}$ implies $\|x^t - x^*\| \leq 3\sqrt{V}$ for all $t = 0,1,\ldots,T$. We will use the induction argument as well. The base is already proven. Next, we assume that $\|x^t - x^*\| \leq 3\sqrt{V}$ for all $t = 0,1,\ldots,t'$ for some $t' < T$. Then 
    \begin{eqnarray}
        \|F(x^*)\| &=& \sqrt{\|F(x^*)\|^2} \leq \sqrt{\sum\limits_{i=1}^n \|F_i(x^*)\|^2} \leq \frac{n\sqrt{V}}{160\gamma \ln \frac{48n(K+1)}{\beta}} < \lambda \label{eq:njscnjdnsdcibv}
    \end{eqnarray}
    and for $t = 0,1,\ldots, t'$
    \begin{eqnarray}
        \|\tx^t - x^*\| &=& \|\prox_{\gamma \Psi}(x^t - \gamma \tg^t) - \prox_{\gamma \Psi}(x^* - \gamma F(x^*))\| \notag\\
        &\leq& \|x^t - x^* - \gamma(\tg^t - F(x^*))\| \leq \|x^t - x^*\| + \gamma \|\tg^t - F(x^*)\|\notag\\
        &\leq& \|x^t - x^*\| + \gamma(\|\tg^k\| + \|F(x^*)\|)  \overset{\eqref{eq:lambda_SEG_monotone}, \eqref{eq:njscnjdnsdcibv}}{\leq} 3\sqrt{V} + 2\gamma\lambda \leq 3\sqrt{V} + \frac{n\sqrt{V}}{30\ln\frac{48(K+1)}{\beta}} \notag\\
        &\leq& 4n\sqrt{V}. \label{eq:gap_thm_SEG_technical_1}
    \end{eqnarray}
    This means that $x^t, \tx^t \in B_{4n\sqrt{V}}(x^*)$ for $t = 0,1,\ldots, t'$ and we can apply Lemma~\ref{lem:SEG_optimization_lemma_monotone}: $E_{T-1}$ implies
    \begin{align*}
        \max_{B_{\sqrt{V}}(x^*)}\Big\{ 2\gamma(t'+1)\Big(\langle F(u), \tx^{t'}_{\avg} - u\rangle &+ \Psi(\tx^{t'}_{\avg}) - \Psi(u)\Big) + \|x^{t'+1} - u\|^2  \Big\}\\
        &\leq \max_{B_{\sqrt{V}}(x^*)}\left\{\|x^0 - u\|^2 + 2\gamma \sum\limits_{l = 0}^{t-1} \langle x^l - u, \theta_l \rangle \right\}\\
        &\quad + \gamma^2 \sum\limits_{l=0}^{t-1}\left(8\|\theta_l\|^2 + 6\|\omega_l\|^2\right)\\
        &\overset{\eqref{eq:induction_inequality_1_SEG}}{\leq} 9V
    \end{align*}
    that gives
    \begin{eqnarray*}
        \|x^{t'+1} - x^*\|^2 &\leq& \max_{B_{\sqrt{V}}(x^*)}\left\{ 2\gamma(t'+1)\left(\langle F(u), \tx^{t'}_{\avg} - u\rangle + \Psi(\tx^{t'}_{\avg}) - \Psi(u)\right) + \|x^{t'+1} - u\|^2  \right\}\\
        &\leq& 9V.
    \end{eqnarray*}
    That is, we showed that $E_{T-1}$ implies $\|x^t - x^*\| \leq 3\sqrt{V}$, $\|\tx^t - x^*\| \leq 4n\sqrt{V}$ and
    \begin{equation}
        \max_{B_{\sqrt{V}}(x^*)}\left\{ 2\gamma(t+1)\left(\langle F(u), \tx^{t}_{\avg} - u\rangle + \Psi(\tx^{t}_{\avg}) - \Psi(u)\right) + \|x^{t+1} - u\|^2  \right\} \leq 9V \label{eq:gap_thm_SEG_technical_1_5}
    \end{equation}
    for all $t = 0, 1, \ldots, T$. Before we proceed, we introduce a new notation:
    \begin{equation*}
        \eta_t = \begin{cases}x^t - x^*,& \text{if } \|x^t - x^*\| \leq 3\sqrt{V},\\ 0,& \text{otherwise,} \end{cases}
    \end{equation*}
    for all $t = 0, 1, \ldots, T$. Random vectors $\{\eta_t\}_{t=0}^T$ are bounded almost surely:
     \begin{equation}
        \|\eta_t\| \leq 3\sqrt{V}  \label{eq:gap_thm_SEG_technical_3}
    \end{equation}
    for all $t = 0, 1, \ldots, T$. In addition, $\eta_t = x^t - x^*$ follows from $E_{T-1}$ for all $t = 0, 1, \ldots, T$ and, thus, $E_{T-1}$ implies
    \begin{eqnarray}
        A_T &\overset{\eqref{eq:induction_inequality_1_SEG}}{=}& \max\limits_{u \in B_{\sqrt{V}}(x^*)}\left\{ \|x^0 - u\|^2 + 2\gamma \sum\limits_{l = 0}^{T-1} \langle x^* - u, \theta_l \rangle \right\}  + 2\gamma \sum\limits_{l = 0}^{T-1} \langle x^l - x^*, \theta_l \rangle \notag\\
        &&\quad + \gamma^2 \sum\limits_{l=0}^{T-1}\left(8\|\theta_l\|^2 + 6\|\omega_l\|^2\right)\notag\\
        &\leq& 4V + 2\gamma\max\limits_{u \in B_{\sqrt{V}}(x^*)}\left\{\left\langle x^* - u, \sum\limits_{l = 0}^{T-1}\theta_l \right\rangle \right\} + 2\gamma \sum\limits_{l = 0}^{T-1} \langle \eta_l, \theta_l \rangle + \gamma^2 \sum\limits_{l=0}^{T-1}\left(8\|\theta_l\|^2 + 6\|\omega_l\|^2\right)\notag\\
        &=& 4V + 2\gamma \sqrt{V} \left\|\sum\limits_{l = 0}^{T-1}\theta_l\right\| + 2\gamma \sum\limits_{l = 0}^{T-1} \langle \eta_l, \theta_l \rangle + \gamma^2 \sum\limits_{l=0}^{T-1}\left(8\|\theta_l\|^2 + 6\|\omega_l\|^2\right). \notag
    \end{eqnarray}

    Using the notation from \eqref{eq:gap_thm_SEG_technical_4_full_theta_omega}-\eqref{eq:gap_thm_SEG_technical_5}, we can rewrite $\|\theta_l\|^2$ as
    \begin{eqnarray}
        \|\theta_l\|^2 &\leq&  2\|\theta_l^u\|^2 + 2\|\theta_l^b\|^2 = \frac{2}{n^2}\left\|\sum\limits_{i=1}^n \theta_{i,l}^u\right\|^2 + 2\|\theta_l^b\|^2\notag \\
        &=& \frac{2}{n^2}\sum\limits_{i=1}^n\|\theta_{i,l}^u\|^2 + \frac{4}{n^2}\sum\limits_{j=2}^n\left\langle \sum\limits_{i=1}^{j-1} \theta_{i,l}^u, \theta_{j,l}^u \right\rangle + 2\|\theta_l^b\|^2 \label{eq:SEG_sum_theta_representation}
    \end{eqnarray}
    and, similarly,
    it holds for $\|\omega_l\|^2$.
    Putting all together, we obtain that $E_{T-1}$ implies
    \begin{eqnarray}
        A_T &\leq&4V + 2\gamma \sqrt{V} \left\|\sum\limits_{l = 0}^{T-1}\theta_l\right\| + \underbrace{\frac{2\gamma}{n} \sum\limits_{l = 0}^{T-1}\sum\limits_{i=1}^n \langle \eta_l, \theta_{i,l}^u \rangle}_{\circledOne} + \underbrace{2\gamma \sum\limits_{l = 0}^{T-1} \langle \eta_l, \theta_l^b \rangle}_{\circledTwo}\notag\\
        &&\quad + \underbrace{\frac{2\gamma^2}{n^2} \sum\limits_{l=0}^{T-1} \sum\limits_{i=1}^n\left(8\EE_{\xi_{2,i}^l}\left[\|\theta_{i,l}^u\|^2\right] + 6\EE_{\xi_{1,i}^l}\left[\|\omega_{i,l}^u\|^2\right]\right)}_{\circledThree} \notag\\
        &&\quad + \underbrace{\frac{2\gamma^2}{n^2} \sum\limits_{l=0}^{T-1} \sum\limits_{i=1}^n\left(8\|\theta_{i,l}^u\|^2 + 6\|\omega_{i,l}^u\|^2 - 8\EE_{\xi_{2,i}^l}\left[\|\theta_{i,l}^u\|^2\right] - 6\EE_{\xi_{1,i}^l}\left[\|\omega_{i,l}^u\|^2\right]\right)}_{\circledFour}\notag\\
        &&\quad + \underbrace{2\gamma^2 \sum\limits_{l=0}^{T-1}\left(8\|\theta_l^b\|^2 + 6\|\omega_l^b\|^2\right)}_{\circledFive} + \underbrace{\frac{32\gamma^2}{n^2} \sum\limits_{l=0}^{T-1} \sum\limits_{j=2}^n \left\langle \sum\limits_{i=1}^{j-1} \theta_{i,l}^u, \theta_{j,l}^u \right\rangle}_{\circledSix}\notag\\
        &&\quad + \underbrace{\frac{24\gamma^2}{n^2} \sum\limits_{l=0}^{T-1} \sum\limits_{j=2}^n \left\langle \sum\limits_{i=1}^{j-1} \omega_{i,l}^u, \omega_{j,l}^u \right\rangle}_{\circledSeven}.\label{eq:gap_thm_SEG_technical_6}
    \end{eqnarray}
    To finish the proof, it remains to estimate $2\gamma \sqrt{V} \left\|\sum_{l = 0}^{T-1}\theta_l\right\|, \circledOne, \circledTwo, \circledThree, \circledFour, \circledFive, \circledSix, \circledSeven$ with high probability. More precisely, the goal is to prove that $2\gamma \sqrt{V} \left\|\sum_{l = 0}^{T-1}\theta_l\right\| + \circledOne + \circledTwo + \circledThree + \circledFour + \circledFive + \circledSix + \circledSeven \leq 5V$ with high probability. Before we proceed, we need to derive several useful inequalities related to $\theta_{i,l}^u, \omega_{i,l}^u, \theta_{l}^b, \omega_{l}^b$. First of all, we have
    \begin{equation}
        \|\theta_{i,l}^u\| \leq 2\lambda,\quad \|\omega_{i,l}^u\| \leq 2\lambda \label{eq:theta_omega_magnitude}
    \end{equation}
    by definition of the clipping operator. Next, probability event $E_{T-1}$ implies
    \begin{eqnarray*}
        \|F_i(x^l)\| &\leq& \|F_i(x^l) - F_i(x^*)\| + \|F_i(x^*)\| \overset{\eqref{eq:L_Lip}}{\leq} L\|x^l - x^*\| + \sqrt{\sum\limits_{i=1}^n \|F_i(x^*)\|^2} \\
        &\leq& 3L\sqrt{V} + \frac{n\sqrt{V}}{160\gamma \ln \frac{48n(K+1)}{\beta}} \overset{\eqref{eq:gamma_SEG_monotone}}{\leq} \frac{n\sqrt{V}}{120\gamma \ln\tfrac{48n(K+1)}{\beta}} \overset{\eqref{eq:lambda_SEG_monotone}}{=} \frac{\lambda}{2},\\
        \|F_i(\tx^l)\| &\leq& \|F_i(\tx^l) - F_i(x^*)\| + \|F_i(x^*)\| \overset{\eqref{eq:L_Lip}}{\leq} L\|\tx^l - x^*\| + \sqrt{\sum\limits_{i=1}^n \|F_i(x^*)\|^2} \\
        &\leq& 4Ln\sqrt{V} + \frac{n\sqrt{V}}{160\gamma \ln \frac{48n(K+1)}{\beta}} \overset{\eqref{eq:gamma_SEG_monotone}}{\leq} \frac{n\sqrt{V}}{120\gamma \ln\tfrac{48n(K+1)}{\beta}} \overset{\eqref{eq:lambda_SEG_monotone}}{=} \frac{\lambda}{2}
    \end{eqnarray*}
    for $l = 0,1,\ldots,T-1$ and $i\in [n]$. Therefore, Lemma~\ref{lem:bias_and_variance_clip} and $E_{T-1}$ imply
    \begin{gather}
        \left\|\theta_l^b\right\| \leq \frac{1}{n}\sum\limits_{i=1}^n \|\theta_{i,l}^b\| \leq \frac{2^\alpha\sigma^\alpha}{\lambda^{\alpha-1}},\quad \left\|\omega_l^b\right\| \leq \frac{1}{n}\sum\limits_{i=1}^n \|\omega_{i,l}^b\| \leq \frac{2^\alpha\sigma^\alpha}{\lambda^{\alpha-1}}, \label{eq:bias_theta_omega}\\
        \EE_{\xi_{2,i}^l}\left[\left\|\theta_{i,l}^u\right\|^2\right] \leq 18 \lambda^{2-\alpha}\sigma^\alpha,\quad \EE_{\xi_{1,i}^l}\left[\left\|\omega_{i,l}^u\right\|^2\right] \leq 18 \lambda^{2-\alpha}\sigma^\alpha, \label{eq:variance_theta_omega}
    \end{gather}
    for all $l = 0,1, \ldots, T-1$ and $i\in [n]$.

    \paragraph{Upper bound for $\circledOne$.} To estimate this sum, we will use Bernstein's inequality. The summands have conditional expectations equal to zero:
    \begin{equation*}
        \EE_{\xi_{2,i}^l}\left[\frac{2\gamma}{n} \langle \eta_l, \theta_{i,l}^u \rangle\right] = \frac{2\gamma}{n} \left\langle \eta_l, \EE_{\xi_{2,i}^l}[\theta_{i,l}^u] \right\rangle = 0.
    \end{equation*}
    Moreover, for all $l = 0,\ldots, T-1$ random vectors $\{\theta_{i,l}^u\}_{i=1}^n$ are independent. Thus, sequence $\left\{\frac{2\gamma}{n} \langle \eta_l, \theta_{i,l}^u \rangle\right\}_{l,i = 0,1}^{T-1,n}$ is a martingale difference sequence. Next, the summands are bounded:
    \begin{eqnarray}
        \left|\frac{2\gamma}{n}\langle \eta_l, \theta_{i,l}^u \rangle \right| \leq \frac{2\gamma}{n} \|\eta_l\|\cdot \|\theta_{i,l}^u\| \overset{\eqref{eq:gap_thm_SEG_technical_3},\eqref{eq:theta_omega_magnitude}}{\leq} \frac{12\gamma}{n} \sqrt{V}\lambda \overset{\eqref{eq:lambda_SEG_monotone}}{\leq} \frac{3V}{10\ln\frac{48n(K+1)}{\beta}} \eqdef c. \label{eq:gap_thm_SEG_technical_6_5}
    \end{eqnarray}
    Finally, conditional variances $\sigma_{i,l}^2 \eqdef \EE_{\xi_{2,i}^l}\left[\frac{4\gamma^2}{n^2} \langle \eta_l, \theta_{i,l}^u \rangle^2\right]$ of the summands are bounded:
    \begin{equation}
        \sigma_{i,l}^2 \leq \EE_{\xi_{2,i}^l}\left[\frac{4\gamma^2 }{n^2}\|\eta_l\|^2\cdot \|\theta_{i,l}^u\|^2\right] \overset{\eqref{eq:gap_thm_SEG_technical_3}}{\leq} \frac{36\gamma^2 V}{n^2} \EE_{\xi_{2,i}^l}\left[\|\theta_{i,l}^u\|^2\right]. \label{eq:gap_thm_SEG_technical_7}
    \end{equation}
    Applying Bernstein's inequality (Lemma~\ref{lem:Bernstein_ineq}) with $X_{i,l} = \frac{2\gamma}{n} \langle \eta_l, \theta_{i,l}^u \rangle$, constant $c$ defined in \eqref{eq:gap_thm_SEG_technical_6_5}, $b = \frac{3V}{10}$, $G = \frac{3V^2}{200\ln\frac{48n(K+1)}{\beta}}$, we get
    \begin{equation*}
        \PP\left\{|\circledOne| > \frac{3V}{10} \text{ and } \sum\limits_{l=0}^{T-1}\sum\limits_{i=1}^n\sigma_{i,l}^2 \leq \frac{3V^2}{200\ln\frac{48n(K+1)}{\beta}}\right\} \leq 2\exp\left(- \frac{b^2}{2G + \nicefrac{2cb}{3}}\right) = \frac{\beta}{24n(K+1)}.
    \end{equation*}
    The above is equivalent to 
    \begin{equation}
        \PP\{E_{\circledOne}\} \geq 1 - \frac{\beta}{24n(K+1)},\; \text{ for }\; E_{\circledOne} = \left\{\text{either } \sum\limits_{l=0}^{T-1}\sum\limits_{i=1}^n\sigma_{i,l}^2 > \frac{3V^2}{200\ln\frac{48n(K+1)}{\beta}} \text{ or } |\circledOne| \leq \frac{3V}{10}\right\}. \label{eq:bound_1_gap_SEG}
    \end{equation}
    Moreover, $E_{T-1}$ implies
    \begin{eqnarray}
        \sum\limits_{l=0}^{T-1}\sum\limits_{i=1}^n\sigma_{i,l}^2 &\overset{\eqref{eq:gap_thm_SEG_technical_7}}{\leq}& \frac{36\gamma^2 V}{n^2} \sum\limits_{l=0}^{T-1} \sum\limits_{i=1}^n \EE_{\xi_{2,i}^l}\left[\|\theta_{i,l}^u\|^2\right] \overset{\eqref{eq:variance_theta_omega}, T \leq K+1}{\leq} \frac{648(K+1)\gamma^2 V \lambda^{2-\alpha} \sigma^\alpha}{n} \notag \\
        &\overset{\eqref{eq:lambda_SEG_monotone}}{\leq}& \frac{648 (K+1)\gamma^{\alpha}\sigma^\alpha V^{2-\frac{\alpha}{2}}}{60^{2-\alpha} n^{\alpha-1} \ln^{2-\alpha}\frac{48n(K+1)}{\beta}} \overset{\eqref{eq:gamma_SEG_monotone}}{\leq} \frac{3V^2}{200\ln\frac{48n(K+1)}{\beta}}. \label{eq:bound_1_variances_gap_SEG}
    \end{eqnarray}

    \paragraph{Upper bound for $\circledTwo$.} Probability event $E_{T-1}$ implies
    \begin{eqnarray}
        \circledTwo &\leq& 2\gamma \sum\limits_{l=0}^{T-1}\|\eta_l\| \cdot \|\theta_l^b\| \overset{\eqref{eq:gap_thm_SEG_technical_3}, \eqref{eq:bias_theta_omega}, T \leq K+1}{\leq} \frac{6\cdot 2^\alpha(K+1)\gamma \sqrt{V}\sigma^\alpha}{\lambda^{\alpha-1}} \notag \\
        &\overset{\eqref{eq:lambda_SEG_monotone}}{=}& \frac{6\cdot 2^\alpha\cdot 60^{\alpha-1} (K+1)\gamma^\alpha \sigma^\alpha \ln^{\alpha-1}\frac{48n(K+1)}{\beta}}{n^{\alpha-1}V^{\frac{\alpha}{2}-1}} \overset{\eqref{eq:gamma_SEG_monotone}}{\leq} \frac{3V}{100}. \label{eq:bound_2_variances_gap_SEG}
    \end{eqnarray}

    \paragraph{Upper bound for $\circledThree$.} Probability event $E_{T-1}$ implies
    \begin{eqnarray}
        \frac{16\gamma^2}{n^2} \sum\limits_{l=0}^{T-1}\sum\limits_{i=1}^n\EE_{\xi_{2,i}^l}[\|\theta_{i,l}^u\|^2] &\overset{\eqref{eq:variance_theta_omega}, T \leq K+1}{\leq}& \frac{288\gamma^2 (K+1)\lambda^{2-\alpha}\sigma^\alpha}{n} \overset{\eqref{eq:lambda_SEG_monotone}}{=} \frac{288\gamma^\alpha (K+1)\sigma^\alpha V^{1-\frac{\alpha}{2}}}{60^{2-\alpha} n^{\alpha-1} \ln^{2-\alpha} \frac{48n(K+1)}{\beta}}\notag\\
        &\overset{\eqref{eq:gamma_SEG_monotone}}{\leq}& \frac{3}{100} V, \label{eq:sum_theta_squared_bound_gap_SEG}\\
        \frac{12\gamma^2}{n^2} \sum\limits_{l=0}^{T-1}\sum\limits_{i=1}^n\EE_{\xi_{1,i}^l}[\|\omega_{i,l}^u\|^2] &\overset{\eqref{eq:variance_theta_omega}, T \leq K+1}{\leq}& \frac{216\gamma^2 (K+1)\lambda^{2-\alpha}\sigma^\alpha}{n} \overset{\eqref{eq:lambda_SEG_monotone}}{=} \frac{216\gamma^\alpha (K+1)\sigma^\alpha V^{1-\frac{\alpha}{2}}}{60^{2-\alpha} n^{\alpha-1} \ln^{2-\alpha} \frac{48n(K+1)}{\beta}}\notag\\
        &\overset{\eqref{eq:gamma_SEG_monotone}}{\leq}& \frac{1}{50} V, \label{eq:sum_omega_squared_bound_gap_SEG}\\
        \circledThree &\overset{\eqref{eq:sum_theta_squared_bound_gap_SEG}, \eqref{eq:sum_omega_squared_bound_gap_SEG}}{\leq}& \frac{1}{20}V. \label{eq:bound_3_variances_gap_SEG}
    \end{eqnarray}

    \paragraph{Upper bound for $\circledFour$.} 
    To estimate this sum, we will use Bernstein's inequality. The summands have conditional expectations equal to zero:
    \begin{equation*}
        \frac{2\gamma^2}{n^2}\EE_{\xi_{1,i}^l,\xi_{2,i}^l}\left[8\|\theta_{i,l}^u\|^2 + 6\|\omega_{i,l}^u\|^2 - 8\EE_{\xi_{2,i}^l}\left[\|\theta_{i,l}^u\|^2\right] - 6\EE_{\xi_{1,i}^l}\left[\|\omega_{i,l}^u\|^2\right]\right] = 0.
    \end{equation*}
    Moreover, for all $l = 0,\ldots, T-1$ random vectors $\{\theta_{i,l}^u\}_{i=1}^n$, $\{\omega_{i,l}^u\}_{i=1}^n$ are independent. Thus, sequence $\left\{\frac{2\gamma^2}{n^2}\left(8\|\theta_{i,l}^u\|^2 + 6\|\omega_{i,l}^u\|^2 - 8\EE_{\xi_{2,i}^l}\left[\|\theta_{i,l}^u\|^2\right] - 6\EE_{\xi_{1,i}^l}\left[\|\omega_{i,l}^u\|^2\right]\right)\right\}_{l,i = 0,1}^{T-1,n}$ is a martingale difference sequence. Next, the summands are bounded:
    \begin{align}
        \frac{2\gamma^2}{n^2}\Big|8\|\theta_{i,l}^u\|^2 + 6\|\omega_{i,l}^u\|^2 &- 8\EE_{\xi_{2,i}^l}\left[\|\theta_{i,l}^u\|^2\right] - 6\EE_{\xi_{1,i}^l}\left[\|\omega_{i,l}^u\|^2\right] \Big| \notag\\
        &\leq \frac{16\gamma^2}{n^2} \left(\|\theta_{i,l}^u\|^2 + \EE_{\xi_{2,i}^l}\left[\|\theta_{i,l}^u\|^2\right]\right) + \frac{12\gamma^2}{n^2} \left(\|\omega_{i,l}^u\|^2 + \EE_{\xi_{1,i}^l}\left[\|\omega_{i,l}^u\|^2\right]\right) \notag\\
        &\overset{\eqref{eq:theta_omega_magnitude}}{\leq} \frac{224\gamma^2 \lambda^2}{n^2} \notag\\
        &\overset{\eqref{eq:lambda_SEG_monotone}}{\leq} \frac{V}{6\ln\frac{48n(K+1)}{\beta}} \eqdef c.\label{eq:gap_thm_SEG_technical_6_5_1}
    \end{align}
    Finally, conditional variances 
    $$\widetilde\sigma_{i,l}^2 \eqdef \frac{4\gamma^4}{n^4}\EE_{\xi_{1,i}^l, \xi_{2,i}^l}\left[\left|8\|\theta_{i,l}^u\|^2 + 6\|\omega_{i,l}^u\|^2 - 8\EE_{\xi_{2,i}^l}\left[\|\theta_{i,l}^u\|^2\right] - 6\EE_{\xi_{1,i}^l}\left[\|\omega_{i,l}^u\|^2\right] \right|^2\right]$$ 
    of the summands are bounded:
    \begin{eqnarray}
        \widetilde\sigma_{i,l}^2 &\overset{\eqref{eq:gap_thm_SEG_technical_6_5_1}}{\leq}& \frac{\gamma^2V}{3n^2\ln\frac{48n(K+1)}{\beta}}\EE_{\xi_{1,i}^l, \xi_{2,i}^l}\left[\left|8\|\theta_{i,l}^u\|^2 + 6\|\omega_{i,l}^u\|^2 - 8\EE_{\xi_{2,i}^l}\left[\|\theta_{i,l}^u\|^2\right] - 6\EE_{\xi_{1,i}^l}\left[\|\omega_{i,l}^u\|^2\right] \right|\right]\notag\\
        &\leq& \frac{4\gamma^2V}{3n^2\ln\frac{48n(K+1)}{\beta}} \EE_{\xi_{1,i}^l, \xi_{2,i}^l}\left[4\|\theta_{i,l}^u\|^2 + 3\|\omega_{i,l}^u\|^2 \right]. \label{eq:gap_thm_SEG_technical_7_1}
    \end{eqnarray}
    Applying Bernstein's inequality (Lemma~\ref{lem:Bernstein_ineq}) with $X_{i,l} = \frac{2\gamma^2}{n^2}\left(8\|\theta_{i,l}^u\|^2 + 6\|\omega_{i,l}^u\|^2 - 8\EE_{\xi_{2,i}^l}\left[\|\theta_{i,l}^u\|^2\right] - 6\EE_{\xi_{1,i}^l}\left[\|\omega_{i,l}^u\|^2\right]\right)$, constant $c$ defined in \eqref{eq:gap_thm_SEG_technical_6_5_1}, $b = \frac{V}{6}$, $G = \frac{V^2}{216\ln\frac{48n(K+1)}{\beta}}$, we get
    \begin{equation*}
        \PP\left\{|\circledFour| > \frac{V}{6} \text{ and } \sum\limits_{l=0}^{T-1}\sum\limits_{i=1}^n\widetilde\sigma_{i,l}^2 \leq \frac{V^2}{216\ln\frac{48n(K+1)}{\beta}}\right\} \leq 2\exp\left(- \frac{b^2}{2G + \nicefrac{2cb}{3}}\right) = \frac{\beta}{24n(K+1)}.
    \end{equation*}
    The above is equivalent to 
    \begin{equation}
        \PP\{E_{\circledFour}\} \geq 1 - \frac{\beta}{24n(K+1)},\; \text{ for }\; E_{\circledFour} = \left\{\text{either } \sum\limits_{l=0}^{T-1}\sum\limits_{i=1}^n\widetilde\sigma_{i,l}^2 > \frac{V^2}{216\ln\frac{48n(K+1)}{\beta}} \text{ or } |\circledFour| \leq \frac{V}{6}\right\}. \label{eq:bound_1_gap_SEG_1}
    \end{equation}
    Moreover, $E_{T-1}$ implies
    \begin{eqnarray}
        \sum\limits_{l=0}^{T-1}\sum\limits_{i=1}^n\widetilde\sigma_{i,l}^2 &\overset{\eqref{eq:gap_thm_SEG_technical_7_1}}{\leq}& \frac{4\gamma^2V}{3n^2\ln\frac{48n(K+1)}{\beta}} \sum\limits_{l=0}^{T-1}\sum\limits_{i=1}^n \EE_{\xi_{1,i}^l, \xi_{2,i}^l}\left[4\|\theta_{i,l}^u\|^2 + 3\|\omega_{i,l}^u\|^2 \right]\notag\\
        &\overset{\eqref{eq:variance_theta_omega}, T \leq K+1}{\leq}& \frac{168(K+1)\gamma^2 V \lambda^{2-\alpha} \sigma^\alpha}{n\ln\frac{48n(K+1)}{\beta}} \notag\\
        &\overset{\eqref{eq:lambda_SEG_monotone}}{\leq}& \frac{168(K+1)\gamma^\alpha V^{2-\frac{\alpha}{2}} \sigma^\alpha}{60^{2-\alpha}n^{\alpha-1}\ln^{3-\alpha}\frac{48n(K+1)}{\beta}}  \overset{\eqref{eq:gamma_SEG_monotone}}{\leq} \frac{V^2}{216\ln\frac{48n(K+1)}{\beta}}. \label{eq:bound_1_variances_gap_SEG_1}
    \end{eqnarray}

    \paragraph{Upper bound for $\circledFive$.} Probability event $E_{T-1}$ implies
    \begin{eqnarray}
        \circledFive &=& 2\gamma^2 \sum\limits_{l=0}^{T-1}\left(8\|\theta_l^b\|^2 + 6\|\omega_l^b\|^2\right) \overset{\eqref{eq:bias_theta_omega}, T\leq K+1}{\leq} \frac{28\cdot 2^{2\alpha}\gamma^2\sigma^{2\alpha} (K+1)}{\lambda^{2\alpha-2}}  \notag\\
        &\overset{\eqref{eq:lambda_SEG_monotone}}{=}&  \frac{28\cdot 2^{2\alpha}\cdot 60^{2\alpha - 2}\gamma^{2\alpha}\sigma^{2\alpha} (K+1) \ln^{2\alpha-2}\frac{48n(K+1)}{\beta}}{n^{2\alpha-2}V^{\alpha-1}} \overset{\eqref{eq:gamma_SEG_monotone}}{\leq} \frac{V}{6}. \label{eq:nsjcbdjhcbfjdhfbj}
    \end{eqnarray}

    \paragraph{Upper bounds for $\circledSix$ and $\circledSeven$.} These sums require more refined analysis. We introduce new vectors:
    \begin{equation}
        \zeta_{j}^l = \begin{cases}
            \frac{\gamma}{n}\sum\limits_{i=1}^{j-1} \theta_{i,l}^u,& \text{if } \left\|\frac{\gamma}{n}\sum\limits_{i=1}^{j-1} \theta_{i,l}^u\right\| \leq \frac{\sqrt{V}}{2},\\
            0,& \text{otherwise,}
        \end{cases}, \quad \delta_{j}^l = \begin{cases}
            \frac{\gamma}{n}\sum\limits_{i=1}^{j-1} \omega_{i,l}^u,& \text{if } \left\|\frac{\gamma}{n}\sum\limits_{i=1}^{j-1} \omega_{i,l}^u\right\| \leq \frac{\sqrt{V}}{2},\\
            0,& \text{otherwise,}
        \end{cases}\label{eq:SEG_zeta_delta_def}
    \end{equation}
    for all $j \in [n]$ and $l = 0,\ldots, T-1$. Then, by definition
    \begin{equation}
        \|\zeta_j^l\| \leq \frac{\sqrt{V}}{2},\quad \|\delta_j^l\| \leq \frac{\sqrt{V}}{2} \label{eq:SEG_zeta_delta_bound}
    \end{equation}
    and
    \begin{eqnarray}
        \circledSix &=& \underbrace{\frac{32\gamma}{n} \sum\limits_{l=0}^{T-1} \sum\limits_{j=2}^n \left\langle \zeta_j^l, \theta_{j,l}^u \right\rangle}_{\circledSix'} + \frac{32\gamma}{n} \sum\limits_{l=0}^{T-1} \sum\limits_{j=2}^n \left\langle \frac{\gamma}{n}\sum\limits_{i=1}^{j-1} \theta_{i,l}^u - \zeta_j^l, \theta_{j,l}^u \right\rangle, \label{eq:SEG_extra_sums_six_distributed} \\
        \circledSeven &=& \underbrace{\frac{24\gamma}{n} \sum\limits_{l=0}^{T-1} \sum\limits_{j=2}^n \left\langle \delta_j^l, \omega_{j,l}^u \right\rangle}_{\circledSeven'} + \frac{24\gamma}{n} \sum\limits_{l=0}^{T-1} \sum\limits_{j=2}^n \left\langle \frac{\gamma}{n}\sum\limits_{i=1}^{j-1} \omega_{i,l}^u - \delta_j^l, \omega_{j,l}^u \right\rangle. \label{eq:SEG_extra_sums_seven_distributed}
    \end{eqnarray}
    We also note here that $E_{T-1}$ implies
    \begin{eqnarray}
        \frac{32\gamma}{n} \sum\limits_{l=0}^{T-1} \sum\limits_{j=2}^n \left\langle \frac{\gamma}{n}\sum\limits_{i=1}^{j-1} \theta_{i,l}^u - \zeta_j^l, \theta_{j,l}^u \right\rangle &=& \frac{32\gamma}{n}  \sum\limits_{j=2}^n \left\langle \frac{\gamma}{n}\sum\limits_{i=1}^{j-1} \theta_{i,T-1}^u - \zeta_j^{T-1}, \theta_{j,T-1}^u \right\rangle, \label{eq:SEG_extra_sums_six_distributed_1}\\
        \frac{24\gamma}{n} \sum\limits_{l=0}^{T-1} \sum\limits_{j=2}^n \left\langle \frac{\gamma}{n}\sum\limits_{i=1}^{j-1} \omega_{i,l}^u - \delta_j^l, \omega_{j,l}^u \right\rangle &=& \frac{24\gamma}{n}  \sum\limits_{j=2}^n \left\langle \frac{\gamma}{n}\sum\limits_{i=1}^{j-1} \omega_{i,T-1}^u - \delta_j^{T-1}, \omega_{j,T-1}^u \right\rangle. \label{eq:SEG_extra_sums_seven_distributed_1}
    \end{eqnarray}
    
    \paragraph{Upper bound for $\circledSix'$.} To estimate this sum, we will use Bernstein's inequality. The summands have conditional expectations equal to zero:
    \begin{equation*}
    \EE_{\xi_{2,j}^l}\left[\frac{32\gamma}{n}\left\langle \zeta_j^l, \theta_{j,l}^u \right\rangle\right] = \frac{32\gamma}{n}\left\langle \zeta_j^l, \EE_{\xi_{2,j}^l}[\theta_{j,l}^u] \right\rangle = 0.
    \end{equation*}
    Moreover, for all $l = 0,\ldots, T-1$ random vectors $\{\theta_{i,l}^u\}_{i=1}^n$ are independent. Thus, sequence $\left\{\frac{32\gamma}{n}\left\langle \zeta_j^l, \theta_{j,l}^u \right\rangle\right\}_{l,j = 0,2}^{T-1,n}$ is a martingale difference sequence. Next, the summands are bounded:
    \begin{eqnarray}
        \left|\frac{32\gamma}{n}\left\langle \zeta_j^l, \theta_{j,l}^u \right\rangle \right| \leq \frac{32\gamma}{n} \left\|\zeta_j^l\right\| \cdot \|\theta_{j,l}^u\| \overset{\eqref{eq:SEG_zeta_delta_bound},\eqref{eq:theta_omega_magnitude}}{\leq} \frac{32\gamma}{n}\cdot \frac{\sqrt{V}}{2}\cdot 2\lambda \overset{\eqref{eq:lambda_SEG_monotone}}{\leq} \frac{4V}{5\ln\frac{48n(K+1)}{\beta}} \eqdef c. \label{eq:gap_thm_SEG_technical_6_5_123}
    \end{eqnarray}
    Finally, conditional variances $\widehat\sigma_{j,l}^2 \eqdef \EE_{\xi_{2,j}^l}\left[\frac{1024\gamma^2}{n^2} \langle \zeta_j^l, \theta_{j,l}^u \rangle^2\right]$ of the summands are bounded:
    \begin{equation}
        \widehat\sigma_{j,l}^2 \leq \EE_{\xi_{2,j}^l}\left[\frac{1024\gamma^2 }{n^2}\|\zeta_j^l\|^2\cdot \|\theta_{j,l}^u\|^2\right] \overset{\eqref{eq:SEG_zeta_delta_bound}}{\leq} \frac{256\gamma^2 V}{n^2} \EE_{\xi_{2,j}^l}\left[\|\theta_{j,l}^u\|^2\right]. \label{eq:gap_thm_SEG_technical_7_123}
    \end{equation}
    Applying Bernstein's inequality (Lemma~\ref{lem:Bernstein_ineq}) with $X_{i,l} = \frac{32\gamma}{n}\left\langle \zeta_j^l, \theta_{j,l}^u \right\rangle$, constant $c$ defined in \eqref{eq:gap_thm_SEG_technical_6_5_123}, $b = \frac{4V}{5}$, $G = \frac{8V^2}{75\ln\frac{48n(K+1)}{\beta}}$, we get
    \begin{equation*}
        \PP\left\{|\circledSix'| > \frac{4V}{5} \text{ and } \sum\limits_{l=0}^{T-1}\sum\limits_{j=2}^n\widehat\sigma_{i,l}^2 \leq \frac{8V^2}{75\ln\frac{48n(K+1)}{\beta}}\right\} \leq 2\exp\left(- \frac{b^2}{2G + \nicefrac{2cb}{3}}\right) = \frac{\beta}{24n(K+1)}.
    \end{equation*}
    The above is equivalent to 
    \begin{equation}
        \PP\{E_{\circledSix'}\} \geq 1 - \frac{\beta}{24n(K+1)},\; \text{ for }\; E_{\circledSix'} = \left\{\text{either } \sum\limits_{l=0}^{T-1}\sum\limits_{j=2}^n\widehat\sigma_{i,l}^2 > \frac{8V^2}{75\ln\frac{48n(K+1)}{\beta}} \text{ or } |\circledSix'| \leq \frac{4V}{5}\right\}. \label{eq:bound_6_gap_SEG}
    \end{equation}
    Moreover, $E_{T-1}$ implies
    \begin{eqnarray}
        \sum\limits_{l=0}^{T-1}\sum\limits_{j=2}^n\widehat\sigma_{j,l}^2 &\overset{\eqref{eq:gap_thm_SEG_technical_7_123}}{\leq}& \frac{256\gamma^2 V}{n^2} \sum\limits_{l=0}^{T-1} \sum\limits_{j=2}^n \EE_{\xi_{2,j}^l}\left[\|\theta_{j,l}^u\|^2\right] \overset{\eqref{eq:variance_theta_omega}, T \leq K+1}{\leq} \frac{4608(K+1)\gamma^2 V \lambda^{2-\alpha} \sigma^\alpha}{n} \notag \\
        &\overset{\eqref{eq:lambda_SEG_monotone}}{\leq}& \frac{4608 (K+1)\gamma^{\alpha}\sigma^\alpha V^{2-\frac{\alpha}{2}}}{40^{2-\alpha} n^{\alpha-1} \ln^{2-\alpha}\frac{48n(K+1)}{\beta}} \overset{\eqref{eq:gamma_SEG_monotone}}{\leq} \frac{8V^2}{75\ln\frac{48n(K+1)}{\beta}}. \label{eq:bound_6_variances_gap_SEG}
    \end{eqnarray}

    \paragraph{Upper bound for $\circledSeven'$.} To estimate this sum, we will use Bernstein's inequality. The summands have conditional expectations equal to zero:
    \begin{equation*}
    \EE_{\xi_{1,j}^l}\left[\frac{24\gamma}{n}\left\langle \delta_j^l, \omega_{j,l}^u \right\rangle\right] = \frac{24\gamma}{n}\left\langle \delta_j^l, \EE_{\xi_{1,j}^l}[\omega_{j,l}^u] \right\rangle = 0.
    \end{equation*}
    Moreover, for all $l = 0,\ldots, T-1$ random vectors $\{\omega_{i,l}^u\}_{i=1}^n$ are independent. Thus, sequence $\left\{\frac{24\gamma}{n}\left\langle \delta_j^l, \omega_{j,l}^u \right\rangle\right\}_{l,j = 0,2}^{T-1,n}$ is a martingale difference sequence. Next, the summands are bounded:
    \begin{eqnarray}
        \left|\frac{24\gamma}{n}\left\langle \delta_j^l, \omega_{j,l}^u \right\rangle \right| \leq \frac{24\gamma}{n} \left\|\delta_j^l\right\| \cdot \|\omega_{j,l}^u\| \overset{\eqref{eq:SEG_zeta_delta_bound},\eqref{eq:theta_omega_magnitude}}{\leq} \frac{24\gamma}{n}\cdot \frac{\sqrt{V}}{2}\cdot 2\lambda \overset{\eqref{eq:lambda_SEG_monotone}}{\leq} \frac{3V}{5\ln\frac{48n(K+1)}{\beta}} \eqdef c. \label{eq:gap_thm_SEG_technical_6_5_1234}
    \end{eqnarray}
    Finally, conditional variances $(\sigma_{j,l}')^2 \eqdef \EE_{\xi_{1,j}^l}\left[\frac{576\gamma^2}{n^2} \langle \delta_j^l, \omega_{j,l}^u \rangle^2\right]$ of the summands are bounded:
    \begin{equation}
        (\sigma_{j,l}')^2 \leq \EE_{\xi_{1,j}^l}\left[\frac{576\gamma^2 }{n^2}\|\delta_j^l\|^2\cdot \|\omega_{j,l}^u\|^2\right] \overset{\eqref{eq:SEG_zeta_delta_bound}}{\leq} \frac{144\gamma^2 V}{n^2} \EE_{\xi_{1,j}^l}\left[\|\omega_{j,l}^u\|^2\right]. \label{eq:gap_thm_SEG_technical_7_1234}
    \end{equation}
    Applying Bernstein's inequality (Lemma~\ref{lem:Bernstein_ineq}) with $X_{i,l} = \frac{24\gamma}{n}\left\langle \delta_j^l, \omega_{j,l}^u \right\rangle$, constant $c$ defined in \eqref{eq:gap_thm_SEG_technical_6_5_1234}, $b = \frac{3V}{5}$, $G = \frac{3V^2}{50\ln\frac{48n(K+1)}{\beta}}$, we get
    \begin{equation*}
        \PP\left\{|\circledSeven'| > \frac{3V}{5} \text{ and } \sum\limits_{l=0}^{T-1}\sum\limits_{j=2}^n(\sigma_{i,l}')^2 \leq \frac{3V^2}{50\ln\frac{48n(K+1)}{\beta}}\right\} \leq 2\exp\left(- \frac{b^2}{2G + \nicefrac{2cb}{3}}\right) = \frac{\beta}{24n(K+1)}.
    \end{equation*}
    The above is equivalent to 
    \begin{equation}
        \PP\{E_{\circledSeven'}\} \geq 1 - \frac{\beta}{24n(K+1)},\; \text{ for }\; E_{\circledSeven'} = \left\{\text{either } \sum\limits_{l=0}^{T-1}\sum\limits_{j=2}^n(\sigma_{i,l}')^2 > \frac{3V^2}{50\ln\frac{48n(K+1)}{\beta}} \text{ or } |\circledSeven'| \leq \frac{3V}{5}\right\}. \label{eq:bound_7_gap_SEG}
    \end{equation}
    Moreover, $E_{T-1}$ implies
    \begin{eqnarray}
        \sum\limits_{l=0}^{T-1}\sum\limits_{j=2}^n(\sigma_{j,l}')^2 &\overset{\eqref{eq:gap_thm_SEG_technical_7_1234}}{\leq}& \frac{144\gamma^2 V}{n^2} \sum\limits_{l=0}^{T-1} \sum\limits_{j=2}^n \EE_{\xi_{1,j}^l}\left[\|\omega_{j,l}^u\|^2\right] \overset{\eqref{eq:variance_theta_omega}, T \leq K+1}{\leq} \frac{2592(K+1)\gamma^2 V \lambda^{2-\alpha} \sigma^\alpha}{n} \notag \\
        &\overset{\eqref{eq:lambda_SEG_monotone}}{\leq}& \frac{2592 (K+1)\gamma^{\alpha}\sigma^\alpha V^{2-\frac{\alpha}{2}}}{60^{2-\alpha} n^{\alpha-1} \ln^{2-\alpha}\frac{48n(K+1)}{\beta}} \overset{\eqref{eq:gamma_SEG_monotone}}{\leq} \frac{3V^2}{50\ln\frac{48n(K+1)}{\beta}}. \label{eq:bound_7_variances_gap_SEG}
    \end{eqnarray}

    \paragraph{Upper bound for $2\gamma \sqrt{V} \left\|\sum_{l=0}^{T-1} \theta_l\right\|$.} We introduce new random vectors:
    \begin{equation*}
        \eta_l' = \begin{cases} \gamma \sum\limits_{r=0}^{l-1}\theta_r,& \text{if } \left\|\gamma \sum\limits_{r=0}^{l-1}\theta_r\right\| \leq \sqrt{V},\\ 0, & \text{otherwise} \end{cases}
    \end{equation*}
    for $l = 1, 2, \ldots, T-1$. With probability $1$ we have
    \begin{equation}
        \|\zeta_l'\| \leq \sqrt{V}.  \label{eq:gap_thm_SEG_technical_8}
    \end{equation}
   Using this and \eqref{eq:induction_inequality_2_SEG}, we obtain that $E_{T-1}$ implies
    \begin{eqnarray}
        2\gamma \sqrt{V}\left\|\sum\limits_{l = 0}^{T-1}\theta_l\right\| &=& 2\sqrt{V}\sqrt{\gamma^2\left\|\sum\limits_{l = 0}^{T-1}\theta_l\right\|^2}\notag\\
        &=& 2\sqrt{V}\sqrt{\gamma^2\sum\limits_{l=0}^{T-1}\|\theta_l\|^2 + 2\gamma\sum\limits_{l=0}^{T-1}\left\langle \gamma\sum\limits_{r=0}^{l-1} \theta_r, \theta_l \right\rangle} \notag\\
        &=& 2\sqrt{V} \sqrt{\gamma^2\sum\limits_{l=0}^{T-1}\|\theta_l\|^2 + 2\gamma \sum\limits_{l=0}^{T-1} \langle \zeta_l', \theta_l\rangle} \notag\\
        &\overset{\eqref{eq:gap_thm_SEG_technical_4}}{\leq}& 2\sqrt{V} \sqrt{\frac{\circledThree + \circledFour + \circledFive + \circledSix}{8} + \underbrace{\frac{2\gamma}{n} \sum\limits_{l=0}^{T-1}\sum\limits_{i=1}^n \langle \zeta_l', \theta_{i,l}^u\rangle}_{\circledEight} + \underbrace{2\gamma \sum\limits_{l=0}^{T-1} \langle \zeta_l', \theta_l^b}_{\circledNine}\rangle}. \label{eq:norm_sum_theta_bound_gap_SEG}
    \end{eqnarray}

    \paragraph{Upper bound for $\circledEight$.} To estimate this sum, we will use Bernstein's inequality. The summands have conditional expectations equal to zero:
    \begin{equation*}
        \EE_{\xi_{2,i}^l}\left[\frac{2\gamma}{n} \langle \zeta_l', \theta_{i,l}^u \rangle\right] = \frac{2\gamma}{n} \left\langle \zeta_l', \EE_{\xi_{2,i}^l}[\theta_{i,l}^u] \right\rangle = 0.
    \end{equation*}
    Moreover, for all $l = 0,\ldots, T-1$ random vectors $\{\theta_{i,l}^u\}_{i=1}^n$ are independent. Thus, sequence $\left\{\frac{2\gamma}{n} \langle \zeta_l', \theta_{i,l}^u \rangle\right\}_{l,i = 0,1}^{T-1,n}$ is a martingale difference sequence. Next, the summands are bounded:
    \begin{eqnarray}
        \left|\frac{2\gamma}{n}\langle \zeta_l', \theta_{i,l}^u \rangle \right| \leq \frac{2\gamma}{n} \|\zeta_l'\|\cdot \|\theta_{i,l}^u\| \overset{\eqref{eq:gap_thm_SEG_technical_8},\eqref{eq:theta_omega_magnitude}}{\leq} \frac{4\gamma}{n} \sqrt{V}\lambda \overset{\eqref{eq:lambda_SEG_monotone}}{\leq} \frac{V}{10\ln\frac{48n(K+1)}{\beta}} \eqdef c. \label{eq:gap_thm_SEG_technical_6_5_8}
    \end{eqnarray}
    Finally, conditional variances $(\widetilde\sigma_{i,l}')^2 \eqdef \EE_{\xi_{2,i}^l}\left[\frac{4\gamma^2}{n^2} \langle \zeta_l', \theta_{i,l}^u \rangle^2\right]$ of the summands are bounded:
    \begin{equation}
        (\widetilde\sigma_{i,l}')^2 \leq \EE_{\xi_{2,i}^l}\left[\frac{4\gamma^2 }{n^2}\|\zeta_l'\|^2\cdot \|\theta_{i,l}^u\|^2\right] \overset{\eqref{eq:gap_thm_SEG_technical_8}}{\leq} \frac{4\gamma^2 V}{n^2} \EE_{\xi_{2,i}^l}\left[\|\theta_{i,l}^u\|^2\right]. \label{eq:gap_thm_SEG_technical_7_8}
    \end{equation}
    Applying Bernstein's inequality (Lemma~\ref{lem:Bernstein_ineq}) with $X_{i,l} = \frac{2\gamma}{n} \langle \zeta_l', \theta_{i,l}^u \rangle$, constant $c$ defined in \eqref{eq:gap_thm_SEG_technical_6_5_8}, $b = \frac{V}{10}$, $G = \frac{V^2}{600\ln\frac{48n(K+1)}{\beta}}$, we get
    \begin{equation*}
        \PP\left\{|\circledEight| > \frac{V}{10} \text{ and } \sum\limits_{l=0}^{T-1}\sum\limits_{i=1}^n(\widetilde\sigma_{i,l}')^2 \leq \frac{V^2}{600\ln\frac{48n(K+1)}{\beta}}\right\} \leq 2\exp\left(- \frac{b^2}{2G + \nicefrac{2cb}{3}}\right) = \frac{\beta}{24n(K+1)}.
    \end{equation*}
    The above is equivalent to 
    \begin{equation}
        \PP\{E_{\circledEight}\} \geq 1 - \frac{\beta}{24n(K+1)},\; \text{ for }\; E_{\circledEight} = \left\{\text{either } \sum\limits_{l=0}^{T-1}\sum\limits_{i=1}^n(\widetilde\sigma_{i,l}')^2 > \frac{V^2}{600\ln\frac{48n(K+1)}{\beta}} \text{ or } |\circledEight| \leq \frac{V}{10}\right\}. \label{eq:bound_8_gap_SEG}
    \end{equation}
    Moreover, $E_{T-1}$ implies
    \begin{eqnarray}
        \sum\limits_{l=0}^{T-1}\sum\limits_{i=1}^n(\widetilde\sigma_{i,l}')^2 &\overset{\eqref{eq:gap_thm_SEG_technical_7_8}}{\leq}& \frac{4\gamma^2 V}{n^2} \sum\limits_{l=0}^{T-1} \sum\limits_{i=1}^n \EE_{\xi_{2,i}^l}\left[\|\theta_{i,l}^u\|^2\right] \overset{\eqref{eq:variance_theta_omega}, T \leq K+1}{\leq} \frac{72(K+1)\gamma^2 V \lambda^{2-\alpha} \sigma^\alpha}{n} \notag \\
        &\overset{\eqref{eq:lambda_SEG_monotone}}{\leq}& \frac{72 (K+1)\gamma^{\alpha}\sigma^\alpha V^{2-\frac{\alpha}{2}}}{60^{2-\alpha} n^{\alpha-1} \ln^{2-\alpha}\frac{48n(K+1)}{\beta}} \overset{\eqref{eq:gamma_SEG_monotone}}{\leq} \frac{V^2}{600\ln\frac{48n(K+1)}{\beta}}. \label{eq:bound_8_variances_gap_SEG}
    \end{eqnarray}

    \paragraph{Upper bound for $\circledNine$.} Probability event $E_{T-1}$ implies
    \begin{eqnarray}
        \circledNine &\leq& 2\gamma \sum\limits_{l=0}^{T-1}\|\zeta_l'\| \cdot \|\theta_l^b\| \overset{\eqref{eq:gap_thm_SEG_technical_8}, \eqref{eq:bias_theta_omega}, T \leq K+1}{\leq} \frac{2\cdot 2^\alpha(K+1)\gamma \sqrt{V}\sigma^\alpha}{\lambda^{\alpha-1}} \notag \\
        &\overset{\eqref{eq:lambda_SEG_monotone}}{=}& \frac{2\cdot 2^\alpha\cdot 60^{\alpha-1} (K+1)\gamma^\alpha \sigma^\alpha \ln^{\alpha-1}\frac{48n(K+1)}{\beta}}{V^{\frac{\alpha}{2}-1}} \overset{\eqref{eq:gamma_SEG_monotone}}{\leq} \frac{V}{100}. \label{eq:bound_9_variances_gap_SEG}
    \end{eqnarray}

    That is, we derive the upper bounds for  $2\gamma \sqrt{V} \left\|\sum_{l = 0}^{T-1}\theta_l\right\|, \circledOne, \circledTwo, \circledThree, \circledFour, \circledFive, \circledSix, \circledSeven$. More precisely, $E_{T-1}$ implies
    \begin{gather*}
        A_T \overset{\eqref{eq:gap_thm_SEG_technical_6}}{\leq} 4V + 2\gamma \sqrt{V}\left\|\sum\limits_{l=0}^{T-1} \theta_l\right\| + \circledOne + \circledTwo + \circledThree + \circledFour + \circledFive + \circledSix + \circledSeven,\\
        \circledSix \overset{\eqref{eq:SEG_extra_sums_six_distributed}}{=} \circledSix' + \frac{32\gamma}{n} \sum\limits_{j=2}^n \left\langle \frac{\gamma}{n}\sum\limits_{i=1}^{j-1} \theta_{i,T-1}^u - \zeta_j^{T-1}, \theta_{j,T-1}^u \right\rangle,\\
        \circledSeven \overset{\eqref{eq:SEG_extra_sums_seven_distributed}}{=} \circledSeven' + \frac{24\gamma}{n} \sum\limits_{j=2}^n \left\langle \frac{\gamma}{n}\sum\limits_{i=1}^{j-1} \omega_{i,T-1}^u - \delta_j^{T-1}, \omega_{j,T-1}^u \right\rangle,\\
        2\gamma \sqrt{V}\left\|\sum\limits_{l=0}^{T-1} \theta_l\right\| \overset{\eqref{eq:norm_sum_theta_bound_gap_SEG}}{\leq} 2\sqrt{V} \sqrt{\frac{\circledThree + \circledFour + \circledFive + \circledSix}{8} + \circledEight + \circledNine},\\
        \circledTwo \overset{\eqref{eq:bound_2_variances_gap_SEG}}{\leq} \frac{3V}{100},\quad \circledThree \overset{\eqref{eq:bound_3_variances_gap_SEG}}{\leq} \frac{V}{20},\quad \circledFive \overset{\eqref{eq:nsjcbdjhcbfjdhfbj}}{\leq} \frac{V}{6},\quad \circledNine \overset{\eqref{eq:bound_9_variances_gap_SEG}}{\leq} \frac{V}{100},\\
        \sum\limits_{l=0}^{T-1}\sum\limits_{i=1}^n\sigma_{i,l}^2 \overset{\eqref{eq:bound_1_variances_gap_SEG}}{\leq}  \frac{3V^2}{200\ln\frac{48n(K+1)}{\beta}},\quad \sum\limits_{l=0}^{T-1}\sum\limits_{i=1}^n\widetilde\sigma_{i,l}^2 \overset{\eqref{eq:bound_1_variances_gap_SEG_1}}{\leq} \frac{V^2}{216\ln\frac{48n(K+1)}{\beta}},\quad \sum\limits_{l=0}^{T-1}\sum\limits_{j=2}^n\widehat\sigma_{j,l}^2 \overset{\eqref{eq:bound_6_variances_gap_SEG}}{\leq}  \frac{8V^2}{75\ln\frac{48n(K+1)}{\beta}},\\
        \sum\limits_{l=0}^{T-1}\sum\limits_{j=2}^n(\sigma_{j,l}')^2 \overset{\eqref{eq:bound_7_variances_gap_SEG}}{\leq} \frac{3V^2}{50\ln\frac{48n(K+1)}{\beta}},\quad \sum\limits_{l=0}^{T-1}\sum\limits_{i=1}^n(\widetilde\sigma_{i,l}')^2 \leq \frac{V^2}{600\ln\frac{48n(K+1)}{\beta}}.
    \end{gather*}
    In addition, we also establish (see \eqref{eq:bound_1_gap_SEG}, \eqref{eq:bound_1_gap_SEG_1}, \eqref{eq:bound_6_gap_SEG}, \eqref{eq:bound_7_gap_SEG}, \eqref{eq:bound_8_gap_SEG} and our induction assumption)
    \begin{gather*}
        \PP\{E_{T-1}\} \geq 1 - \frac{(T-1)\beta}{K+1},\\
        \PP\{E_{\circledOne}\} \geq 1 - \frac{\beta}{24n(K+1)}, \quad \PP\{E_{\circledFour}\} \geq 1 - \frac{\beta}{24n(K+1)}, \quad \PP\{E_{\circledSix'}\} \geq 1 - \frac{\beta}{24n(K+1)},\\
        \PP\{E_{\circledSeven'}\} \geq 1 - \frac{\beta}{24n(K+1)},\quad \PP\{E_{\circledEight}\} \geq 1 - \frac{\beta}{24n(K+1)},
    \end{gather*}
    where
    \begin{eqnarray}
        E_{\circledOne} &=& \left\{\text{either } \sum\limits_{l=0}^{T-1}\sum\limits_{i=1}^n\sigma_{i,l}^2 > \frac{3V^2}{200\ln\frac{48n(K+1)}{\beta}} \text{ or } |\circledOne| \leq \frac{3V}{10}\right\},\notag\\
        E_{\circledFour} &=& \left\{\text{either } \sum\limits_{l=0}^{T-1}\sum\limits_{i=1}^n\widetilde\sigma_{i,l}^2 > \frac{V^2}{216\ln\frac{48n(K+1)}{\beta}} \text{ or } |\circledFour| \leq \frac{V}{6}\right\},\notag\\
        E_{\circledSix'} &=& \left\{\text{either } \sum\limits_{l=0}^{T-1}\sum\limits_{j=2}^n\widehat\sigma_{i,l}^2 > \frac{8V^2}{75\ln\frac{48n(K+1)}{\beta}} \text{ or } |\circledSix'| \leq \frac{4V}{5}\right\},\notag\\
        E_{\circledSeven'} &=& \left\{\text{either } \sum\limits_{l=0}^{T-1}\sum\limits_{j=2}^n(\sigma_{i,l}')^2 > \frac{3V^2}{50\ln\frac{48n(K+1)}{\beta}} \text{ or } |\circledSeven'| \leq \frac{3V}{5}\right\}\notag \\
        E_{\circledEight} &=& \left\{\text{either } \sum\limits_{l=0}^{T-1}\sum\limits_{i=1}^n(\widetilde\sigma_{i,l}')^2 > \frac{V^2}{600\ln\frac{48n(K+1)}{\beta}} \text{ or } |\circledEight| \leq \frac{V}{10}\right\} \notag
    \end{eqnarray}
    Therefore, probability event $E_{T-1} \cap E_{\circledOne} \cap E_{\circledFour} \cap E_{\circledSix'} \cap E_{\circledSeven'} \cap E_{\circledEight}$ implies
    \begin{eqnarray}
        \left\|\gamma\sum\limits_{l=0}^{T-1} \theta_l\right\| &\leq& \sqrt{\frac{1}{8}\left(\frac{V}{20} + \frac{V}{6} + \frac{V}{6} + \frac{4V}{5}\right) + \frac{V}{10} + \frac{V}{100}}\notag\\
        &&\quad + \sqrt{\frac{32\gamma}{n} \sum\limits_{j=2}^n \left\langle \frac{\gamma}{n}\sum\limits_{i=1}^{j-1} \theta_{i,T-1}^u - \zeta_j^{T-1}, \theta_{j,T-1}^u \right\rangle}\notag\\
        &\leq& \sqrt{V} + \sqrt{\frac{32\gamma}{n} \sum\limits_{j=2}^n \left\langle \frac{\gamma}{n}\sum\limits_{i=1}^{j-1} \theta_{i,T-1}^u - \zeta_j^{T-1}, \theta_{j,T-1}^u \right\rangle}, \label{eq:gap_thm_SEG_technical_10}
    \end{eqnarray}
    and
    \begin{eqnarray}
        A_T &\leq& 4V + 2V + 2\sqrt{V}\sqrt{\frac{32\gamma}{n} \sum\limits_{j=2}^n \left\langle \frac{\gamma}{n}\sum\limits_{i=1}^{j-1} \theta_{i,T-1}^u - \zeta_j^{T-1}, \theta_{j,T-1}^u \right\rangle}\notag\\
        &&\quad + \frac{3V}{10} + \frac{3V}{100} + \frac{V}{20} + \frac{V}{6} + \frac{V}{6} + \frac{4V}{5} + \frac{3V}{5}\notag\\
        &&\quad + \frac{32\gamma}{n} \sum\limits_{j=2}^n \left\langle \frac{\gamma}{n}\sum\limits_{i=1}^{j-1} \theta_{i,T-1}^u - \zeta_j^{T-1}, \theta_{j,T-1}^u \right\rangle \notag\\
        &&\quad + \frac{24\gamma}{n} \sum\limits_{j=2}^n \left\langle \frac{\gamma}{n}\sum\limits_{i=1}^{j-1} \omega_{i,T-1}^u - \delta_j^{T-1}, \omega_{j,T-1}^u \right\rangle \notag \\
        &\leq& 9V +  2\sqrt{V}\sqrt{\frac{32\gamma}{n} \sum\limits_{j=2}^n \left\langle \frac{\gamma}{n}\sum\limits_{i=1}^{j-1} \theta_{i,T-1}^u - \zeta_j^{T-1}, \theta_{j,T-1}^u \right\rangle}\notag\\
        &&\quad + \frac{32\gamma}{n} \sum\limits_{j=2}^n \left\langle \frac{\gamma}{n}\sum\limits_{i=1}^{j-1} \theta_{i,T-1}^u - \zeta_j^{T-1}, \theta_{j,T-1}^u \right\rangle \notag\\
        &&\quad + \frac{24\gamma}{n} \sum\limits_{j=2}^n \left\langle \frac{\gamma}{n}\sum\limits_{i=1}^{j-1} \omega_{i,T-1}^u - \delta_j^{T-1}, \omega_{j,T-1}^u \right\rangle, \label{eq:gap_thm_SEG_technical_11}
    \end{eqnarray}

    In the final part of the proof, we will show that $\frac{\gamma}{n}\sum\limits_{i=1}^{j-1} \theta_{i,T-1}^u = \zeta_j^{T-1}$ and $\frac{\gamma}{n}\sum\limits_{i=1}^{j-1} \omega_{i,T-1}^u = \delta_j^{T-1}$ with high probability. In particular, we consider probability event $\widetilde{E}_{T-1,j}$ defined as follows: inequalities
    \begin{equation}
        \left\|\frac{\gamma}{n}\sum\limits_{i=1}^{r-1} \theta_{i,T-1}^u\right\| \leq \frac{\sqrt{V}}{2}, \quad \left\|\frac{\gamma}{n}\sum\limits_{i=1}^{r-1} \omega_{i,T-1}^u\right\| \leq \frac{\sqrt{V}}{2} \notag
    \end{equation}
    hold for $r = 2, \ldots, j$ simultaneously. We want to show that $\PP\{E_{T-1} \cap \widetilde{E}_{T-1,j}\} \geq 1 - \frac{(T-1)\beta}{K+1} - \frac{j\beta}{8n(K+1)}$ for all $j = 2, \ldots, n$. For $j = 2$ the statement is trivial since
    \begin{eqnarray*}
        \left\|\frac{\gamma}{n} \theta_{1,T-1}^u\right\| \overset{\eqref{eq:theta_omega_magnitude}}{\leq} \frac{2\gamma\lambda}{n} \leq \frac{\sqrt{V}}{2},\quad \left\|\frac{\gamma}{n} \omega_{1,T-1}^u\right\| \overset{\eqref{eq:theta_omega_magnitude}}{\leq} \frac{2\gamma\lambda}{n} \leq \frac{\sqrt{V}}{2}.
    \end{eqnarray*}
    Next, we assume that the statement holds for some $j = m-1 < n$, i.e.,  $\PP\{E_{T-1}\cap\widetilde{E}_{T-1,m-1}\} \geq 1 - \frac{(T-1)\beta}{K+1} - \frac{(m-1)\beta}{8n(K+1)}$. Our goal is to prove that $\PP\{E_{T-1}\cap\widetilde{E}_{T-1,m}\} \geq 1 - \frac{(T-1)\beta}{K+1} - \frac{m\beta}{8n(K+1)}$. First, we consider $\left\|\frac{\gamma}{n}\sum\limits_{i=1}^{m-1} \theta_{i,T-1}^u\right\|$:
    \begin{eqnarray}
        \left\|\frac{\gamma}{n}\sum\limits_{i=1}^{m-1} \theta_{i,T-1}^u\right\| &=& \sqrt{\frac{\gamma^2}{n^2}\left\|\sum\limits_{i=1}^{m-1} \theta_{i,T-1}^u\right\|^2} \notag\\
        &=& \sqrt{\frac{\gamma^2}{n^2}\sum\limits_{i=1}^{m-1}\|\theta_{i,T-1}^u\|^2 + \frac{2\gamma}{n}\sum\limits_{i=1}^{m-1}\left\langle \frac{\gamma}{n}\sum\limits_{r=1}^{i-1}\theta_{r,T-1}^u , \theta_{i,T-1}^u \right\rangle} \notag\\
        &\leq& \sqrt{\frac{\gamma^2}{n^2}\sum\limits_{l=0}^{T-1}\sum\limits_{i=1}^{m-1}\|\theta_{i,l}^u\|^2 + \frac{2\gamma}{n}\sum\limits_{i=1}^{m-1}\left\langle \frac{\gamma}{n}\sum\limits_{r=1}^{i-1}\theta_{r,T-1}^u , \theta_{i,T-1}^u \right\rangle}. \notag
    \end{eqnarray}
    Similarly, we have
    \begin{eqnarray}
        \left\|\frac{\gamma}{n}\sum\limits_{i=1}^{m-1} \omega_{i,T-1}^u\right\| &=& \sqrt{\frac{\gamma^2}{n^2}\left\|\sum\limits_{i=1}^{m-1} \omega_{i,T-1}^u\right\|^2} \notag\\
        &=& \sqrt{\frac{\gamma^2}{n^2}\sum\limits_{i=1}^{m-1}\|\omega_{i,T-1}^u\|^2 + \frac{2\gamma}{n}\sum\limits_{i=1}^{m-1}\left\langle \frac{\gamma}{n}\sum\limits_{r=1}^{i-1}\omega_{r,T-1}^u , \omega_{i,T-1}^u \right\rangle} \notag\\
        &\leq& \sqrt{\frac{\gamma^2}{n^2}\sum\limits_{l=0}^{T-1}\sum\limits_{i=1}^{m-1}\|\omega_{i,l}^u\|^2 + \frac{2\gamma}{n}\sum\limits_{i=1}^{m-1}\left\langle \frac{\gamma}{n}\sum\limits_{r=1}^{i-1}\omega_{r,T-1}^u , \omega_{i,T-1}^u \right\rangle}. \notag
    \end{eqnarray}

    Next, we introduce a new notation:
    \begin{gather*}
        \rho_{i,T-1} = \begin{cases}
            \frac{\gamma}{n}\sum\limits_{r=1}^{i-1}\theta_{r,T-1}^u,& \text{if } \left\|\frac{\gamma}{n}\sum\limits_{r=1}^{i-1}\theta_{r,T-1}^u\right\| \leq \frac{\sqrt{V}}{2},\\
            0,& \text{otherwise}
        \end{cases},\\
        \rho_{i,T-1}' = \begin{cases}
            \frac{\gamma}{n}\sum\limits_{r=1}^{i-1}\omega_{r,T-1}^u,& \text{if } \left\|\frac{\gamma}{n}\sum\limits_{r=1}^{i-1}\omega_{r,T-1}^u\right\| \leq \frac{\sqrt{V}}{2},\\
            0,& \text{otherwise}
        \end{cases}
    \end{gather*}
    for $i = 1,\ldots,m-1$. By definition, we have
    \begin{equation}
        \|\rho_{i,T-1}\| \leq \frac{\sqrt{V}}{2},\quad \|\rho_{i,T-1}'\| \leq \frac{\sqrt{V}}{2} \label{eq:SEG_bound_rho}
    \end{equation}
    for $i = 1,\ldots,m-1$. Moreover, $\widetilde{E}_{T-1,m-1}$ implies $\rho_{i,T-1} = \frac{\gamma}{n}\sum\limits_{r=1}^{i-1}\theta_{r,T-1}^u$, $\rho_{i,T-1}' = \frac{\gamma}{n}\sum\limits_{r=1}^{i-1}\omega_{r,T-1}^u$ for $i = 1,\ldots,m-1$ and
    \begin{eqnarray}
        \left\|\frac{\gamma}{n}\sum\limits_{i=1}^{m-1} \theta_{i,l}^u\right\| &\leq& \sqrt{\circledThree + \circledFour + \circledTen}, \notag\\
        \left\|\frac{\gamma}{n}\sum\limits_{i=1}^{m-1} \omega_{i,l}^u\right\| &\leq& \sqrt{\circledThree + \circledFour + \circledTen'}, \notag
    \end{eqnarray}
    where
    \begin{gather*}
        \circledTen =  \frac{2\gamma}{n}\sum\limits_{i=1}^{m-1}\left\langle \rho_{i,T-1} , \theta_{i,T-1}^u \right\rangle,\quad \circledTen' =  \frac{2\gamma}{n}\sum\limits_{i=1}^{m-1}\left\langle \rho_{i,T-1}' , \omega_{i,T-1}^u \right\rangle.
    \end{gather*}
    It remains to estimate $\circledTen$ and $\circledTen'$.

    \textbf{Upper bound for $\circledTen$.} To estimate this sum, we will use Bernstein's inequality. The summands have conditional expectations equal to zero:
    \begin{equation*}
        \EE_{\xi_{2,i}^{T-1}}\left[\frac{2\gamma}{n} \langle \rho_{i,T-1}, \theta_{i,T-1}^u \rangle\right] = \frac{2\gamma}{n} \left\langle \rho_{i,T-1}, \EE_{\xi_{2,i}^{T-1}}[\theta_{i,T-1}^u] \right\rangle = 0,
    \end{equation*}
    since random vectors $\{\theta_{i,T-1}^u\}_{i=1}^n$ are independent. Thus, sequence $\left\{\frac{2\gamma}{n} \langle \rho_{i,T-1}, \theta_{i,T-1}^u \rangle\right\}_{i = 1}^{m-1}$ is a martingale difference sequence. Next, the summands are bounded:
    \begin{eqnarray}
        \left|\frac{2\gamma}{n}\langle \rho_{i,T-1}, \theta_{i,T-1}^u \rangle \right| \leq \frac{2\gamma}{n} \|\rho_{i,T-1}\|\cdot \|\theta_{i,T-1}^u\| \overset{\eqref{eq:SEG_bound_rho},\eqref{eq:theta_omega_magnitude}}{\leq} \frac{2\gamma}{n} \sqrt{V}\lambda \overset{\eqref{eq:lambda_SEG_monotone}}{=} \frac{V}{30\ln\tfrac{48n(K+1)}{\beta}} \eqdef c. \label{eq:gap_thm_SEG_technical_6_5_10}
    \end{eqnarray}
    Finally, conditional variances $(\widehat\sigma_{i,T-1}')^2 \eqdef \EE_{\xi_{2,i}^{T-1}}\left[\frac{4\gamma^2}{n^2} \langle \rho_{i,T-1}, \theta_{i,T-1}^u \rangle^2\right]$ of the summands are bounded:
    \begin{equation}
        (\widehat\sigma_{i,T-1}')^2 \leq \EE_{\xi_{2,i}^{T-1}}\left[\frac{4\gamma^2 }{n^2}\|\rho_{i,T-1}\|^2\cdot \|\theta_{i,T-1}^u\|^2\right] \overset{\eqref{eq:gap_thm_SEG_technical_10}}{\leq} \frac{\gamma^2 V}{n^2} \EE_{\xi_{2,i}^{T-1}}\left[\|\theta_{i,T-1}^u\|^2\right]. \label{eq:gap_thm_SEG_technical_7_10}
    \end{equation}
    Applying Bernstein's inequality (Lemma~\ref{lem:Bernstein_ineq}) with $X_{i} = \frac{2\gamma}{n} \langle \rho_{i,T-1}, \theta_{i,T-1}^u \rangle$, constant $c$ defined in \eqref{eq:gap_thm_SEG_technical_6_5_10}, $b = \frac{V}{30}$, $G = \frac{V^2}{5400\ln\frac{48n(K+1)}{\beta}}$, we get
    \begin{equation*}
        \PP\left\{|\circledTen| > \frac{V}{30} \text{ and } \sum\limits_{i=1}^{m-1}(\widehat\sigma_{i,T-1}')^2 \leq \frac{V^2}{5400\ln\frac{48n(K+1)}{\beta}}\right\} \leq 2\exp\left(- \frac{b^2}{2G + \nicefrac{2cb}{3}}\right) = \frac{\beta}{24n(K+1)}.
    \end{equation*}
    The above is equivalent to 
    \begin{equation}
        \PP\{E_{\circledTen}\} \geq 1 - \frac{\beta}{24n(K+1)},\; \text{ for }\; E_{\circledTen} = \left\{\text{either } \sum\limits_{i=1}^{m-1}(\widehat\sigma_{i,T-1}')^2 > \frac{V^2}{5400\ln\frac{48n(K+1)}{\beta}} \text{ or } |\circledTen| \leq \frac{V}{30}\right\}. \label{eq:bound_10_gap_SEG}
    \end{equation}
    Moreover, $E_{T-1}$ implies
    \begin{eqnarray}
        \sum\limits_{i=1}^{m-1}(\widehat\sigma_{i,T-1}')^2 &\overset{\eqref{eq:gap_thm_SEG_technical_7_10}}{\leq}& \frac{\gamma^2 V}{n^2} \sum\limits_{i=1}^n \EE_{\xi_{2,i}^{T-1}}\left[\|\theta_{i,T-1}^u\|^2\right] \overset{\eqref{eq:variance_theta_omega}}{\leq} \frac{18\gamma^2 V \lambda^{2-\alpha} \sigma^\alpha}{n} \notag \\
        &\overset{\eqref{eq:lambda_SEG_monotone}}{\leq}& \frac{18 \gamma^{\alpha}\sigma^\alpha V^{2-\frac{\alpha}{2}}}{60^{2-\alpha} n^{\alpha-1} \ln^{2-\alpha}\frac{48n(K+1)}{\beta}} \overset{\eqref{eq:gamma_SEG_monotone}}{\leq} \frac{V^2}{5400\ln\tfrac{48n(K+1)}{\beta}}. \label{eq:bound_10_variances_gap_SEG}
    \end{eqnarray}

    \textbf{Upper bound for $\circledTen'$.} To estimate this sum, we will use Bernstein's inequality. The summands have conditional expectations equal to zero:
    \begin{equation*}
        \EE_{\xi_{1,i}^{T-1}}\left[\frac{2\gamma}{n} \langle \rho_{i,T-1}', \omega_{i,T-1}^u \rangle\right] = \frac{2\gamma}{n} \left\langle \rho_{i,T-1}', \EE_{\xi_{1,i}^{T-1}}[\omega_{i,T-1}^u] \right\rangle = 0,
    \end{equation*}
    since random vectors $\{\omega_{i,T-1}^u\}_{i=1}^n$ are independent. Thus, sequence $\left\{\frac{2\gamma}{n} \langle \rho_{i,T-1}', \omega_{i,T-1}^u \rangle\right\}_{i = 1}^{m-1}$ is a martingale difference sequence. Next, the summands are bounded:
    \begin{eqnarray}
        \left|\frac{2\gamma}{n}\langle \rho_{i,T-1}', \omega_{i,T-1}^u \rangle \right| \leq \frac{2\gamma}{n} \|\rho_{i,T-1}'\|\cdot \|\omega_{i,T-1}^u\| \overset{\eqref{eq:SEG_bound_rho},\eqref{eq:theta_omega_magnitude}}{\leq} \frac{2\gamma}{n} \sqrt{V}\lambda \overset{\eqref{eq:lambda_SEG_monotone}}{=} \frac{V}{30\ln\tfrac{48n(K+1)}{\beta}} \eqdef c. \label{eq:gap_thm_SEG_technical_6_5_10'}
    \end{eqnarray}
    Finally, conditional variances $(\widetilde\sigma_{i,T-1}')^2 \eqdef \EE_{\xi_{1,i}^{T-1}}\left[\frac{4\gamma^2}{n^2} \langle \rho_{i,T-1}', \omega_{i,T-1}^u \rangle^2\right]$ of the summands are bounded:
    \begin{equation}
        (\widetilde\sigma_{i,T-1}')^2 \leq \EE_{\xi_{1,i}^{T-1}}\left[\frac{4\gamma^2 }{n^2}\|\rho_{i,T-1}'\|^2\cdot \|\omega_{i,T-1}^u\|^2\right] \overset{\eqref{eq:gap_thm_SEG_technical_10}}{\leq} \frac{\gamma^2 V}{n^2} \EE_{\xi_{1,i}^{T-1}}\left[\|\omega_{i,T-1}^u\|^2\right]. \label{eq:gap_thm_SEG_technical_7_10'}
    \end{equation}
    Applying Bernstein's inequality (Lemma~\ref{lem:Bernstein_ineq}) with $X_{i} = \frac{2\gamma}{n} \langle \rho_{i,T-1}', \omega_{i,T-1}^u \rangle$, constant $c$ defined in \eqref{eq:gap_thm_SEG_technical_6_5_10'}, $b = \frac{V}{30}$, $G = \frac{V^2}{5400\ln\frac{48n(K+1)}{\beta}}$, we get
    \begin{equation*}
        \PP\left\{|\circledTen'| > \frac{V}{30} \text{ and } \sum\limits_{i=1}^{m-1}(\widetilde\sigma_{i,T-1}')^2 \leq \frac{V^2}{5400\ln\frac{48n(K+1)}{\beta}}\right\} \leq 2\exp\left(- \frac{b^2}{2G + \nicefrac{2cb}{3}}\right) = \frac{\beta}{24n(K+1)}.
    \end{equation*}
    The above is equivalent to 
    \begin{equation}
        \PP\{E_{\circledTen'}\} \geq 1 - \frac{\beta}{24n(K+1)},\; \text{ for }\; E_{\circledTen'} = \left\{\text{either } \sum\limits_{i=1}^{m-1}(\widetilde\sigma_{i,T-1}')^2 > \frac{V^2}{5400\ln\frac{48n(K+1)}{\beta}} \text{ or } |\circledTen'| \leq \frac{V}{30}\right\}. \label{eq:bound_10_gap_SEG'}
    \end{equation}
    Moreover, $E_{T-1}$ implies
    \begin{eqnarray}
        \sum\limits_{i=1}^{m-1}(\widetilde\sigma_{i,T-1}')^2 &\overset{\eqref{eq:gap_thm_SEG_technical_7_10'}}{\leq}& \frac{\gamma^2 V}{n^2} \sum\limits_{i=1}^n \EE_{\xi_{1,i}^{T-1}}\left[\|\omega_{i,T-1}^u\|^2\right] \overset{\eqref{eq:variance_theta_omega}}{\leq} \frac{18\gamma^2 V \lambda^{2-\alpha} \sigma^\alpha}{n} \notag \\
        &\overset{\eqref{eq:lambda_SEG_monotone}}{\leq}& \frac{18 \gamma^{\alpha}\sigma^\alpha V^{2-\frac{\alpha}{2}}}{60^{2-\alpha} n^{\alpha-1} \ln^{2-\alpha}\frac{48n(K+1)}{\beta}} \overset{\eqref{eq:gamma_SEG_monotone}}{\leq} \frac{V^2}{5400\ln\tfrac{48n(K+1)}{\beta}}. \label{eq:bound_10_variances_gap_SEG'}
    \end{eqnarray}

    Putting all together we get that $E_{T-1}\cap \widetilde{E}_{T-1, m-1}$ implies

    \begin{gather*}
        \left\|\frac{\gamma}{n}\sum\limits_{i=1}^{m-1} \theta_{i,T-1}^u\right\| \leq \sqrt{\circledThree + \circledFour + \circledTen},\quad \left\|\frac{\gamma}{n}\sum\limits_{i=1}^{m-1} \omega_{i,T-1}^u\right\| \leq \sqrt{\circledThree + \circledFour + \circledTen'},\quad \circledThree \overset{\eqref{eq:bound_3_variances_gap_SEG}}{\leq} \frac{V}{20},\\
        \sum\limits_{l=0}^{T-1}\sum\limits_{i=1}^n\widetilde\sigma_{i,l}^2 \overset{\eqref{eq:bound_1_variances_gap_SEG_1}}{\leq} \frac{V^2}{216\ln\frac{48n(K+1)}{\beta}},\quad \sum\limits_{i=1}^{m-1}(\widehat\sigma_{i,T-1}')^2 \leq \frac{V^2}{5400\ln\tfrac{48n(K+1)}{\beta}},\\
        \sum\limits_{i=1}^{m-1}(\widetilde\sigma_{i,T-1}')^2 \leq \frac{V^2}{5400\ln\tfrac{48n(K+1)}{\beta}}.
    \end{gather*}
    In addition, we also establish (see \eqref{eq:bound_1_gap_SEG_1}, \eqref{eq:bound_10_gap_SEG}, \eqref{eq:bound_10_gap_SEG'} and our induction assumption)
    \begin{gather*}
        \PP\{E_{T-1}\cap \widetilde{E}_{T-1,m-1}\} \geq 1 - \frac{(T-1)\beta}{K+1} - \frac{(m-1)\beta}{8n(K+1)},\\
        \PP\{E_{\circledFour}\} \geq 1 - \frac{\beta}{24n(K+1)}, \quad \PP\{E_{\circledTen}\} \geq 1 - \frac{\beta}{24n(K+1)}, \quad \PP\{E_{\circledTen'}\} \geq 1 - \frac{\beta}{24n(K+1)}
    \end{gather*}
    where
    \begin{eqnarray}
        E_{\circledFour} &=& \left\{\text{either } \sum\limits_{l=0}^{T-1}\sum\limits_{i=1}^n\widetilde\sigma_{i,l}^2 > \frac{V^2}{216\ln\frac{48n(K+1)}{\beta}} \text{ or } |\circledFour| \leq \frac{V}{6}\right\},\notag\\
        E_{\circledTen} &=& \left\{\text{either } \sum\limits_{i=1}^{m-1}(\widehat\sigma_{i,l}')^2 > \frac{V^2}{5400\ln\frac{48n(K+1)}{\beta}} \text{ or } |\circledTen| \leq \frac{V}{30}\right\}, \notag\\
        E_{\circledTen'} &=& \left\{\text{either } \sum\limits_{i=1}^{m-1}(\widetilde\sigma_{i,T-1}')^2 > \frac{V^2}{5400\ln\frac{48n(K+1)}{\beta}} \text{ or } |\circledTen'| \leq \frac{V}{30}\right\} \notag
    \end{eqnarray}
    Therefore, probability event $E_{T-1} \cap \widetilde{E}_{T-1,m-1} \cap E_{\circledFour} \cap E_{\circledTen} \cap E_{\circledTen'}$ implies
    \begin{gather*}
        \left\|\frac{\gamma}{n}\sum\limits_{i=1}^{m-1} \theta_{i,T-1}^u\right\| \leq \sqrt{\frac{V}{20} + \frac{V}{6} + \frac{V}{30}} = \frac{\sqrt{V}}{2},\\
        \left\|\frac{\gamma}{n}\sum\limits_{i=1}^{m-1} \omega_{i,T-1}^u\right\| \leq \sqrt{\frac{V}{20} + \frac{V}{6} + \frac{V}{30}} = \frac{\sqrt{V}}{2}.
    \end{gather*}
    This implies $\widetilde E_{T-1,m}$ and
    \begin{eqnarray*}
        \PP\{E_{T-1} \cap \widetilde{E}_{T-1,m}\} &\geq& \PP\{E_{T-1} \cap \widetilde{E}_{T-1,m-1} \cap E_{\circledFour} \cap E_{\circledTen} \cap E_{\circledTen'}\} \\
        &=& 1 - \PP\left\{\overline{E_{T-1} \cap \widetilde{E}_{T-1,m-1}} \cup \overline{E}_{\circledFour} \cup \overline{E}_{\circledTen} \cup \overline{E}_{\circledTen'}\right\}\\
        &\geq& 1 - \frac{(T-1)\beta}{K+1} - \frac{m\beta}{8n(K+1)}.
    \end{eqnarray*}
    Therefore, for all $m = 2,\ldots,n$ the statement holds and, in particular, $\PP\{E_{T-1} \cap \widetilde{E}_{T-1, n}\} \geq 1 - \frac{(T-1)\beta}{K+1} - \frac{\beta}{8(K+1)}$. Taking into account \eqref{eq:gap_thm_SEG_technical_11}, we conclude that $E_{T-1} \cap \widetilde{E}_{T-1, n} \cap E_{\circledOne} \cap E_{\circledFour} \cap E_{\circledSix'} \cap E_{\circledSeven'} \cap E_{\circledEight}$ implies
    \begin{equation*}
       \left\|\gamma\sum\limits_{l=0}^{T-1} \theta_l\right\| \leq \sqrt{V},\quad A_T \leq 9V
    \end{equation*}
    that is equivalent to \eqref{eq:induction_inequality_1_SEG} and \eqref{eq:induction_inequality_2_SEG} for $t = T$. Moreover,
    \begin{eqnarray*}
        \PP\left\{E_T\right\} &\geq& \PP\left\{E_{T-1} \cap \widetilde{E}_{T-1,n} \cap E_{\circledOne} \cap E_{\circledFour} \cap E_{\circledSix'} \cap E_{\circledSeven'} \cap E_{\circledEight} \right\} \\
        &=& 1 - \PP\left\{\overline{E_{T-1} \cap \widetilde{E}_{n}} \cup \overline{E}_{\circledOne} \cup \overline{E}_{\circledFour} \cup \overline{E}_{\circledSix'} \cup \overline{E}_{\circledSeven'} \cup \overline{E}_{\circledEight}\right\}\\
        &=& 1 - \frac{(T-1)\beta}{K+1} - \frac{\beta}{8(K+1)} - 5\cdot \frac{\beta}{24n(K+1)} = 1 - \frac{T\beta}{K+1}.
    \end{eqnarray*}

    In other words, we showed that $\PP\{E_k\} \geq 1 - \nicefrac{k\beta}{(K+1)}$ for all $k = 0,1,\ldots,K+1$. For $k = K+1$ we have that with probability at least $1 - \beta$
    \begin{eqnarray*}
        \gap_{\sqrt{V}}(\tx^{K}_{\avg}) &=& \max_{B_{\sqrt{V}}(x^*)}\left\{ \langle F(u), \tx^{t'}_{\avg} - u\rangle + \Psi(\tx^{t'}_{\avg}) - \Psi(u)  \right\}\\
        &\leq& \frac{1}{2\gamma(K + 1)}\max_{B_{\sqrt{V}}(x^*)}\left\{ 2\gamma(t'+1)\left(\langle F(u), \tx^{t'}_{\avg} - u\rangle + \Psi(\tx^{t'}_{\avg}) - \Psi(u)\right) + \|x^{t'+1} - u\|^2  \right\}\\
        &\overset{\eqref{eq:gap_thm_SEG_technical_1_5}}{\leq}& \frac{9V}{2\gamma(K+1)}.
    \end{eqnarray*}
    Finally, if 
    \begin{equation*}
        \gamma = \min\left\{\frac{1}{1920 L \ln \frac{48n(K+1)}{\beta}}, \frac{60^{\frac{2-\alpha}{\alpha}}\sqrt{V}n^{\frac{\alpha-1}{\alpha}}}{97200^{\frac{1}{\alpha}}(K+1)^{\frac{1}{\alpha}}\sigma \ln^{\frac{\alpha-1}{\alpha}} \frac{48n(K+1)}{\beta}}\right\}
    \end{equation*}
    then with probability at least $1-\beta$
    \begin{eqnarray*}
        \gap_{\sqrt{V}}(\tx^{K}_{\avg}) &\leq& \frac{9V}{2\gamma(K+1)} = \max\left\{\frac{8640 LV \ln \tfrac{48n(K+1)}{\beta}}{K+1}, \frac{9\cdot 60^{\frac{2-\alpha}{\alpha}} \cdot\sigma R \ln^{\frac{\alpha-1}{\alpha}} \tfrac{48n(K+1)}{\beta}}{2\cdot 97200^{\frac{1}{\alpha}} n^{\frac{\alpha-1}{\alpha}}(K+1)^{\frac{\alpha-1}{\alpha}}}\right\}\\
        &=& \cO\left(\max\left\{\frac{LV\ln\frac{nK}{\beta}}{K}, \frac{\sigma \sqrt{V} \ln^{\frac{\alpha-1}{\alpha}}\frac{nK}{\beta}}{n^{\frac{\alpha-1}{\alpha}}K^{\frac{\alpha-1}{\alpha}}}\right\}\right).
    \end{eqnarray*}
    To get $\gap_{R}(\tx^{K}_{\avg}) \leq \varepsilon$ with probability $\geq 1 - \beta$, $K$ should be
    \begin{equation*}
         K = \cO\left(\frac{LV}{\varepsilon}\ln\frac{nLV}{\varepsilon\beta}, \frac{1}{n}\left(\frac{\sigma \sqrt{V}}{\varepsilon}\right)^{\frac{\alpha}{\alpha-1}}\ln \frac{\sigma \sqrt{V}}{\varepsilon\beta}\right)
    \end{equation*}
    that concludes the proof.
\end{proof}

\subsection{Quasi-Strongly Monotone Case}


We start with the following lemma.

\begin{lemma}\label{lem:SEG_optimization_lemma_QSM}
    Let Assumptions \ref{as:L_Lip} and \ref{as:str_monotonicity} hold for $Q = B_{4n\sqrt{V}}(x^*)$, where $V \geq \|x^0 - x^*\|^2 + \frac{36000000\gamma^2\ln^2\frac{48n(K+1)}{\beta}}{n^2}\sum_{i=1}^n\|F_i(x^*)\|^2$, $\nu = \gamma\mu$, and $0 < \gamma \leq \min\left\{\frac{1}{6L}, \frac{\sqrt{n}}{15000 L \ln \frac{48n(K+1)}{\beta}}, \frac{1}{72000000 \mu \ln^2 \frac{48n(K+1)}{\beta}}\right\}$. If $x^k$ and $\tx^k$ lie in $B_{4n\sqrt{V}}(x^*)$ for all $k = 0,1,\ldots,K$ for some $K \geq 0$, then the iterates produced by \algname{DProx-clipped-SEG-shift} satisfy
    \begin{eqnarray}
       V_{K+1} &\leq& \exp\left(-\frac{\gamma\mu}{2}(K+1)\right) V + 2\gamma \sum\limits_{k=0}^{K}\exp\left(-\frac{\gamma\mu}{2}(K-k)\right)\langle \theta_k, x^k - x^* \rangle\notag\\
       &&\quad + \frac{\gamma^2}{n^2}\sum\limits_{k=0}^{K}\sum\limits_{i=1}^n\exp\left(-\frac{\gamma\mu}{2}(K-k)\right)\left(18\|\theta_{i,k}^u\|^2 + 14\|\omega_{i,k}^u\|^2\right)\notag\\
        &&\quad  + \gamma^2\sum\limits_{k=0}^{K}\sum\limits_{i=1}^n\exp\left(-\frac{\gamma\mu}{2}(K-k)\right)\left(18\|\theta_{i,k}^b\|^2 + 14\|\omega_{i,k}^b\|^2\right)\notag\\
        &&\quad + \frac{32\gamma^2}{n^2}\sum\limits_{k=0}^{K}\sum\limits_{j=2}^n\exp\left(-\frac{\gamma\mu}{2}(K-k)\right)\left\langle \sum\limits_{i=1}^{j-1} \theta_{i,k}^u, \theta_{j,k}^u \right\rangle \notag\\
        &&\quad + \frac{24\gamma^2}{n^2}\sum\limits_{k=0}^{K}\sum\limits_{j=2}^n\exp\left(-\frac{\gamma\mu}{2}(K-k)\right)\left\langle \sum\limits_{i=1}^{j-1} \omega_{i,k}^u, \omega_{j,k}^u \right\rangle, \label{eq:optimization_lemma_SEG_QSM}
    \end{eqnarray}
    where $V_k = \|x^k - x^*\|^2 + \frac{36000000\gamma^2\ln^2\frac{48n(K+1)}{\beta}}{n^2}\sum_{i=1}^n\left(\|\thh_{i}^k - F_i(x^*)\|^2 + \|\hh_{i}^k - F_i(x^*)\|^2\right)$ and $\theta_k$, $\theta_{i,k}^u$, $\theta_{i,k}^b$, $\omega_k$, $\omega_{i,k}^u$, $\omega_{i,k}^b$ are defined in \eqref{eq:theta_k_SEG}, \eqref{eq:omega_k_SEG}, and \eqref{eq:gap_thm_SEG_technical_4}-\eqref{eq:gap_thm_SEG_technical_5}
\end{lemma}
\begin{proof}
    From \eqref{eq:important_optimization_ineq_SEG} with $u = x^*$ we have
    \begin{eqnarray*}
        \gamma\left(\langle F(\tx^k), \tx^k - x^* \rangle +  \Psi(\tx^k) - \Psi(x^*)\right) &\leq& \frac{1}{2}\|x^k - x^*\|^2 - \frac{1}{2}\|x^{k+1} - x^*\|^2  \notag\\
        &&\quad - \frac{1}{2}\left(1 - 6\gamma^2 L^2\right)\|\tx^k - x^k\|^2 - \frac{1}{4}\|x^{k+1} - \tx^k\|^2 \notag\\
        &&\quad + 3\gamma^2 \|\omega_k\|^2 + 3\gamma^2\|\theta_k\|^2  + \gamma \langle \theta_k, \tx^k - x^* \rangle.
    \end{eqnarray*}
    Using quasi-strong monotonicity of $F$, the fact that $-F(x^*) \in \partial \Psi(x^*)$, and convexity of $\Psi(x^*)$, we derive
    \begin{eqnarray*}
        2\gamma\mu\|\tx^k - x^*\|^2 &\leq& 2\gamma \langle F(\tx^k) - F(x^*), \tx^k - x^* \rangle \\
        &\leq& 2\gamma\left(\langle F(\tx^k), \tx^k - x^* \rangle +  \Psi(\tx^k) - \Psi(x^*)\right)\\
        &\leq& \|x^k - x^*\|^2 - \|x^{k+1} - x^*\|^2   - \left(1 - 6\gamma^2 L^2\right)\|\tx^k - x^k\|^2\\
        &&\quad - \frac{1}{2}\|x^{k+1} - \tx^k\|^2  + 6\gamma^2 \|\omega_k\|^2 + 6\gamma^2\|\theta_k\|^2  + 2\gamma \langle \theta_k, \tx^k - x^* \rangle.
    \end{eqnarray*}
    Next, we apply $\|\tx^k - x^*\|^2 \geq \frac{1}{2}\|x^k - x^*\|^2 - \|\tx^k - x^k\|^2$ and rearrange the terms:
    \begin{eqnarray}
        \|x^{k+1} - x^*\|^2 &\leq& \left(1 - \frac{\gamma\mu}{2}\right)\|x^k - x^*\|^2 - \gamma\mu\|\tx^k - x^*\|^2 - \left(1-\gamma\mu - 6\gamma^2 L^2\right)\|\tx^k - x^k\|^2 \notag\\
        &&\quad  + 6\gamma^2 \|\omega_k\|^2 + 6\gamma^2\|\theta_k\|^2  + 2\gamma \langle \theta_k, \tx^k - x^* \rangle\notag.
    \end{eqnarray}
    Since $2\gamma\langle \theta^k, \tx^k - x^* \rangle = 2\gamma\langle \theta^k, x^k - x^* \rangle +  2\gamma\langle \theta^k, \tx^k - x^k \rangle \leq 2\gamma^2\|\theta^k\|^2 + \frac{1}{2}\|\tx^k - x^k\|^2$, we have
    \begin{eqnarray}
        \|x^{k+1} - x^*\|^2 &\leq& \left(1 - \frac{\gamma\mu}{2}\right)\|x^k - x^*\|^2 - \gamma\mu\|\tx^k - x^*\|^2 - \left(\frac{1}{2}-\gamma\mu - 6\gamma^2 L^2\right)\|\tx^k - x^k\|^2 \notag\\
        &&\quad  + 6\gamma^2 \|\omega_k\|^2 + 8\gamma^2\|\theta_k\|^2  + 2\gamma \langle \theta_k, x^k - x^* \rangle. \label{eq:nshjcbjdbhcjsd}
    \end{eqnarray}
    Now, we move on to the shifts: for all $i\in [n]$ (for convenience, we use the new notation: $h_i^* = F_i(x^*)$ for all $i\in [n]$)
    \begin{eqnarray}
        \|\thh_i^{k+1}-h^*_i\|^2 &=& \|\thh^{k}_i-h^*_i\|^2 + 2\nu \left\la \tilde \Delta^k_i, \thh^k_i-h^*_i\right\ra + \nu^2\|\tilde \Delta^k_i\|^2\notag\\
    &=& \|\thh^{k}_i-h^*_i\|^2 + 2\nu \left\la \tg_i^k -\thh^k_i, \thh^k_i-h^*_i\right\ra + \nu^2\|\tg_i^k -\thh^k_i\|^2\notag\\
    &\overset{\nu \leq 1}{\leq}& \|\thh^{k}_i-h^*_i\|^2 + 2\nu \left\la \tg_i^k-\thh^k_i, \thh^k_i-h^*_i\right\ra + \nu\|\tg_i^k-\thh^k_i\|^2\notag\\
    &=&\|\thh^{k}_i-h^*_i\|^2 + \nu \left\la \tg_i^k-\thh^k_i,  \tg_i^k + \thh^k_i- 2h^*_i\right\ra\notag\\
    &\leq& (1-\nu)\|\thh^{k}_i-h^*_i\|^2 + \nu\|\tg_i^k-h^*_i\|^2 \notag\\
    &\leq& (1-\nu)\|\thh^k_i-h^*_i\|^2 + 2\nu\|\tg_i^k- F_i(x^k)\|^2 + 2\nu\|F_i(x^{k})-h^*_i\|^2 \notag\\
    &\leq& (1-\nu)\|\thh^{k}_i-h^*_i\|^2 + 2\nu\|\omega_{i,k}\|^2 + 2\nu\|F_i(x^{k})-h_i^*\|^2\notag\\
    &=& (1-\gamma\mu)\|\thh^{k}_i-h^*_i\|^2 + 2\gamma\mu\|\omega_{i,k}\|^2 + 2\gamma\mu\|F_i(x^{k})-h_i^*\|^2\notag\\
    &\overset{\eqref{eq:L_Lip}}{\leq}&  (1-\gamma\mu)\|\thh^{k}_i-h^*_i\|^2 + 2\gamma\mu\|\omega_{i,k}\|^2 + 2\gamma\mu L^2\|x^{k} - x^*\|^2\notag\\
    &\leq& (1-\gamma\mu)\|\thh^{k}_i-h^*_i\|^2 + 2\gamma\mu\|\omega_{i,k}\|^2 + 4\gamma\mu L^2\|\tx^{k} - x^*\|^2 \notag\\
    &&\quad +  4\gamma\mu L^2\|\tx^{k} - x^k\|^2\label{eq:nvbhjsdbhdbfhjvbdjhvdjbvcs}
    \end{eqnarray}
    and, similarly,
    \begin{eqnarray}
        \|\hh_i^{k+1}-h^*_i\|^2 &\leq& (1-\gamma\mu)\|\hh^{k}_i-h^*_i\|^2 + 2\gamma\mu\|\theta_{i,k}\|^2 + 2\gamma\mu\|F_i(\tx^{k})-h_i^*\|^2 \notag\\
        &\overset{\eqref{eq:L_Lip}}{\leq}& (1-\gamma\mu)\|\hh^{k}_i-h^*_i\|^2 + 2\gamma\mu\|\theta_{i,k}\|^2 + 2\gamma\mu L^2\|\tx^{k} -x^*\|^2.\label{eq:nvbhjsdbhdbfhjvbdjhvdjbvcs_2}
    \end{eqnarray}

    Summing up \eqref{eq:nshjcbjdbhcjsd}, \eqref{eq:nvbhjsdbhdbfhjvbdjhvdjbvcs}, and \eqref{eq:nvbhjsdbhdbfhjvbdjhvdjbvcs_2}, we derive
    \begin{eqnarray}
        V_{k+1} &\leq& \left(1 - \frac{\gamma\mu}{2}\right)\|x^k - x^*\|^2\notag\\
        &&\quad + (1 - \gamma\mu) \frac{36\cdot 10^6\gamma^2\ln^2\frac{48n(K+1)}{\beta}}{n^2}\sum_{i=1}^n\left(\|\thh_i^k - h_i^*\|^2 + \|\hh_i^k - h_i^*\|^2\right)\notag\\
        &&\quad - \left(\gamma\mu - \frac{216\cdot 10^6 \gamma^3\mu L^2 \ln^2\frac{48n(K+1)}{\beta}}{n}\right)\|\tx^k - x^*\|^2\notag\\
        &&\quad - \left(\frac{1}{2}-\gamma\mu - 6\gamma^2 L^2 - \frac{144\cdot 10^6\gamma^3\mu L^2 \ln^2\frac{48n(K+1)}{\beta}}{n}\right)\|\tx^k - x^k\|^2 \notag\\
        &&\quad  + 6\gamma^2 \|\omega_k\|^2 + 8\gamma^2\|\theta_k\|^2  + 2\gamma \langle \theta_k, x^k - x^* \rangle\notag\\
        &&\quad + \frac{72\cdot10^6\gamma^3\mu\ln^2\frac{48n(K+1)}{\beta}}{n^2}\sum_{i=1}^n\left(\|\theta_{i,k}\|^2 + \|\omega_{i,k}\|^2\right) \notag\\
        &\leq& \left(1 - \frac{\gamma\mu}{2}\right) V_k + 6\gamma^2 \|\omega_k\|^2 + 8\gamma^2\|\theta_k\|^2  + 2\gamma \langle \theta_k, x^k - x^* \rangle\notag\\
        &&\quad + \frac{\gamma^2}{n^2}\sum\limits_{i=1}^n\left(\|\theta_{i,k}\|^2 + \|\omega_{i,k}\|^2\right) \notag\\
        &\overset{\eqref{eq:SEG_sum_theta_representation},\eqref{eq:gap_thm_SEG_technical_4_full_theta_omega}}{\leq}& \exp\left(-\frac{\gamma\mu}{2}\right) V_k + 2\gamma \langle \theta_k, x^k - x^* \rangle + \frac{\gamma^2}{n^2}\sum\limits_{i=1}^n\left(18\|\theta_{i,k}^u\|^2 + 14\|\omega_{i,k}^u\|^2\right)\notag\\
        &&\quad  + \gamma^2\sum\limits_{i=1}^n\left(18\|\theta_{i,k}^b\|^2 + 14\|\omega_{i,k}^b\|^2\right) + \frac{32\gamma^2}{n^2}\sum\limits_{j=2}^n\left\langle \sum\limits_{i=1}^{j-1} \theta_{i,k}^u, \theta_{j,k}^u \right\rangle \notag\\
        &&\quad + \frac{24\gamma^2}{n^2}\sum\limits_{j=2}^n\left\langle \sum\limits_{i=1}^{j-1} \omega_{i,k}^u, \omega_{j,k}^u \right\rangle.\notag
    \end{eqnarray}
    Unrolling the recurrence, we get the result.
\end{proof}

Next, we proceed with the full statement of our main result for \algname{DProx-clipped-SEG-shift} in the quasi-strongly monotone case.

\begin{theorem}[Case 1 from Theorem~\ref{thm:D_prox_clipped_SEG_main}]\label{thm:D_prox_clipped_SEG_QSM_appendix}
    Let Assumptions \ref{as:bounded_alpha_moment}, \ref{as:L_Lip} and \ref{as:str_monotonicity} hold for $Q = B_{3n\sqrt{V}}(x^*)$, where $V \geq \|x^0 - x^*\|^2 + \frac{36000000\gamma^2\ln^2 \frac{48n(K+1)}{\beta}}{n^2}\sum_{i=1}^n\|F_i(x^*)\|^2$ and
    \begin{eqnarray}
        0< \gamma &\leq& \min\left\{\frac{1}{72\cdot 10^6 \mu \ln^2 \frac{48n(K+1)}{\beta}}, \frac{1}{6L}, \frac{\sqrt{n}}{15000 L \ln \frac{48n(K+1)}{\beta}}, \frac{2\ln(B_K)}{\mu(K+1)}\right\}, \label{eq:gamma_SEG_str_mon}\\
        B_K &=& \max\left\{2, \frac{n^{\frac{2(\alpha-1)}{\alpha}}(K+1)^{\frac{2(\alpha-1)}{\alpha}}\mu^2V}{3110400^{\frac{2}{\alpha}}\sigma^2\ln^{\frac{2(\alpha-1)}{\alpha}}\left(\frac{48n(K+1)}{\beta}\right)\ln^2(B_K)} \right\} \label{eq:B_K_SEG_str_mon_1} \\
        &=& \cO\left(\max\left\{2, \frac{n^{\frac{2(\alpha-1)}{\alpha}}K^{\frac{2(\alpha-1)}{\alpha}}\mu^2V}{\sigma^2\ln^{\frac{2(\alpha-1)}{\alpha}}\left(\frac{nK}{\beta}\right)\ln^2\left(\max\left\{2, \frac{n^{\frac{2(\alpha-1)}{\alpha}}K^{\frac{2(\alpha-1)}{\alpha}}\mu^2V}{\sigma^2\ln^{\frac{2(\alpha-1)}{\alpha}}\left(\frac{nK}{\beta}\right)} \right\}\right)} \right\}\right), \label{eq:B_K_SEG_str_mon_2} \\
        \lambda_k &=& \frac{n\exp(-\gamma\mu(1 + \nicefrac{k}{4}))\sqrt{V}}{300\gamma \ln \tfrac{48n(K+1)}{\beta}}, \label{eq:lambda_SEG_str_mon}\\
        \nu &=& \gamma\mu \label{eq:nu_SEG_str_mon}
    \end{eqnarray}
    for some $K \geq 1$ and $\beta \in (0,1]$. Then, after $K$ iterations of \algname{DProx-clipped-SEG-shift}, the following inequality holds with probability at least $1 - \beta$:
    \begin{equation}
        \|x^{K+1} - x^*\|^2 \leq 2\exp\left(-\frac{\gamma\mu(K+1)}{2}\right)V. \label{eq:main_result_str_mon}
    \end{equation}
    In particular, when $\gamma$ equals the minimum from \eqref{eq:gamma_SEG_monotone}, then after $K$ iterations of \algname{DProx-clipped-SEG-shift}, we have with probability at least $1-\beta$ that 
    \begin{eqnarray*}
        \|x^{K+1} - x^*\|^2 &=& \cO\Bigg(\max\Bigg\{V\exp\left(- \frac{K}{\ln^2 \tfrac{nK}{\beta}}\right), V\exp\left(- \frac{\mu  K}{L}\right),\notag \\
        &&\quad\quad\quad\quad\quad\quad\quad\quad\quad V\exp\left(- \frac{\mu \sqrt{n} K}{L \ln \tfrac{nK}{\beta}}\right), \frac{\sigma^2\ln^{\frac{2(\alpha-1)}{\alpha}}\left(\frac{nK}{\beta}\right)\ln^2B_K}{n^{\frac{2(\alpha-1)}{\alpha}}K^{\frac{2(\alpha-1)}{\alpha}}\mu^2}\Bigg\}\Bigg),\label{eq:clipped_SEG_str_monotone_case_2_appendix}
    \end{eqnarray*}
    i.e., to achieve $\|x^{K} - x^*\|^2 \leq \varepsilon$ with probability at least $1 - \beta$ \algname{DProx-clipped-SEG-shift} needs
    \begin{align}
        K = \cO\Bigg(\max\Bigg\{\Bigg(\frac{L}{\sqrt{n}\mu}&+\ln\left(\frac{nL}{\mu \beta}\ln\frac{V}{\varepsilon}\right)\Bigg)\ln\left(\frac{V}{\varepsilon}\right)\ln\Bigg(\frac{nL}{\mu \beta}\ln\frac{V}{\varepsilon}\Bigg),\notag\\
        &\frac{L}{\mu}\ln\left(\frac{V}{\varepsilon}\right), \frac{1}{n}\left(\frac{\sigma^2}{\mu^2\varepsilon}\right)^{\frac{\alpha}{2(\alpha-1)}}\ln \left(\frac{n}{\beta} \left(\frac{\sigma^2}{\mu^2\varepsilon}\right)^{\frac{\alpha}{2(\alpha-1)}}\right)\ln^{\frac{\alpha}{\alpha-1}}\left(B_\varepsilon\right)\Bigg\}\Bigg) \label{eq:clipped_SEG_str_monotone_case_complexity_appendix}
    \end{align}
    iterations/oracle calls per worker, where
    \begin{equation*}
        B_\varepsilon = \max\left\{2, \frac{V}{\varepsilon \ln \left(\frac{1}{\beta} \left(\frac{\sigma^2}{\mu^2\varepsilon}\right)^{\frac{\alpha}{2(\alpha-1)}}\right)}\right\}.
    \end{equation*}
\end{theorem}
\begin{proof}
    Similar to previous results, our proof is induction-based.  To formulate the statement rigorously, we introduce probability event $E_k$ for each $k = 0,1,\ldots, K+1$ as follows: inequalities
    \begin{gather}
        V_t \leq 2 \exp\left(-\frac{\gamma\mu t}{2}\right) V \label{eq:induction_inequality_str_mon_SEG}\\
         \left\|\frac{\gamma}{n}\sum\limits_{i=1}^{r-1}\theta_{i,t-1}^u\right\| \leq \exp\left(-\frac{\gamma\mu (t-1)}{4}\right)\frac{\sqrt{V}}{2}, \label{eq:induction_inequality_str_mon_SEG_1}\\
         \left\|\frac{\gamma}{n}\sum\limits_{i=1}^{r-1}\omega_{i,t-1}^u\right\| \leq \exp\left(-\frac{\gamma\mu (t-1)}{4}\right)\frac{\sqrt{V}}{2} \label{eq:induction_inequality_str_mon_SEG_2}
    \end{gather}
    hold for $t = 0,1,\ldots,k$ and $r = 1, 2, \ldots, n$ simultaneously. We will prove by induction that $\PP\{E_k\} \geq  1 - \nicefrac{k\beta}{(K+1)}$ for all $k = 0,1,\ldots,K+1$. The base of induction follows immediately by the definition of $V$. Next, we assume that the statement holds for $k = T-1 \leq K$, i.e., $\PP\{E_{T-1}\} \geq  1 - \nicefrac{(T-1)\beta}{(K+1)}$. Let us show that it also holds for $k = T$, i.e., $\PP\{E_{T}\} \geq  1 - \nicefrac{T\beta}{(K+1)}$.

    Similarly to the monotone case, one can show that due to our choice of the clipping level, we have that $E_{T-1}$ implies $x^t, \tx^t \in B_{4n\sqrt{V}}(x^*)$ for $t=0,\ldots,T-1$. Indeed, for $t = 0,1,\ldots,T-1$ inequality \eqref{eq:induction_inequality_str_mon_SEG} gives $x^t \in B_{2\sqrt{V}}(x^*)$. Next, for $\tx^t$, $t = 0, \ldots, T-1$ event $E_{T-1}$ implies
    \begin{eqnarray*}
        \|\tx^t - x^*\| &=& \|\prox_{\gamma\Psi}(x^t - \gamma \tg^t) - \prox_{\gamma \Psi}(x^* - \gamma F(x^*))\| \\
        &\leq& \|x^t - x^* - \gamma (\tg^t - F(x^*))\| \\
        &\leq& \|x^t - x^*\| + \gamma \left\|\frac{1}{n}\sum\limits_{i=1}^n (\thh_i^t - F_i(x^*)) + \frac{1}{n}\sum\limits_{i=1}^n \tilde\Delta_i^t \right\| \\
        &\leq& 2\sqrt{V} + \gamma \sqrt{\frac{1}{n}\sum\limits_{i=1}^n \|\thh_i^t - F_i(x^*)\|^2} + \frac{\gamma}{n}\sum\limits_{i=1}^n \|\tilde \Delta_i^t\|\\
        &\leq& 2\sqrt{V} + \frac{\sqrt{n}}{6000 \ln \frac{48n(K+1)}{\beta}}\sqrt{V_t} + \gamma \lambda_t\\
        &\overset{\eqref{eq:induction_inequality_str_mon_SEG}}{\leq}& \left(2 + \frac{20n + \sqrt{2n}}{6000 \ln \frac{48n(K+1)}{\beta}}\right)\sqrt{V} \leq 4n\sqrt{V}.
    \end{eqnarray*}

    This means that we can apply Lemma~\ref{lem:SEG_optimization_lemma_QSM}: $E_{T-1}$ implies
    \begin{eqnarray}
       V_{T} &\leq& \exp\left(-\frac{\gamma\mu}{2}T\right) V + 2\gamma \sum\limits_{t=0}^{T-1}\exp\left(-\frac{\gamma\mu}{2}(T-1-t)\right)\langle \theta_t, x^t - x^* \rangle\notag\\
       &&\quad + \frac{\gamma^2}{n^2}\sum\limits_{t=0}^{T-1}\sum\limits_{i=1}^n\exp\left(-\frac{\gamma\mu}{2}(T-1-t)\right)\left(18\|\theta_{i,t}^u\|^2 + 14\|\omega_{i,t}^u\|^2\right)\notag\\
        &&\quad  + \gamma^2\sum\limits_{t=0}^{T-1}\sum\limits_{i=1}^n\exp\left(-\frac{\gamma\mu}{2}(T-1-t)\right)\left(18\|\theta_{i,t}^b\|^2 + 14\|\omega_{i,t}^b\|^2\right)\notag\\
        &&\quad + \frac{32\gamma^2}{n^2}\sum\limits_{t=0}^{T-1}\sum\limits_{j=2}^n\exp\left(-\frac{\gamma\mu}{2}(T-1-t)\right)\left\langle \sum\limits_{i=1}^{j-1} \theta_{i,t}^u, \theta_{j,t}^u \right\rangle \notag\\
        &&\quad + \frac{24\gamma^2}{n^2}\sum\limits_{t=0}^{T-1}\sum\limits_{j=2}^n\exp\left(-\frac{\gamma\mu}{2}(T-1-t)\right)\left\langle \sum\limits_{i=1}^{j-1} \omega_{i,t}^u, \omega_{j,t}^u \right\rangle. \notag
    \end{eqnarray}
    Before we proceed, we introduce a new notation:
    \begin{equation*}
        \eta_t = \begin{cases}x^t - x^*,& \text{if } \|x^t - x^*\| \leq \exp\left(-\frac{\gamma\mu t}{4}\right)\sqrt{2V},\\ 0,& \text{otherwise,} \end{cases}
    \end{equation*}
    for all $t = 0, 1, \ldots, T$. Random vectors $\{\eta_t\}_{t=0}^T$ are bounded almost surely:
     \begin{equation}
        \|\eta_t\| \leq \exp\left(-\frac{\gamma\mu t}{4}\right)\sqrt{2V}  \label{eq:eta_t_bound_SEG_str_mon}
    \end{equation}
    for all $t = 0, 1, \ldots, T$. In addition, $\eta_t = x^t - x^*$ follows from $E_{T-1}$ for all $t = 0, 1, \ldots, T$ and, thus, $E_{T-1}$ implies
    \begin{eqnarray}
       V_{T} &\leq& \exp\left(-\frac{\gamma\mu}{2}T\right) V + \underbrace{\frac{2\gamma}{n} \sum\limits_{l=0}^{T-1}\sum\limits_{i=1}^n\exp\left(-\frac{\gamma\mu}{2}(T-1-l)\right)\langle \theta_{i,l}^u, \eta_l \rangle}_{\circledOne}\notag\\
       &&\quad + \underbrace{2\gamma \sum\limits_{l=0}^{T-1}\exp\left(-\frac{\gamma\mu}{2}(T-1-l)\right)\langle \theta_{l}^b, \eta_l \rangle}_{\circledTwo}\notag\\
       &&\quad + \underbrace{\frac{\gamma^2}{n^2}\sum\limits_{l=0}^{T-1}\sum\limits_{i=1}^n\exp\left(-\frac{\gamma\mu}{2}(T-1-l)\right)\left(18\EE_{\xi_{2,i}^l}\left[\|\theta_{i,l}^u\|^2\right] + 14\EE_{\xi_{1,i}^l}\left[\|\omega_{i,l}^u\|^2\right]\right)}_{\circledThree}\notag\\
       &&\quad + \underbrace{\frac{\gamma^2}{n^2}\sum\limits_{l=0}^{T-1}\sum\limits_{i=1}^n\exp\left(-\frac{\gamma\mu}{2}(T-1-l)\right)\left(18\|\theta_{i,l}^u\|^2 + 14\|\omega_{i,l}^u\|^2 - 18\EE_{\xi_{2,i}^l}\left[\|\theta_{i,l}^u\|^2\right] - 14\EE_{\xi_{1,i}^l}\left[\|\omega_{i,l}^u\|^2\right]\right)}_{\circledFour}\notag\\
        &&\quad  + \underbrace{\frac{\gamma^2}{n^2}\sum\limits_{l=0}^{T-1}\sum\limits_{i=1}^n\exp\left(-\frac{\gamma\mu}{2}(T-1-l)\right)\left(18\|\theta_{i,l}^b\|^2 + 14\|\omega_{i,l}^b\|^2\right)}_{\circledFive}\notag\\
        &&\quad + \underbrace{\frac{32\gamma^2}{n^2}\sum\limits_{l=0}^{T-1}\sum\limits_{j=2}^n\exp\left(-\frac{\gamma\mu}{2}(T-1-l)\right)\left\langle \sum\limits_{i=1}^{j-1} \theta_{i,l}^u, \theta_{j,l}^u \right\rangle}_{\circledSix} \notag\\
        &&\quad + \underbrace{\frac{24\gamma^2}{n^2}\sum\limits_{l=0}^{T-1}\sum\limits_{j=2}^n\exp\left(-\frac{\gamma\mu}{2}(T-1-l)\right)\left\langle \sum\limits_{i=1}^{j-1} \omega_{i,l}^u, \omega_{j,l}^u \right\rangle}_{\circledSeven}. \label{eq:ndvjinuibvbsvdhbvdhbfussdncdni}
    \end{eqnarray}

    To derive high-probability bounds for $\circledOne, \circledTwo, \circledThree, \circledFour, \circledFive, \circledSix, \circledSeven$ we need to establish several useful inequalities related to $\theta_{i,l}^u, \theta_{i,l}^b, \omega_{i,l}^u, \omega_{i,l}^b$. First, by definition of clipping
    \begin{equation}
        \|\theta_{i,l}^u\| \leq 2\lambda_l,\quad \|\omega_{i,l}^u\| \leq 2\lambda_l. \label{eq:theta_omega_magnitude_str_mon}
    \end{equation}
    Next, we notice that $E_{T-1}$ implies
    \begin{eqnarray*}
        \|F_i(x^l) - \thh_i^l\| &\leq& \|F_i(x^l) - F_i(x^*)\| + \|\thh_i^l - F_i(x^*)\|\\
        &\overset{\eqref{eq:L_Lip}}{\leq}& L\|x^l - x^*\| + \sqrt{\sum\limits_{j=1}^n \|\thh_i^l - F_i(x^*)\|^2}\\
        &\leq& L\sqrt{V_l} + \frac{n\sqrt{V_l}}{6000\gamma \ln \frac{48n(K+1)}{\beta}}\\
        &\overset{\eqref{eq:induction_inequality_str_mon_SEG}}{\leq}& \sqrt{2}\left(L + \frac{n}{6000\gamma \ln \frac{48n(K+1)}{\beta}}\right)\exp\left(-\frac{\gamma\mu l}{4}\right)\sqrt{V} \overset{\eqref{eq:gamma_SEG_str_mon}, \eqref{eq:lambda_SEG_str_mon}}{\leq} \frac{\lambda_l}{2}
    \end{eqnarray*}
    and
    \begin{eqnarray*}
        \|F_i(\tx^l) - \hh_i^l\| &\leq& \|F_i(\tx^l) - F_i(x^*)\| + \|\hh_i^l - F_i(x^*)\|\\
        &\overset{\eqref{eq:L_Lip}}{\leq}& L\|\tx^l - x^*\| + \sqrt{\sum\limits_{j=1}^n \|\hh_i^l - F_i(x^*)\|^2}\\
        &\leq& L\|\prox_{\gamma\Psi}(x^l - \gamma \tg^l) - \prox_{\gamma \Psi}(x^* - \gamma F(x^*))\| + \frac{n\sqrt{V_l}}{6000\gamma \ln \frac{48n(K+1)}{\beta}}\\
        &\leq& L\|x^l - x^* - \gamma (\tg^l - F(x^*))\| + \frac{n\sqrt{V_l}}{6000\gamma \ln \frac{48n(K+1)}{\beta}}\\
        &\leq& L\|x^l - x^*\| + L\gamma \left\|\frac{1}{n}\sum\limits_{i=1}^n (\thh_i^l - F_i(x^*)) + \frac{1}{n}\sum\limits_{i=1}^n \tilde\Delta_i^l \right\| + \frac{n\sqrt{V_l}}{6000\gamma \ln \frac{48n(K+1)}{\beta}}\\
        &\leq& \left(L + \frac{n}{6000\gamma \ln \frac{48n(K+1)}{\beta}}\right)\sqrt{V_l} + L\gamma \sqrt{\frac{1}{n}\sum\limits_{i=1}^n \|\thh_i^l - F_i(x^*)\|^2} + \frac{L\gamma}{n}\sum\limits_{i=1}^n \|\tilde \Delta_i^l\|\\
        &\leq& \left(L + \frac{n + L\gamma\sqrt{n}}{6000\gamma \ln \frac{48n(K+1)}{\beta}}\right)\sqrt{V_l} + L\gamma \lambda_t\\
        &\overset{\eqref{eq:induction_inequality_str_mon_SEG}}{\leq}& \sqrt{2}\left(L + \frac{n + L\gamma\sqrt{n}}{6000\gamma \ln \frac{48n(K+1)}{\beta}}\right)\exp\left(-\frac{\gamma\mu l}{4}\right)\sqrt{V} + L\gamma \lambda_l \overset{\eqref{eq:gamma_SEG_str_mon}, \eqref{eq:lambda_SEG_str_mon}}{\leq} \frac{\lambda_l}{2}
    \end{eqnarray*}
    for $l = 0,1,\ldots,T-1$ and $i\in [n]$. Therefore, one can apply Lemma~\ref{lem:bias_and_variance_clip} and get
    \begin{gather}
        \left\|\theta_l^b\right\| \leq \frac{1}{n}\sum\limits_{i=1}^n \|\theta_{i,l}^b\| \leq \frac{2^\alpha\sigma^\alpha}{\lambda_l^{\alpha-1}},\quad \left\|\omega_l^b\right\| \leq \frac{1}{n}\sum\limits_{i=1}^n \|\omega_{i,l}^b\| \leq \frac{2^\alpha\sigma^\alpha}{\lambda_l^{\alpha-1}}, \label{eq:bias_theta_omega_str_mon}\\
        \EE_{\xi_{2,i}^l}\left[\left\|\theta_{i,l}^u\right\|^2\right] \leq 18 \lambda_l^{2-\alpha}\sigma^\alpha,\quad \EE_{\xi_{1,i}^l}\left[\left\|\omega_{i,l}^u\right\|^2\right] \leq 18 \lambda_l^{2-\alpha}\sigma^\alpha, \label{eq:variance_theta_omega_str_mon}
    \end{gather}
    for all $l = 0,1, \ldots, T-1$ and $i\in [n]$.

    \paragraph{Upper bound for $\circledOne$.} To estimate this sum, we will use Bernstein's inequality. The summands have conditional expectations equal to zero:
    \begin{equation*}
        \EE_{\xi_{2,i}^l}\left[\frac{2\gamma}{n}\exp\left(-\frac{\gamma\mu}{2}(T-1-l)\right) \langle \eta_l, \theta_{i,l}^u \rangle\right] = \frac{2\gamma}{n} \exp\left(-\frac{\gamma\mu}{2}(T-1-l)\right)\left\langle \eta_l, \EE_{\xi_{2,i}^l}[\theta_{i,l}^u] \right\rangle = 0.
    \end{equation*}
    Moreover, for all $l = 0,\ldots, T-1$ random vectors $\{\theta_{i,l}^u\}_{i=1}^n$ are independent. Thus, sequence $\left\{\frac{2\gamma}{n}\exp\left(-\frac{\gamma\mu}{2}(T-1-l)\right) \langle \eta_l, \theta_{i,l}^u \rangle\right\}_{l,i = 0,1}^{T-1,n}$ is a martingale difference sequence. Next, the summands are bounded:
    \begin{eqnarray}
        \left|\frac{2\gamma}{n} \exp\left(-\frac{\gamma\mu}{2}(T-1-l)\right)\langle \eta_l, \theta_{i,l}^u \rangle \right| &\leq& \frac{2\gamma}{n}\exp\left(-\frac{\gamma\mu}{2}(T-1-l)\right) \|\eta_l\|\cdot \|\theta_{i,l}^u\|\notag\\
        &\overset{\eqref{eq:eta_t_bound_SEG_str_mon},\eqref{eq:theta_omega_magnitude_str_mon}}{\leq}& \frac{2\sqrt{2V}\gamma \exp\left(-\frac{\gamma\mu(T-1)}{2}\right)}{n} \exp\left(\frac{\gamma\mu l}{4}\right)\lambda_l\notag\\
        &\overset{\eqref{eq:lambda_SEG_monotone}}{=}& \frac{\exp\left(-\frac{\gamma\mu T}{2}\right)V}{100\ln\frac{48n(K+1)}{\beta}} \eqdef c. \label{eq:str_mon_thm_SEG_technical_6_5}
    \end{eqnarray}
    Finally, conditional variances $\sigma_{i,l}^2 \eqdef \EE_{\xi_{2,i}^l}\left[\frac{4\gamma^2}{n^2}\exp\left(-\gamma\mu(T-1-l)\right) \langle \eta_l, \theta_{i,l}^u \rangle^2\right]$ of the summands are bounded:
    \begin{eqnarray}
        \sigma_{i,l}^2 &\leq& \EE_{\xi_{2,i}^l}\left[\frac{4\gamma^2 }{n^2}\exp\left(-\gamma\mu(T-1-l)\right)\|\eta_l\|^2\cdot \|\theta_{i,l}^u\|^2\right] \notag\\
        &\overset{\eqref{eq:eta_t_bound_SEG_str_mon}}{\leq}& \frac{8\gamma^2 V\exp\left(-\gamma\mu\left(T-1-\frac{l}{2}\right)\right)}{n^2} \EE_{\xi_{2,i}^l}\left[\|\theta_{i,l}^u\|^2\right]. \label{eq:str_mon_thm_SEG_technical_7}
    \end{eqnarray}
    Applying Bernstein's inequality (Lemma~\ref{lem:Bernstein_ineq}) with $X_{i,l} = \frac{2\gamma}{n}\exp\left(-\frac{\gamma\mu}{2}(T-1-l)\right) \langle \eta_l, \theta_{i,l}^u \rangle$, constant $c$ defined in \eqref{eq:str_mon_thm_SEG_technical_6_5}, $b = \frac{\exp\left(-\frac{\gamma\mu T}{2}\right)V}{100}$, $G = \frac{\exp\left(-\gamma\mu T\right)V^2}{60000\ln\frac{48n(K+1)}{\beta}}$, we get
    \begin{eqnarray*}
        \PP\left\{|\circledOne| > \frac{\exp\left(-\frac{\gamma\mu T}{2}\right)V}{100} \text{ and } \sum\limits_{l=0}^{T-1}\sum\limits_{i=1}^n\sigma_{i,l}^2 \leq \frac{\exp\left(-\gamma\mu T\right)V^2}{60000\ln\frac{48n(K+1)}{\beta}}\right\} &\leq& 2\exp\left(- \frac{b^2}{2G + \nicefrac{2cb}{3}}\right)\\
        &=& \frac{\beta}{24n(K+1)}.
    \end{eqnarray*}
    The above is equivalent to $\PP\{E_{\circledOne}\} \geq 1 - \frac{\beta}{24n(K+1)}$ for
    \begin{equation}
        E_{\circledOne} = \left\{\text{either } \sum\limits_{l=0}^{T-1}\sum\limits_{i=1}^n\sigma_{i,l}^2 > \frac{\exp\left(-\gamma\mu T\right)V^2}{60000\ln\frac{48n(K+1)}{\beta}} \text{ or } |\circledOne| \leq \frac{\exp\left(-\frac{\gamma\mu T}{2}\right)V}{100}\right\}. \label{eq:bound_1_str_mon_SEG}
    \end{equation}
    Moreover, $E_{T-1}$ implies
    \begin{eqnarray}
        \sum\limits_{l=0}^{T-1}\sum\limits_{i=1}^n\sigma_{i,l}^2 &\overset{\eqref{eq:str_mon_thm_SEG_technical_7}}{\leq}& \frac{8\gamma^2 V\exp\left(-\gamma\mu\left(T-1\right)\right)}{n^2} \sum\limits_{l=0}^{T-1}\exp\left(\frac{\gamma\mu l}{2}\right) \sum\limits_{i=1}^n \EE_{\xi_{2,i}^l}\left[\|\theta_{i,l}^u\|^2\right]\notag\\
        &\overset{\eqref{eq:variance_theta_omega_str_mon}, T \leq K+1}{\leq}&  \frac{144\gamma^2 V\exp\left(-\gamma\mu\left(T-1\right)\right)\sigma^\alpha}{n} \sum\limits_{l=0}^{T-1}\exp\left(\frac{\gamma\mu l}{2}\right) \lambda_l^{2-\alpha}\notag\\
        &\overset{\eqref{eq:lambda_SEG_str_mon}}{\leq}& \frac{144\gamma^\alpha V^{2 - \frac{\alpha}{2}}\exp\left(-\gamma\mu T\right)\sigma^\alpha}{6000^{2-\alpha} n^{\alpha-1} \ln^{2-\alpha}\frac{48n(K+1)}{\beta}} \sum\limits_{l=0}^{T-1}\exp\left(\frac{\gamma\mu l \alpha}{4}\right)\notag\\
        &\leq& \frac{144\gamma^\alpha V^{2 - \frac{\alpha}{2}}\exp\left(-\gamma\mu T\right)\sigma^\alpha (K+1) \exp\left(\frac{\gamma\mu K \alpha}{4}\right)}{6000^{2-\alpha} n^{\alpha-1} \ln^{2-\alpha}\frac{48n(K+1)}{\beta}}\notag\\
        &\overset{\eqref{eq:gamma_SEG_str_mon}}{\leq}& \frac{\exp\left(-\gamma\mu T\right)V^2}{60000\ln\frac{48n(K+1)}{\beta}}.\label{eq:bound_1_variances_str_mon_SEG}
    \end{eqnarray}

    \paragraph{Upper bound for $\circledTwo$.} Probability event $E_{T-1}$ implies
    \begin{eqnarray}
        \circledTwo &\leq& 2\gamma \sum\limits_{l=0}^{T-1}\exp\left(-\frac{\gamma\mu(T-1-l)}{2}\right)\|\eta_l\| \cdot \|\theta_l^b\| \overset{\eqref{eq:eta_t_bound_SEG_str_mon}, \eqref{eq:bias_theta_omega_str_mon}}{\leq} \\
        &\leq& 2^{\alpha+1}\gamma\sigma^\alpha\sqrt{2V}\exp\left(-\frac{\gamma\mu(T-1)}{2}\right) \sum\limits_{l=0}^{T-1}\exp\left(\frac{\gamma\mu l}{4}\right)\frac{1}{\lambda_l^{\alpha-1}}\notag\\
        &\overset{\eqref{eq:lambda_SEG_str_mon}}{\leq}& \frac{2^{\alpha+1}\cdot 120^{\alpha-1} \exp\left(-\frac{\gamma\mu T}{2}\right) (K+1) \exp\left(\frac{\gamma\mu K \alpha}{4}\right)\gamma^\alpha \sigma^\alpha \ln^{\alpha-1}\frac{48n(K+1)}{\beta}}{n^{\alpha-1}V^{\frac{\alpha}{2}-1}} \notag\\
        &\overset{\eqref{eq:gamma_SEG_str_mon}}{\leq}& \frac{3\exp\left(-\frac{\gamma\mu T}{2}\right)V}{100}. \label{eq:bound_2_variances_str_mon_SEG}
    \end{eqnarray}

    \paragraph{Upper bound for $\circledThree$.} Probability event $E_{T-1}$ implies
    \begin{eqnarray}
        \frac{18\gamma^2}{n^2}\sum\limits_{l=0}^{T-1}\sum\limits_{i=1}^n\exp\left(-\frac{\gamma\mu(T-1-l)}{2}\right)\EE_{\xi_{2,i}^l}\left[\|\theta_{i,l}^u\|^2\right] &\overset{\eqref{eq:bound_1_variances_str_mon_SEG}}{\leq}& \frac{\exp\left(-\frac{\gamma\mu T}{2}\right)V}{100} \notag
    \end{eqnarray}
    and, similarly,
    \begin{eqnarray}
        \frac{14\gamma^2}{n^2}\sum\limits_{l=0}^{T-1}\sum\limits_{i=1}^n\exp\left(-\frac{\gamma\mu(T-1-l)}{2}\right)\EE_{\xi_{1,i}^l}\left[\|\omega_{i,l}^u\|^2\right] &\overset{\eqref{eq:bound_1_variances_str_mon_SEG}}{\leq}& \frac{\exp\left(-\frac{\gamma\mu T}{2}\right)V}{100} \notag
    \end{eqnarray}
    that give
    \begin{equation}
        \circledThree \leq \frac{\exp\left(-\frac{\gamma\mu T}{2}\right)V}{50}. \label{eq:bound_3_variances_str_mon_SEG}
    \end{equation}

    \paragraph{Upper bound for $\circledFour$.} 
    To estimate this sum, we will use Bernstein's inequality. The summands have conditional expectations equal to zero:
    \begin{equation*}
        \frac{2\gamma^2}{n^2}\EE_{\xi_{1,i}^l,\xi_{2,i}^l}\left[18\|\theta_{i,l}^u\|^2 + 14\|\omega_{i,l}^u\|^2 - 18\EE_{\xi_{2,i}^l}\left[\|\theta_{i,l}^u\|^2\right] - 14\EE_{\xi_{1,i}^l}\left[\|\omega_{i,l}^u\|^2\right]\right] = 0.
    \end{equation*}
    Moreover, for all $l = 0,\ldots, T-1$ random vectors $\{\theta_{i,l}^u\}_{i=1}^n$, $\{\omega_{i,l}^u\}_{i=1}^n$ are independent. Thus, sequence $\left\{\frac{2\gamma^2}{n^2}\exp\left(-\frac{\gamma\mu}{2}(T-1-l)\right)\left(18\|\theta_{i,l}^u\|^2 + 14\|\omega_{i,l}^u\|^2 - 18\EE_{\xi_{2,i}^l}\left[\|\theta_{i,l}^u\|^2\right] - 14\EE_{\xi_{1,i}^l}\left[\|\omega_{i,l}^u\|^2\right]\right)\right\}_{l,i = 0,1}^{T-1,n}$ is a martingale difference sequence. Next, the summands are bounded:
    \begin{align}
        \frac{2\gamma^2}{n^2}\exp\left(-\frac{\gamma\mu}{2}(T-1-l)\right)\Big|18\|\theta_{i,l}^u\|^2 &+ 14\|\omega_{i,l}^u\|^2 - 18\EE_{\xi_{2,i}^l}\left[\|\theta_{i,l}^u\|^2\right] - 14\EE_{\xi_{1,i}^l}\left[\|\omega_{i,l}^u\|^2\right] \Big| \notag\\
        &\overset{\eqref{eq:theta_omega_magnitude_str_mon}}{\leq} \frac{256\exp\left(-\frac{\gamma\mu}{2}(T-1-l)\right)\gamma^2 \lambda_l^2}{n^2} \notag\\
        &\overset{\eqref{eq:lambda_SEG_str_mon}}{\leq} \frac{\exp\left(-\frac{\gamma\mu T}{2}\right)V}{6\ln\frac{48n(K+1)}{\beta}} \eqdef c.\label{eq:str_mon_thm_SEG_technical_6_5_1}
    \end{align}

    Finally, conditional variance 
    $$\widetilde\sigma^2_{i,l} \eqdef \frac{4\gamma^4}{n^4}\exp\left(-\gamma\mu(T-1-l)\right)\EE_{\xi_{1,i}^l,\xi_{2,i}^l}\left[\left|18\|\theta_{i,l}^u\|^2 + 14\|\omega_{i,l}^u\|^2 - 18\EE_{\xi_{2,i}^l}\left[\|\theta_{i,l}^u\|^2\right] - 14\EE_{\xi_{1,i}^l}\left[\|\omega_{i,l}^u\|^2\right]\right|^2\right] $$
    of the summands are bounded: 
    \begin{eqnarray}
        \widetilde\sigma_{i,l}^2 &\overset{\eqref{eq:str_mon_thm_SEG_technical_6_5_1}}{\leq}& \frac{\gamma^2\exp\left(-\frac{\gamma\mu }{2}(2T-1-l)\right)V}{3n^2\ln\frac{48n(K+1)}{\beta}}\notag\\
        &&\times~ \EE_{\xi_{1,i}^l, \xi_{2,i}^l}\left[\left|18\|\theta_{i,l}^u\|^2 + 14\|\omega_{i,l}^u\|^2 - 18\EE_{\xi_{2,i}^l}\left[\|\theta_{i,l}^u\|^2\right] - 14\EE_{\xi_{1,i}^l}\left[\|\omega_{i,l}^u\|^2\right] \right|\right]\notag\\
        &\leq& \frac{4\gamma^2\exp\left(-\frac{\gamma\mu }{2}(2T-1-l)\right)V}{3n^2\ln\frac{48n(K+1)}{\beta}} \EE_{\xi_{1,i}^l, \xi_{2,i}^l}\left[9\|\theta_{i,l}^u\|^2 + 7\|\omega_{i,l}^u\|^2 \right]. \label{eq:str_mon_thm_SEG_technical_7_1}
    \end{eqnarray}
    Applying Bernstein's inequality (Lemma~\ref{lem:Bernstein_ineq}) with\newline $X_{i,l} = \frac{2\gamma^2}{n^2}\exp\left(-\frac{\gamma\mu}{2}(T-1-l)\right)\left(18\|\theta_{i,l}^u\|^2 + 14\|\omega_{i,l}^u\|^2 - 18\EE_{\xi_{2,i}^l}\left[\|\theta_{i,l}^u\|^2\right] - 14\EE_{\xi_{1,i}^l}\left[\|\omega_{i,l}^u\|^2\right]\right)$, constant $c$ defined in \eqref{eq:str_mon_thm_SEG_technical_6_5_1}, $b = \frac{\exp\left(-\frac{\gamma\mu T}{2}\right)V}{6}$, $G = \frac{\exp\left(-\gamma\mu T\right)V^2}{216\ln\frac{48n(K+1)}{\beta}}$, we get
    \begin{eqnarray*}
        \PP\left\{|\circledFour| > \frac{\exp\left(-\frac{\gamma\mu T}{2}\right)V}{6} \text{ and } \sum\limits_{l=0}^{T-1}\sum\limits_{i=1}^n\widetilde\sigma_{i,l}^2 \leq \frac{\exp\left(-\gamma\mu T\right)V^2}{216\ln\frac{48n(K+1)}{\beta}}\right\} &\leq& 2\exp\left(- \frac{b^2}{2G + \nicefrac{2cb}{3}}\right)\\
        &=& \frac{\beta}{24n(K+1)}.
    \end{eqnarray*}
    The above is equivalent to $\PP\{E_{\circledFour}\} \geq 1 - \frac{\beta}{24n(K+1)}$ for
    \begin{equation}
        E_{\circledFour} = \left\{\text{either } \sum\limits_{l=0}^{T-1}\sum\limits_{i=1}^n\widetilde\sigma_{i,l}^2 > \frac{\exp\left(-\gamma\mu T\right)V^2}{216\ln\frac{48n(K+1)}{\beta}} \text{ or } |\circledFour| \leq \frac{\exp\left(-\frac{\gamma\mu T}{2}\right)V}{6}\right\}. \label{eq:bound_1_str_mon_SEG_1}
    \end{equation}
    Moreover, $E_{T-1}$ implies
    \begin{eqnarray}
        \sum\limits_{l=0}^{T-1}\sum\limits_{i=1}^n\widetilde\sigma_{i,l}^2 &\overset{\eqref{eq:str_mon_thm_SEG_technical_7_1}}{\leq}& \frac{4\exp\left(-\gamma\mu (T-\frac{1}{2})\right)\gamma^2V}{3n^2\ln\frac{48n(K+1)}{\beta}} \sum\limits_{l=0}^{T-1}\exp\left(\frac{\gamma\mu l}{2}\right)\sum\limits_{i=1}^n \EE_{\xi_{1,i}^l, \xi_{2,i}^l}\left[9\|\theta_{i,l}^u\|^2 + 7\|\omega_{i,l}^u\|^2 \right]\notag\\
        &\overset{\eqref{eq:bound_1_variances_str_mon_SEG}}{\leq}& \frac{\exp\left(-\gamma\mu T\right)V}{216\ln\frac{48n(K+1)}{\beta}}. \label{eq:bound_1_variances_str_mon_SEG_1}
    \end{eqnarray}

    \paragraph{Upper bound for $\circledFive$.} Probability event $E_{T-1}$ implies
    \begin{eqnarray}
        \circledFive &\overset{\eqref{eq:bias_theta_omega_str_mon}}{\leq}& 32\gamma^2\sum\limits_{l=0}^{T-1}\exp\left(-\frac{\gamma\mu}{2}(T-1-l)\right) \frac{2^{2\alpha}\sigma^{2\alpha}}{\lambda_{l}^{2\alpha-2}} \overset{\eqref{eq:lambda_SEG_str_mon}, \eqref{eq:gamma_SEG_str_mon}}{\leq}  \frac{\exp\left(-\frac{\gamma\mu T}{2}\right)V}{6}. \label{eq:bound_5_variances_str_mon_SEG}
    \end{eqnarray}

    \paragraph{Upper bounds for $\circledSix$ and $\circledSeven$.} These sums require more refined analysis. We introduce new vectors:
    \begin{eqnarray}
        \zeta_{j}^l = \begin{cases}
            \frac{\gamma}{n}\sum\limits_{i=1}^{j-1} \theta_{i,l}^u,& \text{if } \left\|\frac{\gamma}{n}\sum\limits_{i=1}^{j-1} \theta_{i,l}^u\right\| \leq \exp\left(-\frac{\gamma\mu l}{4}\right)\frac{\sqrt{V}}{2},\\
            0,& \text{otherwise,}
        \end{cases}, \label{eq:SEG_zeta_delta_def_str_mon_1} \\
        \delta_{j}^l = \begin{cases}
            \frac{\gamma}{n}\sum\limits_{i=1}^{j-1} \omega_{i,l}^u,& \text{if } \left\|\frac{\gamma}{n}\sum\limits_{i=1}^{j-1} \omega_{i,l}^u\right\| \leq \exp\left(-\frac{\gamma\mu l}{4}\right)\frac{\sqrt{V}}{2},\\
            0,& \text{otherwise,}
        \end{cases}\label{eq:SEG_zeta_delta_def_str_mon_2}
    \end{eqnarray}
    for all $j \in [n]$ and $l = 0,\ldots, T-1$. Then, by definition
    \begin{equation}
        \|\zeta_j^l\| \leq \exp\left(-\frac{\gamma\mu l}{4}\right)\frac{\sqrt{V}}{2},\quad \|\delta_j^l\| \leq \exp\left(-\frac{\gamma\mu l}{4}\right)\frac{\sqrt{V}}{2} \label{eq:SEG_zeta_delta_bound_str_mon}
    \end{equation}
    and
    \begin{eqnarray}
        \circledSix &=& \underbrace{\frac{32\gamma}{n} \sum\limits_{l=0}^{T-1} \sum\limits_{j=2}^n \exp\left(-\frac{\gamma\mu}{2}(T-1-l)\right)\left\langle \zeta_j^l, \theta_{j,l}^u \right\rangle}_{\circledSix'}\notag\\
        &&\quad + \frac{32\gamma}{n} \sum\limits_{l=0}^{T-1} \sum\limits_{j=2}^n\exp\left(-\frac{\gamma\mu}{2}(T-1-l)\right) \left\langle \frac{\gamma}{n}\sum\limits_{i=1}^{j-1} \theta_{i,l}^u - \zeta_j^l, \theta_{j,l}^u \right\rangle, \label{eq:SEG_extra_sums_six_distributed_str_mon} \\
        \circledSeven &=& \underbrace{\frac{24\gamma}{n} \sum\limits_{l=0}^{T-1} \sum\limits_{j=2}^n \exp\left(-\frac{\gamma\mu}{2}(T-1-l)\right)\left\langle \delta_j^l, \omega_{j,l}^u \right\rangle}_{\circledSeven'}\notag\\
        &&\quad + \frac{24\gamma}{n} \sum\limits_{l=0}^{T-1} \sum\limits_{j=2}^n \exp\left(-\frac{\gamma\mu}{2}(T-1-l)\right) \left\langle \frac{\gamma}{n}\sum\limits_{i=1}^{j-1} \omega_{i,l}^u - \delta_j^l, \omega_{j,l}^u \right\rangle. \label{eq:SEG_extra_sums_seven_distributed_str_mon}
    \end{eqnarray}
    We also note here that $E_{T-1}$ implies
    \begin{align}
        \frac{32\gamma}{n} \sum\limits_{l=0}^{T-1} \sum\limits_{j=2}^n\exp\left(-\frac{\gamma\mu}{2}(T-1-l)\right)& \left\langle \frac{\gamma}{n}\sum\limits_{i=1}^{j-1} \theta_{i,l}^u - \zeta_j^l, \theta_{j,l}^u \right\rangle \notag \\
        &= \frac{32\gamma}{n}  \sum\limits_{j=2}^n \left\langle \frac{\gamma}{n}\sum\limits_{i=1}^{j-1} \theta_{i,T-1}^u - \zeta_j^{T-1}, \theta_{j,T-1}^u \right\rangle, \label{eq:SEG_extra_sums_six_distributed_1_str_mon}\\
        \frac{24\gamma}{n} \sum\limits_{l=0}^{T-1} \sum\limits_{j=2}^n \exp\left(-\frac{\gamma\mu}{2}(T-1-l)\right)&\left\langle \frac{\gamma}{n}\sum\limits_{i=1}^{j-1} \omega_{i,l}^u - \delta_j^l, \omega_{j,l}^u \right\rangle \notag\\
        &= \frac{24\gamma}{n}  \sum\limits_{j=2}^n \left\langle \frac{\gamma}{n}\sum\limits_{i=1}^{j-1} \omega_{i,T-1}^u - \delta_j^{T-1}, \omega_{j,T-1}^u \right\rangle. \label{eq:SEG_extra_sums_seven_distributed_1_str_mon}
    \end{align}

    \paragraph{Upper bound for $\circledSix'$.} To estimate this sum, we will use Bernstein's inequality. The summands have conditional expectations equal to zero:
    \begin{equation*}
        \EE_{\xi_{2,j}^l}\left[\frac{32\gamma}{n}\exp\left(-\frac{\gamma\mu}{2}(T-1-l)\right) \langle \zeta_j^l, \theta_{j,l}^u \rangle\right] = \frac{32\gamma}{n} \exp\left(-\frac{\gamma\mu}{2}(T-1-l)\right)\left\langle \zeta_j^l, \EE_{\xi_{2,j}^l}[\theta_{j,l}^u] \right\rangle = 0.
    \end{equation*}
    Moreover, for all $l = 0,\ldots, T-1$ random vectors $\{\theta_{j,l}^u\}_{j=1}^n$ are independent. Thus, sequence $\left\{\frac{32\gamma}{n}\exp\left(-\frac{\gamma\mu}{2}(T-1-l)\right) \langle \zeta_j^l, \theta_{j,l}^u \rangle\right\}_{l,j = 0,1}^{T-1,n}$ is a martingale difference sequence. Next, the summands are bounded:
    \begin{eqnarray}
        \left|\frac{32\gamma}{n} \exp\left(-\frac{\gamma\mu}{2}(T-1-l)\right)\langle \zeta_j^l, \theta_{j,l}^u \rangle \right| &\leq& \frac{32\gamma}{n}\exp\left(-\frac{\gamma\mu}{2}(T-1-l)\right) \|\zeta_j^l\|\cdot \|\theta_{j,l}^u\|\notag\\
        &\overset{\eqref{eq:SEG_zeta_delta_bound_str_mon},\eqref{eq:theta_omega_magnitude_str_mon}}{\leq}& \frac{16\sqrt{V}\gamma \exp\left(-\frac{\gamma\mu(T-1)}{2}\right)}{n} \exp\left(\frac{\gamma\mu l}{4}\right)\lambda_l\notag\\
        &\overset{\eqref{eq:lambda_SEG_monotone}}{=}& \frac{\exp\left(-\frac{\gamma\mu T}{2}\right)V}{10\ln\frac{48n(K+1)}{\beta}} \eqdef c. \label{eq:str_mon_thm_SEG_technical_6_5_six}
    \end{eqnarray}
    Finally, conditional variances $\widehat \sigma_{j,l}^2 \eqdef \EE_{\xi_{2,j}^l}\left[\frac{1024\gamma^2}{n^2}\exp\left(-\gamma\mu(T-1-l)\right) \langle \zeta_j^l, \theta_{j,l}^u \rangle^2\right]$ of the summands are bounded:
    \begin{eqnarray}
        \widehat \sigma_{j,l}^2 &\leq& \EE_{\xi_{2,j}^l}\left[\frac{1024\gamma^2 }{n^2}\exp\left(-\gamma\mu(T-1-l)\right)\|\zeta_j^l\|^2\cdot \|\theta_{j,l}^u\|^2\right] \notag\\
        &\overset{\eqref{eq:SEG_zeta_delta_bound_str_mon}}{\leq}& \frac{256\gamma^2 V\exp\left(-\gamma\mu\left(T-1-\frac{l}{2}\right)\right)}{n^2} \EE_{\xi_{2,j}^l}\left[\|\theta_{j,l}^u\|^2\right]. \label{eq:str_mon_thm_SEG_technical_7_six}
    \end{eqnarray}
    Applying Bernstein's inequality (Lemma~\ref{lem:Bernstein_ineq}) with $X_{j,l} = \frac{32\gamma}{n}\exp\left(-\frac{\gamma\mu}{2}(T-1-l)\right) \langle \zeta_j^l, \theta_{j,l}^u \rangle$, constant $c$ defined in \eqref{eq:str_mon_thm_SEG_technical_6_5_six}, $b = \frac{\exp\left(-\frac{\gamma\mu T}{2}\right)V}{10}$, $G = \frac{\exp\left(-\gamma\mu T\right)V^2}{600\ln\frac{48n(K+1)}{\beta}}$, we get
    \begin{eqnarray*}
        \PP\left\{|\circledSix'| > \frac{\exp\left(-\frac{\gamma\mu T}{2}\right)V}{10} \text{ and } \sum\limits_{l=0}^{T-1}\sum\limits_{j=2}^n\widehat \sigma_{j,l}^2 \leq \frac{\exp\left(-\gamma\mu T\right)V^2}{600\ln\frac{48n(K+1)}{\beta}}\right\} &\leq& 2\exp\left(- \frac{b^2}{2G + \nicefrac{2cb}{3}}\right)\\
        &=& \frac{\beta}{24n(K+1)}.
    \end{eqnarray*}
    The above is equivalent to $\PP\{E_{\circledSix'}\} \geq 1 - \frac{\beta}{24n(K+1)}$ for
    \begin{equation}
        E_{\circledSix'} = \left\{\text{either } \sum\limits_{l=0}^{T-1}\sum\limits_{j=2}^n\widehat \sigma_{j,l}^2 > \frac{\exp\left(-\gamma\mu T\right)V^2}{600\ln\frac{48n(K+1)}{\beta}} \text{ or } |\circledSix'| \leq \frac{\exp\left(-\frac{\gamma\mu T}{2}\right)V}{10}\right\}. \label{eq:bound_6_str_mon_SEG_six}
    \end{equation}
    Moreover, $E_{T-1}$ implies
    \begin{eqnarray}
        \sum\limits_{l=0}^{T-1}\sum\limits_{j=2}^n \widehat \sigma_{j,l}^2 &\overset{\eqref{eq:str_mon_thm_SEG_technical_7_six}}{\leq}& \frac{256\gamma^2 V\exp\left(-\gamma\mu\left(T-1\right)\right)}{n^2} \sum\limits_{l=0}^{T-1}\exp\left(\frac{\gamma\mu l}{2}\right) \sum\limits_{i=1}^n \EE_{\xi_{2,i}^l}\left[\|\theta_{i,l}^u\|^2\right]\notag\\
        &\overset{\eqref{eq:variance_theta_omega_str_mon}, T \leq K+1}{\leq}&  \frac{4608\gamma^2 V\exp\left(-\gamma\mu\left(T-1\right)\right)\sigma^\alpha}{n} \sum\limits_{l=0}^{T-1}\exp\left(\frac{\gamma\mu l}{2}\right) \lambda_l^{2-\alpha}\notag\\
        &\overset{\eqref{eq:lambda_SEG_str_mon}}{\leq}& \frac{4608\gamma^\alpha V^{2 - \frac{\alpha}{2}}\exp\left(-\gamma\mu T\right)\sigma^\alpha}{300^{2-\alpha} n^{\alpha-1} \ln^{2-\alpha}\frac{48n(K+1)}{\beta}} \sum\limits_{l=0}^{T-1}\exp\left(\frac{\gamma\mu l \alpha}{4}\right)\notag\\
        &\leq& \frac{4608\gamma^\alpha V^{2 - \frac{\alpha}{2}}\exp\left(-\gamma\mu T\right)\sigma^\alpha (K+1) \exp\left(\frac{\gamma\mu K \alpha}{4}\right)}{300^{2-\alpha} n^{\alpha-1} \ln^{2-\alpha}\frac{48n(K+1)}{\beta}}\notag\\
        &\overset{\eqref{eq:gamma_SEG_str_mon}}{\leq}& \frac{\exp\left(-\gamma\mu T\right)V^2}{600\ln\frac{48n(K+1)}{\beta}}.\label{eq:bound_6_variances_str_mon_SEG_six}
    \end{eqnarray}

    \paragraph{Upper bound for $\circledSeven'$.} To estimate this sum, we will use Bernstein's inequality. The summands have conditional expectations equal to zero:
    \begin{equation*}
        \EE_{\xi_{1,j}^l}\left[\frac{24\gamma}{n}\exp\left(-\frac{\gamma\mu}{2}(T-1-l)\right) \langle \delta_j^l, \omega_{j,l}^u \rangle\right] = \frac{24\gamma}{n} \exp\left(-\frac{\gamma\mu}{2}(T-1-l)\right)\left\langle \delta_j^l, \EE_{\xi_{1,j}^l}[\omega_{j,l}^u] \right\rangle = 0.
    \end{equation*}
    Moreover, for all $l = 0,\ldots, T-1$ random vectors $\{\omega_{j,l}^u\}_{j=1}^n$ are independent. Thus, sequence $\left\{\frac{24\gamma}{n}\exp\left(-\frac{\gamma\mu}{2}(T-1-l)\right) \langle \delta_j^l, \omega_{j,l}^u \rangle\right\}_{l,j = 0,1}^{T-1,n}$ is a martingale difference sequence. Next, the summands are bounded:
    \begin{eqnarray}
        \left|\frac{24\gamma}{n} \exp\left(-\frac{\gamma\mu}{2}(T-1-l)\right)\langle \delta_j^l, \omega_{j,l}^u \rangle \right| &\leq& \frac{24\gamma}{n}\exp\left(-\frac{\gamma\mu}{2}(T-1-l)\right) \|\delta_j^l\|\cdot \|\omega_{j,l}^u\|\notag\\
        &\overset{\eqref{eq:SEG_zeta_delta_bound_str_mon},\eqref{eq:theta_omega_magnitude_str_mon}}{\leq}& \frac{12\sqrt{V}\gamma \exp\left(-\frac{\gamma\mu(T-1)}{2}\right)}{n} \exp\left(\frac{\gamma\mu l}{4}\right)\lambda_l\notag\\
        &\overset{\eqref{eq:lambda_SEG_monotone}}{=}& \frac{\exp\left(-\frac{\gamma\mu T}{2}\right)V}{10\ln\frac{48n(K+1)}{\beta}} \eqdef c. \label{eq:str_mon_thm_SEG_technical_6_5_seven}
    \end{eqnarray}
    Finally, conditional variances $(\sigma_{j,l}')^2 \eqdef \EE_{\xi_{1,j}^l}\left[\frac{576\gamma^2}{n^2}\exp\left(-\gamma\mu(T-1-l)\right) \langle \delta_j^l, \omega_{j,l}^u \rangle^2\right]$ of the summands are bounded:
    \begin{eqnarray}
        (\sigma_{j,l}')^2 &\leq& \EE_{\xi_{1,j}^l}\left[\frac{576\gamma^2 }{n^2}\exp\left(-\gamma\mu(T-1-l)\right)\|\delta_j^l\|^2\cdot \|\omega_{j,l}^u\|^2\right] \notag\\
        &\overset{\eqref{eq:SEG_zeta_delta_bound_str_mon}}{\leq}& \frac{288\gamma^2 V\exp\left(-\gamma\mu\left(T-1-\frac{l}{2}\right)\right)}{n^2} \EE_{\xi_{1,j}^l}\left[\|\omega_{j,l}^u\|^2\right]. \label{eq:str_mon_thm_SEG_technical_7_seven}
    \end{eqnarray}
    Applying Bernstein's inequality (Lemma~\ref{lem:Bernstein_ineq}) with $X_{j,l} = \frac{24\gamma}{n}\exp\left(-\frac{\gamma\mu}{2}(T-1-l)\right) \langle \delta_j^l, \omega_{j,l}^u \rangle$, constant $c$ defined in \eqref{eq:str_mon_thm_SEG_technical_6_5_six}, $b = \frac{\exp\left(-\frac{\gamma\mu T}{2}\right)V}{10}$, $G = \frac{\exp\left(-\gamma\mu T\right)V^2}{600\ln\frac{48n(K+1)}{\beta}}$, we get
    \begin{eqnarray*}
        \PP\left\{|\circledSeven'| > \frac{\exp\left(-\frac{\gamma\mu T}{2}\right)V}{10} \text{ and } \sum\limits_{l=0}^{T-1}\sum\limits_{j=2}^n (\sigma_{j,l}')^2 \leq \frac{\exp\left(-\gamma\mu T\right)V^2}{600\ln\frac{48n(K+1)}{\beta}}\right\} &\leq& 2\exp\left(- \frac{b^2}{2G + \nicefrac{2cb}{3}}\right)\\
        &=& \frac{\beta}{24n(K+1)}.
    \end{eqnarray*}
    The above is equivalent to $\PP\{E_{\circledSix'}\} \geq 1 - \frac{\beta}{24n(K+1)}$ for
    \begin{equation}
        E_{\circledSeven'} = \left\{\text{either } \sum\limits_{l=0}^{T-1}\sum\limits_{j=2}^n (\sigma_{j,l}')^2 > \frac{\exp\left(-\gamma\mu T\right)V^2}{600\ln\frac{48n(K+1)}{\beta}} \text{ or } |\circledSeven'| \leq \frac{\exp\left(-\frac{\gamma\mu T}{2}\right)V}{10}\right\}. \label{eq:bound_6_str_mon_SEG_seven}
    \end{equation}
    Moreover, $E_{T-1}$ implies
    \begin{eqnarray}
        \sum\limits_{l=0}^{T-1}\sum\limits_{j=2}^n (\sigma_{j,l}')^2 &\overset{\eqref{eq:str_mon_thm_SEG_technical_7_seven}}{\leq}& \frac{288\gamma^2 V\exp\left(-\gamma\mu\left(T-1\right)\right)}{n^2} \sum\limits_{l=0}^{T-1}\exp\left(\frac{\gamma\mu l}{2}\right) \sum\limits_{i=1}^n \EE_{\xi_{2,i}^l}\left[\|\theta_{i,l}^u\|^2\right]\notag\\
        &\overset{\eqref{eq:variance_theta_omega_str_mon}, T \leq K+1}{\leq}&  \frac{5184\gamma^2 V\exp\left(-\gamma\mu\left(T-1\right)\right)\sigma^\alpha}{n} \sum\limits_{l=0}^{T-1}\exp\left(\frac{\gamma\mu l}{2}\right) \lambda_l^{2-\alpha}\notag\\
        &\overset{\eqref{eq:lambda_SEG_str_mon}}{\leq}& \frac{5184\gamma^\alpha V^{2 - \frac{\alpha}{2}}\exp\left(-\gamma\mu T\right)\sigma^\alpha}{300^{2-\alpha} n^{\alpha-1} \ln^{2-\alpha}\frac{48n(K+1)}{\beta}} \sum\limits_{l=0}^{T-1}\exp\left(\frac{\gamma\mu l \alpha}{4}\right)\notag\\
        &\leq& \frac{5184\gamma^\alpha V^{2 - \frac{\alpha}{2}}\exp\left(-\gamma\mu T\right)\sigma^\alpha (K+1) \exp\left(\frac{\gamma\mu K \alpha}{4}\right)}{300^{2-\alpha} n^{\alpha-1} \ln^{2-\alpha}\frac{48n(K+1)}{\beta}}\notag\\
        &\overset{\eqref{eq:gamma_SEG_str_mon}}{\leq}& \frac{\exp\left(-\gamma\mu T\right)V^2}{600\ln\frac{48n(K+1)}{\beta}}.\label{eq:bound_6_variances_str_mon_SEG_seven}
    \end{eqnarray}

    That is, we derive the upper bounds for  $\circledOne, \circledTwo, \circledThree, \circledFour, \circledFive, \circledSix, \circledSeven$. More precisely, $E_{T-1}$ implies
    \begin{gather*}
        V_T \overset{\eqref{eq:ndvjinuibvbsvdhbvdhbfussdncdni}}{\leq} \exp\left(-\frac{\gamma\mu}{2}T\right)V + \circledOne + \circledTwo + \circledThree + \circledFour + \circledFive + \circledSix + \circledSeven,\\
         \circledSix \overset{\eqref{eq:SEG_extra_sums_six_distributed_str_mon}}{=}  \circledSix' + \frac{32\gamma}{n}  \sum\limits_{j=2}^n \left\langle \frac{\gamma}{n}\sum\limits_{i=1}^{j-1} \theta_{i,T-1}^u - \zeta_j^{T-1}, \theta_{j,T-1}^u \right\rangle,\\
         \circledSeven \overset{\eqref{eq:SEG_extra_sums_seven_distributed_str_mon}}{=} \circledSeven' + \frac{24\gamma}{n}  \sum\limits_{j=2}^n \left\langle \frac{\gamma}{n}\sum\limits_{i=1}^{j-1} \omega_{i,T-1}^u - \delta_j^{T-1}, \omega_{j,T-1}^u \right\rangle,\\
         \circledTwo \overset{\eqref{eq:bound_2_variances_str_mon_SEG}}{\leq} \frac{3\exp\left(-\frac{\gamma\mu T}{2}\right) V}{100},\quad \circledThree \overset{\eqref{eq:bound_3_variances_str_mon_SEG}}{\leq} \frac{\exp\left(-\frac{\gamma\mu T}{2}\right) V}{50},\quad \circledFive \overset{\eqref{eq:bound_5_variances_str_mon_SEG}}{\leq} \frac{\exp\left(-\frac{\gamma\mu T}{2}\right) V}{6},\\
         \sum\limits_{l=0}^{T-1}\sum\limits_{i=1}^n\sigma_{i,l}^2 \overset{\eqref{eq:bound_1_variances_str_mon_SEG}}{\leq} \frac{\exp\left(-\gamma\mu T\right)V^2}{60000\ln\frac{48n(K+1)}{\beta}},\quad \sum\limits_{l=0}^{T-1}\sum\limits_{i=1}^n\widetilde\sigma_{i,l}^2 \overset{\eqref{eq:bound_1_variances_str_mon_SEG_1}}{\leq} \frac{\exp\left(-\gamma\mu T\right)V}{216\ln\frac{48n(K+1)}{\beta}},\\
         \sum\limits_{l=0}^{T-1}\sum\limits_{j=2}^n \widehat \sigma_{j,l}^2 \overset{\eqref{eq:bound_6_variances_str_mon_SEG_six}}{\leq} \frac{\exp\left(-\gamma\mu T\right)V^2}{600\ln\frac{48n(K+1)}{\beta}},\quad \sum\limits_{l=0}^{T-1}\sum\limits_{j=2}^n (\sigma_{j,l}')^2 \overset{\eqref{eq:bound_6_variances_str_mon_SEG_seven}}{\leq} \frac{\exp\left(-\gamma\mu T\right)V^2}{600\ln\frac{48n(K+1)}{\beta}}.
    \end{gather*}
    In addition, we also establish (see \eqref{eq:bound_1_str_mon_SEG}, \eqref{eq:bound_1_str_mon_SEG}, \eqref{eq:bound_6_str_mon_SEG_six}, \eqref{eq:bound_6_str_mon_SEG_six}, and our induction assumption)
    \begin{gather*}
        \PP\{E_{T-1}\} \geq 1 - \frac{(T-1)\beta}{K+1},\quad \PP\{E_{\circledOne}\} \geq 1 - \frac{\beta}{24n(K+1)}, \\
        \PP\{E_{\circledFour}\} \geq 1 - \frac{\beta}{24n(K+1)}, \quad \PP\{E_{\circledSix'}\} \geq 1 - \frac{\beta}{24n(K+1)},\quad \PP\{E_{\circledSeven'}\} \geq 1 - \frac{\beta}{24n(K+1)},
    \end{gather*}
    where
    \begin{eqnarray*}
        E_{\circledOne} &=& \left\{\text{either } \sum\limits_{l=0}^{T-1}\sum\limits_{i=1}^n\sigma_{i,l}^2 > \frac{\exp\left(-\gamma\mu T\right)V^2}{60000\ln\frac{48n(K+1)}{\beta}} \text{ or } |\circledOne| \leq \frac{\exp\left(-\frac{\gamma\mu T}{2}\right)V}{100}\right\},\\
        E_{\circledFour} &=& \left\{\text{either } \sum\limits_{l=0}^{T-1}\sum\limits_{i=1}^n\widetilde\sigma_{i,l}^2 > \frac{\exp\left(-\gamma\mu T\right)V^2}{216\ln\frac{48n(K+1)}{\beta}} \text{ or } |\circledFour| \leq \frac{\exp\left(-\frac{\gamma\mu T}{2}\right)V}{6}\right\},\\
        E_{\circledSix'} &=& \left\{\text{either } \sum\limits_{l=0}^{T-1}\sum\limits_{j=2}^n\widehat \sigma_{j,l}^2 > \frac{\exp\left(-\gamma\mu T\right)V^2}{600\ln\frac{48n(K+1)}{\beta}} \text{ or } |\circledSix'| \leq \frac{\exp\left(-\frac{\gamma\mu T}{2}\right)V}{10}\right\},\\
        E_{\circledSeven'} &=& \left\{\text{either } \sum\limits_{l=0}^{T-1}\sum\limits_{j=2}^n (\sigma_{j,l}')^2 > \frac{\exp\left(-\gamma\mu T\right)V^2}{600\ln\frac{48n(K+1)}{\beta}} \text{ or } |\circledSeven'| \leq \frac{\exp\left(-\frac{\gamma\mu T}{2}\right)V}{10}\right\}.
    \end{eqnarray*}
    Therefore, probability event $E_{T-1} \cap E_{\circledOne} \cap E_{\circledFour} \cap E_{\circledSix'} \cap E_{\circledSeven'}$ implies
    \begin{eqnarray}
        V_T &\leq& \exp\left(-\frac{\gamma\mu}{2}T\right)V \underbrace{\left(1 + \frac{1}{100} + \frac{3}{100} + \frac{1}{50} + \frac{1}{6} + \frac{1}{6} + \frac{1}{10} + \frac{1}{10}\right)}_{\leq 2} \notag\\
        &&\quad + \frac{32\gamma}{n}  \sum\limits_{j=2}^n \left\langle \frac{\gamma}{n}\sum\limits_{i=1}^{j-1} \theta_{i,T-1}^u - \zeta_j^{T-1}, \theta_{j,T-1}^u \right\rangle \notag\\
        &&\quad + \frac{24\gamma}{n}  \sum\limits_{j=2}^n \left\langle \frac{\gamma}{n}\sum\limits_{i=1}^{j-1} \omega_{i,T-1}^u - \delta_j^{T-1}, \omega_{j,T-1}^u \right\rangle. \label{eq:nvsdibhjbsjdbvshbcjs}
    \end{eqnarray}

    To finish the proof, we need to show that $\frac{\gamma}{n}\sum\limits_{i=1}^{j-1} \theta_{i,T-1}^u = \zeta_j^{T-1}$ and $\frac{\gamma}{n}\sum\limits_{i=1}^{j-1} \omega_{i,T-1}^u = \delta_j^{T-1}$ with high probability. In particular, we consider probability event $\widetilde{E}_{T-1,j}$ defined as follows: inequalities
    \begin{equation}
        \left\|\frac{\gamma}{n}\sum\limits_{i=1}^{r-1} \theta_{i,T-1}^u\right\| \leq \exp\left(-\frac{\gamma\mu (T-1)}{4}\right)\frac{\sqrt{V}}{2}, \quad \left\|\frac{\gamma}{n}\sum\limits_{i=1}^{r-1} \omega_{i,T-1}^u\right\| \leq \exp\left(-\frac{\gamma\mu (T-1)}{4}\right)\frac{\sqrt{V}}{2} \notag
    \end{equation}
    hold for $r = 2, \ldots, j$ simultaneously. We want to show that $\PP\{E_{T-1} \cap \widetilde{E}_{T-1,j}\} \geq 1 - \frac{(T-1)\beta}{K+1} - \frac{j\beta}{8n(K+1)}$ for all $j = 2, \ldots, n$. For $j = 2$ the statement is trivial since
    \begin{eqnarray*}
        \left\|\frac{\gamma}{n} \theta_{1,T-1}^u\right\| &\overset{\eqref{eq:theta_omega_magnitude_str_mon}}{\leq}& \frac{2\gamma\lambda_{T-1}}{n} \leq \exp\left(-\frac{\gamma\mu (T-1)}{4}\right)\frac{\sqrt{V}}{2},\\
        \left\|\frac{\gamma}{n} \omega_{1,T-1}^u\right\| &\overset{\eqref{eq:theta_omega_magnitude}}{\leq}& \frac{2\gamma\lambda_{T-1}}{n} \leq \exp\left(-\frac{\gamma\mu (T-1)}{4}\right)\frac{\sqrt{V}}{2}.
    \end{eqnarray*}
    Next, we assume that the statement holds for some $j = m-1 < n$, i.e.,  $\PP\{E_{T-1}\cap\widetilde{E}_{T-1,m-1}\} \geq 1 - \frac{(T-1)\beta}{K+1} - \frac{(m-1)\beta}{8n(K+1)}$. Our goal is to prove that $\PP\{E_{T-1}\cap\widetilde{E}_{T-1,m}\} \geq 1 - \frac{(T-1)\beta}{K+1} - \frac{m\beta}{8n(K+1)}$. First, we consider $\left\|\frac{\gamma}{n}\sum\limits_{i=1}^{m-1} \theta_{i,T-1}^u\right\|$:
    \begin{eqnarray}
        \left\|\frac{\gamma}{n}\sum\limits_{i=1}^{m-1} \theta_{i,T-1}^u\right\| &=& \sqrt{\frac{\gamma^2}{n^2}\left\|\sum\limits_{i=1}^{m-1} \theta_{i,T-1}^u\right\|^2} \notag\\
        &=& \sqrt{\frac{\gamma^2}{n^2}\sum\limits_{i=1}^{m-1}\|\theta_{i,T-1}^u\|^2 + \frac{2\gamma}{n}\sum\limits_{i=1}^{m-1}\left\langle \frac{\gamma}{n}\sum\limits_{r=1}^{i-1}\theta_{r,T-1}^u , \theta_{i,T-1}^u \right\rangle} \notag\\
        &\leq& \sqrt{\frac{\gamma^2}{n^2}\sum\limits_{l=0}^{T-1}\exp\left(-\frac{\gamma\mu(T-1-l)}{2}\right)\sum\limits_{i=1}^{m-1}\|\theta_{i,l}^u\|^2 + \frac{2\gamma}{n}\sum\limits_{i=1}^{m-1}\left\langle \frac{\gamma}{n}\sum\limits_{r=1}^{i-1}\theta_{r,T-1}^u , \theta_{i,T-1}^u \right\rangle}. \notag
    \end{eqnarray}
    Similarly, we have
    \begin{eqnarray}
        \left\|\frac{\gamma}{n}\sum\limits_{i=1}^{m-1} \omega_{i,T-1}^u\right\| &=& \sqrt{\frac{\gamma^2}{n^2}\left\|\sum\limits_{i=1}^{m-1} \omega_{i,T-1}^u\right\|^2} \notag\\
        &=& \sqrt{\frac{\gamma^2}{n^2}\sum\limits_{i=1}^{m-1}\|\omega_{i,T-1}^u\|^2 + \frac{2\gamma}{n}\sum\limits_{i=1}^{m-1}\left\langle \frac{\gamma}{n}\sum\limits_{r=1}^{i-1}\omega_{r,T-1}^u , \omega_{i,T-1}^u \right\rangle} \notag\\
        &\leq& \sqrt{\frac{\gamma^2}{n^2}\sum\limits_{l=0}^{T-1}\exp\left(-\frac{\gamma\mu(T-1-l)}{2}\right)\sum\limits_{i=1}^{m-1}\|\omega_{i,l}^u\|^2 + \frac{2\gamma}{n}\sum\limits_{i=1}^{m-1}\left\langle \frac{\gamma}{n}\sum\limits_{r=1}^{i-1}\omega_{r,T-1}^u , \omega_{i,T-1}^u \right\rangle}. \notag
    \end{eqnarray}

    Next, we introduce a new notation:
    \begin{gather*}
        \rho_{i,T-1} = \begin{cases}
            \frac{\gamma}{n}\sum\limits_{r=1}^{i-1}\theta_{r,T-1}^u,& \text{if } \left\|\frac{\gamma}{n}\sum\limits_{r=1}^{i-1}\theta_{r,T-1}^u\right\| \leq \exp\left(-\frac{\gamma\mu (T-1)}{4}\right)\frac{\sqrt{V}}{2},\\
            0,& \text{otherwise}
        \end{cases},\\
        \rho_{i,T-1}' = \begin{cases}
            \frac{\gamma}{n}\sum\limits_{r=1}^{i-1}\omega_{r,T-1}^u,& \text{if } \left\|\frac{\gamma}{n}\sum\limits_{r=1}^{i-1}\omega_{r,T-1}^u\right\| \leq \exp\left(-\frac{\gamma\mu (T-1)}{4}\right)\frac{\sqrt{V}}{2},\\
            0,& \text{otherwise}
        \end{cases}
    \end{gather*}
    for $i = 1,\ldots,m-1$. By definition, we have
    \begin{equation}
        \|\rho_{i,T-1}\| \leq \exp\left(-\frac{\gamma\mu (T-1)}{4}\right)\frac{\sqrt{V}}{2},\quad \|\rho_{i,T-1}'\| \leq \exp\left(-\frac{\gamma\mu (T-1)}{4}\right)\frac{\sqrt{V}}{2} \label{eq:SEG_bound_rho_str_mon}
    \end{equation}
    for $i = 1,\ldots,m-1$. Moreover, $\widetilde{E}_{T-1,m-1}$ implies $\rho_{i,T-1} = \frac{\gamma}{n}\sum\limits_{r=1}^{i-1}\theta_{r,T-1}^u$, $\rho_{i,T-1}' = \frac{\gamma}{n}\sum\limits_{r=1}^{i-1}\omega_{r,T-1}^u$ for $i = 1,\ldots,m-1$ and
    \begin{eqnarray}
        \left\|\frac{\gamma}{n}\sum\limits_{i=1}^{m-1} \theta_{i,l}^u\right\| &\leq& \sqrt{\circledThree + \circledFour + \circledEight}, \notag\\
        \left\|\frac{\gamma}{n}\sum\limits_{i=1}^{m-1} \omega_{i,l}^u\right\| &\leq& \sqrt{\circledThree + \circledFour + \circledEight'}, \notag
    \end{eqnarray}
    where
    \begin{gather*}
        \circledEight =  \frac{2\gamma}{n}\sum\limits_{i=1}^{m-1}\left\langle \rho_{i,T-1} , \theta_{i,T-1}^u \right\rangle,\quad \circledEight' =  \frac{2\gamma}{n}\sum\limits_{i=1}^{m-1}\left\langle \rho_{i,T-1}' , \omega_{i,T-1}^u \right\rangle.
    \end{gather*}
    It remains to estimate $\circledEight$ and $\circledEight'$.

    \paragraph{Upper bound for $\circledEight$.} To estimate this sum, we will use Bernstein's inequality. The summands have conditional expectations equal to zero:
    \begin{equation*}
        \EE_{\xi_{2,i}^{T-1}}\left[\frac{2\gamma}{n} \langle \rho_{i,T-1}, \theta_{i,T-1}^u \rangle\right] = \frac{2\gamma}{n} \left\langle \rho_{i,T-1}, \EE_{\xi_{2,i}^{T-1}}[\theta_{i,T-1}^u] \right\rangle = 0.
    \end{equation*}
    Thus, sequence $\left\{\frac{2\gamma}{n}\langle \rho_{i,T-1}, \theta_{i,T-1}^u \rangle\right\}_{i = 1}^{n}$ is a martingale difference sequence. Next, the summands are bounded:
    \begin{eqnarray}
        \left|\frac{2\gamma}{n} \langle \rho_{i,T-1}, \theta_{i,T-1}^u \rangle \right| &\leq& \frac{2\gamma}{n}\|\rho_{i,T-1}\|\cdot \|\theta_{i,T-1}^u\|\notag\\
        &\overset{\eqref{eq:SEG_bound_rho_str_mon},\eqref{eq:theta_omega_magnitude_str_mon}}{\leq}& \frac{2\sqrt{V}\gamma \exp\left(-\frac{\gamma\mu(T-1)}{4}\right)}{n} \lambda_{T-1}\notag\\
        &\overset{\eqref{eq:lambda_SEG_monotone}}{\leq}& \frac{\exp\left(-\frac{\gamma\mu (T-1)}{2}\right)V}{80\ln\frac{48n(K+1)}{\beta}} \eqdef c. \label{eq:str_mon_thm_SEG_technical_6_5_eight}
    \end{eqnarray}
    Finally, conditional variances $(\widehat\sigma_{i,T-1}')^2 \eqdef \EE_{\xi_{2,i}^{T-1}}\left[\frac{4\gamma^2}{n^2} \langle \rho_{i,T-1}, \theta_{i,T-1}^u \rangle^2\right]$ of the summands are bounded:
    \begin{eqnarray}
        (\widehat\sigma_{i,T-1}')^2 &\leq& \EE_{\xi_{2,i}^{T-1}}\left[\frac{4\gamma^2 }{n^2}\|\rho_{i,T-1}\|^2\cdot \|\theta_{i,T-1}^u\|^2\right] \notag\\
        &\overset{\eqref{eq:SEG_bound_rho_str_mon}}{\leq}& \frac{\gamma^2 V\exp\left(-\frac{\gamma\mu(T-1)}{2}\right)}{n^2} \EE_{\xi_{2,i}^{T-1}}\left[\|\theta_{i,T-1}^u\|^2\right]. \label{eq:str_mon_thm_SEG_technical_7_eight}
    \end{eqnarray}
    Applying Bernstein's inequality (Lemma~\ref{lem:Bernstein_ineq}) with $X_{i} = \frac{2\gamma}{n}\langle \rho_{i,T-1}, \theta_{i,T-1}^u \rangle$, constant $c$ defined in \eqref{eq:str_mon_thm_SEG_technical_6_5_eight}, $b = \frac{\exp\left(-\frac{\gamma\mu (T-1)}{2}\right)V}{80}$, $G = \frac{\exp\left(-\gamma\mu (T-1)\right)V^2}{38400\ln\frac{48n(K+1)}{\beta}}$, we get
    \begin{eqnarray*}
        \PP\left\{|\circledEight| > \frac{\exp\left(-\frac{\gamma\mu (T-1)}{2}\right)V}{80} \text{ and } \sum\limits_{i=1}^n(\widehat\sigma_{i,T-1}')^2 \leq \frac{\exp\left(-\gamma\mu (T-1)\right)V^2}{38400\ln\frac{48n(K+1)}{\beta}}\right\} &\leq& 2\exp\left(- \frac{b^2}{2G + \nicefrac{2cb}{3}}\right)\\
        &=& \frac{\beta}{24n(K+1)}.
    \end{eqnarray*}
    The above is equivalent to $\PP\{E_{\circledEight}\} \geq 1 - \frac{\beta}{24n(K+1)}$ for
    \begin{equation}
        E_{\circledEight} = \left\{\text{either } \sum\limits_{i=1}^n(\widehat\sigma_{i,T-1}')^2 > \frac{\exp\left(-\gamma\mu (T-1)\right)V^2}{38400\ln\frac{48n(K+1)}{\beta}} \text{ or } |\circledEight| > \frac{\exp\left(-\frac{\gamma\mu (T-1)}{2}\right)V}{80}\right\}. \label{eq:bound_1_str_mon_SEG_eight}
    \end{equation}
    Moreover, $E_{T-1}$ implies
    \begin{eqnarray}
        \sum\limits_{i=1}^n(\widehat\sigma_{i,T-1}')^2 &\overset{\eqref{eq:str_mon_thm_SEG_technical_7_eight}}{\leq}& \frac{\gamma^2 V\exp\left(-\gamma\mu\left(T-1\right)\right)}{n^2} \sum\limits_{i=1}^n \EE_{\xi_{2,i}^T-1}\left[\|\theta_{i,T-1}^u\|^2\right]\notag\\
        &\overset{\eqref{eq:variance_theta_omega_str_mon}}{\leq}&  \frac{18\gamma^2 V\exp\left(-\gamma\mu\left(T-1\right)\right)\sigma^\alpha}{n} \lambda_{T-1}^{2-\alpha}\notag\\
        &\overset{\eqref{eq:lambda_SEG_str_mon}}{\leq}& \frac{18\gamma^\alpha V^{2 - \frac{\alpha}{2}}\exp\left(-\gamma\mu (T-1)\right)\sigma^\alpha}{300^{2-\alpha} n^{\alpha-1} \ln^{2-\alpha}\frac{48n(K+1)}{\beta}} \exp\left(\frac{\gamma\mu (T-1) \alpha}{4}\right)\notag\\
        &\overset{T-1 \leq K}{\leq}& \frac{18\gamma^\alpha V^{2 - \frac{\alpha}{2}}\exp\left(-\gamma\mu (T-1)\right)\sigma^\alpha \exp\left(\frac{\gamma\mu K \alpha}{4}\right)}{300^{2-\alpha} n^{\alpha-1} \ln^{2-\alpha}\frac{48n(K+1)}{\beta}}\notag\\
        &\overset{\eqref{eq:gamma_SEG_str_mon}}{\leq}& \frac{\exp\left(-\gamma\mu (T-1)\right)V^2}{38400\ln\frac{48n(K+1)}{\beta}}.\label{eq:bound_1_variances_str_mon_SEG_eight}
    \end{eqnarray}

    \paragraph{Upper bound for $\circledEight'$.} To estimate this sum, we will use Bernstein's inequality. The summands have conditional expectations equal to zero:
    \begin{equation*}
        \EE_{\xi_{1,i}^{T-1}}\left[\frac{2\gamma}{n} \langle \rho_{i,T-1}', \omega_{i,T-1}^u \rangle\right] = \frac{2\gamma}{n} \left\langle \rho_{i,T-1}', \EE_{\xi_{1,i}^{T-1}}[\omega_{i,T-1}^u] \right\rangle = 0.
    \end{equation*}
    Thus, sequence $\left\{\frac{2\gamma}{n}\langle \rho_{i,T-1}', \omega_{i,T-1}^u \rangle\right\}_{i = 1}^{n}$ is a martingale difference sequence. Next, the summands are bounded:
    \begin{eqnarray}
        \left|\frac{2\gamma}{n} \langle \rho_{i,T-1}', \omega_{i,T-1}^u \rangle \right| &\leq& \frac{2\gamma}{n}\|\rho_{i,T-1}'\|\cdot \|\omega_{i,T-1}^u\|\notag\\
        &\overset{\eqref{eq:SEG_bound_rho_str_mon},\eqref{eq:theta_omega_magnitude_str_mon}}{\leq}& \frac{2\sqrt{V}\gamma \exp\left(-\frac{\gamma\mu(T-1)}{4}\right)}{n} \lambda_{T-1}\notag\\
        &\overset{\eqref{eq:lambda_SEG_monotone}}{=}& \frac{\exp\left(-\frac{\gamma\mu (T-1)}{2}\right)V}{80\ln\frac{48n(K+1)}{\beta}} \eqdef c. \label{eq:str_mon_thm_SEG_technical_6_5_eight'}
    \end{eqnarray}
    Finally, conditional variances $(\widetilde\sigma_{i,T-1}')^2 \eqdef \EE_{\xi_{1,i}^{T-1}}\left[\frac{4\gamma^2}{n^2} \langle \rho_{i,T-1}', \omega_{i,T-1}^u \rangle^2\right]$ of the summands are bounded:
    \begin{eqnarray}
        (\widetilde\sigma_{i,T-1}')^2 &\leq& \EE_{\xi_{1,i}^{T-1}}\left[\frac{4\gamma^2 }{n^2}\|\rho_{i,T-1}'\|^2\cdot \|\omega_{i,T-1}^u\|^2\right] \notag\\
        &\overset{\eqref{eq:SEG_bound_rho_str_mon}}{\leq}& \frac{\gamma^2 V\exp\left(-\frac{\gamma\mu(T-1)}{2}\right)}{n^2} \EE_{\xi_{1,i}^{T-1}}\left[\|\omega_{i,T-1}^u\|^2\right]. \label{eq:str_mon_thm_SEG_technical_7_eight'}
    \end{eqnarray}
    Applying Bernstein's inequality (Lemma~\ref{lem:Bernstein_ineq}) with $X_{i} = \frac{2\gamma}{n}\langle \rho_{i,T-1}', \omega_{i,T-1}^u \rangle$, constant $c$ defined in \eqref{eq:str_mon_thm_SEG_technical_6_5_eight'}, $b = \frac{\exp\left(-\frac{\gamma\mu (T-1)}{2}\right)V}{80}$, $G = \frac{\exp\left(-\gamma\mu (T-1)\right)V^2}{38400\ln\frac{48n(K+1)}{\beta}}$, we get
    \begin{eqnarray*}
        \PP\left\{|\circledEight'| > \frac{\exp\left(-\frac{\gamma\mu (T-1)}{2}\right)V}{80} \text{ and } \sum\limits_{i=1}^n(\widetilde\sigma_{i,T-1}')^2 \leq \frac{\exp\left(-\gamma\mu (T-1)\right)V^2}{38400\ln\frac{48n(K+1)}{\beta}}\right\} &\leq& 2\exp\left(- \frac{b^2}{2G + \nicefrac{2cb}{3}}\right)\\
        &=& \frac{\beta}{24n(K+1)}.
    \end{eqnarray*}
    The above is equivalent to $\PP\{E_{\circledEight'}\} \geq 1 - \frac{\beta}{24n(K+1)}$ for
    \begin{equation}
        E_{\circledEight} = \left\{\text{either } \sum\limits_{i=1}^n(\widetilde\sigma_{i,T-1}')^2 > \frac{\exp\left(-\gamma\mu (T-1)\right)V^2}{38400\ln\frac{48n(K+1)}{\beta}} \text{ or } |\circledEight'| > \frac{\exp\left(-\frac{\gamma\mu (T-1)}{2}\right)V}{80}\right\}. \label{eq:bound_1_str_mon_SEG_eight'}
    \end{equation}
    Moreover, $E_{T-1}$ implies
    \begin{eqnarray}
        \sum\limits_{i=1}^n(\widetilde\sigma_{i,T-1}')^2 &\overset{\eqref{eq:str_mon_thm_SEG_technical_7_eight'}}{\leq}& \frac{\gamma^2 V\exp\left(-\gamma\mu\left(T-1\right)\right)}{n^2} \sum\limits_{i=1}^n \EE_{\xi_{2,i}^T-1}\left[\|\theta_{i,T-1}^u\|^2\right]\notag\\
        &\overset{\eqref{eq:variance_theta_omega_str_mon}}{\leq}&  \frac{18\gamma^2 V\exp\left(-\gamma\mu\left(T-1\right)\right)\sigma^\alpha}{n} \lambda_{T-1}^{2-\alpha}\notag\\
        &\overset{\eqref{eq:lambda_SEG_str_mon}}{\leq}& \frac{18\gamma^\alpha V^{2 - \frac{\alpha}{2}}\exp\left(-\gamma\mu (T-1)\right)\sigma^\alpha}{300^{2-\alpha} n^{\alpha-1} \ln^{2-\alpha}\frac{48n(K+1)}{\beta}} \exp\left(\frac{\gamma\mu (T-1) \alpha}{4}\right)\notag\\
        &\overset{T-1 \leq K}{\leq}& \frac{18\gamma^\alpha V^{2 - \frac{\alpha}{2}}\exp\left(-\gamma\mu (T-1)\right)\sigma^\alpha \exp\left(\frac{\gamma\mu K \alpha}{4}\right)}{300^{2-\alpha} n^{\alpha-1} \ln^{2-\alpha}\frac{48n(K+1)}{\beta}}\notag\\
        &\overset{\eqref{eq:gamma_SEG_str_mon}}{\leq}& \frac{\exp\left(-\gamma\mu (T-1)\right)V^2}{38400\ln\frac{48n(K+1)}{\beta}}.\label{eq:bound_1_variances_str_mon_SEG_eight'}
    \end{eqnarray}

    Putting all together we get that $E_{T-1}\cap \widetilde{E}_{T-1,m-1}$ implies
    \begin{gather*}
        \left\|\frac{\gamma}{n}\sum\limits_{i=1}^{m-1} \theta_{i,T-1}^u\right\| \leq \sqrt{\circledThree + \circledFour + \circledEight},\quad \left\|\frac{\gamma}{n}\sum\limits_{i=1}^{m-1} \omega_{i,T-1}^u\right\| \leq \sqrt{\circledThree + \circledFour + \circledEight'},\quad \circledThree \overset{\eqref{eq:bound_3_variances_str_mon_SEG}}{\leq} \frac{\exp\left(-\frac{\gamma\mu T}{2}\right)V}{50},\\
        \sum\limits_{l=0}^{T-1}\sum\limits_{i=1}^n\widetilde\sigma_{i,l}^2 \overset{\eqref{eq:bound_1_variances_str_mon_SEG_1}}{\leq} \frac{\exp\left(-\gamma\mu T\right)V}{216\ln\frac{48n(K+1)}{\beta}},\quad \sum\limits_{i=1}^{m-1}(\widehat\sigma_{i,T-1}')^2 \leq \frac{\exp\left(-\gamma\mu (T-1)\right)V^2}{38400\ln\frac{48n(K+1)}{\beta}},\\
        \sum\limits_{i=1}^{m-1}(\widetilde\sigma_{i,T-1}')^2 \leq \frac{\exp\left(-\gamma\mu (T-1)\right)V^2}{38400\ln\frac{48n(K+1)}{\beta}}.
    \end{gather*}
    In addition, we also establish (see \eqref{eq:bound_1_str_mon_SEG_1}, \eqref{eq:bound_1_str_mon_SEG_eight}, \eqref{eq:bound_1_str_mon_SEG_eight'} and our induction assumption)
    \begin{gather*}
        \PP\{E_{T-1}\cap \widetilde{E}_{T-1,m-1}\} \geq 1 - \frac{(T-1)\beta}{K+1} - \frac{(m-1)\beta}{8n(K+1)},\\
        \PP\{E_{\circledFour}\} \geq 1 - \frac{\beta}{24n(K+1)}, \quad \PP\{E_{\circledEight}\} \geq 1 - \frac{\beta}{24n(K+1)}, \quad \PP\{E_{\circledEight'}\} \geq 1 - \frac{\beta}{24n(K+1)}
    \end{gather*}
    where
    \begin{eqnarray}
        E_{\circledFour} &=& \left\{\text{either } \sum\limits_{l=0}^{T-1}\sum\limits_{i=1}^n\widetilde\sigma_{i,l}^2 > \frac{\exp\left(-\gamma\mu T\right)V^2}{216\ln\frac{48n(K+1)}{\beta}} \text{ or } |\circledFour| \leq \frac{\exp\left(-\frac{\gamma\mu T}{2}\right)V}{6}\right\},\notag\\
        E_{\circledEight} &=& \left\{\text{either } \sum\limits_{i=1}^n(\widehat\sigma_{i,T-1}')^2 > \frac{\exp\left(-\gamma\mu (T-1)\right)V^2}{38400\ln\frac{48n(K+1)}{\beta}} \text{ or } |\circledEight| > \frac{\exp\left(-\frac{\gamma\mu (T-1)}{2}\right)V}{80}\right\}, \notag\\
        E_{\circledEight'} &=& \left\{\text{either } \sum\limits_{i=1}^n(\widetilde\sigma_{i,T-1}')^2 > \frac{\exp\left(-\gamma\mu (T-1)\right)V^2}{38400\ln\frac{48n(K+1)}{\beta}} \text{ or } |\circledEight'| > \frac{\exp\left(-\frac{\gamma\mu (T-1)}{2}\right)V}{80}\right\}. \notag
    \end{eqnarray}
    Therefore, probability event $E_{T-1} \cap \widetilde{E}_{m-1} \cap E_{\circledFour} \cap E_{\circledEight} \cap E_{\circledEight'}$ implies
    \begin{gather*}
        \left\|\frac{\gamma}{n}\sum\limits_{i=1}^{m-1} \theta_{i,T-1}^u\right\| \leq \exp\left(-\frac{\gamma\mu(T-1)}{4}\right)\sqrt{V}\sqrt{\frac{1}{50} + \frac{1}{6} + \frac{1}{80}} \leq \frac{\exp\left(-\frac{\gamma\mu(T-1)}{4}\right)\sqrt{V}}{2},\\
        \left\|\frac{\gamma}{n}\sum\limits_{i=1}^{m-1} \omega_{i,T-1}^u\right\| \leq \exp\left(-\frac{\gamma\mu(T-1)}{4}\right)\sqrt{V}\sqrt{\frac{1}{50} + \frac{1}{6} + \frac{1}{80}} \leq \frac{\exp\left(-\frac{\gamma\mu(T-1)}{4}\right)\sqrt{V}}{2}.
    \end{gather*}
    This implies $\widetilde E_{T-1,m}$ and
    \begin{eqnarray*}
        \PP\{E_{T-1} \cap \widetilde{E}_{T-1,m}\} &\geq& \PP\{E_{T-1} \cap \widetilde{E}_{T-1,m-1} \cap E_{\circledFour} \cap E_{\circledEight} \cap E_{\circledEight'}\} \\
        &=& 1 - \PP\left\{\overline{E_{T-1} \cap \widetilde{E}_{T-1,m-1}} \cup \overline{E}_{\circledFour} \cup \overline{E}_{\circledEight} \cup \overline{E}_{\circledEight'}\right\}\\
        &\geq& 1 - \frac{(T-1)\beta}{K+1} - \frac{m\beta}{8n(K+1)}.
    \end{eqnarray*}
    Therefore, for all $m = 2,\ldots,n$ the statement holds and, in particular, $\PP\{E_{T-1} \cap \widetilde{E}_{T-1, n}\} \geq 1 - \frac{(T-1)\beta}{K+1} - \frac{\beta}{8(K+1)}$, i.e., \eqref{eq:induction_inequality_str_mon_SEG_1} and \eqref{eq:induction_inequality_str_mon_SEG_2} hold. Taking into account \eqref{eq:nvsdibhjbsjdbvshbcjs}, we conclude that $E_{T-1} \cap \widetilde{E}_{T-1, n} \cap E_{\circledOne} \cap E_{\circledFour} \cap E_{\circledSix'} \cap E_{\circledSeven'} \cap E_{\circledEight}$ implies
    \begin{equation*}
        V_T \leq 2\exp\left(-\frac{\gamma\mu}{2}T\right)V
    \end{equation*}
    that is equivalent to \eqref{eq:induction_inequality_str_mon_SEG} for $t = T$. Moreover,
    \begin{eqnarray*}
        \PP\left\{E_T\right\} &\geq& \PP\left\{E_{T-1} \cap \widetilde{E}_{T-1,n} \cap E_{\circledOne} \cap E_{\circledFour} \cap E_{\circledSix'} \cap E_{\circledSeven'} \right\} \\
        &=& 1 - \PP\left\{\overline{E_{T-1} \cap \widetilde{E}_{n}} \cup \overline{E}_{\circledOne} \cup \overline{E}_{\circledFour} \cup \overline{E}_{\circledSix'} \cup \overline{E}_{\circledSeven'}\right\}\\
        &=& 1 - \frac{(T-1)\beta}{K+1} - \frac{\beta}{8(K+1)} - 4\cdot \frac{\beta}{24n(K+1)} = 1 - \frac{T\beta}{K+1}.
    \end{eqnarray*}

    In other words, we showed that $\PP\{E_k\} \geq 1 - \nicefrac{k\beta}{(K+1)}$ for all $k = 0,1,\ldots,K+1$. For $k = K+1$ we have that with probability at least $1 - \beta$
    \begin{equation*}
       \|x^{K+1} - x^*\|^2 \leq V_{K+1} \leq 2\exp\left(-\frac{\gamma\mu(K+1)}{2}\right)V.
    \end{equation*}

    Finally, if 
    \begin{eqnarray*}
        \gamma &=& \min\left\{\frac{1}{72\cdot10^6 \mu \ln^2 \frac{48n(K+1)}{\beta}}, \frac{1}{6L}, \frac{\sqrt{n}}{15000 L \ln \frac{48n(K+1)}{\beta}}, \frac{2\ln(B_K)}{\mu(K+1)}\right\}, \notag\\
        B_K &=& \max\left\{2, \frac{n^{\frac{2(\alpha-1)}{\alpha}}(K+1)^{\frac{2(\alpha-1)}{\alpha}}\mu^2V}{3110400^{\frac{2}{\alpha}}\sigma^2\ln^{\frac{2(\alpha-1)}{\alpha}}\left(\frac{48n(K+1)}{\beta}\right)\ln^2(B_K)} \right\} \notag \\
        &=& \cO\left(\max\left\{2, \frac{n^{\frac{2(\alpha-1)}{\alpha}}K^{\frac{2(\alpha-1)}{\alpha}}\mu^2V}{\sigma^2\ln^{\frac{2(\alpha-1)}{\alpha}}\left(\frac{nK}{\beta}\right)\ln^2\left(\max\left\{2, \frac{n^{\frac{2(\alpha-1)}{\alpha}}K^{\frac{2(\alpha-1)}{\alpha}}\mu^2V}{\sigma^2\ln^{\frac{2(\alpha-1)}{\alpha}}\left(\frac{nK}{\beta}\right)} \right\}\right)} \right\}\right) \notag
    \end{eqnarray*}
    then with probability at least $1-\beta$
    \begin{eqnarray*}
        \|x^{K+1} - x^*\|^2 &\leq& 2\exp\left(-\frac{\gamma\mu(K+1)}{2}\right)V\\
        &=& 2V\max\Bigg\{\exp\left(-\frac{K+1}{144\cdot 10^6 \ln^2 \frac{48n(K+1)}{\beta}}\right), \exp\left(-\frac{\mu(K+1)}{12 L}\right),\\
        &&\quad\quad\quad\quad\quad\quad\quad\quad\quad\quad\quad\quad\quad\quad\quad \exp\left(-\frac{\mu\sqrt{n}(K+1)}{30000 L \ln \frac{48n(K+1)}{\beta}}\right), \frac{1}{B_K} \Bigg\}\\
        &=& \cO\Bigg(\max\Bigg\{V\exp\left(- \frac{K}{\ln^2 \tfrac{nK}{\beta}}\right), V\exp\left(- \frac{\mu  K}{L}\right),\\
        &&\quad\quad\quad\quad\quad\quad\quad\quad\quad V\exp\left(- \frac{\mu \sqrt{n} K}{L \ln \tfrac{nK}{\beta}}\right), \frac{\sigma^2\ln^{\frac{2(\alpha-1)}{\alpha}}\left(\frac{nK}{\beta}\right)\ln^2B_K}{n^{\frac{2(\alpha-1)}{\alpha}}K^{\frac{2(\alpha-1)}{\alpha}}\mu^2}\Bigg\}\Bigg).
    \end{eqnarray*}
    To get $\|x^{K+1} - x^*\|^2 \leq \varepsilon$ with probability $\geq 1 - \beta$, $K$ should be
    \begin{align}
        K = \cO\Bigg(\max\Bigg\{\Bigg(\frac{L}{\sqrt{n}\mu}&+\ln\left(\frac{nL}{\mu \beta}\ln\frac{V}{\varepsilon}\right)\Bigg)\ln\left(\frac{V}{\varepsilon}\right)\ln\Bigg(\frac{nL}{\mu \beta}\ln\frac{V}{\varepsilon}\Bigg),\notag\\
        &\frac{L}{\mu}\ln\left(\frac{V}{\varepsilon}\right), \frac{1}{n}\left(\frac{\sigma^2}{\mu^2\varepsilon}\right)^{\frac{\alpha}{2(\alpha-1)}}\ln \left(\frac{n}{\beta} \left(\frac{\sigma^2}{\mu^2\varepsilon}\right)^{\frac{\alpha}{2(\alpha-1)}}\right)\ln^{\frac{\alpha}{\alpha-1}}\left(B_\varepsilon\right)\Bigg\}\Bigg), \notag
    \end{align}
    where
    \begin{equation*}
        B_\varepsilon = \max\left\{2, \frac{V}{\varepsilon \ln \left(\frac{1}{\beta} \left(\frac{\sigma^2}{\mu^2\varepsilon}\right)^{\frac{\alpha}{2(\alpha-1)}}\right)}\right\}.
    \end{equation*}
\end{proof}

\newpage
\revision{\section{Numerical Experiments}
In this section we provide numerical experiments for the following simple problem:
\begin{equation}
    \label{eq:experiment_problem}
    \min_{x \in B_r(\hat{x})} f(x),
\end{equation}
where a radius  $r=1$, a central point $\hat{x} = (3,3,\dots,3)^{\top} \in \mathbb{R}^{10}$, and $f(x) = \frac{1}{2}\|x\|^2$, $f_{\xi}(x) = \frac{1}{2}\|x\|^2 +\langle \xi, x\rangle$, where $\xi$ comes from the symmetric Levy $\alpha$-stable distribution $\alpha = \frac{3}{2}$. We use the following parameters: $\gamma = 0.001$, $x^0 = \hat{x} + r \frac{e}{\|e\|}$, where $e = (1,1,\dots,1)^{\top}$. We tried three values of $\lambda$: $0.1$, $0.01$ and $0.001$. 

\begin{figure}[h]
    \centering
    \includegraphics[width=0.32\textwidth]{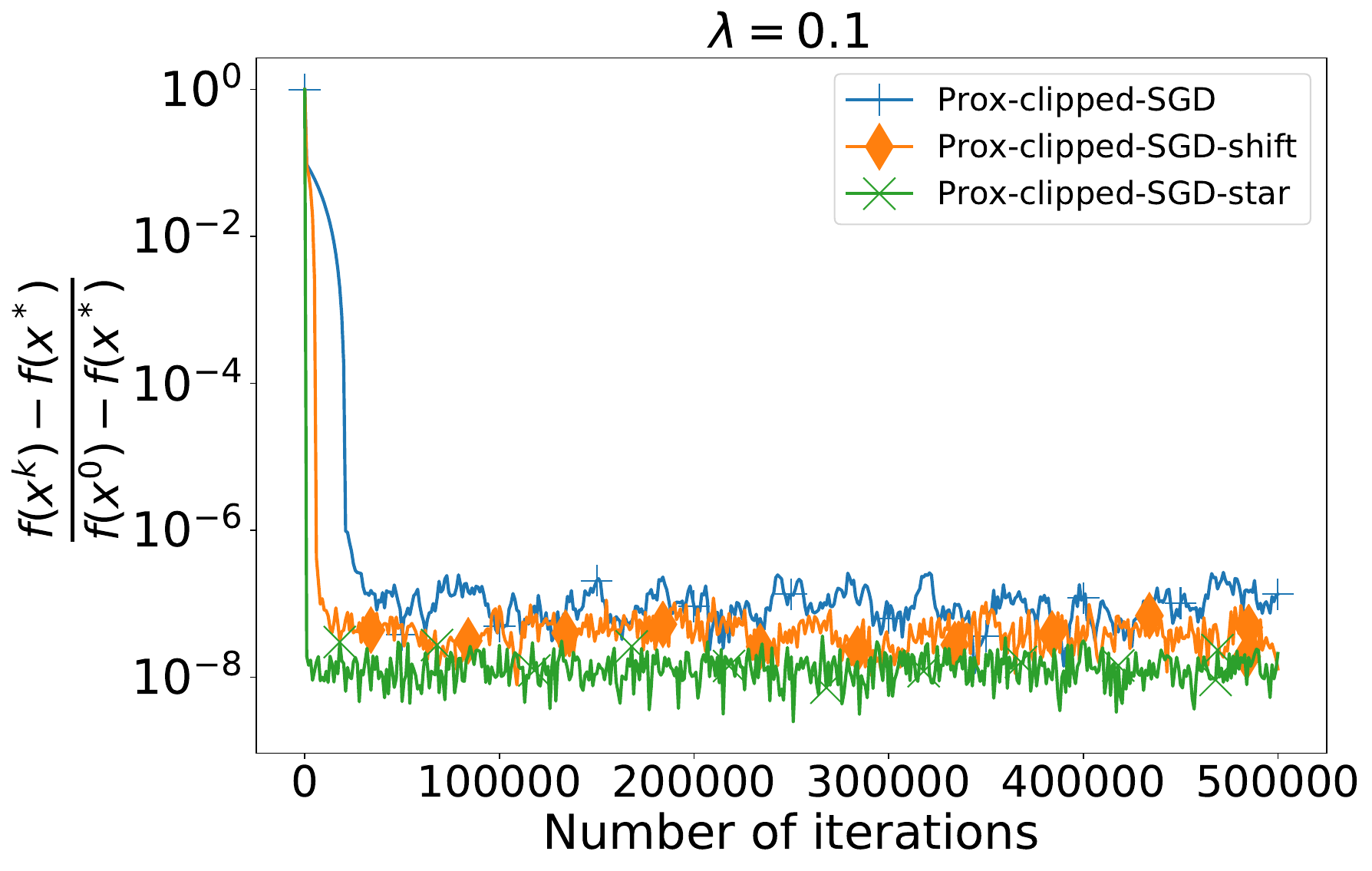}
    \includegraphics[width=0.32\textwidth]{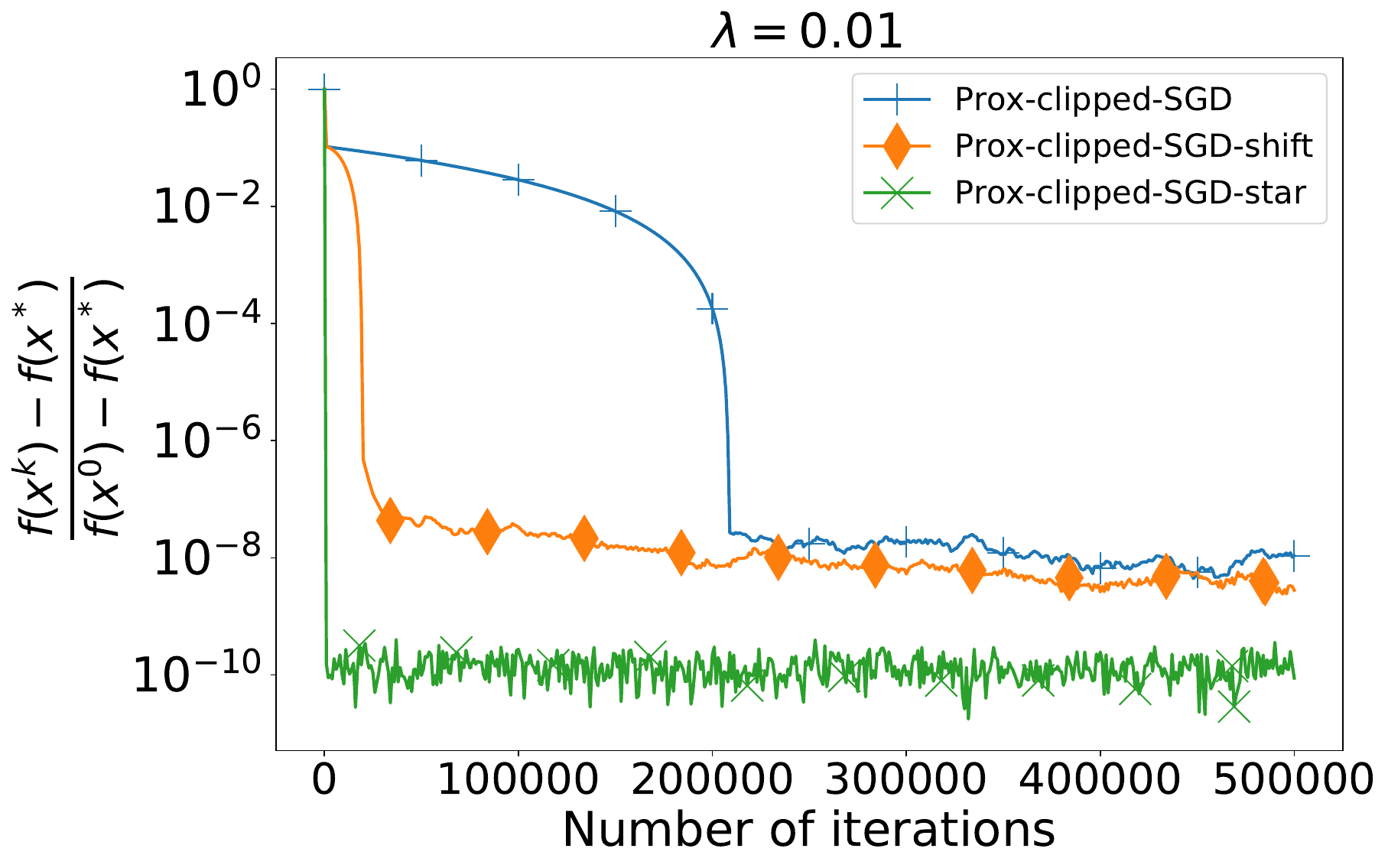}
    \includegraphics[width=0.32\textwidth]{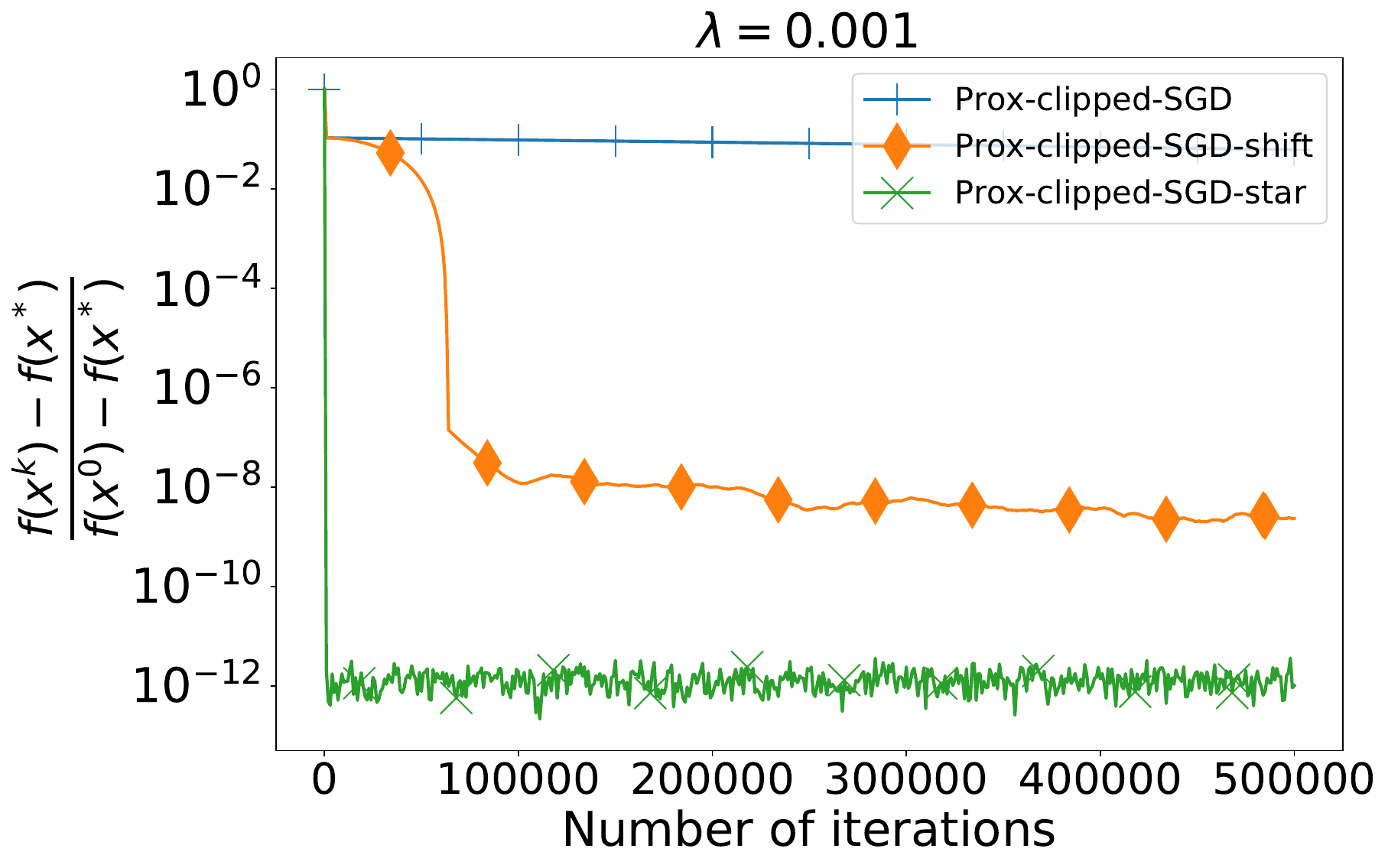}
    \caption{\revision{Comparison between performances of \algname{Prox-clipped-SGD}, \algname{Prox-clipped-SGD-star}, \algname{Prox-clipped-SGD-shift} in solving problem \eqref{eq:experiment_problem} with fixed clipping level for each of them $\lambda \in \{0.1, 0.01, 0.001\}$.}}
    \label{fig:exp_1}
\end{figure}



In our numerical experiments (see Figure~\ref{fig:exp_1}), we observe that the na\"ive \algname{Prox-clipped-SGD} converges slower than \algname{Prox-clipped-SGD-star} and \algname{Prox-clipped-SGD-shift}. Moreover, when the clipping level is small \algname{Prox-clipped-SGD} converges extremely slow, while \algname{Prox-clipped-SGD-shifts} takes some time to learn the shift and then converges to much better accuracy. We also see that the smaller clipping level is, the better accuracy \algname{Prox-clipped-SGD-star} achieves. For \algname{Prox-clipped-SGD-shift} we observe the same phenomenon when we reduce $\lambda$  from $0.1$ to $0.01$ . We expect the improvement in the accuracy even further if we decrease the stepsizes $\gamma$ and $\nu$ .}

\end{document}